\documentclass[a4paper]{article}

\usepackage[utf8]{inputenc}

\usepackage[top=1in,bottom=1in,left=1.25in,right=1.25in]{geometry}
\usepackage[font=small,skip=0pt]{caption}

\usepackage{stmaryrd}
\usepackage{multirow}
\usepackage{float}
\usepackage{diagbox}
\usepackage{tabularx}
\usepackage{multirow}
\usepackage{graphicx}
\usepackage{enumitem}

\usepackage[all]{xy}

\usepackage{xcolor,palatino}
\definecolor{ultramarine}{RGB}{0,32,96}
\colorlet{mygreen}{green!20!gray}
\colorlet{myultramarine}{ultramarine!20!gray}

\usepackage{amsmath}
\usepackage{amsthm}
\usepackage{amssymb}
\usepackage{mathrsfs} 
\usepackage{mathtools}
\mathtoolsset{showonlyrefs}

\usepackage[colorlinks=true,citecolor=mygreen, linkcolor=mygreen, urlcolor=myultramarine,breaklinks, naturalnames]{hyperref}

\usepackage[backrefs]{amsrefs}

\numberwithin{equation}{section}
\numberwithin{table}{section}

\allowdisplaybreaks[4]

\DeclareFontFamily{U}{BOONDOX-calo}{\skewchar\font=45 }
\DeclareFontShape{U}{BOONDOX-calo}{m}{n}{
   <-> s*[1.05] BOONDOX-r-calo}{}
\DeclareFontShape{U}{BOONDOX-calo}{b}{n}{
   <-> s*[1.05] BOONDOX-b-calo}{}
\DeclareMathAlphabet{\mathcalboondox}{U}{BOONDOX-calo}{m}{n}
\SetMathAlphabet{\mathcalboondox}{bold}{U}{BOONDOX-calo}{b}{n}
\DeclareMathAlphabet{\mathbcalboondox}{U}{BOONDOX-calo}{b}{n}

\newtheorem{thm}{Theorem}[section]

\newtheorem{rk}[thm]{Remark}
\newtheorem{lem}[thm]{Lemma}

\newtheorem{prop}[thm]{Proposition}
\newtheorem{cor}[thm]{Corollary}
\newtheorem{fact}[thm]{Fact}
\newtheorem*{prop*}{Proposition}

\newcommand\FK{{\operatorname{FK}}}
\newcommand\id{{\operatorname{id}}}
\newcommand\Img{{\operatorname{Im}}}
\newcommand\Ker{{\operatorname{Ker}}}

\newcommand\T{{\mathbb{T}}}
\newcommand\NN{{\mathbb{N}}}
\newcommand\RR{{\mathbb{R}}}
\newcommand\ZZ{{\mathbb{Z}}}

\newcommand{\suline}[1]{\smash{\underline{\vphantom{\mathstrut}#1}}}

\title{\texorpdfstring{Hochschild and cyclic (co)homology of the Fomin–Kirillov algebra on $3$ generators}{Hochschild and cyclic (co)homology of the Fomin–Kirillov algebra on 3 generators}}
\author{Estanislao Herscovich and Ziling Li}
\date{}

\begin{document}

\maketitle

\begin{abstract}
The goal of this article is to explicitly compute the Hochschild (co)homology of the Fomin-Kirillov algebra on three generators over a field of characteristic different from $2$ and $3$. 
We also obtain the cyclic (co)homology of the Fomin-Kirillov algebra in case the characteristic of the field is zero. 
Moreover, we compute the algebra structure of the Hochschild cohomology. 
\end{abstract}




\section{Introduction}

In their study of the Schubert calculus of flag manifolds, S. Fomin and A. Kirillov introduced a family of quadratic algebras over a field $\Bbbk$ of characteristic zero, now called the {\color{ultramarine}{\textbf{Fomin-Kirillov algebras}}} $\operatorname{FK}(n)$, indexed by the positive integers $n \in \NN$ (see \cites{FK99, Ki00, Ki16}). 
In case the index $n$ takes the values $3$, $4$ or $5$, the Fomin-Kirillov algebras are also Nichols algebras (see \cites{MS00,GV16}), which appear in the classification of finite dimensional pointed Hopf algebras (see \cite{AS10}). 
Their (co)homological properties have gained some importance, in particular in relation to the conjecture by P. Etingof and V. Ostrik that claims that the Yoneda algebra of every finite dimensional 
Hopf algebra is finitely generated. 
In particular, the Yoneda algebra of the Fomin-Kirillov algebra $\operatorname{FK}(3)$ was first computed by D. \c{S}tefan and C. Vay in \cite{SV16} 
using several calculations involving spectral sequences. 
The Yoneda algebra of $\operatorname{FK}(3)$ was more recently obtained in \cite{es} by more direct methods, namely by explicitly computing the minimal projective resolution of the trivial 
module $\Bbbk$ in the category of bounded below graded modules.
The aim of this article is to explicitly compute the Hochschild (co)homology of $\operatorname{FK}(3)$ over a field $\Bbbk$ 
of characteristic different from $2$ and $3$, and the cyclic (co)homology if the field $\Bbbk$ has characteristic zero. 

The contents of the article are as follows. 
After recalling some basic facts about the Fomin-Kirillov algebra $A = \operatorname{FK}(3)$ on three generators in Section \ref{section:generalities}, 
we explicitly construct in Section \ref{section:min-proj-res} the minimal projective resolution of the standard bimodule $A$ in the category of bounded below graded bimodules (see Proposition \ref{prb}), building upon the minimal projective resolution of the trivial module $\Bbbk$ in the category of bounded below graded modules of \cite{es}. 
Using this resolution we then compute in Section \ref{section:homology} explicit bases for the Hochschild homology groups of $A$. 
In particular, we prove the following result. 

\begin{prop*}[see Proposition \ref{prop dim HHA}]
	The dimension of $\operatorname{HH}_n(A)$ is given by 
	\[ \operatorname{dim}\operatorname{HH}_n(A) 
	= 
	\begin{cases}
		6,  &\text{if $n=0$}, \\
		\frac{5}{2}n+5, &\text{if $n=4r$ for $r\in\NN$},
		\\
		\frac{5n+13}{2}, &\text{if $n=4r+1$ for $r\in\NN_0$},
		\\
		\frac{5}{2}n+6, &\text{if $n=4r+2$ for $r\in\NN_0$},
		\\
		\frac{5n+9}{2}, &\text{if $n=4r+3$ for $r\in\NN_0$}.
	\end{cases} 
	\]
\end{prop*}

We also compute their full Hilbert series with respect to the internal degree of $A$ (see Corollary \ref{cor hilbert series homology}). 
Moreover, (the Hilbert series of) the cyclic homology is immediately obtained from the Hochschild homology by means of Goodwillie's theorem (see Corollary \ref{hseriescyclichomology}) in case the characteristic of the field is zero. 

Analogously, using again the minimal projective resolution of $A$ in the category of bounded below graded bimodules, 
we compute in Section \ref{Hochschild cohomology} explicit bases for the Hochschild cohomology groups of $A$. 
In particular, we prove the following result. 

\begin{prop*} [see Proposition \ref{proposition:cohomology}]
	The dimension of $\operatorname{HH}^n(A)$ is given by 
	\[ \operatorname{dim}\operatorname{HH}^n(A)=
\begin{cases}
		 \frac{5}{2}n+4, & \text{if $n=4r$ for $r\in\NN_0$}, 
		 \\
		 \frac{5}{2}n+5, &\text{if $n=4r+2$ for $r\in\NN_0$},\\
		\frac{5n+9}{2}, &\text{if $n=2r+1$ for $r\in\NN_0$}.
	\end{cases} \]
\end{prop*}

We also compute the full Hilbert series of the Hochschild cohomology with respect to the internal degree of $A$ (see Corollary \ref{corollary:cohomology-hilbert-series}). 
Finally, using techniques from Gröbner bases, we prove in Section \ref{section:algebra-cohomology} the main results of this article, Theorem \ref{thm:HHcohomologymain} and Corollary \ref{cor:HHcohomologymain}, 
which explicitly describe the algebra structure of the Hochschild cohomology of $A$ by generators and relations. 
Namely, $\operatorname{HH}^{\bullet}(A)$ is given as a quotient of a free graded-commutative algebra (for the homological degree) with $14$ homogeneous generators (see Proposition \ref{proposition:generators-cohomology}) 
modulo the homogeneous ideal generated by the $63$ relations listed in \eqref{eq:relations-cohomology-2}. 

\section{Generalities} 
\label{section:generalities}

In this section we will review the basic definitions of quadratic algebras, in particular applied to Fomin-Kirillov algebras.

We will denote by $\NN$ (resp., $\NN_{0}$) the set of positive (resp., nonnegative) integers. 
Given $i\in\ZZ$, we will denote by $\ZZ_{\leqslant i}$ the set $\{m\in\ZZ|m\leqslant i \}$.
Given $i,j\in \mathbb{Z}$ with $i\leqslant j$, we will denote by $\llbracket i,j \rrbracket =\{m\in \mathbb{Z}|i\leqslant m\leqslant j \}$ the integer interval, and we define $\chi_n=0$ if $n$ is an odd integer and $\chi_n=1$ if $n$ is an even integer. 
Moreover, given $r \in \RR$, we set $\lfloor r\rfloor =\sup \{n\in \ZZ|n\leqslant r \}$. 

From now on, $\Bbbk$ will be a field of characteristic different from $2$ and $3$. 
All maps between $\Bbbk$-vector spaces will be $\Bbbk$-linear.
Given an integer $N \geqslant 2$, recall that an algebra $A$ is said to be {\color{ultramarine}{\textbf{$N$-homogeneous}}} if it is of the form $\T V/(R)$, where $V$ is a $\Bbbk$-vector space, 
$\T V = \oplus_{n \in \NN_{0}} V^{\otimes n}$ is the tensor algebra on $V$ and $R \subseteq V^{\otimes N}$. 
An $N$-homogeneous algebra with $N = 2$ is called {\color{ultramarine}{\textbf{quadratic}}}. 
The grading of an $N$-homogeneous algebra $A = \oplus_{n \in \NN_{0}} A_{n}$ induced by setting $V$ to be concentrated in degree $1$ is called the {\color{ultramarine}{\textbf{Adams grading}}} or the {\color{ultramarine}{\textbf{internal grading}}} of $A$. 
As usual, we consider $A$ to be $\ZZ$-graded with $A_{n} = 0$ if $n \in \ZZ \setminus \NN_{0}$.

Let $V^*$ be the dual vector space of $V$ and define $\gamma : (V^{*})^{\otimes 2} \otimes V^{\otimes 2} \rightarrow \Bbbk$ by $\gamma (f_1 \otimes f_2 , v_1 \otimes v_2) = f_1(v_1) f_2(v_2)$, for all $v_1, v_2 \in V$ and $f_1, f_2 \in V^{*}$.  
Given a quadratic algebra $A = \T V/(R)$, let $A^!=\mathbb{T}(V^*)/(R^{\bot})=\oplus_{n\in\NN_0} A^!_{-n}$ be the {\color{ultramarine}{\textbf{quadratic dual}}} of $A$, where 
\begin{equation}
\label{eq:perpendicular-relations}
       R^{\bot}=\{ \alpha \in (V^*)^{\otimes 2}| \gamma(\alpha,r)=0 \textnormal{ for all } r \in R \}. 
\end{equation} 
Note that $A^!_{-n}$ is concentrated in Adams degree $-n$ for $n \in \NN_{0}$, and we consider $A^!$ to be $\ZZ$-graded with $A^!_{n} = 0$ for 
$n \in \NN$. 

Recall that, given a $\Bbbk$-algebra $A$, denote by $A^{op}$ the {\color{ultramarine}{\textbf{opposite algebra}}} of $A$, which is the $\Bbbk$-module $A$ with multiplication $x \cdot_{A^{op}} y=yx$ for $x,y\in A$. 
The algebra $A$ itself is an $A$-bimodule under left and right multiplications. 
We also recall that an $A$-bimodule $M$ can also be viewed as a left $A^e$-module, where $A^e=A\otimes A^{op}$ and the action is $(x\otimes y) m=xmy$ for $x,y\in A$ and $m\in M$. 
Analogously, 
it can also be regarded as a right $A^e$-module, where the action is $m (x\otimes y)=ymx$ for $x,y\in A$ and $m\in M$.

We recall that the 
{\color{ultramarine}{\textbf{Fomin–Kirillov algebra}}} 
on $3$ generators is the $\Bbbk$-algebra $\FK(3)$ generated by the $\Bbbk$-vector space $V$ spanned by three elements $a,b,c$, modulo the ideal generated by the vector space $R \subseteq V^{\otimes 2}$ spanned by 
\begin{equation}
    \{ a^2, b^2, c^2, ab+bc+ca, ba+ac+cb \}. 
\end{equation}        
As usual, we will omit the tensor symbol $\otimes$ when denoting the product of the elements of $\T V $. 
We refer the reader to \cites{FK99,MS00} for more information on Fomin-Kirillov algebras.
For simplicity, from now on we will denote the Fomin-Kirillov algebra $\FK(3)$ simply by $A$. 
Note that the algebra $A$ is quadratic and that $A=\oplus_{m\in \llbracket 0,4 \rrbracket} A_m$. 
It is easy to see that 
\begin{equation}
\label{eq:basis-FK}
\{1,a,b,c,ab,bc,ba,ac,aba,abc,bac,abac  \}   
\end{equation}     
is a basis of $A$ (see \cite{FK99}). 
Given, $m\in \llbracket 0,4 \rrbracket$, we will denote by $\mathcalboondox{B}_m$ the subset of \eqref{eq:basis-FK} that is a basis of $A_{m}$. 

We will denote by $\{A,B,C\}$ the basis of $V^*$ dual to the basis $\{a,b,c\}$ of $V$. 
Then, it is easy to see that 
\[     A^!=\Bbbk\langle A,B,C\rangle/ (BA-AC, CA-AB, AB-BC, CB-BA).     \]
Note that $A^!_0=\Bbbk$ and $A^!_{-1}\cong V^*$. 
We recall that $\mathcalboondox{B}^!_n=\{A^n, B^n, C^n, A^{n-1}B, A^{n-1}C, A^{n-2}B^2 \}$ is a basis of $A^!_{-n}$ for all integers $n \geqslant 2$, where we follow the convention that $A^{0}B^{2} = B^{2}$ (see \cite{SV16}, Lemma 4.4). 

Let $(A_{-n}^!)^*$ be the dual space of $A_{-n}^!$ and $\mathcalboondox{B}^{!*}_n = \{\alpha_n, \beta_n,\gamma_n,\alpha_{n-1}\beta, \alpha_{n-1}\gamma, \alpha_{n-2}\beta_2\} \setminus \{ \mathbf{0} \}$ the dual basis of $\mathcalboondox{B}^!_n$ for $n\in \NN$, where we will follow the convention that if the index of some letter in an element of the previous sets is less than or equal to zero, this element is zero $\mathbf{0}$.
We will omit the index $1$ for the elements of the previous bases and write $\epsilon^!$ for the basis of $(A_{0}^!)^*$. 
The previous bases for the homogeneous components of $A$, $A^{!}$ or $(A^!)^{\#}  =\oplus_{n\in\NN_0}(A_{-n}^!)^*$ will be called {\color{ultramarine}{\textbf{usual}}}. 

Recall that  $(A^!)^{\#}$ is a graded bimodule over $A^!$ via $(ufv)(w)=f(vwu)$ for $u,v,w \in A^!$ and $f\in (A^!)^{\#}$. 
Using this definition of action of $A^!$ together with \cite{es}, Fact 3.6, and the clear identities 
of $A^{!}$ given by 
\begin{equation}
\label{eq:identites-a-shriek}
X Y^{n} = A^{n} X \hskip 0.5cm \text{ and } \hskip 0.5cm X A^{n} = \begin{cases}
A^{n} X, &\text{if $n$ is even,}
\\
A^{n} Y, &\text{if $n$ is odd,}
\end{cases}
\end{equation}
for $n \in \NN$ and $\{ X, Y \} = \{ B , C \}$, we immediately get that 
\begin{equation}
	\label{eq:action_1}
	\begin{split}
	 A\alpha & =B\beta =C\gamma = \alpha A=\beta B=\gamma C = \epsilon^!, 
	\\
	 A\beta & =A\gamma =B\alpha =B\gamma =C\alpha =C\beta = \beta A=\gamma A=\alpha B=\gamma B=\alpha C=\beta C = 0.
	\end{split}
\end{equation} 
Moreover, for $n\geqslant 2$, we have 
\begin{small}
\begin{equation}
	\label{eq:action_n>1}
	\begin{aligned}
	A\alpha_n & =\alpha_n A=\alpha_{n-1},
	\quad &
	A\beta_n & =\beta_n A=A\gamma_n=\gamma_n A=0,
	\\
	A\alpha_{n-1}\beta & =\chi_n\gamma_{n-1}+\alpha_{n-2}\gamma,
	\quad &
	\alpha_{n-1}\beta A & =\chi_n\beta_{n-1}+\alpha_{n-2}\beta,
	\\
	A\alpha_{n-1}\gamma & =\chi_n\beta_{n-1}+\alpha_{n-2}\beta,
	\quad &
	\alpha_{n-1}\gamma A & =\chi_n\gamma_{n-1}+\alpha_{n-2}\gamma,
	\\
	A\alpha_{n-2}\beta_2 & =\chi_{n+1}(\beta_{n-1}+\gamma_{n-1})+\alpha_{n-3}\beta_2,
	\quad &
	\alpha_{n-2}\beta_2 A & =\chi_{n+1}(\beta_{n-1}+\gamma_{n-1})+\alpha_{n-3}\beta_2,
	\\
	B\beta_n & =\beta_n B=\beta_{n-1},
	\quad &
	B\alpha_{n} & =\alpha_{n} B=B\gamma_{n} =\gamma_{n} B=0,
	\\
	 B\alpha_{n-1}\beta & =\alpha_{n-1}+\chi_{n+1}\gamma_{n-1}+\alpha_{n-3}\beta_2,
	\quad &
	\alpha_{n-1}\beta B & =\gamma_{n-1}+\chi_n\alpha_{n-2}\gamma+\chi_{n+1}(\alpha_{n-1}+\alpha_{n-3}\beta_2),
	\\
	B\alpha_{n-1}\gamma & =\chi_{n}\gamma_{n-1}+\alpha_{n-2}\gamma,
	\quad &
	\alpha_{n-1}\gamma B & =\chi_{n+1}\alpha_{n-2}\beta+\chi_n(\alpha_{n-1}+\alpha_{n-3}\beta_2),
	\\
	B\alpha_{n-2}\beta_2 & =\alpha_{n-2}\beta,
	\quad &
	\alpha_{n-2}\beta_2 B & =\chi_n\alpha_{n-2}\beta+\chi_{n+1}\alpha_{n-2}\gamma,
	\\
	C\gamma_n & =\gamma_n C=\gamma_{n-1},
	\quad &
	C\alpha_n & =\alpha_n C=C\beta_n=\beta_n C=0,
	\\
	C\alpha_{n-1}\beta & =\chi_{n}\beta_{n-1}+\alpha_{n-2}\beta,
	\quad &
	\alpha_{n-1}\beta C & 
	=\chi_{n+1}\alpha_{n-2}\gamma+\chi_n(\alpha_{n-1}+\alpha_{n-3}\beta_2),
	\\
	C\alpha_{n-1}\gamma & =\alpha_{n-1}+\chi_{n+1}\beta_{n-1}+\alpha_{n-3}\beta_2,
	\quad &
	\alpha_{n-1}\gamma C & =\beta_{n-1}+\chi_n\alpha_{n-2}\beta+\chi_{n+1}(\alpha_{n-1}+\alpha_{n-3}\beta_2),
	\\
	C\alpha_{n-2}\beta_2 & =\alpha_{n-2}\gamma,
	\quad &
	\alpha_{n-2}\beta_2 C & =\chi_{n+1}\alpha_{n-2}\beta+\chi_n \alpha_{n-2}\gamma.
\end{aligned}
\end{equation} 
\end{small}

Finally, the following elementary result, whose proof is immediate, will be useful in the sequel to establish the linear independence of several sets of (co)boundaries and (co)cycles. 
\begin{fact} 
\label{indep}
Let $V$ be a $\Bbbk$-vector space of dimension $n\in\NN$ and $\{v_1,\cdots, v_n\}$ a basis of $V$. Let $r\leqslant n$ be a positive integer and $U=\{\sum_{j=1}^n c^k_jv_j\in V|c^k_j\in \Bbbk, k\in \llbracket 1, r\rrbracket \}$ a set of $r$ elements. If there is an injective map $\varphi: \llbracket 1,r \rrbracket \to \llbracket 1, n \rrbracket $ such that for all $k\in \llbracket 1,r\rrbracket$, $c^k_{\varphi(k)}\neq 0$, but $c^i_{\varphi(k)}=0$ for $i\in\llbracket 1, k-1 \rrbracket$, then the elements in $U$ are linearly independent.
\end{fact}

Instead of writing the specific map $\varphi$, in the cases of (ordered) sets $U$ we will consider in the sequel we will simply underline the corresponding term $c^k_{\varphi(k)} v_{\varphi(k)}$. 
This will be the case in particular in Subsections \ref{subsection: boundaries} - \ref{subsection:homology} and \ref{subsection: com coboun} - \ref{subsection: cohomology}. 
In that situation, the basis $\{v_1,\cdots, v_n\}$ of the larger vector space will be an usual basis and the condition on $\varphi$ is tantamount to the fact that the underlined term of an element does not appear (with nonzero coefficient) in the expressions of the previous elements of the same set.

\section{\texorpdfstring{The projective bimodule resolution of the Fomin-Kirillov algebra on $3$ generators}{The projective bimodule resolution of the Fomin-Kirillov algebra on 3 generators}}
\label{section:min-proj-res} 

In this section we will explicitly construct the minimal projective resolution of the standard bimodule $A$ in the category of bounded below graded $A$-bimodules. 

\subsection{The bimodule Koszul complex}

In the article \cite{berger} R. Berger and N. Marconnet introduced the 
{\color{ultramarine}{\textbf{bimodule Koszul complex}}}
for any $N$-homogeneous algebra.
We will recall this for the special case of the Fomin–Kirillov algebra (so $N=2$). 
Given $n \in \NN_0$, let $K_n^b= A \otimes (A_{-n}^!)^*\otimes A$ be the bimodule over $A$ for the outer action. 
Define the maps 
 \[     i_{\mathcalboondox{l}}, i_{\mathcalboondox{r}} : A \otimes (A^!)^{\#} \otimes A \rightarrow A \otimes (A^!)^{\#} \otimes A     \]
given by $i_{\mathcalboondox{l}} (x\otimes u \otimes y)=xa\otimes uA \otimes y+xb\otimes uB\otimes y+xc\otimes uC \otimes y$ 
and $i_{\mathcalboondox{r}}(x\otimes u\otimes y)= x\otimes Au\otimes ay+x\otimes Bu\otimes by+x\otimes Cu\otimes cy$ for $x,y\in A$ and $u\in (A^!)^{\#}$.
Note that $i_{\mathcalboondox{l}}^2=0$, $i_{\mathcalboondox{r}}^2=0$ and $i_{\mathcalboondox{l}}i_{\mathcalboondox{r}}=i_{\mathcalboondox{r}}i_{\mathcalboondox{l}}$. 
Indeed, the first identity follows from the fact that 
\[     (a\otimes A+b\otimes B+c\otimes C)^2=ba\otimes BA+ca\otimes CA+ab\otimes AB+cb\otimes CB+ac\otimes AC+ bc\otimes BC \]     
is trivially zero by applying the relations in $A$ and $A^!$ and the fact that $i_{\mathcalboondox{l}}^2(x\otimes u\otimes y)=(x\otimes u\otimes y)(a\otimes A \otimes 1 +b\otimes B \otimes 1+c\otimes C \otimes 1)^2$. 
The identity $i_{\mathcalboondox{r}}^2=0$ is proved in the same way.
Since the left and right actions of $A^!$ on $(A^!)^{\#}$ are compatible, the maps $i_{\mathcalboondox{l}}$ and $i_{\mathcalboondox{r}}$ commute.

\begin{fact}
   Take $x, y \in A$. 
   To reduce space, we will typically use vertical bars instead of the tensor product symbols $\otimes$. 
   The map $i_{\mathcalboondox{l}}|_{A\otimes (A_{-1}^!)^*\otimes A}: A\otimes (A_{-1}^!)^*\otimes A \to A\otimes (A_{0}^!)^*\otimes A$ sends $x|\alpha|y$ to $xa|\epsilon^!|y$, $x|\beta|y$ to $xb|\epsilon^!|y$, and $x|\gamma|y$ to $xc|\epsilon^!|y$. For $n\geqslant 2$,  $i_{\mathcalboondox{l}}|_{A\otimes (A_{-n}^!)^*\otimes A}:A\otimes (A_{-n}^!)^*\otimes A\to A\otimes (A_{-(n-1)}^!)^*\otimes A $ is given by  
   \begin{align*}
   x|\alpha_n|y &\mapsto xa|\alpha_{n-1}|y,\quad
   x|\beta_n |y \mapsto xb|\beta_{n-1}|y,\quad 
   x|\gamma_n |y \mapsto xc|\gamma_{n-1}|y,
   \\
   x|\alpha_{n-1}\beta|y &\mapsto xa| (\chi_n\beta_{n-1}+\alpha_{n-2}\beta)|y+xb| (\gamma_{n-1}+\chi_n\alpha_{n-2}\gamma+\chi_{n+1}(\alpha_{n-1}+\alpha_{n-3}\beta_2))|y
   \\
   &\phantom{ \mapsto \;} 
   +xc| (\chi_{n+1}\alpha_{n-2}\gamma+\chi_n(\alpha_{n-1}+\alpha_{n-3}\beta_2))|y,
   \\
   x|\alpha_{n-1}\gamma|y &\mapsto xa| (\chi_n\gamma_{n-1}+\alpha_{n-2}\gamma)|y+xb| (\chi_{n+1}\alpha_{n-2}\beta+\chi_n(\alpha_{n-1}+\alpha_{n-3}\beta_2))|y
   \\
   & \phantom{ \mapsto \;}
   +xc| (\beta_{n-1}+\chi_n\alpha_{n-2}\beta+\chi_{n+1}(\alpha_{n-1}+\alpha_{n-3}\beta_2))|y,
   \\
   x|\alpha_{n-2}\beta_2|y&\mapsto xa|(\chi_{n+1}(\beta_{n-1}+\gamma_{n-1})+\alpha_{n-3}\beta_2)|y+xb|(\chi_n\alpha_{n-2}\beta+\chi_{n+1}\alpha_{n-2}\gamma)|y
   \\
   & \phantom{ \mapsto \;}
   +xc|(\chi_{n+1}\alpha_{n-2}\beta+\chi_n \alpha_{n-2}\gamma)|y.
   \end{align*}

   The map $i_{\mathcalboondox{r}}|_{A\otimes (A_{-1}^!)^*\otimes A}: A\otimes (A_{-1}^!)^*\otimes A\to A \otimes (A_{0}^!)^*\otimes A$ sends $x|\alpha|y$ to $x|\epsilon^!|ay$, $x|\beta|y$ to $x|\epsilon^!|by$, and $x|\gamma|y$ to $x|\epsilon^!|cy$. 
   For $n\geqslant 2$,  $i_{\mathcalboondox{r}}|_{A\otimes (A_{-n}^!)^*\otimes A}:A\otimes (A_{-n}^!)^*\otimes A\to A\otimes (A_{-(n-1)}^!)^*\otimes A$ is given by 
   \begin{align*}
	x|\alpha_n|y &\mapsto x|\alpha_{n-1}|ay,\quad
	x|\beta_n|y\mapsto x|\beta_{n-1}|by,\quad 
	x|\gamma_n|y\mapsto x|\gamma_{n-1}|cy,
	\\
	x|\alpha_{n-1}\beta|y &\mapsto x|(\chi_n \gamma_{n-1}+\alpha_{n-2}\gamma)|ay+x|(\alpha_{n-1}+\chi_{n+1}\gamma_{n-1}+\alpha_{n-3}\beta_2)|by\\
	& \phantom{ \mapsto \;}
	+x|(\chi_n\beta_{n-1}+\alpha_{n-2}\beta)|cy,
	\\
	x|\alpha_{n-1}\gamma|y&\mapsto x|(\chi_n\beta_{n-1}+\alpha_{n-2}\beta)|ay+x|(\chi_n\gamma_{n-1}+\alpha_{n-2}\gamma)|by
	\\
	& \phantom{ \mapsto \;}
	+x|(\alpha_{n-1}+\chi_{n+1}\beta_{n-1}+\alpha_{n-3}\beta_2)|cy,
	\\
	x|\alpha_{n-2}\beta_2|y&\mapsto x|(\chi_{n+1}(\beta_{n-1}+\gamma_{n-1})+\alpha_{n-3}\beta_2)|ay+x|\alpha_{n-2}\beta|by+x|\alpha_{n-2}\gamma|cy.
	\end{align*}
\end{fact}

Following \cite{berger}, we now set $d_{n}^b:K_n^b\to K_{n-1}^b$ by $d_{n}^b=(-1)^n i_{\mathcalboondox{l}}+i_{\mathcalboondox{r}}$ for $n\in \NN $. It is easy to see that $d^b_{n}d^b_{n+1}=-i_{\mathcalboondox{l}}^2+i_{\mathcalboondox{r}}^2=0$ for $n\in \NN$. Then $(K^b_{\bullet}, d^b_{\bullet})$ is a complex in the category of bounded below graded $A$-bimodules, called the {\color{ultramarine}{\textbf{bimodule Koszul complex}}}
over $A$. It is clear that $\Bbbk\otimes_A (K^b_{\bullet}, d^b_{\bullet})\cong (K_{\bullet},d_{\bullet})$, where $(K_{\bullet},d_{\bullet})$, considered in \cite{es}, is the Koszul complex of the trivial right $A$-module $\Bbbk$ in the category of graded right $A$-modules.

\begin{rk}
The bimodule Koszul complex $(K^b_{\bullet}, d^b_{\bullet})$ is minimal, since the complex $\Bbbk\otimes_{A^e}(K^b_{\bullet}, d^b_{\bullet})\cong \Bbbk\otimes_A (K^b_{\bullet}, d^b_{\bullet})\otimes_A \Bbbk\cong (K_{\bullet}, d_{\bullet})\otimes_A \Bbbk$ has zero differentials.
\end{rk}

 We recall the following result.
\begin{prop}[\cite{berger}, Proposition 4.1]
	\label{berger}
Let $B$ be a nonnegatively graded connected $\Bbbk$-algebra, and let
\[  M_1\xrightarrow{f} M_2\xrightarrow{g} M_3  \]
be a sequence of graded-free $B$-modules, with $M_1$ bounded below and $gf=0$. Then this sequence is exact if 
\[ \Bbbk\otimes_B M_1\xrightarrow{\id_{\Bbbk}\otimes_B f} \Bbbk\otimes_B M_2\xrightarrow{\id_{\Bbbk}\otimes_B g} \Bbbk\otimes_B M_3 \]
is exact.
\end{prop}

\begin{cor}
	 We have $\operatorname{H}_n(K^b_{\bullet}, d^b_{\bullet})=0$ for $n$ different from $0$ and $3$.
\end{cor}

\begin{proof}
Recall that $\operatorname{H}_{n}(K_{\bullet},d_{\bullet})=0$ for $n\neq 0,3$, by \cite{es}, Proposition 3.1. Applying Proposition \ref{berger}, we get the result.
\end{proof}

\subsection{The minimal projective bimodule resolution}
In this subsection, we will explicitly describe the minimal projective resolution of $A$ in the category of bounded below graded bimodules. We recall the following result (see Proposition 3.3 in \cite{es}).
\begin{prop}\label{es}
Let $(K_{\bullet},d_{\bullet})$ be the Koszul complex of the trivial right $A$-module $\Bbbk$ in the category of graded right $A$-modules. The minimal projective resolution $(P_{\bullet},\delta_{\bullet})$ of $\Bbbk$ in the category of bounded below graded $A$-modules is given as follows. For $n\in\NN_0$, set 
\[ P_n=\underset{i\in \llbracket 0, \lfloor n/4\rfloor \rrbracket}{\bigoplus}\omega_i K_{n-4i},\]
where $\omega_i$ is a symbol of internal degree $6i$ for all $i\in \NN_0$, and the differential $\delta_{n}:P_n\to P_{n-1}$ for $n\in\NN$ is given by 
\[ \delta_{n}\bigg(\underset{i\in \llbracket 0, \lfloor n/4\rfloor \rrbracket}{\sum}\omega_i\rho_{n-4i}\bigg)=\underset{i\in \llbracket 0, \lfloor n/4\rfloor \rrbracket}{\sum}\big(\omega_i d_{n-4i}(\rho_{n-4i})+\omega_{i-1}f_{n-4i}(\rho_{n-4i})\big),\]
where $\rho_j\in K_j$ for $j\in\NN_0$, $\omega_{-1}=0$ and $f_j:K_j\to K_{j+3}$ are morphisms of graded right $A$-modules of internal degree $6$ such that $d_{j+4}f_{j+1}=-f_jd_{j+1}$ for $j\in\NN_0$, $d_3f_0=0$ and $\Img(f_0)\nsubseteq \Img(d_4)$. This gives a minimal projective resolution of the trivial right $A$-module $\Bbbk$ by means of the augmentation $\epsilon:P_0=K_0\to\Bbbk$ of the Koszul complex. We usually omit $\omega_0$ for simplicity.
\end{prop}

We further provide an explicit family of morphisms $\{f_{\bullet}\}_{\bullet\in\NN_0}$ satisfying the above conditions, since we will need it for the calculations. Indeed, a lengthy but straightforward computation shows that 
    \begin{equation}
    \label{fn}
		\begin{split}
	f_0(\epsilon^{!} |1)&=2\alpha_3|bac+2\beta_3|abc-2\gamma_3|aba-\alpha_2\beta|abc+\alpha_2\gamma|aba-\alpha\beta_2|bac,\\
	f_n(\alpha_n|1)&=(2\alpha_{n+3}-\alpha_{n+1}\beta_2)|bac+\chi_n\beta_{n+3}|abc-\chi_n\gamma_{n+3}|aba,\\
	f_n(\beta_n|1)&=(2\beta_{n+3}-\chi_n\alpha_{n+2}\beta-\chi_{n+1}\alpha_{n+1}\beta_2)|abc+\chi_n\alpha_{n+3}|bac-\chi_n\gamma_{n+3}|aba,\\
	f_n(\gamma_n|1)&=(-2\gamma_{n+3}+\chi_n\alpha_{n+2}\gamma+\chi_{n+1}\alpha_{n+1}\beta_2)|aba+\chi_n\alpha_{n+3}|bac+\chi_n\beta_{n+3}|abc,\\
	f_n(\alpha_{n-1}\beta|1)&=(n-1)\chi_{n+1}\beta_{n+3}|abc,\\	
	f_n(\alpha_{n-1}\gamma|1)&=-(n-1)\chi_{n+1}\gamma_{n+3}|aba,\\	
	f_n(\alpha_{n-2}\beta_2|1)&=((n-2)+\chi_{n+1})\alpha_{n+3}|bac+(n-2)\chi_n\beta_{n+3}|abc-(n-2)\chi_n\gamma_{n+3}|aba,
	\end{split}
    \end{equation}
for $n\in\NN$, satisfy the conditions of Proposition \ref{es}. Note that $f_0$ already appeared in \cite{es}.

Given $n\in\NN_0$, we now define the morphisms of $A$-bimodules $f_n^b:K_n^b\to K_{n+3}^b$ by 
\begin{small}
	\allowdisplaybreaks
\begin{align*}
f_0^b(1|\epsilon^!|1)&=	2|\alpha_3|bac+2|\beta_3|abc-2|\gamma_3|aba-1|\alpha_2\beta|abc+1|\alpha_2\gamma|aba-1|\alpha\beta_2|bac\\
	&\quad  +a|\alpha_2\beta|(ba+ac)-c|\alpha_2\beta|ab-a|\alpha_2\gamma|bc-b|\alpha_2\gamma|ac+b|\alpha\beta_2|(ab+bc)-c|\alpha\beta_2|ba\\
	&\quad -2b|\alpha_3|(ab+bc)+2c|\alpha_3|ba-2a|\beta_3|(ba+ac)+2c|\beta_3|ab+2a|\gamma_3|bc+2b|\gamma_3|ac\\
	&\quad -bc|\alpha_2\beta|a-ba|\alpha_2\beta|c+(ab+bc)|\alpha_2\gamma|b+(ba+ac)|\alpha_2\gamma|a-ab|\alpha\beta_2|c-ac|\alpha\beta_2|b\\
	&\quad +2ab|\alpha_3|c+2ac|\alpha_3|b+2bc|\beta_3|a+2ba|\beta_3|c-2(ab+bc)|\gamma_3|b-2(ba+ac)|\gamma_3|a\\
	&\quad +2bac|\alpha_3|1+2abc|\beta_3|1-2aba|\gamma_3|1-abc|\alpha_2\beta|1+aba|\alpha_2\gamma|1-bac|\alpha\beta_2|1,\\
f_n^b(1|\alpha_n|1)&=2|\alpha_{n+3}|bac+\chi_n|\beta_{n+3}|abc-\chi_n|\gamma_{n+3}|aba-1|\alpha_{n+1}\beta_2|bac -\chi_n c|\alpha_{n+2}\beta|ab\\
	&\quad -\chi_n b|\alpha_{n+2}\gamma|ac+\chi_{n+1} b|\alpha_{n+1}\beta_2|ac+\chi_{n+1}c|\alpha_{n+1}\beta_2|ab -\chi_n b|\alpha_{n+3}|(ab+bc)\\
	&\quad + \chi_n c|\alpha_{n+3}|ba+(-1)^n 2c|\beta_{n+3}|ab-\chi_n a|\beta_{n+3}|ac-\chi_n a|\gamma_{n+3}|ab +(-1)^n 2b|\gamma_{n+3}|ac\\
	&\quad -\chi_n ba|\alpha_{n+2}\beta|c+\chi_n (ab+bc)|\alpha_{n+2}\gamma|b+\chi_{n+1} (ab+bc)|\alpha_{n+1}\beta_2|b\\
	&\quad -\chi_{n+1} ba|\alpha_{n+1}\beta_2|c +\chi_n ac|\alpha_{n+3}|b+\chi_n ab|\alpha_{n+3}|c+2ba|\beta_{n+3}|c+\chi_n (ab+bc)|\beta_{n+3}|a\\
	&\quad -\chi_n ba|\gamma_{n+3}|a-2(ab+bc)|\gamma_{n+3}|b+(-1)^n 2bac|\alpha_{n+3}|1+\chi_n abc|\beta_{n+3}|1\\
	&\quad -\chi_n aba|\gamma_{n+3}|1 +(-1)^{n+1}bac|\alpha_{n+1}\beta_2|1,	\\
f_n^b(1|\beta_n|1)&=2|\beta_{n+3}|abc-\chi_n |\gamma_{n+3}|aba+\chi_n |\alpha_{n+3}|bac-\chi_n|\alpha_{n+2}\beta|abc-\chi_{n+1}|\alpha_{n+1}\beta_2|abc\\
	&\quad -\chi_n a|\alpha_{n+2}\gamma|bc+(-1)^{n+1} c|\alpha_{n+1}\beta_2|ba+\chi_{n+1}a|\alpha_{n+1}\beta_2|bc+\chi_n c|\beta_{n+3}|ab\\
	&\quad  -\chi_n a|\beta_{n+3}|(ba+ac)
	 +(-1)^n 2a|\gamma_{n+3}|bc-\chi_n b|\gamma_{n+3}|ba-\chi_n b|\alpha_{n+3}|bc\\
	 &\quad +(-1)^n2c|\alpha_{n+3}|ba+\chi_n (ba+ac)|\alpha_{n+2}\gamma|a- ab|\alpha_{n+1}\beta_2|c +\chi_{n+1}(ba+ac)|\alpha_{n+1}\beta_2|a\\
	 &\quad +\chi_n ba|\beta_{n+3}|c+\chi_n bc|\beta_{n+3}|a-2(ba+ac)|\gamma_{n+3}|a-\chi_n ab|\gamma_{n+3}|b\\
	 &\quad +\chi_n (ba+ac)|\alpha_{n+3}|b+2ab|\alpha_{n+3}|c+(-1)^n2abc|\beta_{n+3}|1-\chi_n aba|\gamma_{n+3}|1\\
	 &\quad +\chi_n bac|\alpha_{n+3}|1 -\chi_n abc|\alpha_{n+2}\beta|1+\chi_{n+1}abc|\alpha_{n+1}\beta_2|1,\\
f_n^b(1|\gamma_n|1)&=-2|\gamma_{n+3}|aba+\chi_n|\alpha_{n+3}|bac+\chi_n|\beta_{n+3}|abc+\chi_n|\alpha_{n+2}\gamma|aba+\chi_{n+1}|\alpha_{n+1}\beta_2|aba\\
	 &\quad  +\chi_n a|\alpha_{n+2}\beta|(ba+ac)-\chi_{n+1} a|\alpha_{n+1}\beta_2|(ba+ac)+(-1)^n b|\alpha_{n+1}\beta_2|(ab+bc)\\
	 &\quad +\chi_n a|\gamma_{n+3}|bc+\chi_n b|\gamma_{n+3}|ac+(-1)^{n+1}2b|\alpha_{n+3}|(ab+bc)+\chi_n c|\alpha_{n+3}|(ba+ac)\\
	 &\quad +\chi_n c|\beta_{n+3}|(ab+bc)+(-1)^{n+1}2a|\beta_{n+3}|(ba+ac)-\chi_n bc|\alpha_{n+2}\beta|a
	 \stepcounter{equation}\tag{\theequation}\label{fbn}\\
	 &\quad -\chi_{n+1} bc|\alpha_{n+1}\beta_2|a -ac|\alpha_{n+1}\beta_2|b-\chi_n (ba+ac)|\gamma_{n+3}|a-\chi_n (ab+bc)|\gamma_{n+3}|b\\
	 &\quad +2ac|\alpha_{n+3}|b-\chi_n bc|\alpha_{n+3}|c -\chi_n ac|\beta_{n+3}|c+2bc|\beta_{n+3}|a+(-1)^{n+1}2aba|\gamma_{n+3}|1 \\
	 &\quad +\chi_n bac|\alpha_{n+3}|1+\chi_n abc|\beta_{n+3}|1+\chi_n aba|\alpha_{n+2}\gamma|1-\chi_{n+1} aba|\alpha_{n+1}\beta_2|1,\\
	f_n^b(1|\alpha_{n-1}\beta|1)&=\chi_{n+1}[(n-1)|\beta_{n+3}|abc+a|\alpha_{n+3}|ab-(n-2)c|\alpha_{n+3}|ba+c|\alpha_{n+3}|ac-a|\beta_{n+3}|ab\\
		&\quad +c|\beta_{n+3}|(ba+ac)-a|\gamma_{n+3}|ab-c|\gamma_{n+3}|(ba+ac)-(n-1)a|\gamma_{n+3}|bc-ba|\alpha_{n+3}|a\\
		&\quad +(n-1)ab|\alpha_{n+3}|c+bc|\alpha_{n+3}|c+ba|\beta_{n+3}|a+bc|\beta_{n+3}|c-(n-2)ba|\gamma_{n+3}|a\\
		&\quad -(n-1)ac|\gamma_{n+3}|a-bc|\gamma_{n+3}|c-(n-1)abc|\beta_{n+3}|1],	\\
		f_n^b(1|\alpha_{n-1}\gamma|1)&=\chi_{n+1}[-(n-1)|\gamma_{n+3}|aba+b|\beta_{n+3}|bc+(n-1)a|\beta_{n+3}|ba+(n-2)a|\beta_{n+3}|ac\\
		&\quad -b|\gamma_{n+3}|bc-a|\gamma_{n+3}|ac+a|\alpha_{n+3}|ac+(n-1)b|\alpha_{n+3}|ab+(n-2)b|\alpha_{n+3}|bc\\
		&\quad +(ba+ac)|\beta_{n+3}|b+(n-2)bc|\beta_{n+3}|a-ab|\beta_{n+3}|a-(ba+ac)|\gamma_{n+3}|b\\
		&\quad -(ab+bc)|\gamma_{n+3}|a+(n-2)ac|\alpha_{n+3}|b-ba|\alpha_{n+3}|b+(ab+bc)|\alpha_{n+3}|a\\
		&\quad +(n-1)aba|\gamma_{n+3}|1],\\
f_n^b(1|\alpha_{n-2}\beta_2|1)&=\chi_{n+1}[(n-1)|\alpha_{n+3}|bac-c|\gamma_{n+3}|(ab+bc)-(n-1)b|\gamma_{n+3}|ac-b|\gamma_{n+3}|ba\\
		&\quad +c|\alpha_{n+3}|(ab+bc)-b|\alpha_{n+3}|ba+b|\beta_{n+3}|ba-(n-2)c|\beta_{n+3}|ab+c|\beta_{n+3}|bc\\
		&\quad -ac|\gamma_{n+3}|c-(n-1)bc|\gamma_{n+3}|b-(n-2)ab|\gamma_{n+3}|b+ac|\alpha_{n+3}|c+ab|\alpha_{n+3}|b\\
		&\quad +ac|\beta_{n+3}|c+(n-1)ba|\beta_{n+3}|c-ab|\beta_{n+3}|b-(n-1)bac|\alpha_{n+3}|1]\\
		&\quad +\chi_n (n-2)	[1|\alpha_{n+3}|bac+1|\beta_{n+3}|abc-1|\gamma_{n+3}|aba-b|\alpha_{n+3}|(ab+bc)\\
		&\quad +c|\alpha_{n+3}|ba+c|\beta_{n+3}|ab-a|\beta_{n+3}|(ba+ac)+a|\gamma_{n+3}|bc+b|\gamma_{n+3}|ac\\
		&\quad +ac|\alpha_{n+3}|b+ab|\alpha_{n+3}|c+ba|\beta_{n+3}|c+bc|\beta_{n+3}|a-(ba+ac)|\gamma_{n+3}|a\\
		&\quad -(ab+bc)|\gamma_{n+3}|b+bac|\alpha_{n+3}|1+abc|\beta_{n+3}|1-aba|\gamma_{n+3}|1],
	\end{align*}
\end{small}
where $n\in \NN$.

The proof of the following result is a tedious but straightforward computation, that we leave to the reader. 

\begin{lem}\label{fbn1}
The $A$-bimodule morphisms $f_n^b:K_n^b\to K_{n+3}^b$ defined above are homogeneous morphisms of internal degree $6$, such that $d_3^b f_0^b=0$, $d_{n+4}^b f_{n+1}^b+f_{n}^bd_{n+1}^b=0$ and $\id_{\Bbbk}\otimes_A f^b_n=f_n$ for $n\in \NN_0$, where $f_n$ are the specific morphisms given in \eqref{fn}. 
\end{lem}

Using the previous lemma, we can now prove the main result of this section. 

\begin{prop}\label{prb}
The minimal projective resolution $(P^b_{\bullet},\delta^b_{\bullet})$ of $A$ in the category of bounded below graded $A$-bimodules is given as follows. For $n\in\NN_0$, set 
\[ P^b_n=\underset{i\in \llbracket 0, \lfloor n/4\rfloor \rrbracket}{\bigoplus}\omega_i K^b_{n-4i}=\underset{i\in \llbracket 0, \lfloor n/4\rfloor \rrbracket}{\bigoplus}\omega_i A\otimes (A^!_{-(n-4i)})^*\otimes A,\]
where $\omega_i$ is a symbol of internal degree $6i$ for all $i\in\NN_0$, the $A$-bimodule structure of $P^b_n$ is given by 
$x'(\omega_i x\otimes u\otimes y)y'=\omega_i x'x\otimes u\otimes yy'$ for all $x,x',y,y'\in A$ and $u\in (A^!_{-(n-4i)})^*$, and the differential $\delta^b_{n}:P^b_n\to P^b_{n-1}$ for $n\in \NN$ is given by 
\[\delta^b_{n}\bigg( \underset{i\in \llbracket 0, \lfloor n/4\rfloor \rrbracket}{\sum} \omega_i\rho_{n-4i}\bigg)=\underset{i\in \llbracket 0, \lfloor n/4\rfloor \rrbracket}{\sum} \big( \omega_i d^b_{n-4i}(\rho_{n-4i})+\omega_{i-1}f^b_{n-4i}(\rho_{n-4i}) \big),\]
where $\rho_j\in K^b_j$ for $j\in\NN_0$, $\omega_{-1}=0$ and $f^b_j:K^b_j\to K^b_{j+3}$ are the morphisms in \eqref{fbn}. This gives a minimal projective resolution of $A$ by means of the augmentation $\epsilon^b:P^b_0=A\otimes (A^!_0)^*\otimes A\to A$, where $\epsilon^b(x|\epsilon^!|y)=xy$ for $x,y\in A$.
\end{prop}

\begin{proof}
It is clear that $P^b_{\bullet}\to A\to 0$ is a complex of graded-free (left) $A$-modules by Lemma \ref{fbn1}, $(\Bbbk\otimes_A P^b_{\bullet},\id_{\Bbbk}\otimes_A\delta^b_{\bullet})\cong(P_{\bullet}, \delta_{\bullet})$ and $\id_{\Bbbk}\otimes_A \epsilon^b\cong\epsilon$. Proposition \ref{es} tells us that $\Bbbk\otimes_A P^b_{\bullet}\to\Bbbk\otimes_A A\to 0$ is exact. Proposition \ref{berger} in turn shows that the complex $P^b_{\bullet}\to A\to 0$ is also exact. Moreover, the bimodule resolution $(P^b_{\bullet},\delta^b_{\bullet})$ is minimal since $\id_{\Bbbk}\otimes_A \delta^b_{\bullet}\otimes_A \id_{\Bbbk}=0$.
\end{proof}

We follow the convention that $P^b_n=0$, $K^b_n=0$ for $n\in \ZZ  \setminus \NN_0 $, and $\delta^b_n=0$, $d^b_n=0$ for $n\in \ZZ  \setminus \NN$ in the following sections.

\section{\texorpdfstring{Hochschild  and cyclic homology of $\FK(3)$}{Hochschild and cyclic homology of FK(3)}}
\label{section:homology}

Using the minimal projective bimodule resolution $(P_{\bullet}^b, \delta^b_{\bullet})$ of $A$ in Proposition \ref{prb}, we will compute the linear structure of the Hochschild homology $\operatorname{HH}_{\bullet}(A)=\operatorname{Tor}_{\bullet}^{A^e}(A,A)=\operatorname{H}_{\bullet}(A\otimes_{A^e}P_{\bullet}^b)$. 
For further information about Hochschild and cyclic homology, we refer the reader to \cite{Louis}. 

\subsection{Recursive description of the spaces}
\label{subsection:rds}
 
Let $\tilde{K}_n=A\otimes (A_{-n}^!)^*$ for $n\in \NN_0$ and $\tilde{K}_n=0$ for $n\in\ZZ  \setminus \NN_0 $. 
We have $A\otimes_{A^e}P_{n}^b\cong \tilde{P}_n$ as $\Bbbk$-vector spaces, where $\tilde{P}_n= \oplus_ {i\in \llbracket 0, \lfloor n/4 \rfloor \rrbracket}\omega_{i}\tilde{K}_{n-4i}$ for $n\in\NN_0$ and $\tilde{P}_n=0$ for $n\in \ZZ  \setminus \NN_0 $. 
We will denote by $\partial_{n}: \tilde{P}_{n}\to \tilde{P}_{n-1}$ the differential $\id_A\otimes_{A^e}\delta_{n}^b$, 
$\tilde{\partial}_{n}: \tilde{K}_{n}\to \tilde{K}_{n-1}$ the differential $\id_A\otimes_{A^e} d_{n}^b$ for $n\in \ZZ$, and $\tilde{f}_n$ the map $\id_{A}\otimes_{A^e} f^b_n$ for $n\in\NN_0$. 
Then the differential $\partial_{n}$ for $n\in\NN$ is given by 
\[ \partial_{n}\bigg(\underset{i\in \llbracket 0, \lfloor n/4\rfloor \rrbracket}{\sum}\omega_i\rho_{n-4i}\bigg)=\underset{i\in \llbracket 0, \lfloor n/4\rfloor \rrbracket}{\sum}\big(\omega_i \tilde{\partial}_{n-4i}(\rho_{n-4i})+\omega_{i-1}\tilde{f}_{n-4i}(\rho_{n-4i}) \big),\]
where $\rho_j\in \tilde{K}_j$ for $j\in \NN_0$. Note that $\partial_{n}=\tilde{\partial}_n=0$ for $n\in\ZZ  \setminus \NN $.

The aim of this section is to compute the homology of the complex $(\tilde{P}_{\bullet}, \partial_{\bullet})$. 
Let $\tilde{K}_{n,m}=A_m\otimes (A_{-n}^!)^*$ for $(n, m) \in  \NN_0\times \llbracket 0, 4 \rrbracket$ 
and $\tilde{K}_{n,m}=0$ for $(n,m)\in  \ZZ^{2} \setminus (\NN_0 \times \llbracket 0, 4 \rrbracket)$.
Let $\tilde{P}_{n,m}= \oplus_ {i\in \llbracket 0, \lfloor n/4 \rfloor \rrbracket}\omega_{i} \tilde{K}_{n-4i, m-2i}$ for $m,n \in \NN_0$ 
and $\tilde{P}_{n,m}=0$ for $(n,m)\in  \ZZ^{2} \setminus \NN_0^{2}$, 
where the symbol $\omega_i$ has homological degree $4i$ and internal degree $6i$ for $i\in \NN_0$, and we usually omit $\omega_0$ for simplicity. 
The spaces $\tilde{K}_{n,m}$ and $\tilde{P}_{n,m}$ are concentrated in homological degree $n$ and internal degree $m+n$. 
We have $\tilde{P}_n=\oplus_{m\in \NN_0}\tilde{P}_{n,m}$.  
Let $\partial_{n,m}=\partial_{n}|_{\tilde{P}_{n,m}}:\tilde{P}_{n,m}\to \tilde{P}_{n-1,m+1}$, and $\tilde{\partial}_{n,m}=\tilde{\partial}_{n}|_{\tilde{K}_{n,m}}:\tilde{K}_{n,m}\to \tilde{K}_{n-1,m+1}$.
Let $D_{n,m}=\Ker(\partial_{n,m})$, $B_{n,m}=\Img(\partial_{n+1,m-1})$ for $m,n\in \NN_0$, 
and $\tilde{D}_{n,m}=\Ker(\tilde{\partial}_{n,m})$, $\tilde{B}_{n,m}=\Img(\tilde{\partial}_{n+1,m-1})$ for $(n,m)\in \NN_0 \times \llbracket 0,4 \rrbracket$. 
Notice that $D_{n,m}=B_{n,m}=0$ for $(n,m)\in \ZZ^2 \setminus \NN_0^2$, and $\tilde{D}_{n,m}=\tilde{B}_{n,m}=0$ for $(n,m)\in  \ZZ^{2} \setminus (\NN_0 \times \llbracket 0, 4 \rrbracket)$.
\begin{prop}\label{bdh}
For integers $m\geqslant 3$ and $n\in\NN_0$, we have 
\begin{equation}\label{bnm}
	\begin{split}
 B_{n,m}=
 \begin{cases}
\omega_{\frac{m-3}{2}}B_{n-2m+6,3},  & \text{if $m$ is odd}, 
\\
 \omega_{\frac{m}{2}-2}B_{n-2m+8,4}, &  \text{if $m$ is even},
\end{cases}
\end{split}
\end{equation}
and
\begin{equation}\label{dnm}
	\begin{split}
D_{n,m}=
\begin{cases}
\omega_{\frac{m-3}{2}}D_{n-2m+6,3},  & \text{if $m$ is odd},
\\
 \omega_{\frac{m}{2}-2}D_{n-2m+8,4}, & \text{if $m$ is even},
\end{cases}
\end{split}
\end{equation}
where we follow the convention that $\omega_i\omega_j=\omega_{i+j}$ for $i,j\in \NN_0$ and $\omega_i=0$ for $i\in \ZZ\setminus \NN_0$. 
\end{prop}
\begin{proof}
Consider $\tilde{P}_{n,m}= \oplus_ {i\in \llbracket 0, \lfloor n/4 \rfloor \rrbracket}\omega_{i} \tilde{K}_{n-4i,m-2i}$ for $m,n\in\NN_0$. For the index $m-2i$ of $\tilde{K}_{n-4i,m-2i}$, we have $ m-2i \in \llbracket 0, 4 \rrbracket$. If $m$ is odd, then $m-2i=1$ or $3$, \textit{i.e.} $i=(m-1)/2$ or $(m-3)/2$. Since $n-4i\in \NN_0$, we have 
\begin{equation}\label{pnmodd}
	\begin{split}
\tilde{P}_{n,m}=
\begin{cases}
\omega_{\frac{m-3}{2}}\tilde{K}_{n-2m+6,3}\oplus \omega_{\frac{m-1}{2}}\tilde{K}_{n-2m+2,1}, &\text{if $n\geqslant 2m-2$},
\\
\omega_{\frac{m-3}{2}}\tilde{K}_{n-2m+6,3} , & \text{if $2m-6\leqslant n<2m-2$},
\\
0 ,&\text{if $0\leqslant n< 2m-6$}.
\end{cases} 
\end{split}
\end{equation}
If $m$ is even, then $m-2i=0,2$ or $4$, \textit{i.e.} $i=m/2$, $m/2-1$ or $m/2-2$. Then 
\begin{equation}\label{pnmeven}
	\begin{split}
 \tilde{P}_{n,m}=
 \begin{cases}
\omega_{\frac{m}{2}-2}\tilde{K}_{n-2m+8,4}\oplus \omega_{\frac{m}{2}-1}\tilde{K}_{n-2m+4,2} \oplus \omega_{\frac{m}{2}}\tilde{K}_{n-2m,0}, & \text{if $n\geqslant 2m$},
\\
\omega_{\frac{m}{2}-2}\tilde{K}_{n-2m+8,4}\oplus\omega_{\frac{m}{2}-1}\tilde{K}_{n-2m+4,2}, &\text{if $2m-4\leqslant n<2m$},
\\
\omega_{\frac{m}{2}-2}\tilde{K}_{n-2m+8,4}, & \text{if $2m-8\leqslant n<2m-4$}, 
\\
0, &\text{if $0\leqslant n< 2m-8$}.
\end{cases}
\end{split}
\end{equation}
Hence,
\begin{equation}\label{pnm1}
 \begin{split}
	\tilde{P}_{n,m}=
	\begin{cases}
     \omega_{\frac{m-3}{2}}\tilde{P}_{n-2m+6,3}, &\text{if $m\geqslant 3$ is odd}, 
     \\
     \omega_{\frac{m}{2}-2}\tilde{P}_{n-2m+8,4}, &\text{if $m\geqslant 4$ is even}.
     \end{cases}
 \end{split}
\end{equation}
Since the identities \eqref{bnm} and \eqref{dnm} for $m=3$ are immediate, we suppose $m\geqslant 4$ from now on.

Assume that $m$ is even. 
Then \eqref{pnm1} tells us that the sequence 
\[\tilde{P}_{n+1,m-1}\xrightarrow{\partial{n+1,m-1}} \tilde{P}_{n,m}\xrightarrow{\partial_{n,m}} \tilde{P}_{n-1,m+1}\]
of graded $\Bbbk$-vector spaces is of the form
\[\omega_{\frac{m}{2}-2}\tilde{P}_{n-2m+9,3}\xrightarrow{\partial_{n+1,m-1}} \omega_{\frac{m}{2}-2}\tilde{P}_{n-2m+8,4}\xrightarrow{\partial_{n,m}} \omega_{\frac{m}{2}-1}\tilde{P}_{n-2m+3,3}.\]
Since $\tilde{P}_{n-2m+7,5}=\omega_1\tilde{P}_{n-2m+3,3}$ by \eqref{pnm1}, the above sequence is of the form 
\[\omega_{\frac{m}{2}-2}\tilde{P}_{n-2m+9,3}\xrightarrow{\partial_{n+1,m-1}} \omega_{\frac{m}{2}-2}\tilde{P}_{n-2m+8,4}\xrightarrow{\partial_{n,m}} \omega_{\frac{m}{2}-2}\tilde{P}_{n-2m+7,5}.\]
Note further that $\partial_{n,m}=\omega_{\frac{m}{2}-2}\partial_{n-2m+8,4}$ and $\partial_{n+1,m-1}=\omega_{\frac{m}{2}-2}\partial_{n-2m+9,3}$, where the differential $\omega_j \partial_{n',m'}:\omega_j \tilde{P}_{n',m'}\rightarrow \omega_j \tilde{P}_{n'-1,m'+1}$ maps $\omega_j x$ to $\omega_j \partial_{n',m'}(x)$ for all $x\in \tilde{P}_{n',m'}$ and $j,m',n'\in \NN_0$.
Hence, $B_{n,m}=\omega_{\frac{m}{2}-2}B_{n-2m+8,4}$ and $D_{n,m}=\omega_{\frac{m}{2}-2}D_{n-2m+8,4}$.

Assume that $m$ is odd (so $m\geqslant 5$). Then \eqref{pnm1} tells us that the sequence
\[\tilde{P}_{n+1,m-1}\xrightarrow{\partial_{n+1,m-1}} \tilde{P}_{n,m}\xrightarrow{\partial_{n,m}}\tilde{P}_{n-1,m+1}\] 
of graded $\Bbbk$-vector spaces is of the form
\[\omega_{\frac{m-5}{2}}\tilde{P}_{n-2m+11,4}\xrightarrow{\partial_{n+1,m-1}} \omega_{\frac{m-3}{2}}\tilde{P}_{n-2m+6,3}\xrightarrow{\partial_{n,m}} \omega_{\frac{m-3}{2}}\tilde{P}_{n-2m+5,4}.\]
Note that $\partial_{n,m}=\omega_{\frac{m-3}{2}}\partial_{n-2m+6,3}$ and $\partial_{n+1,m-1}=\omega_{\frac{m-5}{2}}\partial_{n-2m+11,4}$.
Since 
\[\tilde{P}_{n-2m+11,4}=\omega_0 \tilde{K}_{n-2m+11,4}\oplus\omega_1\tilde{P}_{n-2m+7,2}\] 
by \eqref{pnmeven}, 
\[ \partial_{n-2m+11,4}(\omega_0 x+\omega_1 y)=\omega_0\tilde{\partial}_{n-2m+11,4}(x)+\omega_1\partial_{n-2m+7,2}(y) \] 
for all $x\in \tilde{K}_{n-2m+11,4}$ and $y\in \tilde{P}_{n-2m+7,2}$, and $\tilde{\partial}_{n-2m+11,4}(\tilde{K}_{n-2m+11,4})=0$, 
it is sufficient to consider the following sequence  
\[\omega_{\frac{m-3}{2}}\tilde{P}_{n-2m+7,2}\xrightarrow{\omega_{\frac{m-3}{2}}\partial_{n-2m+7,2}} \omega_{\frac{m-3}{2}}\tilde{P}_{n-2m+6,3}\xrightarrow{\omega_{\frac{m-3}{2}}\partial_{n-2m+6,3}} \omega_{\frac{m-3}{2}}\tilde{P}_{n-2m+5,4}.\]
Hence, $B_{n,m}=\omega_{\frac{m-3}{2}}B_{n-2m+6,3}$ and $D_{n,m}=\omega_{\frac{m-3}{2}}D_{n-2m+6,3}$, as was to be shown.
\end{proof}
\begin{prop}\label{d4n}
For $n\in\NN_0$, we have $D_{n,4}= \tilde{K}_{n,4}\oplus \omega_1 D_{n-4,2}$.
\end{prop}
\begin{proof}
This follows directly from the facts that $\tilde{P}_{n,4}=\tilde{K}_{n,4}\oplus \omega_1 \tilde{P}_{n-4,2}$, $\tilde{P}_{n-1,5}=\omega_1 \tilde{P}_{n-5,3}$ and $\partial_{n,4}(\tilde{K}_{n,4})=0$.
\end{proof}

In order to compute $B_{n,m}$ and $D_{n,m}$, it is sufficient to compute the case $m \in \llbracket 0,4 \rrbracket$ according to Proposition \ref{bdh}. 
First, we will compute the boundaries, and then we will compute the cycles. 
Since this will require handling elements of $\tilde{K}_{n,m}$ and $\tilde{P}_{n,m}$ for $n \in \NN_{0}$ and $m \in \llbracket 0 , 4 \rrbracket$, we will use the basis $\{ x \otimes y | x \in \mathcalboondox{B}_{m}, y \in \mathcalboondox{B}_{n}^{!*} \}$ of 
$\tilde{K}_{n,m}$ and the basis $\{ \omega_{i} x \otimes y | i \in \llbracket 0 , \lfloor n/4 \rfloor \rrbracket, x \in \mathcalboondox{B}_{m-2i}, y \in \mathcalboondox{B}_{n-4i}^{!*} \}$ of 
$\tilde{P}_{n,m}$, both of which will be called {\color{ultramarine}{\textbf{usual}}} bases, constructed from the usual bases of the homogeneous components of $A$ and $(A^{!})^{\#}$, introduced in Section \ref{section:generalities}. 

\subsection{Explicit description of the differentials}
\label{subsection:explicit-description-diff-homology} 

Recall the isomorphism $ A\otimes_{A^e}(A\otimes (A^!_{-n})^*\otimes A)\to A\otimes (A^!_{-n})^*$ given by $x\otimes_{A^e}(y|u|z)\mapsto zxy|u$, and its inverse $ A\otimes (A^!_{-n})^*\to A\otimes_{A^e}(A\otimes (A^!_{-n})^*\otimes A)$ given by $ x|u\mapsto x\otimes_{A^e} (1|u|1)$ for all $x,y,z\in A$, $u\in (A^!_{-n})^*$ and $n\in\NN_0$. We will use them together with Proposition \ref{prb} to explicitly describe  $\tilde{\partial}_n$ and $\tilde{f}_n$, which were defined at the beginning of Subsection \ref{subsection:rds}.

Let $x\in A $. It is then straightforward to see that the differential $\tilde{\partial}_1:A\otimes (A^!_{-1})^*\to A\otimes (A^!_{0})^*$ is given by $\tilde{\partial}_1(x|\alpha)=(ax-xa)|\epsilon^!$, $\tilde{\partial}_1(x|\beta)=(bx-xb)|\epsilon^!$ and $\tilde{\partial}_1(x|\gamma)=(cx-xc)|\epsilon^! $. Analogously, for $n\geqslant 2$ and $n$  even, $\tilde{\partial}_{n}: A\otimes (A^!_{-n})^*\to A\otimes (A^!_{-(n-1)})^* $ is given by
\begin{equation}\label{tildepartialeven}
	\begin{split}
x|\alpha_n & \mapsto (xa+ax)|\alpha_{n-1}, \quad x|\beta_n\mapsto (xb+bx)|\beta_{n-1}, \quad x|\gamma_n\mapsto (xc+cx)|\gamma_{n-1},
\\
x|\alpha_{n-1}\beta & \mapsto (xa+cx)|(\beta_{n-1}+\alpha_{n-2}\beta)+(xb+ax)|(\gamma_{n-1}+\alpha_{n-2}\gamma)
\\
& \phantom{ \mapsto \;}                       
+(xc+bx)|(\alpha_{n-1}+\alpha_{n-3}\beta_2),
\\
x|\alpha_{n-1}\gamma & \mapsto (xa+bx)|(\gamma_{n-1}+\alpha_{n-2}\gamma)+(xb+cx)|(\alpha_{n-1}+\alpha_{n-3}\beta_2)
\\
& \phantom{ \mapsto \;}                 
+(xc+ax)|(\beta_{n-1}+\alpha_{n-2}\beta),
\\
 x|\alpha_{n-2}\beta_2 & \mapsto (xa+ax)|\alpha_{n-3}\beta_2+(xb+bx)|\alpha_{n-2}\beta+(xc+cx)|\alpha_{n-2}\gamma,
	\end{split}
\end{equation}
whereas, for $n\geqslant 3$ and $n$ odd, $\tilde{\partial}_{n}: A\otimes (A^!_{-n})^*\to A\otimes (A^!_{-(n-1)})^* $ is given by
\begin{equation}\label{tildepartialodd}
	\begin{split}
x|\alpha_n & \mapsto (ax-xa)|\alpha_{n-1}, \quad x|\beta_n\mapsto (bx-xb)|\beta_{n-1},\quad x|\gamma_n\mapsto (cx-xc)|\gamma_{n-1},
\\
x|\alpha_{n-1}\beta & \mapsto (cx-xa)|\alpha_{n-2}\beta+(ax-xc)|\alpha_{n-2}\gamma+(bx-xb)|(\alpha_{n-1}+\gamma_{n-1}+\alpha_{n-3}\beta_2),
\\
x|\alpha_{n-1}\gamma & \mapsto (ax-xb)|\alpha_{n-2}\beta+(bx-xa)|\alpha_{n-2}\gamma+(cx-xc)|(\alpha_{n-1}+\beta_{n-1}+\alpha_{n-3}\beta_2),
\\
x|\alpha_{n-2}\beta_2 & \mapsto (ax-xa)|(\beta_{n-1}+\gamma_{n-1}+\alpha_{n-3}\beta_2)+(bx-xc)|\alpha_{n-2}\beta+(cx-xb)|\alpha_{n-2}\gamma.
  \end{split}
\end{equation}

For the reader's convenience, we list the images of the differentials $\tilde{\partial}_n$ evaluated at elements of the usual $\Bbbk$-basis of the respective domain. 
In the following tables, $\tilde{\partial}_{n,m}(x|y)$ is the entry appearing in the column indexed by $y$ and the row indexed by $x$, where $m$ is the internal degree of $x$ and $n$ is the internal degree of $y$. The differential $\tilde{\partial}_1$ is given by 
\begin{table}[H]
	\begin{center}
	\begin{tabular}{|c|c|c|c|}
		\hline
		 \diagbox[width=13mm ,height=5.4mm ]{$x$}{$y$}  & $\alpha$                & $\beta$                 & $\gamma$  \\
		\hline
		$1$          & $0$                     & $0$                     & $0$    \\
		\hline
		$a$          & $0$                     & $(ba-ab)|\epsilon^!$    & $(-ab-bc-ac)|\epsilon^!$ \\
		\hline     
		$b$          & $(ab-ba)|\epsilon^!$    & $0$                     & $(-ba-ac-bc)|\epsilon^!$  \\
		\hline
        $c$          & $(ab+bc+ac)|\epsilon^!$ & $(ba+ac+bc)|\epsilon^!$ & $0$  \\
		\hline
		$ab$         & $-aba|\epsilon^!$       & $aba|\epsilon^!$        & $(bac-abc)|\epsilon^!$    \\
		\hline
		$bc$         & $(aba+abc)|\epsilon^!$  & $bac|\epsilon^!$        & $-bac|\epsilon^!$  \\
		\hline
		$ba$         & $aba|\epsilon^!$        & $-aba|\epsilon^!$       & $(abc-bac)|\epsilon^!$  \\
		\hline 
		$ac$         & $abc|\epsilon^!$        & $(aba+bac)|\epsilon^!$  & $-abc|\epsilon^!$  \\
		\hline
		$aba$        & $0$                     & $0$                     & $-2abac|\epsilon^!$ \\
		\hline 
		$abc$        & $0$                     & $2abac|\epsilon^!$      & $0$            \\
		\hline
		$bac$        & $2abac|\epsilon^!$      & $0$                     & $0$        \\
		\hline
		$abac$       & $0$                     & $0$                     & $0$     \\
		\hline
	\end{tabular}
	\end{center}
	\caption{Images of $\tilde{\partial}_1$.}	
	\label{tpar1}
	\end{table}
For $n\geqslant 2$ and $n$ even, $\tilde{\partial}_n$ is given by 
\begin{table}[H]
	\begin{center}
	\begin{tabular}{|c|c|c|c|}
		\hline
		     \diagbox[width=13mm ,height=5.4mm ]{$x$}{$y$}      & $\alpha_n$              & $\beta_n$               & $\gamma_n$                         \\
		\hline
		$1$          & $2a|\alpha_{n-1}$       & $2b|\beta_{n-1}$        & $2c|\gamma_{n-1}$                  \\
		\hline
		$a$          & $0$                     & $(ab+ba)|\beta_{n-1}$    & $(ac-ab-bc)|\gamma_{n-1}$                          \\
		\hline     
		$b$          & $(ab+ba)|\alpha_{n-1}$  & $0$                     & $(bc-ba-ac)|\gamma_{n-1}$                        \\
		\hline
        $c$          & $(ac-ab-bc)|\alpha_{n-1}$ & $(bc-ba-ac)|\beta_{n-1}$ & $0$                                           \\
		\hline
		$ab$         & $aba|\alpha_{n-1}$       & $aba|\beta_{n-1}$        & $(abc+bac)|\gamma_{n-1}$                       \\
		\hline
		$bc$         & $(abc-aba)|\alpha_{n-1}$  & $-bac|\beta_{n-1}$        & $-bac|\gamma_{n-1}$                   \\
		\hline
		$ba$         & $aba|\alpha_{n-1}$        & $aba|\beta_{n-1}$       & $(abc+bac)|\gamma_{n-1}$                  \\
		\hline 
		$ac$         & $-abc|\alpha_{n-1}$        & $(bac-aba)|\beta_{n-1}$  & $-abc|\gamma_{n-1}$                   \\
		\hline
		$aba$        & $0$                     & $0$                     & $0$                 \\
		\hline 
		$abc$        & $0$                     & $0$      & $0$                                    \\
		\hline
		$bac$        & $0$      & $0$                     & $0$                                  \\
		\hline
		$abac$       & $0$                     & $0$                     & $0$                                    \\
		\hline
	\end{tabular}
	\end{center}
	\caption{Images of $\tilde{\partial}_n$ for $n\geqslant 2$ and $n$ even.}	
	\label{tparevenalphan}
	\end{table}

\noindent together with

	\begin{table}[H]
		\begin{center}
		\begin{tabular}{|c|c|c|c|}
			\hline
\diagbox[width=13mm ,height=5.4mm ]{$x$}{$y$}  & $\alpha_{n-1}\beta$                                     \\
			\hline
			$1$      & $(a+c)|(\beta_{n-1}+\alpha_{n-2}\beta)+(b+a)|(\gamma_{n-1}+\alpha_{n-2}\gamma)+(c+b)|(\alpha_{n-1}+\alpha_{n-3}\beta_2) $ \\
			\hline
			$a$      & $-(ab+bc)|(\beta_{n-1}+\alpha_{n-2}\beta)+ab|(\gamma_{n-1}+\alpha_{n-2}\gamma)+(ba+ac)|(\alpha_{n-1}+\alpha_{n-3}\beta_2)$ \\
			\hline     
			$b$      & $-ac|(\beta_{n-1}+\alpha_{n-2}\beta)+ab|(\gamma_{n-1}+\alpha_{n-2}\gamma)+bc|(\alpha_{n-1}+\alpha_{n-3}\beta_2)$   \\
			\hline
			$c$      & $-(ab+bc)|(\beta_{n-1}+\alpha_{n-2}\beta)-ba|(\gamma_{n-1}+\alpha_{n-2}\gamma)+bc|(\alpha_{n-1}+\alpha_{n-3}\beta_2)$ \\
			\hline
			$ab$     & $(aba+bac)|(\beta_{n-1}+\alpha_{n-2}\beta)+(aba+abc)|(\alpha_{n-1}+\alpha_{n-3}\beta_2)$               \\
			\hline
			$bc$     & $(-aba-bac)|(\beta_{n-1}+\alpha_{n-2}\beta)+(abc-bac)|(\gamma_{n-1}+\alpha_{n-2}\gamma)$                \\
			\hline
			$ba$     & $abc|(\beta_{n-1}+\alpha_{n-2}\beta)+2aba|(\gamma_{n-1}+\alpha_{n-2}\gamma)+bac|(\alpha_{n-1}+\alpha_{n-3}\beta_2)$  \\
			\hline 
			$ac$     & $-2abc|(\beta_{n-1}+\alpha_{n-2}\beta)-aba|(\gamma_{n-1}+\alpha_{n-2}\gamma)+bac|(\alpha_{n-1}+\alpha_{n-3}\beta_2)$ \\
			\hline
			$aba$    & $abac|(-\beta_{n-1}-\alpha_{n-2}\beta+\alpha_{n-1}+\alpha_{n-3}\beta_2)$                                      \\
			\hline 
			$abc$    & $abac|(-\gamma_{n-1}-\alpha_{n-2}\gamma+\alpha_{n-1}+\alpha_{n-3}\beta_2)$                              \\
			\hline
			$bac$    & $abac|(-\beta_{n-1}-\alpha_{n-2}\beta+\gamma_{n-1}+\alpha_{n-2}\gamma)$                                       \\
			\hline
			$abac$   & $0$                                                        \\
			\hline
		\end{tabular}
		\end{center}
		\caption{Images of $\tilde{\partial}_n$ for $n\geqslant 2$ and $n$ even.}	
		\label{tparevenalphabeta}
		\end{table}

\noindent and

		\begin{table}[H]
			\begin{center}
			\begin{tabular}{|c|c|c|c|}
				\hline
			\diagbox[width=13mm ,height=5.4mm ]{$x$}{$y$}		 & $\alpha_{n-1}\gamma$                                     \\
				\hline
				$1$      & $(a+b)|(\gamma_{n-1}+\alpha_{n-2}\gamma)+(b+c)|(\alpha_{n-1}+\alpha_{n-3}\beta_2)+(c+a)|(\beta_{n-1}+\alpha_{n-2}\beta)$ \\
				\hline
				$a$      & $ba|(\gamma_{n-1}+\alpha_{n-2}\gamma)-bc|(\alpha_{n-1}+\alpha_{n-3}\beta_2)+ac|(\beta_{n-1}+\alpha_{n-2}\beta)$  \\
				\hline     
				$b$      &  $ba|(\gamma_{n-1}+\alpha_{n-2}\gamma)-(ba+ac))|(\alpha_{n-1}+\alpha_{n-3}\beta_2)+(ab+bc)|(\beta_{n-1}+\alpha_{n-2}\beta) $ \\
				\hline
				$c$      & $-ab|(\gamma_{n-1}+\alpha_{n-2}\gamma)-(ba+ac)|(\alpha_{n-1}+\alpha_{n-3}\beta_2)+ac|(\beta_{n-1}+\alpha_{n-2}\beta) $  \\
				\hline
				$ab$     & $2aba|(\gamma_{n-1}+\alpha_{n-2}\gamma)+bac|(\alpha_{n-1}+\alpha_{n-3}\beta_2)+abc|(\beta_{n-1}+\alpha_{n-2}\beta) $ \\
				\hline
				$bc$     & $-aba|(\gamma_{n-1}+\alpha_{n-2}\gamma)-2bac|(\alpha_{n-1}+\alpha_{n-3}\beta_2)+abc|(\beta_{n-1}+\alpha_{n-2}\beta) $ \\
				\hline
				$ba$     &  $(aba+abc)|(\alpha_{n-1}+\alpha_{n-3}\beta_2)+(aba+bac)|(\beta_{n-1}+\alpha_{n-2}\beta) $ \\
				\hline 
				$ac$     &  $(bac-abc)|(\gamma_{n-1}+\alpha_{n-2}\gamma)-(aba+abc)|(\alpha_{n-1}+\alpha_{n-3}\beta_2) $ \\
				\hline
				$aba$    &    $abac|(-\alpha_{n-1}-\alpha_{n-3}\beta_2+\beta_{n-1}+\alpha_{n-2}\beta)  $              \\
				\hline 
				$abc$    &    $abac|(\gamma_{n-1}+\alpha_{n-2}\gamma-\alpha_{n-1}-\alpha_{n-3}\beta_2)  $           \\
				\hline
				$bac$    &    $ abac|(-\gamma_{n-1}-\alpha_{n-2}\gamma+\beta_{n-1}+\alpha_{n-2}\beta)   $                  \\
				\hline
				$abac$   & $0$                                                        \\
				\hline
			\end{tabular}
			\end{center}
			\caption{Images of $\tilde{\partial}_n$ for $n\geqslant 2$ and $n$ even.}	
			\label{tparevenalphagamma}
			\end{table}

\noindent as well as

	\begin{table}[H]
				\begin{center}
				\begin{tabular}{|c|c|c|c|}
					\hline
				\diagbox[width=13mm ,height=5.4mm ]{$x$}{$y$}	 & $\alpha_{n-2}\beta_2$                                     \\
					\hline
					$1$      & $2a|\alpha_{n-3}\beta_2+2b|\alpha_{n-2}\beta+2c|\alpha_{n-2}\gamma$  \\
					\hline
					$a$      &  $ (ab+ba)|\alpha_{n-2}\beta+(ac-ab-bc)|\alpha_{n-2}\gamma$ \\
					\hline     
					$b$      &  $ (ab+ba)|\alpha_{n-3}\beta_2+(bc-ba-ac)|\alpha_{n-2}\gamma$ \\
					\hline
					$c$      &  $ (ac-ab-bc)|\alpha_{n-3}\beta_2+(bc-ba-ac)|\alpha_{n-2}\beta$ \\
					\hline
					$ab$     &  $aba|\alpha_{n-3}\beta_2+aba|\alpha_{n-2}\beta+(abc+bac)|\alpha_{n-2}\gamma$ \\
					\hline
					$bc$     &  $(abc-aba)|\alpha_{n-3}\beta_2-bac|\alpha_{n-2}\beta-bac|\alpha_{n-2}\gamma$ \\
					\hline
					$ba$     &  $aba|\alpha_{n-3}\beta_2+aba|\alpha_{n-2}\beta+(abc+bac)|\alpha_{n-2}\gamma$ \\
					\hline 
					$ac$     &  $-abc|\alpha_{n-3}\beta_2+(bac-aba)|\alpha_{n-2}\beta-abc|\alpha_{n-2}\gamma$ \\
					\hline
					$aba$    &    $0 $              \\
					\hline 
					$abc$    &    $0 $           \\
					\hline
					$bac$    &    $0 $                  \\
					\hline
					$abac$   & $0$                                                        \\
					\hline
				\end{tabular}
				\end{center}
				\caption{Images of $\tilde{\partial}_n$ for $n\geqslant 4$ and $n$ even.}	
				\label{tparevenalphabeta2}
				\end{table}

For $n\geqslant 3$ and $n$ odd, $\tilde{\partial}_n$ is given by

\begin{table}[H]
	\begin{center}
	\begin{tabular}{|c|c|c|c|}
		\hline
		\diagbox[width=13mm ,height=5.4mm ]{$x$}{$y$}        & $\alpha_n$                & $\beta_n$                 & $\gamma_n$  \\
		\hline
		$1$          & $0$                     & $0$                     & $0$    \\
		\hline
		$a$          & $0$                     & $(ba-ab)|\beta_{n-1}$    & $(-ab-bc-ac)|\gamma_{n-1}$ \\
		\hline     
		$b$          & $(ab-ba)|\alpha_{n-1}$    & $0$                     & $(-ba-ac-bc)|\gamma_{n-1}$  \\
		\hline
        $c$          & $(ab+bc+ac)|\alpha_{n-1}$ & $(ba+ac+bc)|\beta_{n-1}$ & $0$  \\
		\hline
		$ab$         & $-aba|\alpha_{n-1}$       & $aba|\beta_{n-1}$        & $(bac-abc)|\gamma_{n-1}$    \\
		\hline
		$bc$         & $(aba+abc)|\alpha_{n-1}$  & $bac|\beta_{n-1}$        & $-bac|\gamma_{n-1}$  \\
		\hline
		$ba$         & $aba|\alpha_{n-1}$        & $-aba|\beta_{n-1}$       & $(abc-bac)|\gamma_{n-1}$  \\
		\hline 
		$ac$         & $abc|\alpha_{n-1}$        & $(aba+bac)|\beta_{n-1}$  & $-abc|\gamma_{n-1}$  \\
		\hline
		$aba$        & $0$                     & $0$                     & $-2abac|\gamma_{n-1}$ \\
		\hline 
		$abc$        & $0$                     & $2abac|\beta_{n-1}$      & $0$            \\
		\hline
		$bac$        & $2abac|\alpha_{n-1}$      & $0$                     & $0$        \\
		\hline
		$abac$       & $0$                     & $0$                     & $0$     \\
		\hline
	\end{tabular}
	\end{center}
	\caption{Images of $\tilde{\partial}_n$ for $n\geqslant 3$ and $n$ odd.}	
	\label{tparoddalphan}
	\end{table}

\noindent together with

	\begin{table}[H]
		\begin{center}
		\begin{tabular}{|c|c|c|c|}
			\hline
		\diagbox[width=13mm ,height=5.4mm ]{$x$}{$y$}	 & $\alpha_{n-1}\beta$                                     \\
			\hline
			$1$      & $(c-a)|\alpha_{n-2}\beta+(a-c)|\alpha_{n-2}\gamma$ \\
			\hline
			$a$      & $-(ab+bc)|\alpha_{n-2}\beta-ac|\alpha_{n-2}\gamma+(ba-ab)|(\alpha_{n-1}+\gamma_{n-1}+\alpha_{n-3}\beta_2)$ \\
			\hline     
			$b$      &  $(-2ba-ac)|\alpha_{n-2}\beta+(ab-bc)|\alpha_{n-2}\gamma$  \\
			\hline
			$c$      &  $(ab+bc)|\alpha_{n-2}\beta+ac|\alpha_{n-2}\gamma+(ba+ac+bc)|(\alpha_{n-1}+\gamma_{n-1}+\alpha_{n-3}\beta_2)$\\
			\hline
			$ab$     &   $(bac-aba)|\alpha_{n-2}\beta-abc|\alpha_{n-2}\gamma+aba|(\alpha_{n-1}+\gamma_{n-1}+\alpha_{n-3}\beta_2)$         \\
			\hline
			$bc$     &   $(aba-bac)|\alpha_{n-2}\beta+abc|\alpha_{n-2}\gamma+bac|(\alpha_{n-1}+\gamma_{n-1}+\alpha_{n-3}\beta_2)$        \\
			\hline
			$ba$     &   $abc|\alpha_{n-2}\beta+(aba-bac)|\alpha_{n-2}\gamma-aba|(\alpha_{n-1}+\gamma_{n-1}+\alpha_{n-3}\beta_2)$ \\
			\hline 
			$ac$     &   $(aba+bac)|(\alpha_{n-1}+\gamma_{n-1}+\alpha_{n-3}\beta_2)$   \\
			\hline
			$aba$    &    $-abac|(\alpha_{n-2}\beta+\alpha_{n-2}\gamma)$                                \\
			\hline 
			$abc$    &  $2abac|(\alpha_{n-1}+\gamma_{n-1}+\alpha_{n-3}\beta_2)$                  \\
			\hline
			$bac$    &  $abac|(\alpha_{n-2}\beta+\alpha_{n-2}\gamma)$ \\
			\hline
			$abac$   & $0$                                                        \\
			\hline
		\end{tabular}
		\end{center}
		\caption{Images of $\tilde{\partial}_n$ for $n\geqslant 3$ and $n$ odd.}	
		\label{tparoddalphabeta}
		\end{table}

\noindent and

		\begin{table}[H]
			\begin{center}
			\begin{tabular}{|c|c|c|c|}
				\hline
					\diagbox[width=13mm ,height=5.4mm ]{$x$}{$y$}	 & $\alpha_{n-1}\gamma$                              \\
				\hline
				$1$      & $(a-b)|\alpha_{n-2}\beta+(b-a)|\alpha_{n-2}\gamma$ \\
				\hline
				$a$      & $-ab|\alpha_{n-2}\beta+ba|\alpha_{n-2}\gamma-(ab+bc+ac)|(\alpha_{n-1}+\beta_{n-1}+\alpha_{n-3}\beta_2)$ \\
				\hline     
				$b$      &  $ab|\alpha_{n-2}\beta-ba|\alpha_{n-2}\gamma-(ba+ac+bc)|(\alpha_{n-1}+\beta_{n-1}+\alpha_{n-3}\beta_2)$  \\
				\hline
				$c$      &  $(2ac+ba)|\alpha_{n-2}\beta+(ab+2bc)|\alpha_{n-2}\gamma$ \\
				\hline
				$ab$     &  $(bac-abc)|(\alpha_{n-1}+\beta_{n-1}+\alpha_{n-3}\beta_2)$         \\
				\hline
				$bc$     &   $(abc+bac)|\alpha_{n-2}\beta+aba|\alpha_{n-2}\gamma-bac|(\alpha_{n-1}+\beta_{n-1}+\alpha_{n-3}\beta_2)$        \\
				\hline
				$ba$     &   $(abc-bac)|(\alpha_{n-1}+\beta_{n-1}+\alpha_{n-3}\beta_2)$ \\
				\hline 
				$ac$     &   $aba|\alpha_{n-2}\beta+(abc+bac)|\alpha_{n-2}\gamma-abc|(\alpha_{n-1}+\beta_{n-1}+\alpha_{n-3}\beta_2)$   \\
				\hline
				$aba$    &    $-2abac|(\alpha_{n-1}+\beta_{n-1}+\alpha_{n-3}\beta_2)$                                \\
				\hline 
				$abc$    &  $abac|(\alpha_{n-2}\beta+\alpha_{n-2}\gamma)$                  \\
				\hline
				$bac$    &  $abac|(\alpha_{n-2}\beta+\alpha_{n-2}\gamma)$ \\
				\hline
				$abac$   & $0$                                                        \\
				\hline
			\end{tabular}
			\end{center}
			\caption{Images of $\tilde{\partial}_n$ for $n\geqslant 3$ and $n$ odd.}	
			\label{tparoddalphagamma}
			\end{table}

\noindent as well as

			\begin{table}[H]
				\begin{center}
				\begin{tabular}{|c|c|c|c|}
					\hline
					\diagbox[width=13mm ,height=5.4mm ]{$x$}{$y$}	 & $\alpha_{n-2}\beta_2$                              \\
					\hline
					$1$      & $(b-c)|\alpha_{n-2}\beta+(c-b)|\alpha_{n-2}\gamma$ \\
					\hline
					$a$      & $(ba-ac)|\alpha_{n-2}\beta-(2ab+bc)|\alpha_{n-2}\gamma$ \\
					\hline     
					$b$      &  $(ab-ba)|(\beta_{n-1}+\gamma_{n-1}+\alpha_{n-3}\beta_2)-bc|\alpha_{n-2}\beta-(ba+ac)|\alpha_{n-2}\gamma$  \\
					\hline
					$c$      &  $ (ab+bc+ac)|(\beta_{n-1}+\gamma_{n-1}+\alpha_{n-3}\beta_2)+bc|\alpha_{n-2}\beta+(ba+ac)|\alpha_{n-2}\gamma $ \\
					\hline
					$ab$     &  $ -aba|(\beta_{n-1}+\gamma_{n-1}+\alpha_{n-3}\beta_2)+(aba-abc)|\alpha_{n-2}\beta+bac|\alpha_{n-2}\gamma  $ \\
					\hline
					$bc$     &   $(aba+abc)|(\beta_{n-1}+\gamma_{n-1}+\alpha_{n-3}\beta_2)$        \\
					\hline
					$ba$     &   $aba|(\beta_{n-1}+\gamma_{n-1}+\alpha_{n-3}\beta_2)-bac|\alpha_{n-2}\beta+(abc-aba)|\alpha_{n-2}\gamma$ \\
					\hline 
					$ac$     &   $abc|(\beta_{n-1}+\gamma_{n-1}+\alpha_{n-3}\beta_2)+bac|\alpha_{n-2}\beta+(aba-abc)|\alpha_{n-2}\gamma$   \\
					\hline
					$aba$    &    $-abac|(\alpha_{n-2}\beta+\alpha_{n-2}\gamma)$                                \\
					\hline 
					$abc$    &  $abac|(\alpha_{n-2}\beta+\alpha_{n-2}\gamma)$                  \\
					\hline
					$bac$    &  $2abac|(\beta_{n-1}+\gamma_{n-1}+\alpha_{n-3}\beta_2)$ \\
					\hline
					$abac$   & $0$                                                        \\
					\hline
				\end{tabular}
				\end{center}
				\caption{Images of $\tilde{\partial}_n$ for $n\geqslant 3$ and $n$ odd.}	
				\label{tparoddalphabeta2}
				\end{table}

Let us now turn to the maps $\tilde{f}_n$. Note first that the $\Bbbk$-linear maps $\tilde{f}_n:A\otimes (A^!_{-n})^*\to A\otimes (A^!_{-(n+3)})^*$ are homogeneous of homological degree $3$ and internal degree $6$. By degree reasons we see that $\tilde{f}_n(x|y)=0$ for all $x\in A_m$, $y\in (A^!_{-n})^*$, with $m\in \llbracket 2,4\rrbracket$ and $n\in\NN_0$. A straightforward computation using \eqref{fbn} tells us that the map $\tilde{f}_0$ is given by
\begin{equation}
\label{eq:f0-homology}
\begin{split}
\tilde{f}_0(1|\epsilon^!) & = 12bac|\alpha_3+12abc|\beta_3-12aba|\gamma_3-6abc|\alpha_2\beta+6aba|\alpha_2\gamma-6bac|\alpha\beta_2,
\\
\tilde{f}_0(a|\epsilon^!) & =\tilde{f}_0(b|\epsilon^!)=\tilde{f}_0(c|\epsilon^!)= 0.
\end{split}
\end{equation}
Analogously, if $n\in\NN$ is odd, then 
\begin{align*}
 \tilde{f}_n(a|\alpha_n)&=\tilde{f}_n(b|\beta_n)=\tilde{f}_n(c|\gamma_n)= -4abac|\alpha_{n+3}-4abac|\beta_{n+3}-4abac|\gamma_{n+3}+6abac|\alpha_{n+1}\beta_2,
 \\
\tilde{f}_n(b|\alpha_{n-1}\beta)&=\tilde{f}_n(c|\alpha_{n-1}\gamma)=\tilde{f}_n(a|\alpha_{n-2}\beta_2)
\\
  &= -2(n-1)abac|\alpha_{n+3}-2(n-1)abac|\beta_{n+3}-2(n-1)abac|\gamma_{n+3},
\end{align*}
and $\tilde{f}_n(x)=0$ for
\begin{equation}
\label{eq:fn-homology}
\begin{split}
x\in \big\{ 
& 1|\alpha_n, 1\beta_n, 1|\gamma_n, 1|\alpha_{n-1}\beta, 1|\alpha_{n-1}\gamma,1|\alpha_{n-2}\beta_2,b|\alpha_n, c|\alpha_n, a|\beta_n,  c|\beta_n,a|\gamma_n,
\\
& b|\gamma_n,a|\alpha_{n-1}\beta,c|\alpha_{n-1}\beta,
 a|\alpha_{n-1}\gamma,b|\alpha_{n-1}\gamma,b|\alpha_{n-2}\beta_2,c|\alpha_{n-2}\beta_2 
 \big\} .
 \end{split}
\end{equation}
Finally, if $n\geqslant 2$ is even, then  
\begin{align*}
\tilde{f}_n(1|\alpha_n)&=\tilde{f}_n(1|\beta_n)=\tilde{f}_n(1|\gamma_n)
=8bac|\alpha_{n+3}+8abc|\beta_{n+3}-8aba|\gamma_{n+3}-2abc|\alpha_{n+2}\beta
\\
& \phantom{ = \;}
+2aba|\alpha_{n+2}\gamma-2bac|\alpha_{n+1}\beta_2,
\\
\tilde{f}_n(1|\alpha_{n-1}\beta)&=\tilde{f}_n(1|\alpha_{n-1}\gamma)=0,\quad 
\tilde{f}_n(1|\alpha_{n-2}\beta_2)=6(n-2)(bac|\alpha_{n+3}+abc|\beta_{n+3}-aba|\gamma_{n+3}),
\end{align*}
and $\tilde{f}_n(x)=0$ for $x\in A_1\otimes (A^!_{-n})^*$. 

\subsection{Computation of the boundaries}
\label{subsection: boundaries} 

In this subsection, we will explicitly construct bases $\tilde{\mathfrak{B}}_{n,m}$ and $\mathfrak{B}_{n,m}$ of the $\Bbbk$-vector spaces $\tilde{B}_{n,m}=\Img(\tilde{\partial}_{n+1,m-1})$ and $B_{n,m}=\Img(\partial_{n+1,m-1})$ for $m\in \llbracket 0,4\rrbracket$ and $n\in\NN_0$ respectively, defined before Proposition \ref{bdh}. 
This will be done by simply applying the corresponding differential $\tilde{\partial}_{n+1,m-1}$ or $\partial_{n+1,m-1}$ to the usual basis of its domain and extracting a linearly independent generating subset. 

\subsubsection{\texorpdfstring{Computation of $\tilde{\mathfrak{B}}_{n,m}$}{Computation of Bnm}}
\label{subsubsection:boundaries-homology-1}

Recall that $\tilde{B}_{n,m}=\Img(\tilde{\partial}_{n+1,m-1})$ and $\tilde{\partial}_{n,m}:\tilde{K}_{n,m}=A_m\otimes (A^!_{-n})^*\to \tilde{K}_{n-1,m+1}=A_{m+1}\otimes (A^!_{-(n-1)})^*$ was defined in Subsection \ref{subsection:rds}. 
Obviously, 
$\tilde{B}_{n,0}=\Img(\tilde{\partial}_{n+1,-1})=0$
for $n\in\NN_0$. 
Then we define 
$\tilde{\mathfrak{B}}_{n,0}=\emptyset$
for $n\in\NN_0$. 

Suppose $m=1$. Table \ref{tpar1} shows that $\tilde{\partial}_{1,0}(\tilde{K}_{1,0})=0$, so 
$\tilde{B}_{0,1}=\Img(\tilde{\partial}_{1,0})=0$.
We define 
$\tilde{\mathfrak{B}}_{0,1}=\emptyset$.
For $n\in\NN$ with $n$ odd, Tables \ref{tparevenalphan} - \ref{tparevenalphabeta2} show that
\begin{align*}
& a|\alpha_n=(1/2)\tilde{\partial}_{n+1,0}(1|\alpha_{n+1}),
\quad 
b|\beta_n=(1/2)\tilde{\partial}_{n+1,0}(1|\beta_{n+1}), 
\quad 
c|\gamma_n=(1/2)\tilde{\partial}_{n+1,0}(1|\gamma_{n+1}), 
\\
& (a+c)|(\beta_n+\alpha_{n-1}\beta)+(b+a)|(\gamma_n+\alpha_{n-1}\gamma)+(c+b)|(\alpha_{n}+\alpha_{n-2}\beta_2) =\tilde{\partial}_{n+1,0}(1|\alpha_{n}\beta)
\\
& \qquad =\tilde{\partial}_{n+1,0}(1|\alpha_{n}\gamma), 
\\
& a|\alpha_{n-2}\beta_2+b|\alpha_{n-1}\beta+c|\alpha_{n-1}\gamma=(1/2)\tilde{\partial}_{n+1,0}(1|\alpha_{n-1}\beta_2).
\end{align*}
These five elements are linearly independent if none of them vanishes, so they form a $\Bbbk$-basis of $\tilde{B}_{n,1}$. 
If $n=1$, we define a basis of $\tilde{B}_{1,1}$ by \[
\tilde{\mathfrak{B}}_{1,1}= \big\{ \suline{a|\alpha}, 
\suline{b|\beta},
\suline{c|\gamma}, (\suline{a}+c)|\suline{\beta}+(b+a)|\gamma+(c+b)|\alpha \big\}.
\] 
If $n\geqslant 3$ is odd, we define a basis of $\tilde{B}_{n,1}$ by  
\begin{align*}
\tilde{\mathfrak{B}}_{n,1}= \big\{ 
& \suline{a|\alpha_{n}}, 
\suline{b|\beta_n}, 
\suline{c|\gamma_{n}}, (a+\suline{c})|(\beta_n+\suline{\alpha_{n-1}\beta})+(b+a)|(\gamma_n+\alpha_{n-1}\gamma)+(c+b)|(\alpha_{n}+\alpha_{n-2}\beta_2),\\
& \suline{a|\alpha_{n-2}\beta_2}+b|\alpha_{n-1}\beta+c|\alpha_{n-1}\gamma \big\}.
\end{align*}
If $n\geqslant 2$ is even, Tables \ref{tparoddalphan} - \ref{tparoddalphabeta2} show that
\begin{align*}
	& 0=\tilde{\partial}_{n+1,0}(1|\alpha_{n+1})=\tilde{\partial}_{n+1,0}(1|\beta_{n+1})=\tilde{\partial}_{n+1,0}(1|\gamma_{n+1}),\\
	& (c-a)|(\alpha_{n-1}\beta-\alpha_{n-1}\gamma)=\tilde{\partial}_{n+1,0}(1|\alpha_{n}\beta),\quad 
	(a-b)|(\alpha_{n-1}\beta-\alpha_{n-1}\gamma)=\tilde{\partial}_{n+1,0}(1|\alpha_{n}\gamma),\\
	& (b-c)|(\alpha_{n-1}\beta-\alpha_{n-1}\gamma)=\tilde{\partial}_{n+1,0}(1|\alpha_{n-1}\beta_2)=-\tilde{\partial}_{n+1,0}(1|\alpha_{n}\beta)-\tilde{\partial}_{n+1,0}(1|\alpha_{n}\gamma).
\end{align*}
Since the elements $(c-a)|(\alpha_{n-1}\beta-\alpha_{n-1}\gamma)$ and $(a-b)|(\alpha_{n-1}\beta-\alpha_{n-1}\gamma) $ are linearly independent, we define a basis of $\tilde{B}_{n,1}$ by 
\[ 
	\tilde{\mathfrak{B}}_{n,1}=\big\{
	(\suline{c}-a)|(\alpha_{n-1}\beta-\suline{\alpha_{n-1}\gamma}),
	(a-\suline{b})|(\suline{\alpha_{n-1}\beta}-\alpha_{n-1}\gamma) \big\}.
\] 
The dimension of $\tilde{B}_{n,1}$ is then given by 
\begin{equation}\label{dimbtilde1}
	\begin{split}
		\operatorname{dim} \tilde{B}_{n,1} =
		\begin{cases}
			0,  & \text{if $n=0$}, 
			\\
			4,  & \text{if $n=1$}, 
			\\
			2,  & \text{if $n\geqslant 2$ is even},
			\\
			5,  & \text{if $n\geqslant 3$ is odd}.
			\end{cases} 
	\end{split}
\end{equation}

Suppose now $m=2$. Table \ref{tpar1} shows that $\tilde{B}_{0,2}$ is spanned by $(ab-ba)|\epsilon^!$, $(ba+ac+bc)|\epsilon^!$ and $(ab+bc+ac)|\epsilon^!$. Since $(ab+ba)|\epsilon^!$ and $(ba+ac+bc)|\epsilon^!$ are linearly independent, and $(ab+bc+ac)|\epsilon^!=(ab-ba)|\epsilon^!+(ba+ac+bc)|\epsilon^!$, we see that  
\[ \tilde{\mathfrak{B}}_{0,2}=\big\{ (ab-\suline{ba})|\suline{\epsilon^!}, (ba+\suline{ac}+bc)|\suline{\epsilon^!} \big\}\]
is a basis of $\tilde{B}_{0,2}$. 
If $n\in\NN$ is odd, let 
\begin{align*}
	\mathcal{E}_{n,2}=\big\{
		 e_{1,n,2}& =(\suline{ab }+ba)|\suline{\alpha_n} =\tilde{\partial}_{n+1,1}(b|\alpha_{n+1}), \\
		 e_{2,n,2}& =(\suline{bc}-ba-ac)|\suline{\alpha_n}=-\tilde{\partial}_{n+1,1}(b|\alpha_{n+1})-\tilde{\partial}_{n+1,1}(c|\alpha_{n+1}), \\
		 e_{3,n,2}& =(\suline{ab}+ba)|\suline{\beta_n}=\tilde{\partial}_{n+1,1}(a|\beta_{n+1}),\\
		 e_{4,n,2}& =(\suline{bc}-ba-ac)|\suline{\beta_n} =\tilde{\partial}_{n+1,1}(c|\beta_{n+1}),\\
		 e_{5,n,2}& =(\suline{ab}+ba)|\suline{\gamma_n} =-\tilde{\partial}_{n+1,1}(a|\gamma_{n+1})-\tilde{\partial}_{n+1,1}(b|\gamma_{n+1}),\\
		 e_{6,n,2}& =(\suline{bc}-ba-ac)|\suline{\gamma_n}= \tilde{\partial}_{n+1,1}(b|\gamma_{n+1}),\\
		 e_{7,n,2}& =bc|(\alpha_n+\alpha_{n-2}\beta_2)-ac|(\beta_n+\alpha_{n-1}\beta)+\suline{ab}|(\gamma_n+\suline{\alpha_{n-1}\gamma})= \tilde{\partial}_{n+1,1}(b|\alpha_{n}\beta)
		\big\}.
\end{align*}
Then we define the set 
$\tilde{\mathfrak{B}}_{1,2}=\mathcal{E}_{1,2}$, 
and  
\begin{align*}
\tilde{\mathfrak{B}}_{n,2}=\mathcal{E}_{n,2}\cup \big\{
e_{8,n,2} & =(ab+\suline{ba})|\suline{\alpha_{n-1}\gamma} =\tilde{\partial}_{n+1,1}(b|\alpha_n\beta)+ \tilde{\partial}_{n+1,1}(a|\alpha_n\gamma)- e_{5,n,2}, 
 \\
e_{9,n,2} & =(\suline{ab}+ba)|(\suline{\alpha_{n-1}\beta}+\alpha_{n-2}\beta_2)
 \\
& = \tilde{\partial}_{n+1,1}(a|\alpha_{n-1}\beta_2)+\tilde{\partial}_{n+1,1}(b|\alpha_{n-1}\beta_2)+e_{8,n,2},
\\
e_{10,n,2}& =(bc-ba-\suline{ac})|\suline{\alpha_{n-2}\beta_2 }
\\
& =\tilde{\partial}_{n+1,1}(c|\alpha_n\beta)-\tilde{\partial}_{n+1,1}(a|\alpha_n\beta)+e_{8,n,2}+e_{5,n,2}-e_{2,n,2},
\\
e_{11,n,2} & =(ab+ba)|\alpha_{n-1}\beta-(\suline{bc}-ba-ac)|\suline{\alpha_{n-1}\gamma}=\tilde{\partial}_{n+1,1}(a|\alpha_{n-1}\beta_2)+e_{8,n,2},
\\
e_{12,n,2} & =(\suline{bc}-ba-ac)|\suline{\alpha_{n-1}\beta}-(ab+ba)|\alpha_{n-2}\beta_2 \\
& =\tilde{\partial}_{n+1,1}(c|\alpha_{n-1}\beta_2)+e_{10,n,2} 
\big\}
\end{align*}
for $n\geqslant 3$ with $n$ odd.
We will show that $\tilde{\mathfrak{B}}_{n,2}$ is a basis of $\tilde{B}_{n,2}$ for $n\in\NN$ with $n$ odd.
As noted before, $\tilde{\mathfrak{B}}_{n,2}\subseteq \tilde{B}_{n,2}$. 
Since 
\begin{align*}
	&\tilde{\partial}_{n+1,1}(a|\alpha_{n+1})=\tilde{\partial}_{n+1,1}(b|\beta_{n+1})=\tilde{\partial}_{n+1,1}(c|\gamma_{n+1})=0, \\
	&\tilde{\partial}_{n+1,1}(b|\alpha_{n+1})=e_{1,n,2}, \quad 
	\tilde{\partial}_{n+1,1}(c|\alpha_{n+1})=-e_{1,n,2}-e_{2,n,2}, \quad 
	\tilde{\partial}_{n+1,1}(a|\beta_{n+1})=e_{3,n,2}, \\
	&\tilde{\partial}_{n+1,1}(c|\beta_{n+1})=e_{4,n,2}, \quad 
	\tilde{\partial}_{n+1,1}(a|\gamma_{n+1})=-e_{5,n,2}-e_{6,n,2}, \quad 
	\tilde{\partial}_{n+1,1}(b|\gamma_{n+1})=e_{6,n,2}, \\
	&\tilde{\partial}_{n+1,1}(a|\alpha_n\beta)=-e_{2,n,2}-e_{3,n,2}-e_{4,n,2}+e_{7,n,2}-e_{9,n,2}-e_{10,n,2}-e_{12,n,2}, \\ 
	& \tilde{\partial}_{n+1,1}(b|\alpha_n\beta)=e_{7,n,2}, \\
	& \tilde{\partial}_{n+1,1}(c|\alpha_n\beta)=-e_{3,n,2}-e_{4,n,2}-e_{5,n,2}+e_{7,n,2}-e_{8,n,2}-e_{9,n,2}-e_{12,n,2}, \\
	& \tilde{\partial}_{n+1,1}(a|\alpha_n\gamma)=e_{5,n,2}-e_{7,n,2}+e_{8,n,2}, \quad
	 \tilde{\partial}_{n+1,1}(b|\alpha_n\gamma)=-\tilde{\partial}_{n+1,1}(c|\alpha_n\beta)+e_{2,n,2}+e_{10,n,2} , \\
	&\tilde{\partial}_{n+1,1}(c|\alpha_n\gamma)=-\tilde{\partial}_{n+1,1}(b|\alpha_n\beta)+e_{2,n,2}+e_{10,n,2}, \quad 
	\tilde{\partial}_{n+1,1}(a|\alpha_{n-1}\beta_2)=-e_{8,n,2}+e_{11,n,2}, \\
	& \tilde{\partial}_{n+1,1}(b|\alpha_{n-1}\beta_2)=e_{9,n,2}-e_{11,n,2},\quad 
	 \tilde{\partial}_{n+1,1}(c|\alpha_{n-1}\beta_2)=-e_{10,n,2}+e_{12,n,2},
\end{align*}
the elements in $\tilde{\mathfrak{B}}_{n,2}$ span the space $\tilde{B}_{n,2}$. 
By Fact \ref{indep}, the elements in $\tilde{\mathfrak{B}}_{n,2}$ are linearly independent, so $\tilde{\mathfrak{B}}_{n,2}$ is a basis of $\tilde{B}_{n,2}$, as claimed. 
If $n\geqslant 2$ is even, let
\begin{align*}
	\mathcal{G}_{n,2}=\big\{ 
		g_{1,n,2} & =(\suline{ba}-ab)|\suline{\alpha_n}=-\tilde{\partial}_{n+1,1}(b|\alpha_{n+1}), 
		\\
		g_{2,n,2} & =(bc+ba+\suline{ac})|\suline{\alpha_n}=\tilde{\partial}_{n+1,1}(c|\alpha_{n+1})-\tilde{\partial}_{n+1,1}(b|\alpha_{n+1}),
		\\
		g_{3,n,2} & =(\suline{ba}-ab)|\suline{\beta_n}=\tilde{\partial}_{n+1,1}(a|\beta_{n+1}),
		\quad
		g_{4,n,2}=(bc+ba+\suline{ac})|\suline{\beta_n}=\tilde{\partial}_{n+1,1}(c|\beta_{n+1}), 
		\\
		g_{5,n,2} & =(\suline{ba}-ab)|\suline{\gamma_n}=\tilde{\partial}_{n+1,1}(a|\gamma_{n+1})-\tilde{\partial}_{n+1,1}(b|\gamma_{n+1}),
		\\
		g_{6,n,2} & =(bc+ba+\suline{ac})|\suline{\gamma_n}=-\tilde{\partial}_{n+1,1}(b|\gamma_{n+1}),
		\\
		g_{9,n,2} & =ab|\alpha_{n-1}\beta-\suline{ba|\alpha_{n-1}\gamma} =
		(1/3)\big[ \tilde{\partial}_{n+1,1}(2a|\alpha_n\beta+2c|\alpha_n\beta +3b|\alpha_n\gamma
		+b|\alpha_{n-1}\beta_2
		\\
 & \phantom{= \; }
 +c|\alpha_{n-1}\beta_2)
 -2g_{1,n,2}+g_{2,n,2}+2g_{3,n,2}+2g_{4,n,2}-3g_{6,n,2} \big], 
\\
		g_{10,n,2} & = \suline{ba|\alpha_{n-1}\beta}-ab|\alpha_{n-1}\gamma=(1/3)\tilde{\partial}_{n+1,1}(a|\alpha_{n-1}\beta_2-b|\alpha_n\beta),
		\\
		g_{11,n,2} & =\suline{ac|\alpha_{n-1}\beta}+(ab+bc)|\alpha_{n-1}\gamma=-(1/3)\tilde{\partial}_{n+1,1}(b|\alpha_n\beta+2a|\alpha_{n-1}\beta_2),
		\\
g_{12,n,2} & =(ab+bc)|\alpha_{n-1}\beta+\suline{ac|\alpha_{n-1}\gamma}=(1/3)\big[ \tilde{\partial}_{n+1,1}(c|\alpha_n\beta-2a|\alpha_n\beta-b|\alpha_{n-1}\beta_2 
\\
& \phantom{= \ \, }
-c|\alpha_{n-1}\beta_2)
+2g_{1,n,2}-g_{2,n,2}-2g_{3,n,2}+g_{4,n,2} \big]
\big\},
\end{align*}
Then we define  
$\tilde{\mathfrak{B}}_{2,2}=\mathcal{G}_{2,2}$,
and 
\begin{align*}
 \tilde{\mathfrak{B}}_{n,2}=\mathcal{G}_{n,2}\cup \big\{
g_{7,n,2} & =(ba-\suline{ab})|\suline{\alpha_{n-2}\beta_2 } =(1/3)[\tilde{\partial}_{n+1,1}(a|\alpha_n\beta+c|\alpha_n\beta-b|\alpha_{n-1}\beta_2 
\\
& \phantom{ = \ \,}
-c|\alpha_{n-1}\beta_2)
-g_{1,n,2}-g_{2,n,2}-2g_{3,n,2}+g_{4,n,2}-3g_{5,n,2} ],\\
g_{8,n,2} & =(\suline{bc}+ba+ac)|\suline{\alpha_{n-2}\beta_2 } =(1/3)\big[ \tilde{\partial}_{n+1,1}(2a|\alpha_n\beta+2c|\alpha_n\beta+b|\alpha_{n-1}\beta_2 
\\
& \phantom{= \ \, }
+c|\alpha_{n-1}\beta_2)
-2g_{1,n,2}-2g_{2,n,2}+2g_{3,n,2}-g_{4,n,2}-3g_{6,n,2} \big]
\big\}
\end{align*}
for $n\geqslant 4$ with $n$ even.
We will show that $\tilde{\mathfrak{B}}_{n,2}$ is a basis of $\tilde{B}_{n,2}$ for $n\geqslant 2$ with $n$ even. 
From the definition, we see that $\tilde{\mathfrak{B}}_{n,2}\subseteq \tilde{B}_{n,2}$. Since
\begin{align*} 
& \tilde{\partial}_{n+1,1}(a|\alpha_{n+1})=\tilde{\partial}_{n+1,1}(b|\beta_{n+1})=\tilde{\partial}_{n+1,1}(c|\gamma_{n+1})=0, \\ 
&\tilde{\partial}_{n+1,1}(b|\alpha_{n+1})=-g_{1,n,2},\quad 
\tilde{\partial}_{n+1,1}(c|\alpha_n)=g_{2,n,2}-g_{1,n,2}, \quad 
\tilde{\partial}_{n+1,1}(a|\beta_{n+1})=g_{3,n,2}, \\
&\tilde{\partial}_{n+1,1}(c|\beta_{n+1})=g_{4,n,2}, \quad
\tilde{\partial}_{n+1,1}(a|\gamma_{n+1})=g_{5,n,2}-g_{6,n,2}, \quad
\tilde{\partial}_{n+1,1}(b|\gamma_{n+1})=-g_{6,n,2}, \\
& \tilde{\partial}_{n+1,1}(a|\alpha_{n}\beta)=g_{1,n,2}+g_{5,n,2}+g_{7,n,2}-g_{12,n,2},\quad 
 \tilde{\partial}_{n+1,1}(b|\alpha_n\beta)=-2g_{10,n,2}-g_{11,n,2},\\
& \tilde{\partial}_{n+1,1}(c|\alpha_n\beta)=g_{2,n,2}+g_{6,n,2}+g_{8,n,2}+g_{12,n,2}, \\
&\tilde{\partial}_{n+1,1}(a|\alpha_n\gamma)=g_{1,n,2}-g_{2,n,2}+g_{3,n,2}-g_{4,n,2}+g_{7,n,2}-g_{8,n,2}-g_{9,n,2}, \\
&\tilde{\partial}_{n+1,1}(b|\alpha_n\gamma)=-g_{2,n,2}-g_{4,n,2}-g_{8,n,2}+g_{9,n,2}, \quad 
\tilde{\partial}_{n+1,1}(c|\alpha_n\gamma)=g_{10,n,2}+2g_{11,n,2}, \\ 
& \tilde{\partial}_{n+1,1}(a|\alpha_{n-1}\beta_2)=g_{10,n,2}-g_{11,n,2}, \\
& \tilde{\partial}_{n+1,1}(b|\alpha_{n-1}\beta_2)=-g_{3,n,2}-g_{5,n,2}-g_{7,n,2}+g_{9,n,2}-g_{12,n,2},\\
&\tilde{\partial}_{n+1,1}(c|\alpha_{n-1}\beta_2)=-g_{3,n,2}+g_{4,n,2}-g_{5,n,2}+g_{6,n,2}-g_{7,n,2}+g_{8,n,2}-g_{9,n,2}+g_{12,n,2},
\end{align*}
the elements in $\tilde{\mathfrak{B}}_{n,2}$ span the space $\tilde{B}_{n,2}$. 
By Fact \ref{indep}, the elements in $\tilde{\mathfrak{B}}_{n,2}$ are linearly independent, so $\tilde{\mathfrak{B}}_{n,2}$ is a basis of $\tilde{B}_{n,2}$, as claimed. 
The dimension of $\tilde{B}_{n,2}$ is thus given by 
\begin{equation}\label{dimbtilde2}
	\begin{split}
		\operatorname{dim} \tilde{B}_{n,2} =
		\begin{cases}
			2,  &  \text{if $n=0$}, 
			\\
			7,  &  \text{if $n=1$}, 
			\\
			10,  &  \text{if $n=2$}, 
			\\
			12,  &  \text{if $n\geqslant 3$}. 
			\end{cases} 
	\end{split}
\end{equation}

Suppose now $m=3$. Table \ref{tpar1} shows that $\tilde{\partial}_{1,2}$ is surjective. We thus define a basis of $\tilde{B}_{0,3}$ by the usual basis of $\tilde{K}_{0,3}$.
If $n\in\NN$ is odd, let 
\begin{align*}
	\mathcal{E}_{n,3}=\big\{ 
		& e_{1,n,3}=\suline{aba|\alpha_n}=\tilde{\partial}_{n+1,2}(ab|\alpha_{n+1}),   \quad 
		e_{2,n,3}=\suline{abc|\alpha_n}=-\tilde{\partial}_{n+1,2}(ac|\alpha_{n+1}),\\
	   & e_{3,n,3}=\suline{aba|\beta_n}=\tilde{\partial}_{n+1,2}(ab|\beta_{n+1}),\quad 
		e_{4,n,3}=\suline{bac|\beta_n}=-\tilde{\partial}_{n+1,2}(bc|\beta_{n+1}),\\
	   & e_{5,n,3}=\suline{abc|\gamma_n}=-\tilde{\partial}_{n+1,2}(ac|\gamma_{n+1}),\quad 
		e_{6,n,3}=\suline{bac|\gamma_n}=-\tilde{\partial}_{n+1,2}(bc|\gamma_{n+1}), \\
	   & e_{7,n,3}=\suline{abc}|(\suline{\beta_n}+\alpha_{n-1}\beta)+aba|(\gamma_n+\alpha_{n-1}\gamma)=(1/3)\tilde{\partial}_{n+1,2}((ba-ac)|\alpha_n\beta), \\
	   & e_{8,n,3}=aba|(\gamma_n+\alpha_{n-1}\gamma)+\suline{bac}|(\suline{\alpha_n}+\alpha_{n-2}\beta_2)=(1/3)\tilde{\partial}_{n+1,2}((2ba+ac)|\alpha_n\beta)	
	\big\}.
\end{align*}
Then we define 
$\tilde{\mathfrak{B}}_{1,3}=\mathcal{E}_{1,3}$, 
and 
\begin{align*}
\tilde{\mathfrak{B}}_{n,3}=\mathcal{E}_{n,3}\cup \big\{
e_{9,n,3} & =aba|\alpha_{n-1}\beta+(abc+bac)|\alpha_{n-1}\gamma+aba|\alpha_{n-2}\beta_2=\tilde{\partial}_{n+1,2}(ba|\alpha_{n-1}\beta_2), 
\\
e_{10,n,3} & =bac|\alpha_{n-1}\beta+bac|\alpha_{n-1}\gamma+(aba-abc)|\alpha_{n-2}\beta_2=-\tilde{\partial}_{n+1,2}(bc|\alpha_{n-1}\beta_2),
\\
 e_{11,n,3} & =(aba+bac)|\alpha_{n-1}\beta+(bac-abc)|\alpha_{n-1}\gamma 
 \\
&  =-\tilde{\partial}_{n+1,2}(bc|\alpha_n\beta)-e_{3,n,3}-e_{4,n,3}+e_{5,n,3}-e_{6,n,3}, 
\\
e_{12,n,3} & =(abc-bac)|\alpha_{n-1}\gamma+(aba+abc)|\alpha_{n-2}\beta_2 
\\
&  =\tilde{\partial}_{n+1,2}((ab+bc)|\alpha_n\beta)-e_{1,n,3}-e_{2,n,3}-e_{5,n,3}+e_{6,n,3} 
\big\} 
\end{align*}
for $n\geqslant 3$ with $n$ odd.
We now show that $\tilde{\mathfrak{B}}_{n,3}$ is a basis of $\tilde{B}_{n,3}$ for $n\in\NN$ with $n$ odd.
As noted before, $ \tilde{\mathfrak{B}}_{n,3}\subseteq \tilde{B}_{n,3}$. 
Since
\begin{align*}
&\tilde{\partial}_{n+1,2}(ab|\alpha_{n+1})=\tilde{\partial}_{n+1,2}(ba|\alpha_{n+1})=e_{1,n,3}, \quad
\tilde{\partial}_{n+1,2}(bc|\alpha_{n+1})=e_{2,n,3}-e_{1,n,3},\\
&\tilde{\partial}_{n+1,2}(ac|\alpha_{n+1})=-e_{2,n,3}, \quad
\tilde{\partial}_{n+1,2}(ab|\beta_{n+1})=\tilde{\partial}_{n+1,2}(ba|\beta_{n+1})=e_{3,n,3}, \\
& \tilde{\partial}_{n+1,2}(bc|\beta_{n+1})=-e_{4,n,3}, \quad 
\tilde{\partial}_{n+1,2}(ac|\beta_{n+1})=e_{4,n,3}-e_{3,n,3}, \\
&\tilde{\partial}_{n+1,2}(ab|\gamma_{n+1})=\tilde{\partial}_{n+1,2}(ba|\gamma_{n+1})=e_{5,n,3}+e_{6,n,3}, \quad 
\tilde{\partial}_{n+1,2}(bc|\gamma_{n+1})=-e_{6,n,3}, \\
& \tilde{\partial}_{n+1,2}(ac|\gamma_{n+1})=-e_{5,n,3},\quad 
\tilde{\partial}_{n+1,2}(ab|\alpha_n\beta)=e_{1,n,3}+e_{2,n,3}+e_{3,n,3}+e_{4,n,3}+e_{11,n,3}+e_{12,n,3}, \\
&\tilde{\partial}_{n+1,2}(bc|\alpha_n\beta)=-e_{3,n,3}-e_{4,n,3}+e_{5,n,3}-e_{6,n,3}-e_{11,n,3}, \\
& \tilde{\partial}_{n+1,2}(ba|\alpha_n\beta)=\tilde{\partial}_{n+1,2}(ab|\alpha_n\gamma)=e_{7,n,3}+e_{8,n,3}, \quad
\tilde{\partial}_{n+1,2}(ac|\alpha_n\beta)=-2e_{7,n,3}+e_{8,n,3}, \\
&\tilde{\partial}_{n+1,2}(bc|\alpha_n\gamma)=e_{7,n,3}-2e_{8,n,3}, \\
& \tilde{\partial}_{n+1,2}(ba|\alpha_n\gamma)=e_{1,n,3}+e_{2,n,3}+e_{3,n,3}+e_{4,n,3}+ e_{11,n,3}+e_{12,n,3}, \\
&\tilde{\partial}_{n+1,2}(ac|\alpha_n\gamma)=-e_{1,n,3}-e_{2,n,3}-e_{5,n,3}+e_{6,n,3}  -e_{12,n,3},\\
& \tilde{\partial}_{n+1,2}(ab|\alpha_{n-1}\beta_2)=\tilde{\partial}_{n+1,2}(ba|\alpha_{n-1}\beta_2)=e_{9,n,3}, \quad 
\tilde{\partial}_{n+1,2}(bc|\alpha_{n-1}\beta_2)=-e_{10,n,3}, \\
& \tilde{\partial}_{n+1,2}(ac|\alpha_{n-1}\beta_2)=e_{10,n,3}-e_{9,n,3},
\end{align*}
the elements in $\tilde{\mathfrak{B}}_{n,3}$ span the space $\tilde{B}_{n,3}$. 
By Fact \ref{indep}, the elements $e_{\ell,n,3}$ for $\ell \in \llbracket 1,8\rrbracket$ are linearly independent.  
The reader can easily verify that the elements 
$e_{\ell,n,3}$ for $\ell \in \llbracket 9,12\rrbracket$ are linearly independent. 
Since the underlined terms of $e_{\ell,n,3}$ for $\ell \in \llbracket 1,8\rrbracket$ do not appear in $e_{\ell,n,3}$ for $\ell \in \llbracket 9,12\rrbracket$, 
the elements in $\tilde{\mathfrak{B}}_{n,3}$ are linearly independent. So $\tilde{\mathfrak{B}}_{n,3}$ is a basis of $\tilde{B}_{n,3}$, as claimed. 
If $n\geqslant 2$ is even, let 
\begin{align*}
\tilde{\mathfrak{B}}_{n,3}=\big\{  
g_{1,n,3} & =\suline{aba|\alpha_n}=\tilde{\partial}_{n+1,2}(ba|\alpha_{n+1}),\quad
g_{2,n,3}=\suline{abc|\alpha_n}=\tilde{\partial}_{n+1,2}(ac|\alpha_{n+1}), 
\\
g_{3,n,3} & =\suline{aba|\beta_n}=\tilde{\partial}_{n+1,2}(ab|\beta_{n+1}),\quad 
g_{4,n,3}=\suline{bac|\beta_n}=\tilde{\partial}_{n+1,2}(bc|\beta_{n+1}), 
\\
g_{5,n,3} & = \suline{abc|\gamma_n}=-\tilde{\partial}_{n+1,2}(ac|\gamma_{n+1}),\quad
g_{6,n,3}= \suline{bac|\gamma_n}=-\tilde{\partial}_{n+1,2}(bc|\gamma_{n+1}), 
\\
g_{7,n,3} & = bac|(\alpha_n+\alpha_{n-2}\beta_2)+\suline{aba}|(\suline{\gamma_n}+\alpha_{n-2}\beta_2) 
\\
&  =\tilde{\partial}_{n+1,2}((ab+bc)|\alpha_n\beta)-g_{1,n,3}-g_{6,n,3}, 
\\
g_{8,n,3} & =\suline{abc}|(\suline{\beta_n}+\alpha_{n-2}\beta_2)+aba|(\gamma_n+\alpha_{n-2}\beta_2) 
\\
& =\tilde{\partial}_{n+1,2}((ba+ac)|\alpha_{n-1}\beta_2)-g_{3,n,3}-g_{5,n,3},
\\
g_{9,n,3} & =\suline{bac}|(\suline{\alpha_{n-1}\beta}+\alpha_{n-1}\gamma-\alpha_n-\alpha_{n-2}\beta_2)+aba|(\gamma_n+\alpha_{n-2}\beta_2) 
\\
&  = \tilde{\partial}_{n+1,2}(ab|\alpha_{n}\beta+ac|\alpha_n\gamma)-e_{1,n,3}+e_{2,n,3}-e_{7,n,3}+e_{8,n,3}, \\
g_{10,n,3} & =\suline{aba}|(\alpha_{n-1}\beta+\suline{\alpha_{n-1}\gamma}-2\gamma_n-2\alpha_{n-2}\beta_2) 
\\
&  =-\tilde{\partial}_{n+1,2}(ab|\alpha_n\beta+ba|\alpha_{n-1}\beta_2)+g_{1,n,3}+g_{3,n,3}, 
\\
g_{11,n,3} & =(bac-aba)|\alpha_{n-1}\beta-\suline{abc|\alpha_{n-1}\gamma}+aba|(\gamma_n+\alpha_{n-2}\beta_2)
\\
&  =\tilde{\partial}_{n+1,2}(ab|\alpha_n\beta)-g_{1,n,3},
\\
g_{12,n,3} & =\suline{abc|\alpha_{n-1}\beta}+(aba-bac)|\alpha_{n-1}\gamma-aba|(\gamma_n+\alpha_{n-2}\beta_2)
\\
& =\tilde{\partial}_{n+1,2}(ba|\alpha_n\beta)+g_{1,n,3}
 \big\}.
\end{align*}
We then show that $\tilde{\mathfrak{B}}_{n,3}$ is a basis of $\tilde{B}_{n,3}$.
It follows from the definition that $\tilde{\mathfrak{B}}_{n,3}\subseteq \tilde{B}_{n,3}$.
Since 
\begin{align*}
& \tilde{\partial}_{n+1,2}(ab|\alpha_{n+1})=-\tilde{\partial}_{n+1,2}(ba|\alpha_{n+1})=-g_{1,n,3}, \quad \tilde{\partial}_{n+1,2}(bc|\alpha_{n+1})=g_{1,n,3}+g_{2,n,3}, 
\\
& \tilde{\partial}_{n+1,2}(ac|\alpha_{n+1})=g_{2,n,3},\quad 
\tilde{\partial}_{n+1,2}(ab|\beta_{n+1})=-\tilde{\partial}_{n+1,2}(ba|\beta_{n+1})=g_{3,n,3}, 
\\
& \tilde{\partial}_{n+1,2}(bc|\beta_{n+1})=g_{4,n,3},\quad 
 \tilde{\partial}_{n+1,2}(ac|\beta_{n+1})=g_{3,n,3}+g_{4,n,3}, 
 \\
& \tilde{\partial}_{n+1,2}(ab|\gamma_{n+1})=-\tilde{\partial}_{n+1,2}(ba|\gamma_{n+1})=g_{6,n,3}-g_{5,n,3}, \quad  
\tilde{\partial}_{n+1,2}(bc|\gamma_{n+1})=-g_{6,n,3},
\\
& \tilde{\partial}_{n+1,2}(ac|\gamma_{n+1})=-g_{5,n,3}, \quad
\tilde{\partial}_{n+1,2}(ab|\alpha_n\beta)=g_{1,n,3}+g_{11,n,3},
\\
& \tilde{\partial}_{n+1,2}(bc|\alpha_n\beta)=g_{6,n,3}+g_{7,n,3}-g_{11,n,3},\quad
\tilde{\partial}_{n+1,2}(ba|\alpha_n\beta)=g_{12,n,3}-g_{1,n,3},
\\
&\tilde{\partial}_{n+1,2}(ac|\alpha_n\beta)=g_{1,n,3}+g_{6,n,3}+g_{7,n,3},\quad 
 \tilde{\partial}_{n+1,2}(ab|\alpha_n\gamma)=-g_{2,n,3}+g_{4,n,3}+g_{7,n,3}-g_{8,n,3},
\\
&\tilde{\partial}_{n+1,2}(bc|\alpha_n\gamma)=-g_{4,n,3}+g_{9,n,3}+g_{12,n,3},\quad 
\tilde{\partial}_{n+1,2}(ba|\alpha_n\gamma)=g_{2,n,3}-g_{4,n,3}-g_{7,n,3}+g_{8,n,3},
\\
&\tilde{\partial}_{n+1,2}(ac|\alpha_n\gamma)=-g_{2,n,3}+g_{7,n,3}-g_{8,n,3}+g_{9,n,3}-g_{11,n,3}, 
\\
 & \tilde{\partial}_{n+1,2}(ab|\alpha_{n-1}\beta_2)=-g_{3,n,3}+g_{10,n,3}-g_{12,n,3},\quad 
\tilde{\partial}_{n+1,2}(bc|\alpha_{n-1}\beta_2)=g_{3,n,3}+g_{5,n,3}+g_{8,n,3}, 
\\ 
 & \tilde{\partial}_{n+1,2}(ba|\alpha_{n-1}\beta_2)=g_{3,n,3}-g_{10,n,3}-g_{11,n,3},
 \\
& \tilde{\partial}_{n+1,2}(ac|\alpha_{n-1}\beta_2)=g_{5,n,3}+g_{8,n,3}+g_{10,n,3}+g_{11,n,3}, 
\end{align*}
the elements in $\tilde{\mathfrak{B}}_{n,3}$ span the space $\tilde{B}_{n,3}$. 
By Fact \ref{indep}, the elements in $\tilde{\mathfrak{B}}_{n,3}$ are linearly independent, so $\tilde{\mathfrak{B}}_{n,3}$ is a basis of $\tilde{B}_{n,3}$, as claimed. 
Hence, the dimension of $\tilde{B}_{n,3}$ is given by 
\begin{equation}\label{dimbtilde3}
	\begin{split}
		\operatorname{dim} \tilde{B}_{n,3} =
		\begin{cases}
			3,  &  \text{if $n=0$}, 
			\\
			8,  &  \text{if $n=1$}, 
			\\
			12, &  \text{if $n\geqslant 2$}.
			\end{cases} 
	\end{split}
\end{equation}

Suppose now $m=4$. Table \ref{tpar1} tells us that
the usual basis of $\tilde{K}_{0,4}$
is a basis of $\tilde{B}_{0,4}$.
If $n\in\NN$ is odd, Tables \ref{tparevenalphan} - \ref{tparevenalphabeta2} show that $\tilde{B}_{n,4}$ is spanned by 
$\tilde{\partial}_{n+1,3}(aba|\alpha_{n}\beta)$, 
$\tilde{\partial}_{n+1,3}(abc|\alpha_{n}\beta)$
and 
$\tilde{\partial}_{n+1,3}(bac|\alpha_{n}\beta)$. 
Since 
\[     \tilde{\partial}_{n+1,3}(bac|\alpha_{n}\beta)=\tilde{\partial}_{n+1,3}(aba|\alpha_{n}\beta)-\tilde{\partial}_{n+1,3}(abc|\alpha_{n}\beta),     \] 
and the elements $\tilde{\partial}_{n+1,3}(aba|\alpha_{n}\beta)$ and 
$\tilde{\partial}_{n+1,3}(abc|\alpha_{n}\beta)$ are linearly independent, we define a basis of $\tilde{B}_{n,4}$ by  
\begin{align*} 
\tilde{\mathfrak{B}}_{n,4}=\big\{
	&\tilde{\partial}_{n+1,3}(aba|\alpha_{n}\beta)=\suline{abac}|(\alpha_n+\alpha_{n-2}\beta_2-\suline{\beta_n}-\alpha_{n-1}\beta), \\
	& \tilde{\partial}_{n+1,3}(aba|\alpha_{n}\beta)=\suline{abac}|(\alpha_n+\alpha_{n-2}\beta_2-\suline{\gamma_n}-\alpha_{n-1}\gamma) \big\}.
\end{align*}
If $n=2$, by Tables \ref{tparoddalphan} - \ref{tparoddalphabeta2}, we define a basis of $\tilde{B}_{2,4}$ by 
\[
\tilde{\mathfrak{B}}_{2,4}=\big\{ \suline{abac|\alpha_2}, \suline{abac|\beta_2},
\suline{abac|\gamma_2}, \suline{abac}|(\alpha\beta+\suline{\alpha\gamma}) \big\}.
\]
If $n\geqslant 4$ is even,  
we note that $abac|\alpha_{n-2}\beta_2=(1/2)\tilde{\partial}_{n+1,3}(abc|\alpha_n\beta-bac|\alpha_{n+1}+aba|\gamma_{n+1})$. So we can define a basis of $\tilde{B}_{n,4}$ by 
\[
\tilde{\mathfrak{B}}_{n,4}=\big\{ \suline{abac|\alpha_n}, 
\suline{abac|\beta_n}, 
\suline{abac|\gamma_n},
\suline{abac}|(\alpha_{n-1}\beta+\suline{\alpha_{n-1}\gamma}), 
\suline{abac|\alpha_{n-2}\beta_2}
\big\}.
\]
In conclusion, the dimension of $\tilde{B}_{n,4}$ is given by 
\begin{equation}\label{dimbtilde4}
	\begin{split}
		\operatorname{dim} \tilde{B}_{n,4} =
		\begin{cases}
			1,  & \text{if $n=0$}, 
			\\
			2,  & \text{if $n\in\NN$ is odd}, 
			\\
			4 , & \text{if $n= 2$}, 
			\\
			5 , & \text{if $n\geqslant 4$ is even}.
			\end{cases} 
	\end{split}
\end{equation}

\subsubsection{\texorpdfstring{Computation of $\mathfrak{B}_{n,m}$}{Computation of Bnm}}
\label{subsubsection:boundaries-homology-2}

Recall that $B_{n,m}=\Img(\partial_{n+1,m-1})$ and $\partial_{n,m}:\tilde{P}_{n,m}\to \tilde{P}_{n-1,m+1}$. 
Since $\partial_{n,m}=\tilde{\partial}_{n,m}$ for either $m=-1,0,1$ and $n\in \NN$, or $m=2,3$ and $n=1,2,3$, 
we get $B_{n,m}=\tilde{B}_{n,m}$ for either $m=0,1,2$ and $n\in \NN_0$, or $m=3,4$ and $n=0,1,2$. 
So we define a basis of $B_{n,m}$ by 
$\mathfrak{B}_{n,m}=\mathfrak{\tilde{B}}_{n,m}$
for either $m=0,1,2$ and $n\in \NN_0$, or $m=3,4$ and $n=0,1,2$.

Suppose $m=3$. 
Consider $\partial_{n+1,2}:\tilde{K}_{n+1,2}\oplus \omega_1\tilde{K}_{n-3,0}\to \tilde{K}_{n,3}\oplus \omega_1 \tilde{K}_{n-4,1}$. 
If $n=3$, the element $ 2bac|\alpha_3+2abc|\beta_3-2aba|\gamma_3-abc|\alpha_2\beta+aba|\alpha_2\gamma-bac|\alpha\beta_2=(1/6)\partial_{4,2}(\omega_1 1|\epsilon^!)$ is not in the space $\tilde{B}_{3,3}$.
So we define a basis of $B_{3,3}$ by  
\[\mathfrak{B}_{3,3}=\tilde{\mathfrak{B}}_{3,3}\cup \big\{ 2bac|\alpha_3+2abc|\beta_3-2aba|\gamma_3-abc|\alpha_2\beta+aba|\alpha_2\gamma-bac|\alpha\beta_2 \big\}.\]
If $n=5$, we define the set 
\begin{align*}
	\mathfrak{B}_{5,3}=
	 \tilde{\mathfrak{B}}_{5,3}\cup 
	 \big\{ 
	 & 4bac|\alpha_5+4abc|\beta_5-4abc|\gamma_5-abc|\alpha_{4}\beta+aba|\alpha_{4}\gamma-bac|\alpha_{3}\beta_2+\suline{\omega_1 a|\alpha}
	\\
	&
	=(1/2)\partial_{6,2}(\omega_1 1|\alpha_{2}),
	\\
	&  4bac|\alpha_5+4abc|\beta_5-4abc|\gamma_5-abc|\alpha_{4}\beta+aba|\alpha_{4}\gamma-bac|\alpha_{3}\beta_2+ \suline{\omega_1 b|\beta }
	\\
	& 
	=(1/2)\partial_{6,2}(\omega_1 1|\beta_{2}),
	\\
	& 
	4bac|\alpha_5+4abc|\beta_5-4abc|\gamma_5-abc|\alpha_{4}\beta+aba|\alpha_{4}\gamma-bac|\alpha_{3}\beta_2+ \suline{\omega_1 c|\gamma}
	\\
	&
	=(1/2)\partial_{6,2}(\omega_1 1|\gamma_{2}),
	\\
	& 
	\suline{\omega_1}[(\suline{a}+c)|\suline{\beta}+(b+a)|\gamma+(c+b)|\alpha]
	= \partial_{6,2}(\omega_1 1|\alpha\beta)=\partial_{6,2}(\omega_1 1|\alpha\gamma) \big\}.
\end{align*}
If $n\geqslant 7$ is odd, we define the set 
\begin{align*}
	\mathfrak{B}_{n,3}=
	& \tilde{\mathfrak{B}}_{n,3}\cup 
	\\
	& \big\{4bac|\alpha_n+4abc|\beta_n-4abc|\gamma_n-abc|\alpha_{n-1}\beta+aba|\alpha_{n-1}\gamma-bac|\alpha_{n-2}\beta_2+\suline{\omega_1 a|\alpha_{n-4}}
	\\
	& \phantom{ \big\{  } 
	=(1/2)\partial_{n+1,2}(\omega_1 1|\alpha_{n-3}),
	\\
	& \phantom{ \big\{  }  4bac|\alpha_n+4abc|\beta_n-4abc|\gamma_n-abc|\alpha_{n-1}\beta+aba|\alpha_{n-1}\gamma-bac|\alpha_{n-2}\beta_2+ \suline{\omega_1 b|\beta_{n-4} }
	\\
	& \phantom{ \big\{  } 
	=(1/2)\partial_{n+1,2}(\omega_1 1|\beta_{n-3}),
	\\
	& \phantom{ \big\{ } 
	4bac|\alpha_n+4abc|\beta_n-4abc|\gamma_n-abc|\alpha_{n-1}\beta+aba|\alpha_{n-1}\gamma-bac|\alpha_{n-2}\beta_2+ \suline{\omega_1 c|\gamma_{n-4}}
	\\
	& \phantom{ \big\{ } 
	=(1/2)\partial_{n+1,2}(\omega_1 1|\gamma_{n-3}),
	\\
	& \phantom{ \big\{ } 
	\suline{\omega_1}[(a+\suline{c})|(\beta_{n-4}+\suline{\alpha_{n-5}\beta})+(b+a)|(\gamma_{n-4}+\alpha_{n-5}\gamma)+(c+b)|(\alpha_{n-4}+\alpha_{n-6}\beta_2)]
	\\
	& \phantom{ \big\{ } 
	= \partial_{n+1,2}(\omega_1 1|\alpha_{n-4}\beta)=\partial_{n+1,2}(\omega_1 1|\alpha_{n-4}\gamma),
	\\
	& \phantom{ \big\{ }
	3(n-5)(bac|\alpha_n+abc|\beta_n-aba|\gamma_n)+\suline{\omega_1}(\suline{a|\alpha_{n-6}\beta_2}+b|\alpha_{n-5}\beta+c|\alpha_{n-5}\gamma) \\
	& \phantom{ \big\{ }
	=\partial_{n+1,2}(\omega_1 1|\alpha_{n-5}\beta_2)
	\big\}.
\end{align*}
By Fact \ref{indep}, the elements in $\mathfrak{B}_{n,3}$ are linearly independent, so $\mathfrak{B}_{n,3}$ is a basis of $B_{n,3}$ for $n\geqslant 5$ with $n$ odd. 
If $n\geqslant 4$ is even, then $\tilde{f}_{n-3}(\tilde{K}_{n-3,0})=0$ since $\tilde{f}$ vanishes on the elements given by \eqref{eq:fn-homology}. Hence, $B_{n,3}=\tilde{B}_{n,3}\oplus\omega_1 \tilde{B}_{n-4,1}$. We define a basis of $B_{4,3}$ by 
$\mathfrak{B}_{4,3}=\tilde{\mathfrak{B}}_{4,3}$, 
and we define a basis of $B_{n,3}$ by 
\[ 
	\mathfrak{B}_{n,3}=\tilde{\mathfrak{B}}_{n,3}\cup \big\{\suline{\omega_1}(\suline{c}-a)|(\alpha_{n-5}\beta-\suline{\alpha_{n-5}\gamma}),\omega_1(a-\suline{b})|(\suline{\alpha_{n-5}\beta}-\alpha_{n-5}\gamma)  \big\}
\] 
for $n\geqslant 6$ with $n$ even.
The dimension of $B_{n,3}$ is then given by 
\begin{equation}\label{dimb3}
	\begin{split}
		\operatorname{dim} B_{n,3} =
		\begin{cases}
			3,  &  \text{if $n=0$}, 
			\\
			8, &  \text{if $n=1$}, 
			\\
			12,  & \text{if $n= 2,4$}, 
			\\
			13,  & \text{if $n=3$}, 
			\\
			16,  & \text{if $n=5$}, 
			\\
			14, &  \text{if $n\geqslant 6$ is even},
			\\
			17,  &  \text{if $n\geqslant 7$ is odd}.
			\end{cases}
	\end{split}
\end{equation}

Suppose $m=4$. 
Consider $\partial_{n+1,3}:\tilde{K}_{n+1,3}\oplus \omega_1\tilde{K}_{n-3,1}\to \tilde{K}_{n,4}\oplus \omega_1 \tilde{K}_{n-4,2}\oplus \omega_2 \tilde{K}_{n-8,0}$.
If $n=3$, then $B_{3,4}=\tilde{B}_{3,4}$ since $\tilde{f}_0(\tilde{K}_{0,1})=0$ by the second line of \eqref{eq:f0-homology}.
So, we define 
$\mathfrak{B}_{3,4}=\tilde{\mathfrak{B}}_{3,4}$. 
If $n\geqslant 4$ is even, since 
$\tilde{f}_{n-3}(\tilde{K}_{n-3,1})\subseteq \tilde{B}_{n,4} $, we have $B_{n,4}=\tilde{B}_{n,4}\oplus \omega_1 B_{n-4,2}=\tilde{B}_{n,4}\oplus \omega_1 \tilde{B}_{n-4,2}$.
If $n\geqslant 5$ is odd, since $\tilde{f}_{n-3}(\tilde{K}_{n-3,1})=0$ by the last identity of Subsection \ref{subsection:explicit-description-diff-homology}, we have $B_{n,4}=\tilde{B}_{n,4}\oplus \omega_1 \tilde{B}_{n-4,2}$.
Hence, for $n\geqslant 4$, we define a basis of $B_{n,4}$ by 
$\mathfrak{B}_{n,4}=\tilde{\mathfrak{B}}_{n,4}\cup \omega_1 \tilde{\mathfrak{B}}_{n-4,2}$. 
The dimension of $B_{n,4}$ is then given by 
\begin{equation}\label{dimb4}
	\begin{split}
		\operatorname{dim} B_{n,4} =
		\begin{cases}
			1,  & \text{if $n=0$}, 
			\\
			2,  & \text{if $n=1,3$},
			\\
			4,  & \text{if $n=2$}, 
			\\
			7,  & \text{if $n=4$}, 
			\\
			9,  & \text{if $n=5$}, 
			\\
			15, & \text{if $n=6$}, 
			\\
			14, & \text{if $n\geqslant 7$ is odd},
			\\
			17, & \text{if $n\geqslant 8$ is even}.
			\end{cases}
	\end{split}
\end{equation}

\subsection{Computation of the cycles}
\label{subsection:com cycles}

As one can remark rather easily, from the computations in the previous subsection we can already deduce the dimensions of the homogeneous components of the spaces of cycles and thus of the Hochschild homology groups. 
However, since having specific representatives of bases of homology classes is relevant for other computations involving the Hochschild homology groups, we will proceed to do so.
More precisely, in this subsection, we will explicitly construct bases $\tilde{\mathfrak{D}}_{n,m}$ and $\mathfrak{D}_{n,m}$ of the $\Bbbk$-vector spaces $\tilde{D}_{n,m}=\Ker(\tilde{\partial}_{n,m})$ and $D_{n,m}=\Ker(\partial_{n,m})$ for $m\in \llbracket 0,4\rrbracket$ and $n\in\NN_0$ respectively, defined before Proposition \ref{bdh}.

\subsubsection{\texorpdfstring{Computation of $\tilde{\mathfrak{D}}_{n,m}$}{Computation of Dnm}}
\label{subsubsection:cycles-homology-1}

Recall that $\tilde{D}_{n,m}=\Ker(\tilde{\partial}_{n,m})$ and $\tilde{\partial}_{n,m}:\tilde{K}_{n,m}=A_m\otimes (A^!_{-n})^*\to \tilde{K}_{n-1,m+1}=A_{m+1}\otimes (A^!_{-(n-1)})^*$ was defined in Subsection \ref{subsection:rds}. 
Since $\tilde{K}_{n,m}/\tilde{D}_{n,m}\cong \tilde{B}_{n-1,m+1}$, we see that 
\begin{equation}\label{dimrelationtilded}
	\begin{split}
		\operatorname{dim}\tilde{D}_{n,m}=\operatorname{dim} \tilde{K}_{n,m}-\operatorname{dim} \tilde{B}_{n-1,m+1}. 
	\end{split}
\end{equation}
Hence, from the dimension of $\tilde{B}_{n-1,m+1}$ computed in Subsubsection \ref{subsubsection:boundaries-homology-1} as well as the dimension of $\tilde{K}_{n,m}$ (see the last paragraph of Subsection \ref{subsection:rds}), we deduce the value of the dimension of $\tilde{D}_{n,m}$. 
We will present them explicitly in the computations below. 

For every $(n,m)\in \NN_0\times \llbracket 0,4 \rrbracket $, we are going to provide a set $\tilde{\mathfrak{D}}_{n,m}\subseteq \tilde{D}_{n,m}$ such that $ \# \tilde{\mathfrak{D}}_{n,m}=\operatorname{dim} \tilde{D}_{n,m}$ and the elements in $\tilde{\mathfrak{D}}_{n,m}$ are linearly independent. 
As a consequence, $\tilde{\mathfrak{D}}_{n,m}$ is a basis of $\tilde{D}_{n,m}$. 
If $\tilde{D}_{n,m}=\tilde{K}_{n,m}$, we pick the usual basis of $\tilde{K}_{n,m}$, defined at the end of Subsection \ref{subsection:rds}.
We leave to the reader the easy verification in each case that the set $\tilde{\mathfrak{D}}_{n,m}$ satisfies these conditions.

Obviously, 
$ \tilde{D}_{0,m}=\tilde{K}_{0,m} $
for $m\in \llbracket 0,4 \rrbracket $.
Then we define the set $\tilde{\mathfrak{D}}_{0,m}$ by the usual basis of $\tilde{K}_{0,m}$.

Suppose $m=0$. By \eqref{dimbtilde1}, the dimension of $\tilde{D}_{n,0}$ is given by 
\begin{equation}\label{dimdtilde0}
	\begin{split}
		\operatorname{dim} \tilde{D}_{n,0} =
		\begin{cases}
			3,  &  \text{if $n=1$},
			\\
			1, &  \text{if $n\in\NN_0$ is even},
			\\
			4,  &  \text{if $n\geqslant 3$ is odd}.
			\end{cases}
	\end{split}
\end{equation}
If $n=1$, then 
$\tilde{D}_{1,0}=\tilde{K}_{1,0}$
since $\operatorname{dim} \tilde{D}_{1,0}=3=\operatorname{dim} \tilde{K}_{1,0}$. 
If $n\geqslant 3$ is odd, we define the set  
\[ 
\tilde{\mathfrak{D}}_{n,0}=\big\{ \suline{1|\alpha_n},\suline{1|\beta_n},\suline{1|\gamma_n},\suline{1}|(\suline{\alpha_{n-1}\beta}+\alpha_{n-1}\gamma+\alpha_{n-2}\beta_2) \big\}.    
\]  
If $n\geqslant 2$ is even, we define the set  
\[  \tilde{\mathfrak{D}}_{n,0}=\big\{ \suline{1}|(\suline{\alpha_{n-1}\beta}-\alpha_{n-1}\gamma) \big\}.\]

Suppose $m=1$. By \eqref{dimbtilde2}, the dimension of $\tilde{D}_{n,1}$ is given by 
\begin{equation}\label{dimdtilde1}
	\begin{split}
		\operatorname{dim} \tilde{D}_{n,1} =
		\begin{cases}
				3,  &  \text{if $n=0$},
				\\
			7,  &  \text{if $n=1$},
			\\
			8,  &  \text{if $n=2,3$},
			\\
			6,  & \text{if $n\geqslant 4$}.
			\end{cases}
	\end{split}
\end{equation}
We define the sets  
\[ 
\tilde{\mathfrak{D}}_{1,1}=\big\{  \suline{a|\alpha},\suline{b|\beta},\suline{c|\gamma}, \suline{a|\beta}+c|\alpha-c|\beta,\suline{ a|\gamma}+c|\alpha, \suline{b|\alpha}-c|\alpha+c|\beta, \suline{b|\gamma}+c|\beta \big\}\subseteq \tilde{D}_{1,1}, 
\] 
\begin{equation}
\begin{split}
    \tilde{\mathfrak{D}}_{2,1}=\big\{
		& \suline{a|\alpha_2},\suline{b|\beta_2},\suline{c|\gamma_2}, (\suline{c}-a)|(\alpha\beta-\suline{\alpha\gamma}), (a-\suline{b})|(\suline{\alpha\beta}-\alpha\gamma), (\suline{a}+c)|\suline{\beta_2}+a|\alpha\gamma+c|\alpha\beta, \\
	& (\suline{a}+b)|\suline{\gamma_2}+a|\alpha\beta+b|\alpha\gamma, (\suline{b}+c)|\suline{\alpha_2}+b|\alpha\gamma+c|\alpha\beta   \big\}\subseteq \tilde{D}_{2,1},
	\end{split}
\end{equation}
and 
\begin{align*}
	\tilde{\mathfrak{D}}_{3,1}=\big\{
		& \suline{a|\alpha_3}, \suline{b|\beta_3},\suline{c|\gamma_3}, \suline{b|\alpha_2\beta}+c|\alpha_2\gamma+a|\alpha\beta_2, \suline{a}|(\beta_3+\suline{\alpha_2\beta})+b|(\gamma_3+\alpha_2\gamma)+c|(\alpha_3+\alpha\beta_2),\\
	& -a|(\beta_3+\gamma_3)+b|(2\alpha_3+\gamma_3-\alpha_2\gamma)+\suline{c}|(\suline{\alpha_2\beta}-\alpha\beta_2-2\alpha_3),\\
	&  \suline{a}|(\suline{\alpha_2\gamma}-\beta_3)+b|(\alpha_2\gamma-\alpha_3)+2c|(\alpha_3+\beta_3),\\
	&  2a|(\beta_3+\gamma_3)+\suline{b}|(\suline{\alpha\beta_2}-\gamma_3)+c|(\alpha\beta_2-\beta_3)    \big\}  \subseteq \tilde{D}_{3,1}.	
\end{align*}
Moreover, if $n\geqslant 4$ is even, we define 
\begin{align*}
	\tilde{\mathfrak{D}}_{n,1}=\big\{
		& \suline{a|\alpha_n},\suline{b|\beta_n},\suline{c|\gamma_n}, (\suline{c}-a)|(\alpha_{n-1}\beta-\suline{\alpha_{n-1}\gamma}), (a-\suline{b})|(\suline{\alpha_{n-1}\beta}-\alpha_{n-1}\gamma), \\
	& (\suline{a}+b+c)|(\alpha_{n-1}\beta+\alpha_{n-1}\gamma+\alpha_{n-2}\beta_2+\alpha_n+\beta_n+\suline{\gamma_n}) \big\}\subseteq \tilde{D}_{n,1},
\end{align*}
and if $n\geqslant 5$ is odd, we set 
\begin{align*}
	\tilde{\mathfrak{D}}_{n,1}=\big\{
		& \suline{a|\alpha_n},
		\suline{b|\beta_n},
		\suline{c|\gamma_n}, \suline{b|\alpha_{n-1}\beta}+c|\alpha_{n-1}\gamma+a|\alpha_{n-2}\beta_2, \\
	& \suline{a}|(\beta_n+\suline{\alpha_{n-1}\beta})+b|(\gamma_{n}+\alpha_{n-1}\gamma)+c|(\alpha_n+\alpha_{n-2}\beta_2), \\
	&  \suline{c}|(\beta_n+\suline{\alpha_{n-1}\beta})+a|(\gamma_n+\alpha_{n-1}\gamma)+b|(\alpha_n+\alpha_{n-2}\beta_2)\big\}\subseteq \tilde{D}_{n,1}.	
\end{align*}

Suppose $m=2$. By \eqref{dimbtilde3}, the dimension of $\tilde{D}_{n,2}$ is given by 
\begin{equation}\label{dimdtilde2}
	\begin{split}
		\operatorname{dim} \tilde{D}_{n,2} =
		\begin{cases}
			4,  & \text{if $n=0$},
		    \\
			9,  & \text{if $n=1$},
			\\
			12, & \text{if $n\geqslant 2$}.
			\end{cases} 
	\end{split}
\end{equation}
We define the sets 
%
\begin{align*}
	\tilde{\mathfrak{D}}_{1,2}=\big\{
		& (\suline{ab}+ba)|\suline{\gamma}, (\suline{bc}-ba-ac)|\suline{\gamma},
		\suline{ac}|(\suline{\alpha}+\gamma),
		(\suline{ba}+ac)|(\suline{\beta}+\gamma), 
		\suline{bc|\alpha}-ac|\beta+ab|\gamma, 
		\\
		& (\suline{ab}+ba)|\suline{\beta}, 
		 (\suline{bc}-ba-ac)|\suline{\beta},
		 (bc-\suline{ba}-ac)|\suline{\alpha}, 
		 (\suline{ab}+ba)|\suline{\alpha}
	\big\} \subseteq \tilde{D}_{1,2}, 
\end{align*}
and
\begin{align*}
	\tilde{\mathfrak{D}}_{2,2}=\big\{
		& (ba-\suline{ab})|\suline{\alpha_2}, (\suline{bc}+ba+ac)|\suline{\alpha_2}, (ba-\suline{ab})|\suline{\beta_2}, (\suline{bc}+ba+ac)|\suline{\beta_2}, (ba-\suline{ab})|\suline{\gamma_2},\\
	& (\suline{bc}+ba+ac)|\suline{\gamma_2}, \suline{ab}|(\suline{\alpha\beta}-2\alpha_2-\beta_2)+bc|(\beta_2-\alpha_2), ab|(\beta_2-\gamma_2)+\suline{bc}|(\suline{\alpha\beta}-\beta_2-2\gamma_2),\\
	&  \suline{ba|\alpha\beta}-ab|\alpha\gamma,
	ab|\alpha\beta-\suline{ba|\alpha\gamma},
	\suline{ac|\alpha\beta}+(ab+bc)|\alpha\gamma,
	(ab+bc)|\alpha\beta+\suline{ac|\alpha\gamma} \big\} \subseteq \tilde{D}_{2,2}.	
\end{align*}
Moreover, if $n\geqslant$ 3 is odd, we define 
\begin{align*}
	\tilde{\mathfrak{D}}_{n,2}=\big\{
		& (\suline{ab}+ba)|\suline{\alpha_n}, (\suline{bc}-ba-ac)|\suline{\alpha_n}, (\suline{ab}+ba)|\suline{\beta_n}, (\suline{bc}-ba-ac)|\suline{\beta_n}, (\suline{ab}+ba)|\suline{\gamma_n}, \\
	& (\suline{bc}-ba-ac)|\suline{\gamma_n},
	(\suline{ab}+ba)|(\suline{\alpha_{n-1}\beta}+\alpha_{n-2}\beta_2), (\suline{ab}+ba)|\suline{\alpha_{n-1}\gamma}, (\suline{bc}-ba-ac)|\suline{\alpha_{n-2}\beta_2},\\
	&  (ab+ba)|\alpha_{n-1}\beta-(\suline{bc}-ba-ac)|\suline{\alpha_{n-1}\gamma},\\
	&  bc|(\alpha_n+\alpha_{n-2}\beta_2)-\suline{ac}|(\beta_n+\suline{\alpha_{n-1}\beta})+ab|(\gamma_n+\alpha_{n-1}\gamma),\\ 
	& (\suline{bc}-ba-ac)|\suline{\alpha_{n-1}\beta}-(ab+ba)|\alpha_{n-2}\beta_2
	\big\} \subseteq \tilde{D}_{n,2},
\end{align*}
and if $n\geqslant 4$ is even, we set
\begin{align*}
	\tilde{\mathfrak{D}}_{n,2}& =\big\{ 
		(ba-\suline{ab})|\suline{\alpha_n}, (\suline{bc}+ba+ac)|\suline{\alpha_n}, (ba-\suline{ab})|\suline{\beta_n}, (\suline{bc}+ba+ac)|\suline{\beta_n},
		(ba-\suline{ab})|\suline{\gamma_n},\\
	& \phantom{ =\big\{ \, } (\suline{bc}+ba+ac)|\suline{\gamma_n}, (ba-\suline{ab})|\suline{\alpha_{n-2}\beta_2}, (\suline{bc}+ba+ac)|\suline{\alpha_{n-2}\beta_2}, 
	\suline{ba|\alpha_{n-1}\beta}-ab|\alpha_{n-1}\gamma,\\
	& \phantom{ =\big\{ \, }  \suline{ab|\alpha_{n-1}\beta}-ba|\alpha_{n-1}\gamma, 
	\suline{ac|\alpha_{n-1}\beta}+(ab+bc)|\alpha_{n-1}\gamma, (ab+\suline{bc})|\suline{\alpha_{n-1}\beta}+ac|\alpha_{n-1}\gamma  \big\} \\
	& \subseteq \tilde{D}_{n,2}.
\end{align*}

Suppose $m=3$. By \eqref{dimbtilde4}, the dimension of $\tilde{D}_{n,3}$ is given by 
\begin{equation}\label{dimdtilde3}
	\begin{split}
		\operatorname{dim} \tilde{D}_{n,3} =
		\begin{cases}
			3,  & \text{if $n=0$},
			\\
			8,  & \text{if $n=1$},
			\\
			13, & \text{if $n=2$ or $n\geqslant 5$ is odd},
			\\
			14, & \text{if $n=3$},
			\\
			16, & \text{if $n\geqslant 4$ is even}.
			\end{cases} 
	\end{split}
\end{equation}
We define the sets 
\[
	\tilde{\mathfrak{D}}_{1,3}=\big\{
	\suline{aba|\alpha},
	\suline{abc|\alpha},
	\suline{aba|\beta},
	\suline{bac|\beta},
	\suline{abc|\gamma},
	\suline{bac|\gamma},
	\suline{bac|\alpha}+aba|\gamma, \suline{abc|\beta}+aba|\gamma   \big\} \subseteq \tilde{D}_{1,3}, 
\] 
and
\begin{align*}
	\tilde{\mathfrak{D}}_{3,3}=\big\{
		& \suline{aba|\alpha_3}, 
		\suline{abc|\alpha_3},
		\suline{aba|\beta_3},
		\suline{bac|\beta_3},
		\suline{abc|\gamma_3},
		\suline{bac|\gamma_3}, aba|\alpha_2\beta+\suline{bac|\alpha_2\beta},
		aba|\alpha_2\beta+\suline{abc|\alpha_2\gamma}, \\
	& aba|\alpha_2\beta+\suline{bac|\alpha_2\gamma},
	aba|\alpha_2\beta-\suline{aba|\alpha\beta_2}, aba|\alpha_2\beta+\suline{abc|\alpha\beta_2}, \suline{abc|\alpha_2\beta}-bac|\alpha_3+aba|\gamma_3,\\
	& \suline{aba|\alpha_2\gamma}+bac|\alpha_3+abc|\beta_3, \suline{bac|\alpha\beta_2}-abc|\beta_3+aba|\gamma_3  \big\} \subseteq \tilde{D}_{3,3}.
\end{align*}
Moreover, if $n\geqslant 2$ is even, let 
\begin{align*}
	\mathcal{G}_{n,3} =\big\{
		& \suline{aba|\alpha_n}, \suline{abc|\alpha_n},
		\suline{bac|\alpha_n},
		\suline{aba|\beta_n},
		\suline{abc|\beta_n},
		\suline{bac|\beta_n},
		\suline{aba|\gamma_n},
		\suline{abc|\gamma_n},
		\suline{bac|\gamma_n},\\
	 & \suline{aba|\alpha_{n-1}\beta}+(abc+bac)|\alpha_{n-1}\gamma ,
	 \suline{abc}|(\suline{\alpha_{n-1}\beta}+\alpha_{n-1}\gamma), 
	\suline{ bac}|(\suline{\alpha_{n-1}\beta}+\alpha_{n-1}\gamma), \\
	 & \suline{aba}|(\alpha_{n-1}\beta+\suline{\alpha_{n-1}\gamma})
	 \big\}.
\end{align*}
Then we define the set 
$\tilde{\mathfrak{D}}_{2,3}=\mathcal{G}_{2,3}$, 
and 
\[
	\tilde{\mathfrak{D}}_{n,3}= 
	\mathcal{G}_{n,3}\cup \big\{ \suline{aba|\alpha_{n-2}\beta_2},  \suline{abc|\alpha_{n-2}\beta_2}, \suline{bac|\alpha_{n-2}\beta_2 }  \big\}\subseteq \tilde{D}_{n,3}
\] 
for $n\geqslant 4$ with $n$ even.
If $n\geqslant 5$ is odd, then we define
\begin{align*}
	\tilde{\mathfrak{D}}_{n,3}=\big\{
		& \suline{aba|\alpha_n},
		\suline{abc|\alpha_n},
		\suline{aba|\beta_n},
		\suline{bac|\beta_n},
		\suline{abc|\gamma_n},
		\suline{bac|\gamma_n}, aba|\alpha_{n-1}\beta+\suline{bac|\alpha_{n-1}\beta}, \\
	& aba|\alpha_{n-1}\beta+\suline{abc|\alpha_{n-1}\gamma},
	aba|\alpha_{n-1}\beta+\suline{bac|\alpha_{n-1}\gamma},
	aba|\alpha_{n-1}\beta-\suline{aba|\alpha_{n-2}\beta_2}, \\
	&  aba|\alpha_{n-1}\beta+\suline{abc|\alpha_{n-2}\beta_2}, abc|(\beta_n+\alpha_{n-1}\beta)+\suline{aba}|(\gamma_n+\suline{\alpha_{n-1}\gamma}), \\
	&  \suline{bac}|(\alpha_n+\suline{\alpha_{n-2}\beta_2})-abc|(\beta_n+\alpha_{n-1}\beta)  \big\} \subseteq \tilde{D}_{n,3}.
\end{align*}

Finally, if $m=4$, we immediately see that $\tilde{D}_{n,4}=\tilde{K}_{n,4}$. So we define the set $\tilde{\mathfrak{D}}_{n,4}$ by the usual basis of $\tilde{K}_{n,4}$. The dimension of $\tilde{D}_{n,4}$ is given by 
\begin{equation}\label{dimdtilde4}
	\begin{split}
		\operatorname{dim} \tilde{D}_{n,4} =
		\begin{cases}
			1,  &  \text{if $n=0$},
			\\
			3,  &  \text{if $n=1$},
			\\
			5, &  \text{if $n=2$},
			\\
			6, & \text{if $n\geqslant 3$}.
			\end{cases} 
	\end{split}
\end{equation}

\subsubsection{\texorpdfstring{Computation of $\mathfrak{D}_{n,m}$}{Computation of Dnm}} 
\label{subsubsection:cycles-homology-2}

Recall that $D_{n,m}=\Ker(\partial_{n,m})$ and $\partial_{n,m}:\tilde{P}_{n,m}\to \tilde{P}_{n-1,m+1}$. 
The isomorphism $\tilde{P}_{n,m}/D_{n,m}\cong B_{n-1,m+1}$ tells us that 
\begin{equation}\label{dimrelationd}
	\begin{split}
		\operatorname{dim} D_{n,m}=\operatorname{dim} \tilde{P}_{n,m}-\operatorname{dim} B_{n-1,m+1}.
	\end{split}
\end{equation}
Hence, from the dimension of $B_{n-1,m+1}$ computed in Subsubsection \ref{subsubsection:boundaries-homology-2} as well as the dimension of $\tilde{P}_{n,m}$ (see the last paragraph of Subsection \ref{subsection:rds}), we deduce the value of the dimension of $D_{n,m}$. 
We will present them explicitly in the computations below. 

For integers $(n,m)\in\NN_0 \times \llbracket 0,4 \rrbracket$, we are going to provide a set $\mathfrak{D}_{n,m}\subseteq D_{n,m}$ such that $ \# \mathfrak{D}_{n,m}=\operatorname{dim} D_{n,m}$ and the elements in $\mathfrak{D}_{n,m}$ are linearly independent. 
As a consequence, $\mathfrak{D}_{n,m}$ is a basis of $D_{n,m}$.
We leave to the reader the easy verification in each case that the set $\mathfrak{D}_{n,m}$ satisfies these conditions.

For either $m=0,1$ and $n\in \NN_0 $, or $m=2,3,4$ and $n=0,1,2,3$, note that $\partial_{n,m}=\tilde{\partial}_{n,m}$, then $D_{n,m}=\tilde{D}_{n,m}$. 
So we define the basis of $D_{n,m}$ by 
$\mathfrak{D}_{n,m}=\tilde{\mathfrak{D}}_{n,m}$.

Suppose $m=2$. By \eqref{dimb3}, the dimension of $D_{n,2}$ is given by 
\begin{equation}\label{dimd2}
	\begin{split}
		\operatorname{dim} D_{n,2} =
		\begin{cases}
			4,  &  \text{if $n=0$},
			\\
			9,  &  \text{if $n=1$},
			\\
			12, &  \text{if $n=2,3,4$},
			\\
			15, & \text{if $n=5$},
			\\
			13, & \text{if $n\geqslant 6$ is even}, 
			\\
			16, & \text{if $n\geqslant 7$ is odd}.
			\end{cases}
	\end{split}
\end{equation}
If $n=4$, we define the set  
$\mathfrak{D}_{4,2}=\tilde{\mathfrak{D}}_{4,2} \subseteq D_{4,2}$. 
If $n\geqslant 6$ is even, we define 
\[ \mathfrak{D}_{n,2}= \tilde{\mathfrak{D}}_{n,2}\cup \big\{ \suline{\omega_1 1}|(\suline{\alpha_{n-5}\beta}-\alpha_{n-5}\gamma) \big\}\subseteq D_{n,2}.\] 
If $n\geqslant 5$ is odd, we define 
$\mathfrak{D}_{n,2}=\tilde{\mathfrak{D}}_{n,2}\cup \omega_1 \tilde{\mathfrak{D}}_{n-4,0}\subseteq D_{n,2}$.

Suppose $m=3$. By \eqref{dimb4}, the dimension of $D_{n,3}$ is given by 
\begin{equation}\label{dimd3}
	\begin{split}
		\operatorname{dim} D_{n,3} = 
		\begin{cases}
			3,  &  \text{if $n=0$},
			\\
			8,  &  \text{if $n=1$},
			\\
			13, &  \text{if $n=2$},
			\\
			14, &  \text{if $n=3$},
			\\
			19, &  \text{if $n=4$ or $n\geqslant 9$ is odd},
			\\
			20, &  \text{if $n=5$},
			\\
			24, &  \text{if $n=6$}, 
			\\
			21, &  \text{if $n=7$}, 
			\\
			22, &  \text{if $n\geqslant 8$ is even}.
			\end{cases}
	\end{split}
\end{equation}
We define the sets 
\begin{align*}
	\mathfrak{D}_{5,3}=\tilde{\mathfrak{D}}_{5,3}\cup \big\{ 
	& 
	5bac|\alpha_5+2abc|\beta_5-5aba|\gamma_5-3abc|\alpha_4\beta+\suline{\omega_1 a|\alpha},
	\\
	& 
	5bac|\alpha_5+2abc|\beta_5-5aba|\gamma_5-3abc|\alpha_4\beta+\suline{\omega_1 b|\beta},
	\\
	& 
	5bac|\alpha_5+2abc|\beta_5-5aba|\gamma_5-3abc|\alpha_4\beta+\suline{\omega_1 c|\gamma}, \suline{\omega_1}(\suline{a|\beta}+c|\alpha-c|\beta),
	\\
	&
	\suline{\omega_1}(\suline{a|\gamma}+c|\alpha), 
	\suline{\omega_1}(\suline{b|\alpha}-c|\alpha+c|\beta), 
	\suline{\omega_1}(\suline{b|\gamma}+c|\beta)  \big\} \subseteq D_{5,3}, 	
\end{align*}
and 
\begin{align*}
	\mathfrak{D}_{7,3}=\tilde{\mathfrak{D}}_{7,3}\cup \big\{  
	& 
	5bac|\alpha_7+2abc|\beta_7-5aba|\gamma_7-3abc|\alpha_6\beta+\suline{\omega_1 a|\alpha_3},
	\\
	&
	5bac|\alpha_7+2abc|\beta_7-5aba|\gamma_7-3abc|\alpha_6\beta+\suline{\omega_1 b|\beta_3},
	\\
	&
	5bac|\alpha_7+2abc|\beta_7-5aba|\gamma_7-3abc|\alpha_6\beta+\suline{\omega_1 c|\gamma_3},
	\\
	& 
	6bac|\alpha_7+6abc|\beta_7-6aba|\gamma_7+\suline{\omega_1}(\suline{b|\alpha_2\beta}+c|\alpha_2\gamma+a|\alpha\beta_2), 
	\\
	&
	\suline{\omega_1}[\suline{a}|(\beta_3+\suline{\alpha_2\beta})+b|(\gamma_3+\alpha_2\gamma)+c|(\alpha_3+\alpha\beta_2)],
	\\
	&
	\suline{\omega_1}[-a|(\beta_3+\gamma_3)+b|(2\alpha_3+\gamma_3-\alpha_2\gamma)+\suline{c}|(\suline{\alpha_2\beta}-\alpha\beta_2-2\alpha_3)],
	\\
	&
	\suline{\omega_1}[\suline{a}|(\suline{\alpha_2\gamma}-\beta_3)+b|(\alpha_2\gamma-\alpha_3)+2c|(\alpha_3+\beta_3)],
	\\
	&
	\suline{\omega_1}[2a|(\beta_3+\gamma_3)+\suline{b}|(\suline{\alpha\beta_2}-\gamma_3)+c|(\alpha\beta_2-\beta_3)] \big\} 
	\subseteq D_{7,3}.	
\end{align*}
Moreover, if $n\geqslant 4$ is even, we define the set
$
\mathfrak{D}_{n,3}=\tilde{\mathfrak{D}}_{n,3}\cup \omega_1 \tilde{\mathfrak{D}}_{n-4,1} \subseteq D_{n,3}$, 
and if $n\geqslant 9$ is odd, we define the set
\begin{align*}
	\mathfrak{D}_{n,3}=\tilde{\mathfrak{D}}_{n,3}\cup \big\{
	&
	5bac|\alpha_n+2abc|\beta_n-5aba|\gamma_n-3abc|\alpha_{n-1}\beta+\suline{\omega_1 a|\alpha_{n-4}},
	\\
	&
	5bac|\alpha_n+2abc|\beta_n-5aba|\gamma_n-3abc|\alpha_{n-1}\beta+\suline{\omega_1 b|\beta_{n-4}},
	\\
	&
	5bac|\alpha_n+2abc|\beta_n-5aba|\gamma_n-3abc|\alpha_{n-1}\beta+\suline{\omega_1 c|\gamma_{n-4}},
	\\
	&
	3(n-5)(bac|\alpha_n+abc|\beta_n-aba|\gamma_n)+\suline{\omega_1}(\suline{b|\alpha_{n-5}\beta}+c|\alpha_{n-5}\gamma+a|\alpha_{n-6}\beta_2),
	\\
	&
	\suline{\omega_1}[\suline{a}|(\beta_{n-4}+\suline{\alpha_{n-5}\beta})+b|(\gamma_{n-4}+\alpha_{n-5}\gamma)+c|(\alpha_{n-5}+\alpha_{n-6}\beta_2)],
	\\
	&
	\suline{\omega_1}[\suline{c}|(\beta_{n-4}+\suline{\alpha_{n-5}\beta})+a|(\gamma_{n-4}+\alpha_{n-5}\gamma)+b|(\alpha_{n-4}+\alpha_{n-6}\beta_2)]   \big\}\subseteq D_{n,3}.
\end{align*}

Suppose $m=4$. 
The space $D_{n,4}$ is given by Proposition \ref{d4n}. So $\mathfrak{D}_{n,4}$ is given by the usual basis of $\tilde{K}_{n,4}$ and $\omega_1\mathfrak{D}_{n-4,2}$. 
The dimension of $D_{n,4}$ is then given by 
\begin{equation}\label{dimd4}
	\begin{split}
		\operatorname{dim} D_{n,4} =
		\begin{cases}
			1,  &  \text{if $n=0$},
			\\
			3,  &  \text{if $n=1$},
			\\
			5,  &  \text{if $n=2$},
			\\
			6,  &  \text{if $n=3$},
			\\
			10, &  \text{if $n=4$},
			\\
			15, &  \text{if $n=5$},
			\\
			18, &  \text{if $n=6,7,8$}, 
			\\
			21, &  \text{if $n=9$}, 
			\\
			19, &  \text{if $n\geqslant 10$ is even},
			\\
			22, &  \text{if $n\geqslant 11$ is odd}.
			\end{cases} 
	\end{split}
\end{equation}

\subsection{Hochschild homology}
\label{subsection:homology}
In this subsection, we will explicitly construct a subspace $H_{n,m}$ of $D_{n,m}$ such that 
$D_{n,m}= H_{n,m}\oplus B_{n,m}$ for $m,n\in \NN_0$, 
and we define $H_{n,m}=0$ for $(n,m)\in \ZZ^2\setminus \NN_0^2$. 
By Proposition \ref{bdh}, we have the following similar recursive description. 
\begin{cor}\label{cor recursive homology}
	For integers $m\geqslant 3$ and $n\in \NN_0$, we have 
\begin{equation}\label{homology recursive}
	\begin{split}
	 H_{n,m} \cong
	 \begin{cases}
	\omega_{\frac{m-3}{2}}H_{n-2m+6,3},  & \text{if $m$ is odd}, 
	\\
	\omega_{\frac{m}{2}-2}H_{n-2m+8,4}, & \text{if $m$ is even}.
	\end{cases} 
\end{split}
\end{equation}
\end{cor}
So it is also sufficient to compute the case $m\in \llbracket 0,4 \rrbracket $. 
Recall that 
\begin{equation}\label{hdimrelation}
	\begin{split}
		\operatorname{dim}H_{n,m}=\operatorname{dim}D_{n,m}-\operatorname{dim}B_{n,m}=\operatorname{dim}\tilde{P}_{n,m}-\operatorname{dim}B_{n-1,m+1}-\operatorname{dim}B_{n,m}.
	\end{split}
\end{equation} 
Hence, from the dimension of $D_{n,m}$ computed in Subsubsection \ref{subsubsection:cycles-homology-2} as well as the dimension of 
$B_{n,m}$ computed in Subsubsection \ref{subsubsection:boundaries-homology-2}, 
we deduce the value of the dimension of $H_{n,m}$. 
We will present them explicitly in the computations below. 

For every $(n,m)\in \NN_0\times \llbracket 0,4 \rrbracket $, we are going to provide a set $\mathfrak{H}_{n,m}\subseteq D_{n,m}$ such that $ \# \mathfrak{H}_{n,m}=\operatorname{dim} H_{n,m}$ and the elements in $ \mathfrak{H}_{n,m}\cup \mathfrak{B}_{n,m}$ are linearly independent. 
As a consequence, the space $H_{n,m}$ spanned by $\mathfrak{H}_{n,m}$ satisfies $D_{n,m}= H_{n,m}\oplus B_{n,m}$.
We leave to the reader the easy verification in each case that the set $\mathfrak{H}_{n,m}$ satisfies these conditions. Note that,
unless stated otherwise, the linear independence of the elements in $ \mathfrak{H}_{n,m}\cup \mathfrak{B}_{n,m}$ is from Fact \ref{indep}, where we put the elements in $\mathfrak{H}_{n,m}$ before the elements in $\mathfrak{B}_{n,m}$.

Suppose $m=0$. We get immediately
$\mathfrak{H}_{n,0}=\mathfrak{D}_{n,0}$
since $B_{n,0}=0$ for $n\in \NN_0$. The dimension of $H_{n,0}$ is given by 
\begin{equation}\label{dimhomology0}
	\begin{split}
	 \operatorname{dim} H_{n,0}=
	 \begin{cases}
	1, & \text{if $n\in\NN_0$ is even}, 
	\\
	3, & \text{if $n=1$}, 
	\\
	4, & \text{if $n\geqslant 3$ is odd}.
	\end{cases}
\end{split}
\end{equation}

Suppose $m=1$. The dimension of $H_{n,1}$ is given by 
\begin{equation}\label{dimhomology1}
	\begin{split}
	\operatorname{dim} H_{n,1}=
	\begin{cases}
	3, &\text{if $n=0,1,3$}, 
	\\
	6, &\text{if $n=2$}, 
	\\
	4, &\text{if $n\geqslant 4$ is even}, 
	\\
	1, &\text{if $n\geqslant 5$ is odd}.
	\end{cases}
\end{split}
\end{equation}
We define the sets 
$\mathfrak{H}_{0,1}=\mathfrak{D}_{0,1}$, 
\[ \mathfrak{H}_{1,1}=\big\{\suline{a|\gamma}+c|\alpha, \suline{b|\alpha}-c|\alpha+c|\beta, \suline{b|\gamma}+c|\beta   \big\}, \] 
\begin{equation}
\begin{split}
	\mathfrak{H}_{2,1}=\mathfrak{D}_{2,1} \backslash \mathfrak{B}_{2,1}=\big\{
	& 
	\suline{a|\alpha_2}, 
	\suline{b|\beta_2}, 
	\suline{c|\gamma_2}, \suline{a}|(\suline{\beta_2}+\alpha\gamma)+c|(\beta_2+\alpha\beta), \suline{a}|(\suline{\gamma_2}+\alpha\beta)+b|(\gamma_2+\alpha\gamma),\\
	&
	\suline{b}|(\suline{\alpha_2}+\alpha\gamma)+c|(\alpha_2+\alpha\beta)   \big\},
\end{split}
\end{equation}
and
\begin{equation}
\begin{split}
	\mathfrak{H}_{3,1}=\big\{
	& 
	\suline{a}|(\beta_3+\suline{\alpha_2\beta})+b|(\gamma_3+\alpha_2\gamma)+c|(\alpha_3+\alpha\beta_2), \\
	&
	\suline{a}|(\suline{\alpha_2\gamma}-\beta_3)+b|(\alpha_2\gamma-\alpha_3)+2c|(\alpha_3+\beta_3),\\
	&
	2a|(\beta_3+\gamma_3)+\suline{b}|(\suline{\alpha\beta_2}-\gamma_3)+c|(\alpha\beta_2-\beta_3)   \big\}.
\end{split}
\end{equation}
Moreover, if $n\geqslant 4$ is even, we define the set
\[ 
	\mathfrak{H}_{n,1}=\mathfrak{D}_{n,1}\backslash \mathfrak{B}_{n,1}=\big\{
	\suline{a|\alpha_n},
	\suline{b|\beta_n},
	\suline{c|\gamma_n}, (\suline{a}+b+c)|(\alpha_{n-1}\beta+\alpha_{n-1}\gamma+\alpha_{n-2}\beta_2+\alpha_n+\beta_n+\suline{\gamma_n})  \big\},
\] 
and if $n\geqslant 5$ is odd, we define  
\[ 
	\mathfrak{H}_{n,1}=\big\{\suline{a}|(\beta_{n}+\suline{\alpha_{n-1}\beta})+b|(\gamma_n+\alpha_{n-1}\gamma)+c|(\alpha_n+\alpha_{n-2}\beta_2) \big\}.
\]

Suppose $m=2$. The dimension of $H_{n,2}$ is given by 
\begin{equation}\label{dimhomology2}
	\begin{split}
	\operatorname{dim} H_{n,2}=
	\begin{cases}
	2, & \text{if $n=0,1,2$}, 
	\\
	0, &  \text{if $n=3,4$}, 
	\\
	3, & \text{if $n=5$},
	\\
	1, & \text{if $n\geqslant 6$ is even}, 
	\\
	4, &  \text{if $n\geqslant 7$ is odd}.
	\end{cases} 
\end{split}
\end{equation}
We define the sets 
\[ 
\mathfrak{H}_{0,2}=\big\{
\suline{ab|\epsilon^!},
\suline{bc|\epsilon^!}
\big\},	
\quad 
	\mathfrak{H}_{1,2}=\big\{(\suline{ba}+ac)|(\suline{\beta}+\gamma), \suline{ac}|(\suline{\alpha}+\gamma)  \big\},	
\] 
\[ 
	\mathfrak{H}_{2,2}=\big\{
	\suline{ab}|(\suline{\beta_2}-\gamma_2)+bc|(\alpha\beta-\beta_2-2\gamma_2),
	\suline{ab}|(\alpha\beta-2\suline{\alpha_2}-\beta_2)+bc|(\beta_2-\alpha_2)  \big\},
\] 
and 
$ 
	\mathfrak{H}_{3,2}=\mathfrak{H}_{4,2}=\emptyset
$.
Moreover, if $n\geqslant 5$ is odd, we define the set
$\mathfrak{H}_{n,2}=\omega_1 \mathfrak{D}_{n-4,0}$,
and if $n\geqslant 6$ is even, we define 
\[ 
	\mathfrak{H}_{n,2}=\big\{
	\suline{\omega_1 1}|(\suline{\alpha_{n-5}\beta}-\alpha_{n-5}\gamma)
	\big\}.
\]

Suppose $m=3$. The dimension of $H_{n,3}$ is given by 
\begin{equation}\label{dimhomology3}
	\begin{split}
	\operatorname{dim} H_{n,3}=
	\begin{cases}
	0, &  \text{if $n=0,1$}, 
	\\
	1, &  \text{if $n=2,3$}, 
	\\
	7, &  \text{if $n=4$}, 
	\\
	4, &  \text{if $n=5,7$},
	\\
	10,&  \text{if $n=6$},
	\\
	8, &  \text{if $n\geqslant 8$ is even}, 
	\\
	2, &  \text{if $n\geqslant 9$ is odd}.
	\end{cases}
\end{split}
\end{equation}
We define the sets 
$\mathfrak{H}_{0,3}=\mathfrak{H}_{1,3}=\emptyset$, 
\[ \mathfrak{H}_{2,3}=\big\{ \suline{bac|\alpha_2}
\big\}, 
\quad 
\mathfrak{H}_{3,3}=\big\{ aba|\alpha_2\beta+bac|\alpha_2\beta   \big\}, \] 
\[ 
	\mathfrak{H}_{4,3}=\big\{ 
	\suline{bac|\alpha_4}, 
	\suline{aba|\alpha_{2}\beta_2}, \suline{abc|\alpha_{2}\beta_2}, \suline{bac|\alpha_{2}\beta_2}, \suline{\omega_1a|\epsilon^!}, \suline{\omega_1b|\epsilon^!},  \suline{\omega_1c|\epsilon^!} 
	\big\},
\] 
\[ 
	\mathfrak{H}_{5,3}=\big\{(aba+bac)|\alpha_4\beta, \suline{\omega_1}(\suline{a|\gamma}+c|\alpha), 
	\suline{\omega_1}(\suline{b|\alpha}-c|\alpha+c|\beta), 
	\suline{\omega_1}(\suline{b|\gamma}+c|\beta)  \big\},
\] 
\begin{equation}
\begin{split}
	\mathfrak{H}_{6,3}=\big\{ 
	& 
	\suline{bac|\alpha_6},
	\suline{aba|\alpha_{4}\beta_2},
	\suline{abc|\alpha_{4}\beta_2},
	\suline{bac|\alpha_{4}\beta_2}, \suline{\omega_1a|\alpha_2},
	\suline{\omega_1b|\beta_2},
	\suline{\omega_1c|\gamma_2},
	\\
	&
	\suline{\omega_1}[\suline{a}|(\suline{\beta_2}+\alpha\gamma)+c|(\beta_2+\alpha\beta)], 
	\suline{\omega_1}[\suline{a}|(\suline{\gamma_2}+\alpha\beta)+b|(\gamma_2+\alpha\gamma)],\\
	&
	\suline{\omega_1}[\suline{b}|(\suline{\alpha_2}+\alpha\gamma)+c|(\alpha_2+\alpha\beta)]  \big\},
\end{split}
\end{equation}
and 
\begin{align*}
	\mathfrak{H}_{7,3}=\big\{
	& 
	(aba+bac)|\alpha_6\beta, \suline{\omega_1}[\suline{a}|(\beta_3+\suline{\alpha_2\beta})+b|(\gamma_3+\alpha_2\gamma)+c|(\alpha_3+\alpha\beta_2)], \\
	&
	\suline{\omega_1}[\suline{a}|(\suline{\alpha_2\gamma}-\beta_3)+b|(\alpha_2\gamma-\alpha_3)+2c|(\alpha_3+\beta_3)],
	\\
	& \suline{\omega_1}[2a|(\beta_3+\gamma_3)+\suline{b}|(\suline{\alpha\beta_2}-\gamma_3)+c|(\alpha\beta_2-\beta_3)] \big\}.
\end{align*}
Moreover, if $n\geqslant 8$ is even, we define 
\begin{align*}
	\mathfrak{H}_{n,3}=\big\{ 
	& 
	\suline{bac|\alpha_n},
	\suline{aba|\alpha_{n-2}\beta_2},
	\suline{abc|\alpha_{n-2}\beta_2},
	\suline{bac|\alpha_{n-2}\beta_2},
	\suline{\omega_1a|\alpha_{n-4}},
	\suline{\omega_1b|\beta_{n-4}},
	\suline{\omega_1c|\gamma_{n-4}},
	\\
	&
	\suline{\omega_1}(\suline{a}+b+c)|(\alpha_{n-5}\beta+\alpha_{n-5}\gamma+\alpha_{n-6}\beta_2+\alpha_{n-4}+\beta_{n-4}+\suline{\gamma_{n-4}}) \big\},
\end{align*}
and if $n\geqslant 9$ is odd, we define 
\begin{align*}
	\mathfrak{H}_{n,3}=\big\{ (aba+bac)|\alpha_{n-1}\beta,
	\suline{\omega_1}[\suline{a}|(\beta_{n-4}+\suline{\alpha_{n-5}\beta})+b|(\gamma_{n-4}+\alpha_{n-5}\gamma)+c|(\alpha_{n-4}+\alpha_{n-6}\beta_2)] \big\}.
\end{align*}
Moreover, the set $\mathfrak{H}_{n,3}\cup \mathfrak{B}_{n,3}$ for $n\geqslant 3$ and $n$ odd is linearly independent. 
Indeed, Fact \ref{indep} tells us that the elements containing underlined terms do form a linearly independent set. 
It is then easy to prove that the elements of $\mathfrak{H}_{n,3}\cup \mathfrak{B}_{n,3}$ 
without any underlining are not a linear combination of the remaining elements, proving the claim.  

Suppose $m=4$. The dimension of $H_{n,4}$ is given by 
\begin{equation}\label{dimhomology4}
	\begin{split}
	\operatorname{dim} H_{n,4}=
	\begin{cases}
	0, & \text{if $n=0$}, 
	\\
	1, & \text{if $n=1,2,8$}, 
	\\
	4, & \text{if $n=3,7$}, 
	\\
	3, & \text{if $n=4,6$},
	\\
	6, & \text{if $n=5$},
	\\
	7, & \text{if $n=9$},
	\\
	2, & \text{if $n\geqslant 10$ is even}, 
	\\
	8, & \text{if $n\geqslant 11$ is odd}.
	\end{cases}
\end{split}
\end{equation}
We define the sets  
$\mathfrak{H}_{0,4}=\emptyset$, 
\[ \mathfrak{H}_{1,4}=\big\{ \suline{abac|\alpha} 
\big\},
\quad 
 \mathfrak{H}_{2,4}=\big\{
\suline{abac|\alpha\beta} 
\big\},  \]
\[
\mathfrak{H}_{3,4}=\big\{
\suline{abac|\alpha_3},\suline{ abac|\alpha_2\beta}, 
\suline{abac|\alpha_2\gamma},
\suline{abac|\alpha\beta_2 } 
\big\},  
\quad 
 \mathfrak{H}_{4,4}=\big\{ \suline{abac|\alpha_3\beta}, \suline{\omega_1 ab|\epsilon^!}, \suline{\omega_1 bc|\epsilon^!} 
\big\}, \]
\[ \mathfrak{H}_{5,4}=\big\{
\suline{abac|\alpha_5},
\suline{abac|\alpha_4\beta}, 
\suline{abac|\alpha_4\gamma},
\suline{abac|\alpha_3\beta_2}, \suline{\omega_1}(\suline{ba}+ac)|(\suline{\beta}+\gamma), \suline{\omega_1ac}|(\suline{\alpha}+\gamma)  \big\}, \] 
\begin{equation}
\begin{split}
	 \mathfrak{H}_{6,4}=\big\{ 
	 & \suline{abac|\alpha_5\beta},
	 \suline{\omega_1}[\suline{ab}|(\suline{\beta_2}-\gamma_2)+bc|(\alpha\beta-\beta_2-2\gamma_2)],
	 \\
	&
	\suline{\omega_1}[\suline{ab}|(\alpha\beta-2\suline{\alpha_2}-\beta_2)+bc|(\beta_2-\alpha_2)]
	\big\},
\end{split}
\end{equation}
\[\mathfrak{H}_{7,4}=\big\{
\suline{abac|\alpha_7}, 
\suline{abac|\alpha_6\beta},
\suline{abac|\alpha_6\gamma},
\suline{abac|\alpha_5\beta_2} 
\big\},   \] 
and 
\[ \mathfrak{H}_{8,4}=\big\{ \suline{abac|\alpha_7\beta} 
\big\}. \] 
Moreover, if $n\geqslant 9$ is odd, we define the set 
\[	\mathfrak{H}_{n,4}=\big\{  \suline{abac|\alpha_n}, \suline{abac|\alpha_{n-1}\beta}, \suline{abac|\alpha_{n-1}\gamma}, \suline{abac|\alpha_{n-2}\beta_2 }
\big\}\cup \omega_2 \mathfrak{D}_{n-8,0}, \] 
and if $n\geqslant 10$ is even, we define  
\[ \mathfrak{H}_{n,4}=\big\{  \suline{abac|\alpha_{n-1}\beta}, \suline{\omega_2 1}|(\suline{\alpha_{n-9}\beta}-\alpha_{n-9}\gamma)\big\}. \]

The previous results can be restated as follows.

\begin{cor}\label{corollary:tildeH}
Let $m\in \llbracket 0,4 \rrbracket$ and  $n\in \NN_0$. Then $H_{n,m}=\tilde{H}_{n,m}\oplus \omega_1 H_{n-4, m-2}$ except for $(n,m)=(4,2)$. Moreover, $H_{4,2}=\tilde{H}_{4,2}=0$. Here, $\tilde{H}_{n,m}$ is the $\Bbbk$-vector space spanned by the set $\tilde{\mathfrak{H}}_{n,m}$. The set $\tilde{\mathfrak{H}}_{n,m}$ is defined as follows.
If $m=0$ or $1$, we define the set 
$\tilde{\mathfrak{H}}_{n,m}=\mathfrak{H}_{n,m}$ 
for $n\in\NN_0$. 
If $m=2$, we define the sets 
\begin{align*}
    \tilde{\mathfrak{H}}_{0,2} &  =\big\{ab|\epsilon^!,bc|\epsilon^!\big\},
    \quad 
    \tilde{\mathfrak{H}}_{1,2}=\big\{(ba+ac)|(\beta+\gamma), ac|(\alpha+\gamma)  \big\},	
    \\
    \tilde{\mathfrak{H}}_{2,2}& =\big\{ab|(\alpha\beta-2\alpha_2-\beta_2)+bc|(\beta_2-\alpha_2), ab|(\beta_2-\gamma_2)+bc|(\alpha\beta-\beta_2-2\gamma_2)    \big\},
\end{align*}
and $\tilde{\mathfrak{H}}_{n,2}=\emptyset$ for $n\geqslant 3$.
If $m=3$, we define the set  	 
$\tilde{\mathfrak{H}}_{0,3}=\emptyset$,  
and   
\[ \tilde{\mathfrak{H}}_{n,3}=\big\{(aba+bac)|\alpha_{n-1}\beta\big\} 
\] 
for $n\in\NN$ with $n$ odd, 
together with 
\[ \tilde{\mathfrak{H}}_{n,3}=\big\{bac|\alpha_n, aba|\alpha_{n-2}\beta_2, abc|\alpha_{n-2}\beta_2, bac|\alpha_{n-2}\beta_2 \big\}\] 
for $n\geqslant 2$ with $n$ even. 
If $m=4$, we define the set 
$\tilde{\mathfrak{H}}_{0,4}=\emptyset$, 
and
\[ \tilde{\mathfrak{H}}_{n,4}=\big\{abac|\alpha_n, abac|\alpha_{n-1}\beta, abac|\alpha_{n-1}\gamma, \alpha_{n-2}\beta_2\big\}
\] 
for $n\in\NN$ with $n$ odd,
together with 
	\[ \tilde{\mathfrak{H}}_{n,4}=\big\{abac|\alpha_{n-1}\beta \big\}
	\] 
for $n\geqslant 2$ with $n$ even. 
Furthermore, if we define $\tilde{H}_{n,m}=0$ for $(n,m)\in \ZZ^2\setminus (\NN_0\times \llbracket 0,4 \rrbracket)$, then $H_{n,m}=\tilde{H}_{n,m}\oplus \omega_1 H_{n-4, m-2}$ holds for $(n,m)\in \ZZ^2\setminus\{ (4,2) \}$ by applying Corollary \ref{cor recursive homology}.
\end{cor}

\begin{rk}
	The reader can easily check that  $\tilde{D}_{n,m}= \tilde{H}_{n,m}\oplus \tilde{B}_{n,m}$ except the case $m=n=3$.
\end{rk}

Recall that the Hochschild homology is decomposed as $\operatorname{HH}_n(A)=\oplus_{m\in\NN_0}H_{n,m}$ for $n\in\NN_0$.

\begin{prop}
\label{prop:linear HHnA}
Let $n\in\NN$. 
Then
\[ \operatorname{HH}_n(A) = \mathop{\bigoplus}\limits_{\substack{i\in \llbracket 0, \lfloor n/4 \rfloor \rrbracket, \\
m\in \llbracket 0,4 \rrbracket}}\omega_{i}  \tilde{H}_{n-4i,m} \] 
for $4\nmid n$, and 
\[ \operatorname{HH}_n(A)=\bigg(\mathop{\bigoplus}\limits_ {\substack{i\in \llbracket 0, n/4-1 \rrbracket, \\
m \in \llbracket 0,4 \rrbracket}}\omega_{i}  \tilde{H}_{n-4i,m}\bigg) \oplus \bigg(\mathop{\bigoplus}\limits_{m\in \llbracket 1,4 \rrbracket}\omega_{n/4}\tilde{H}_{0,m}\bigg) \] 
for $4|n$.
\end{prop}
\begin{proof}
By Corollary \ref{corollary:tildeH}, we have 
\begin{equation}\label{h420}
    \begin{split}
    H_{n,2}& =\tilde{H}_{n,2}\oplus \omega_1 \tilde{H}_{n-4,0} \textnormal{ for $n\in \NN_0\setminus \{4 \}$}, \quad H_{4,2}= \tilde{H}_{4,2},\\
    H_{n,3}& =\tilde{H}_{n,3}\oplus \omega_1 \tilde{H}_{n-4,1} \textnormal{ for $n\in\NN_0$}, \\
    H_{n,4}& =\tilde{H}_{n,4}\oplus \omega_1\tilde{H}_{n-4,2}\oplus \omega_2\tilde{H}_{n-8,0}   
    \textnormal{ for $n\in \NN_0\setminus \{8 \}$},\quad
    H_{8,4}=\tilde{H}_{8,4}\oplus\omega_1\tilde{H}_{4,2}.
    \end{split}
\end{equation}
If $4\nmid n$, using Corollary \ref{cor recursive homology} and \eqref{h420}, we get 
%
%
%
\begin{align*}
\operatorname{HH}_n(A)
& =\bigoplus_{ m\in \llbracket 0, 2\lfloor n/4 \rfloor +4 \rrbracket}H_{n,m}
\\
& = H_{n,0}\oplus H_{n,1} \oplus H_{n,2}
\oplus \bigg( \bigoplus_{i\in \llbracket 0, \lfloor n/4 \rfloor \rrbracket }\omega_i H_{n-4i,3} \bigg) 
\oplus \bigg( \bigoplus_{i\in \llbracket 0, \lfloor n/4 \rfloor \rrbracket }\omega_i H_{n-4i,4} \bigg) 
\\
& = \tilde{H}_{n,0}\oplus \tilde{H}_{n,1}
\oplus (\tilde{H}_{n,2}\oplus \omega_1 \tilde{H}_{n-4,0})
\oplus \bigg( \bigoplus_{i\in \llbracket 0, \lfloor n/4 \rfloor \rrbracket }\omega_i (\tilde{H}_{n-4i,3}\oplus \omega_1\tilde{H}_{n-4i-4,1}) \bigg) \\
& \phantom{= \; }
\oplus \bigg( \bigoplus_{i\in \llbracket 0, \lfloor n/4 \rfloor \rrbracket }\omega_i ( \tilde{H}_{n-4i,4}\oplus \omega_1 \tilde{H}_{n-4i-4,2}\oplus \omega_2\tilde{H}_{n-4i-8,0}  )\bigg)
\\
& = \tilde{H}_{n,0}\oplus \tilde{H}_{n,1}
\oplus \tilde{H}_{n,2}\oplus \omega_1 \tilde{H}_{n-4,0}
\oplus \bigg( \bigoplus_{i\in \llbracket 0, \lfloor n/4 \rfloor \rrbracket }\omega_i \tilde{H}_{n-4i,3} \bigg)
\\
& \phantom{= \; } 
\oplus \bigg( \bigoplus_{i\in \llbracket 0, \lfloor n/4 \rfloor \rrbracket }\omega_{i+1}\tilde{H}_{n-4i-4,1} \bigg)
\oplus \bigg( \bigoplus_{i\in \llbracket 0, \lfloor n/4 \rfloor \rrbracket }\omega_i \tilde{H}_{n-4i,4} \bigg)
\\
& \phantom{= \; } 
\oplus \bigg( \bigoplus_{i\in \llbracket 0, \lfloor n/4 \rfloor \rrbracket }\omega_{i+1} \tilde{H}_{n-4i-4,2} \bigg)
\oplus \bigg( \bigoplus_{i\in \llbracket 0, \lfloor n/4 \rfloor \rrbracket }\omega_{i+2} \tilde{H}_{n-4i-8,0} \bigg)
\\
& = \tilde{H}_{n,0}\oplus \tilde{H}_{n,1}
\oplus \tilde{H}_{n,2}\oplus \omega_1 \tilde{H}_{n-4,0}
\oplus \bigg( \bigoplus_{i\in \llbracket 0, \lfloor n/4 \rfloor \rrbracket }\omega_i \tilde{H}_{n-4i,3} \bigg)
\oplus \bigg( \bigoplus_{i\in \llbracket 1, \lfloor n/4 \rfloor \rrbracket }\omega_{i}\tilde{H}_{n-4i,1} \bigg)
\\
& \phantom{= \; } 
\oplus \bigg( \bigoplus_{i\in \llbracket 0, \lfloor n/4 \rfloor \rrbracket }\omega_i \tilde{H}_{n-4i,4} \bigg)
\oplus \bigg( \bigoplus_{i\in \llbracket 1, \lfloor n/4 \rfloor \rrbracket }\omega_{i} \tilde{H}_{n-4i,2} \bigg)
\oplus \bigg( \bigoplus_{i\in \llbracket 2, \lfloor n/4 \rfloor \rrbracket }\omega_{i} \tilde{H}_{n-4i,0} \bigg)
\\
& = \mathop{\bigoplus}\limits_{\substack{i\in \llbracket 0, \lfloor n/4 \rfloor \rrbracket, \\
m\in \llbracket 0,4 \rrbracket}}\omega_{i}  \tilde{H}_{n-4i,m}.
\end{align*}
If $4|n$, the proof is similar to above. 
Note that if $n=4$, there is no term $\omega_1\tilde{H}_{0,0}$ when decomposing $H_{4,2}$, 
and if $n\geqslant 8$, there is no term $\omega_{n/4}\tilde{H}_{0,0}$ when decomposing $\omega_{n/4-2}H_{8,4}$. 
\end{proof}

Here is a table of the dimensions of $H_{n,m}$ and $\operatorname{HH}_n(A)$ for $n \in \llbracket 0, 19 \rrbracket$ and $m \in \llbracket 0, 12 \rrbracket$. 

\begin{table}[H]
\begin{center}
\resizebox{\textwidth}{30mm}{
\begin{tabular}{|c|c|c|c|c|c|c|c|c|c|c|c|c|c|c|c|c|c|c|c|c|c|}
	\hline
	\diagbox{$m$}{$n$}&0&1&2&3&4&5&6&7&8&9&10&11&12&13&14&15&16&17&18&19\\
	\hline
	0 &1&3&1&4&1&4&1&4&1&4&1&4&1&4&1&4&1&4&1&4\\
	1 &3&3&6&3&4&1&4&1&4&1&4&1&4&1&4&1&4&1&4&1\\
	2 &2&2&2&0&0&3&1&4&1&4&1&4&1&4&1&4&1&4&1&4\\
	3 &0&0&1&1&7&4&10&4&8&2&8&2&8&2&8&2&8&2&8&2\\
	4 &0&1&1&4&3&6&3&4&1&7&2&8&2&8&2&8&2&8&2&8\\
	5 & & & & &0&0&1&1&7&4&10&4&8&2&8&2&8&2&8&2\\
	6 & & & & &0&1&1&4&3&6&3&4&1&7&2&8&2&8&2&8 \\
	7 & & & & & & & & &0&0&1&1&7&4&10&4&8&2&8&2\\
	8 & & & & & & & & &0&1&1&4&3&6&3&4&1&7&2&8\\
	9 & & & & & & & & & & & & &0&0&1&1&7&4&10&4\\
	10& & & & & & & & & & & & &0&1&1&4&3&6&3&4\\
	11& & & & & & & & & & & & & & & & &0&0&1&1\\
	12& & & & & & & & & & & & & & & & &0&1&1&4\\
	\hline
	$\operatorname{HH}_n$&6&9&11&12&15&19&21&22&25&29&31&32&35&39&41&42&45&49&51&52    \\
	\hline
\end{tabular}
}
\end{center}
\caption{Dimension of $H_{n,m}$ and $\operatorname{HH}_n(A)$.}	
\end{table}

\begin{prop}\label{prop dim HHA}
The dimension of $\operatorname{HH}_n(A)$ is given by 
\[ 
\operatorname{dim}\operatorname{HH}_n(A) =
\begin{cases}
6,  & \text{if $n=0$}, 
\\
\frac{5}{2}n+5,  &\text{if $n=4r$ for $r\in\NN$},
\\
\frac{5n+13}{2}, &\text{if $n=4r+1$ for $r\in\NN_0$},
\\
\frac{5}{2}n+6,  & \text{if $n=4r+2$ for $r\in\NN_0$},
\\
\frac{5n+9}{2},  & \text{if $n=4r+3$ for $r\in\NN_0$}.
\end{cases} 
\] 
\end{prop}

The Hilbert series of $\operatorname{HH}_n(A)$ is $h_n(t)=\sum_{m\in\NN_0}\operatorname{dim}(H_{n,m})t^{m+n}$ for $n\in\NN_0$. Note that $m+n$ is the internal degree of $H_{n,m}$.

\begin{cor}\label{cor hilbert series homology}
The Hilbert series $h_n(t)$ of $\operatorname{HH}_n(A)$ is given as follows. Let $n\geqslant 6$. Then 
\[ 
    h_n(t) =t^n\big [1+3\chi_{n+1}+(3\chi_n+1)t+(1+3\chi_{n+1})t^{2}+\sum\limits_{i=0}^{\mu_n}\big( (2+6\chi_n)t^{3+2i}
    +(2+6\chi_{n+1})t^{4+2i}\big) +p_n(t) \big],
\] 
where
\[ p_n(t)=
      \begin{cases}
      8t^{2\lfloor \frac{n}{4}\rfloor -1} 
      +t^{2\lfloor \frac{n}{4}\rfloor}
      +7t^{2\lfloor \frac{n}{4}\rfloor+1}
      +3t^{2\lfloor \frac{n}{4}\rfloor+2}, 
      &  \text{if $n\equiv 0$ $(\operatorname{mod}$ $4)$},
      \\
      2t^{2\lfloor \frac{n}{4}\rfloor-1}
      +7t^{2\lfloor \frac{n}{4}\rfloor}
      +4t^{2\lfloor \frac{n}{4}\rfloor+1}
      +6t^{2\lfloor \frac{n}{4}\rfloor+2}
      +t^{2\lfloor \frac{n}{4}\rfloor+4}, 
      & \text{if $n\equiv 1$ $(\operatorname{mod}$ $4)$},
      \\
      10t^{2\lfloor \frac{n}{4}\rfloor+1}
      +3t^{2\lfloor \frac{n}{4}\rfloor+2}
      +t^{2\lfloor \frac{n}{4}\rfloor+3}
      +t^{2\lfloor \frac{n}{4}\rfloor+4}, 
      & \text{if $n\equiv 2$ $(\operatorname{mod}$ $4)$},
      \\
      4t^{2\lfloor \frac{n}{4}\rfloor+1}
      +4t^{2\lfloor \frac{n}{4}\rfloor+2}
      +t^{2\lfloor \frac{n}{4}\rfloor+3}
      +4t^{2\lfloor \frac{n}{4}\rfloor+4}, 
      & \text{if $n\equiv 3$ $(\operatorname{mod}$ $4)$},
      \end{cases}
\] 
and 
\[ \mu_n=
\begin{cases}
\lfloor \frac{n}{4}\rfloor-3,&\text{if $n\equiv 0,1$ $(\operatorname{mod}$ $4)$},
\\
\lfloor \frac{n}{4}\rfloor-2,&\text{if $n\equiv 2,3$ $(\operatorname{mod}$ $4)$}.
\end{cases}
\]
Moreover,
\begin{align*}
h_0(t)&=1+3t+2t^2, \quad &
h_1(t)&=3t+3t^2+2t^3+t^5,
\\
h_2(t)& =t^2+6t^3+2t^4+t^5+t^6, \quad &
h_3(t)& =4t^3+3t^4+t^6+4t^7,
\\
h_4(t)& =t^4+4t^5+7t^7+3t^8, \quad &
h_5(t)& =4t^5+t^6+3t^7+4t^8+6t^9+t^{11}.	
\end{align*}
\end{cor}

\begin{rk}
As we mentioned at the beginning of Subsection \ref{subsection:com cycles}, one can obtain Proposition \ref{prop dim HHA} and Corollary \ref{cor hilbert series homology} directly from the computations in Subsection \ref{subsection: boundaries} together with Corollary \ref{cor recursive homology}, but a specific choice of cycles for the Hochschild homology can be useful for later computations.
\end{rk}

\subsection{Cyclic homology}

In this subsection we assume that the characteristic of the field $\Bbbk$ is zero. 
Recall that the 
{\color{ultramarine}{\textbf{reduced Hochschild homology}}}
of $A$ is given by 
\begin{equation}\label{reduced Hochschild homology}
	\begin{split}
	\overline{\operatorname{HH}}_{n}(A)=
	\begin{cases}
	\operatorname{HH}_0(A)/\Bbbk, & \text{if $n=0$}, 
	\\
	\operatorname{HH}_n(A), & \text{if $n\in\NN$}, 
	\end{cases} 
\end{split}
\end{equation}
and the 
{\color{ultramarine}{\textbf{reduced cyclic homology}}} 
of $A$ is given by 
\begin{equation}\label{reduced cyclic homology}
	\begin{split}
	\overline{\operatorname{HC}}_{n}(A)=
	\begin{cases}
	\operatorname{HC}_n(A)/\Bbbk, &  \text{if $n\in \NN_0$ is even}, 
	\\
	\operatorname{HC}_n(A), & \text{if $n\in \NN$ is odd}, 
	\end{cases}
\end{split}
\end{equation}
where $\operatorname{HC}_n(A)$ for $n\in\NN_0$ is the cyclic homology of $A$ (see \cite{Louis}).
As a consequence of Goodwillie's Theorem (see \cite{Weibel}, Thm. 9.9.1), we have the isomorphism of graded vector spaces
\begin{equation}\label{Goodwillie}
	\begin{split}
	\overline{\operatorname{HC}}_{n}(A)\cong
	\begin{cases}
	\overline{\operatorname{HH}}_0(A), & \text{if $n= 0$}, 
	\\
	\overline{\operatorname{HH}}_n(A)/\overline{\operatorname{HC}}_{n-1}(A), & \text{if $n\in\NN$}. 
	\end{cases}
\end{split}
\end{equation}

\begin{cor}\label{hseriescyclichomology} 
Assume that the characteristic of $\Bbbk$ is zero. 
Let $g_{n}(t)$ be the Hilbert series of $\overline{\operatorname{HC}}_{n}(A)$ for $n\in\NN_0$. 
Then
\begin{align*}
g_0(t)=3t+2t^2, \quad 
g_1(t)=t^2+2t^3+t^5, \quad 
g_2(t)=4t^3+2t^4+t^6, \quad 
g_3(t)=t^4+4t^7,
\end{align*}
and for $n\geqslant 4$,
\[   
g_n(t)=t^{n+1}\big[ 
1+3\chi_n+\sum\limits_{i=0}^{\lfloor\frac{n}{4}\rfloor-2}\big( (1+3\chi_n)t^{2+2i}+(1+3\chi_{n+1})t^{3+2i}  \big)+t^{2\lfloor\frac{n}{4}\rfloor}q_n(t)
\big],
\] 
where
\[ 
q_n(t)=
\begin{cases}
3+3t, 
& \text{if $n\equiv 0$ $(\operatorname{mod}$ $4)$},
\\
1+6t+t^3,
& \text{if $n\equiv 1$ $(\operatorname{mod}$ $4)$},
\\
4+3t+t^3,
& \text{if $n\equiv 2$ $(\operatorname{mod}$ $4)$},
\\
1+4t+4t^3,
& \text{if $n\equiv 3$ $(\operatorname{mod}$ $4)$}.
\end{cases}
\] 
%
%
%
\end{cor}
\begin{proof}
By \eqref{Goodwillie}, we have 
\begin{equation}\label{Goodwillie1}
	\begin{split}
	g_n(t)= 
	\begin{cases}
	h_0(t)-1,          & \text{if $n= 0$}, 
	\\
	h_n(t)-g_{n-1}(t), & \text{if $n\in\NN$}. 
	\end{cases} 
\end{split}
\end{equation}
Then we get the result by induction.
\end{proof}

\begin{rk}\label{rk:cyclic coho}
	The cyclic cohomology of $A$ is isomorphic to the dual space of the cyclic homology of $A$, so their Hilbert series coincide (see \cite{Louis}).
\end{rk}

\section{\texorpdfstring{Hochschild cohomology of $\FK(3)$}{Hochschild cohomology of FK(3)}}
\label{Hochschild cohomology}  

In this section we will compute the linear structure of the Hochschild cohomology $\operatorname{HH}^{\bullet}(A)=\operatorname{Ext}_{A^e}^{\bullet}(A,A)$ by means of the complex $\operatorname{H}^{\bullet}(\operatorname{Hom}_{A^e}(P_{\bullet}^b,A))$. 
We refer the reader to \cite{Sarah} for further information about Hochschild cohomology. 

\subsection{Recursive description of the spaces}
\label{rdsco}

Let $K^n=\operatorname{Hom}_{\Bbbk}((A_{-n}^!)^*,A)$ for $n\in\NN_0$ and $K^n=0$ for $n\in \ZZ\setminus \NN_0$. We have $\operatorname{Hom}_{A^e}(P_{n}^b,A)\cong Q^n $ as $\Bbbk$-vector spaces, where $Q^n= \oplus_ {i\in \llbracket 0, \lfloor n/4 \rfloor \rrbracket} \omega^*_i K^{n-4i}$ for $n\in\NN_0$ and $Q^n=0$ for $n\in \ZZ\setminus \NN_0$. We will denote by $\partial^{n}:Q^n\to Q^{n+1}$ the differential
\[
(\delta_{n+1}^b)^*: \operatorname{Hom}_{A^e}(P_{n}^b,A)\to \operatorname{Hom}_{A^e}(P_{n+1}^b,A), 
\] 
by $d^n:K^n\to K^{n+1}$ the differential 
\[ 
(d_{n+1}^b)^*: \operatorname{Hom}_{A^e}(K_{n}^b,A)\to \operatorname{Hom}_{A^e}(K_{n+1}^b,A),
\] 
and by $f^n:K^{n+3}\to K^{n}$ the map
\[ 
(f_n^b)^{*}:\operatorname{Hom}_{A^e}(K_{n+3}^b,A)\to \operatorname{Hom}_{A^e}(K_n^b,A)\] 
for $n\in\ZZ$. 
Then the differential $\partial^n$ for $n\in\NN_0$ is given by 
\begin{equation}
\label{eq:diff-partial-cohomology}
 \partial^{n}\bigg(\sum\limits_{\substack{i\in \llbracket 0, \lfloor n/4\rfloor \rrbracket}}\omega^*_i\xi _{n-4i}\bigg)=\sum\limits_{\substack{i\in \llbracket 0, \lfloor n/4\rfloor \rrbracket}}\big(\omega^*_i d^{n-4i}(\xi_{n-4i})+\omega^*_{i+1}f^{n-4i-3}(\xi_{n-4i}) \big), 
 \end{equation}
where $\xi_j\in K^j$ for $j\in\NN_0$. Note that $\partial^{n}=\tilde{\partial}^{n}=0$ for $n\in \ZZ\setminus \NN_0$.

Our aim is to compute the cohomology of $(Q^{\bullet}, \partial^{\bullet})$. Let $K^n_m=\operatorname{Hom}_{\Bbbk}((A_{-n}^!)^*,A_m)$ for $(n,m)\in \NN_0 \times \llbracket 0,4 \rrbracket$ and $K^n_m=0$ for $(n,m)\in \ZZ^2\setminus (\NN_0 \times \llbracket 0,4 \rrbracket)$. 
Let $Q_{m}^n=\oplus_ {i\in \llbracket 0, \lfloor n/4 \rfloor \rrbracket}\omega^*_i K^{n-4i}_{m+2i}$ for $(n,m)\in \NN_0\times \ZZ_{\leqslant 4}$ 
and $Q^n_m=0$ for $(n,m)\in\ZZ^2\setminus (\NN_0\times \ZZ_{\leqslant 4})$, 
where the symbol $\omega^*_i$ has cohomological degree $4i$ and internal degree $-6i$ for $i\in\NN_0$, and we usually omit $\omega^*_0$ for simplicity. 
The spaces $K^n_m$ and $Q^n_m$ are concentrated in cohomological degree $n$ and internal degree $m-n$.
We have $Q^n=\oplus_{m\leqslant 4}Q^n_m$. 
Let $\partial^n_m=\partial^n|_{Q^n_m} :Q^n_m\to Q^{n+1}_{m+1}$, 
and $d^n_m=d^n|_{K^n_m}:K^n_m\to K^{n+1}_{m+1}$. 
Let $D^n_m=\Ker(\partial^n_m)$, $B^n_m=\Img(\partial^{n-1}_{m-1})$ for $(n,m)\in \NN_0\times \ZZ_{\leqslant 4}$, and $\tilde{D}^n_m=\Ker(d^n_m)$, $\tilde{B}^n_m=\Img(d^{n-1}_{m-1})$ for $(n,m)\in \NN_0\times \llbracket 0,4\rrbracket $. 
Notice that $D^n_m=B^n_m=0$ for $(n,m)\in \ZZ^2\setminus ( \NN_0\times \ZZ_{\leqslant 4})$, and $\tilde{D}^n_m=\tilde{B}^n_m=0$ for $(n,m)\in \ZZ^2\setminus (\NN_0\times \llbracket 0,4\rrbracket) $.

\begin{rk}\label{mrange}
	We have $Q^n=\oplus_{m\in\llbracket -2\lfloor n/4\rfloor, 4\rrbracket }Q^n_m$ since the indices in $Q_{m}^n=\oplus_ {i\in \llbracket 0, \lfloor n/4 \rfloor \rrbracket}\omega^*_i K^{n-4i}_{m+2i}$ satisfy $n-4i\in \NN_0$ and $ m+2i\in \llbracket 0, 4 \rrbracket$.
\end{rk}

\begin{prop}\label{bdnmco}
	For integers $m\leqslant 1$ and $n\in\NN_0$, we have 
	\begin{equation}\label{bnmco}
		\begin{split}
	 B_{m}^n=
	 \begin{cases}
	\omega^*_{\frac{1-m}{2}}B_{1}^{n+2m-2},  &  \text{if $m$ is odd}, 
	\\
	\omega^*_{-\frac{m}{2}}B_0^{n+2m}, & \text{if $m$ is even},
	\end{cases}
\end{split}
\end{equation}
and
\begin{equation}\label{dnmco}
\begin{split}
D_{m}^n=
\begin{cases}
	\omega^*_{\frac{1-m}{2}}D_{1}^{n+2m-2},  &  \text{if $m$ is odd}, 
	\\
	\omega^*_{-\frac{m}{2}}D_{0}^{n+2m}, & \text{if $m$ is even},
	\end{cases}
\end{split}
\end{equation}
where we follow the convention that $\omega^*_i\omega^*_j=\omega^*_{i+j}$ for $i,j\in \NN_0$ and $\omega^*_i=0$ for $i\in \ZZ\setminus \NN_0$. 
\end{prop}
\begin{proof}
Consider $Q_{m}^n=\oplus_ {i\in \llbracket 0, \lfloor n/4 \rfloor \rrbracket}\omega^*_i K^{n-4i}_{m+2i}$ for integers $m\leqslant 4$ and $n\in\NN_0$.
If $m$ is odd, then $m+2i=1$ or $3$, \textit{i.e.}
$i=(1-m)/2$ or $(3-m)/2$. 
We have 
\begin{equation}\label{qnmodd}
\begin{split}
	Q_{m}^n=
	\begin{cases}
		\omega^*_{\frac{1-m}{2}}K^{n+2m-2}_{1}\oplus \omega^*_{\frac{3-m}{2}}K^{n+2m-6}_{3},   & \text{if $n\geqslant 6-2m$},
		\\
		\omega^*_{\frac{1-m}{2}}K^{n+2m-2}_{1}, & \text{if $2-2m\leqslant n<6-2m$},
		\\
	    0,  & \text{if $0\leqslant n<2-2m$}.
	\end{cases} 
\end{split}
\end{equation}
If $m$ is even, then $m+2i=0, 2$ or $4$, \textit{i.e.}
$i=-m/2$, $1-m/2$ or $2-m/2$. We have 
\begin{equation}\label{qnmeven}
	\begin{split}
	Q_{m}^n=
	\begin{cases}
		\omega^*_{-\frac{m}{2}}K^{n+2m}_{0}\oplus \omega^*_{1-\frac{m}{2}}K^{n+2m-4}_{2} \oplus \omega^*_{2-\frac{m}{2}}K^{n+2m-8}_{4},  & \text{if $n\geqslant 8-2m$},
		\\
		\omega^*_{-\frac{m}{2}}K^{n+2m}_{0}\oplus \omega^*_{1-\frac{m}{2}}K^{n+2m-4}_{2},   & \text{if $4-2m\leqslant n<8-2m$}, 
		\\
		\omega^*_{-\frac{m}{2}}K^{n+2m}_{0}, & \text{if $-2m\leqslant n<4-2m$}, 
	    \\
	    0,  & \text{if $0\leqslant n<-2m$}.
	\end{cases} 
\end{split}
\end{equation}
Hence, 
\begin{equation}\label{qnm}
	\begin{split}
	 Q_{m}^n=
	 \begin{cases}
	\omega^*_{\frac{1-m}{2}}Q_{1}^{n+2m-2},  & \text{if $m\leqslant 1$ is odd}, 
	\\
	\omega^*_{-\frac{m}{2}}Q_{0}^{n+2m}, & \text{if $m\leqslant 0$ is even}.
	\end{cases}
\end{split}
\end{equation}
Since the identities \eqref{bnmco} and \eqref{dnmco} for $m=1$ are immediate, we suppose $m\leqslant 0$ from now on.

Assume that $m$ is even. 
Then \eqref{qnm} tells us that the sequence 
\[ 
Q^{n-1}_{m-1}\xrightarrow{\partial^{n-1}_{m-1}}  Q^n_m \xrightarrow{\partial^n_m} Q^{n+1}_{m+1}
\] 
of graded $\Bbbk$-vector spaces is of the form 
\[ 
\omega^*_{1-\frac{m}{2}}Q_1^{n+2m-5}
\xrightarrow{\partial^{n-1}_{m-1}} 
\omega^*_{-\frac{m}{2}}Q^{n+2m}_{0}
\xrightarrow{\partial^n_m} 
\omega^*_{-\frac{m}{2}}Q_{1}^{n+2m+1}.
\]
Since $Q_{-1}^{n+2m-1}=\omega^*_1 Q_1^{n+2m-5}$ by \eqref{qnm}, 
the above sequence is of the form 
\[
\omega^*_{-\frac{m}{2}}Q_{-1}^{n+2m-1}
\xrightarrow{\partial^{n-1}_{m-1}} 
\omega^*_{-\frac{m}{2}}Q^{n+2m}_{0}
\xrightarrow{\partial^n_m}
\omega^*_{-\frac{m}{2}}Q_{1}^{n+2m+1}.
\]
Note further that 
$\partial^n_m=\omega^*_{-\frac{m}{2}}\partial^{n+2m}_0$ 
and 
$\partial^{n-1}_{m-1}=\omega^*_{1-\frac{m}{2}}\partial^{n+2m-5}_{1}=\omega^*_{-\frac{m}{2}}\partial^{n+2m-1}_{-1}   $, 
where the differential $\omega^*_j\partial^{n'}_{m'}:\omega^*_j Q^{n'}_{m'}\to \omega^*_j Q^{n'+1}_{m'+1}$ maps $\omega^*_jx$ to $\omega^*_j\partial^{n'}_{m'}(x)$ for all $x\in Q^{n'}_{m'}$, $j,n'\in \NN_0$ and for all integers $m'\leqslant 4$. 
Hence, $B_m^n=\omega^*_{-\frac{m}{2}}B^{n+2m}_0$ and $D_m^n=\omega^*_{-\frac{m}{2}}D^{n+2m}_0$.

Assume that $m$ is odd (so $m\leqslant -1$). 
Then \eqref{qnm} tells us that the sequence 
\[ Q^{n-1}_{m-1}\xrightarrow{\partial^{n-1}_{m-1}}  Q^n_m \xrightarrow{\partial^n_m} Q^{n+1}_{m+1}\] 
of graded $\Bbbk$-vector spaces is of the form
\[\omega^*_{\frac{1-m}{2}}Q_0^{n+2m-3}
\xrightarrow{\partial^{n-1}_{m-1}}
\omega^*_{\frac{1-m}{2}}Q^{n+2m-2}_{1}
\xrightarrow{\partial^{n}_{m}}
\omega^*_{-\frac{m+1}{2}}Q_{0}^{n+2m+3}.\] 
Note that  
$\partial^{n-1}_{m-1}=\omega^*_{\frac{1-m}{2}}\partial^{n+2m-3}_0$.	
Moreover, we also have 
\begin{equation}
\label{eq:decomposition-Q-chomology}
    \omega^*_{-\frac{m+1}{2}}Q_0^{n+2m+3}=\omega^*_{-\frac{m+1}{2}} K^{n+2m+3}_0\oplus \omega^*_{\frac{1-m}{2}} Q^{n+2m-1}_2 
\end{equation}
by \eqref{qnmeven}, 
and the image of $\omega^*_{\frac{1-m}{2}}Q_1^{n+2m-2}$ is contained in $\omega^*_{\frac{1-m}{2}}Q_2^{n+2m-1}$ by the explicit expression of the differential \eqref{eq:diff-partial-cohomology}. 
Furthermore, the composition of $\partial^n_m$ with the canonical projection 
\[       \omega^*_{-\frac{m+1}{2}}Q_0^{n+2m+3} \longrightarrow  \omega^*_{\frac{1-m}{2}} Q^{n+2m-1}_2        \]
induced by \eqref{eq:decomposition-Q-chomology} is precisely
$\omega^*_{\frac{1-m}{2}}\partial^{n+2m-2}_1$.
It is thus sufficient to consider the sequence 
\[ \omega^*_{\frac{1-m}{2}}Q_0^{n+2m-3}
\xrightarrow{\omega^*_{\frac{1-m}{2}}\partial^{n+2m-3}_0}
\omega^*_{\frac{1-m}{2}}Q^{n+2m-2}_{1}
\xrightarrow{\omega^*_{\frac{1-m}{2}}\partial^{n+2m-2}_1}
\omega^*_{\frac{1-m}{2}}Q_{2}^{n+2m-1}. \]
Hence, $B_{m}^n=\omega^*_{\frac{1-m}{2}}B_1^{n+2m-2}$ and $D_{m}^n=\omega^*_{\frac{1-m}{2}}D_1^{n+2m-2}$, as was to be shown.
\end{proof}
	
Throughout the Sections \ref{Hochschild cohomology} and \ref{section:algebra-cohomology} we will use the symbol $y|x$, where $x\in \mathcalboondox{B}_{m}$ and $y\in \mathcalboondox{B}_{n}^{!*}$, to denote the $\Bbbk$-linear map in $K^n=\operatorname{Hom}_{\Bbbk}((A^!_{-n})^*,A)$, which maps $y$ to $x$ and sends the other usual basis elements of $(A^!_{-n})^*$ to zero. 
Even though one usually denotes the previous map by $y||x$, we will use $y|x$ for the sake of reducing space in the expressions of the next subsection.

In order to compute $B^n_{m}$ and $D^n_{m}$, it is sufficient to compute the case $m\in \llbracket 0,4 \rrbracket$ according to Proposition \ref{bdnmco}.
First, we will compute the coboundaries, and then we will compute the cocycles. 
Since this will require handling elements of $K^n_{m}$ and $Q^n_{m}$ for $n \in \NN_{0}$ and $m \in \llbracket 0 , 4 \rrbracket$, we will use the basis $\{ y | x \hskip1mm | \hskip1mm x \in \mathcalboondox{B}_{m}, y \in \mathcalboondox{B}_{n}^{!*} \}$ of 
$K^n_{m}$ 
and the basis 
$\{ \omega^*_{i}  y|x  \hskip1mm | \hskip1mm i \in \llbracket 0 , \lfloor n/4 \rfloor \rrbracket, x \in \mathcalboondox{B}_{m+2i}, y \in \mathcalboondox{B}_{n-4i}^{!*} \}$ of 
$Q^n_{m}$, both of which will be called {\color{ultramarine}{\textbf{usual}}} bases, constructed from the usual bases of the homogeneous components of $A$ and $(A^{!})^{\#}$, introduced in Section \ref{section:generalities}.

\subsection{Explicit description of the differentials}
\label{subsection:explicit-description-diff-cohomology}

Recall the isomorphism $\operatorname{Hom}_{A^e}(A\otimes (A^!_{-n})^*\otimes A,A)\cong \operatorname{Hom}_{\Bbbk}((A^!_{-n})^*,A)$.
We will use it together with Proposition \ref{prb} to explicitly describe $d^n$ and $f^n$, which were defined at the beginning of Subsection \ref{rdsco}.

Let $x\in A$. It is then straightforward to see that the differential 
$d^0:\operatorname{Hom}_{\Bbbk}((A^!_{0})^*,A)\to \operatorname{Hom}_{\Bbbk}((A^!_{-1})^*,A) $ is given by $d^0(\epsilon^!|x)=\alpha|(xa-ax)+\beta|(xb-bx)+\gamma|(xc-cx)$. 
Analogously, for $n\in\NN$, $d^n:\operatorname{Hom}_{\Bbbk}((A^!_{-n})^*,A)\to \operatorname{Hom}_{\Bbbk}((A^!_{-(n+1)})^*,A)$ is given by
\begin{equation}\label{codifferential}
	\begin{split}
	\alpha_n|x & \mapsto   \alpha_{n+1}|[(-1)^{n+1}ax+xa]+\alpha_{n}\beta|(\chi_{n+1}cx-\chi_n bx+xb)
	\\
	& \phantom{ \mapsto \; } +\alpha_n\gamma|(\chi_{n+1}bx-\chi_n cx+xc),
	\\
	\beta_n|x&\mapsto \beta_{n+1}|[(-1)^{n+1}bx+xb]+\chi_{n+1}\alpha_n\beta|(ax+xc)+\alpha_n\gamma|[(-1)^{n+1}cx+\chi_{n+1}xa+\chi_n xc]
	\\
	&\phantom{ \mapsto \; }
	+\chi_n \alpha_{n-1}\beta_2|(xa-ax),
	\\
	\gamma_n|x &\mapsto \gamma_{n+1}|[(-1)^{n+1}cx+xc]+\alpha_n\beta|[(-1)^{n+1}bx+\chi_{n+1}xa+\chi_n xb]+\chi_{n+1}\alpha_n\gamma|(ax+xb)
	\\
	&\phantom{ \mapsto \; } +\chi_n\alpha_{n-1}\beta_2|(xa-ax),
	\\
	\alpha_{n-1}\beta|x & \mapsto \alpha_n\beta|[(-1)^{n+1}ax+xc]+\alpha_n\gamma|(\chi_{n+1}cx-\chi_n bx+xa)
	\\
	   & \phantom{ \mapsto \; }
	   +\alpha_{n-1}\beta_2|(\chi_{n+1}bx-\chi_n cx+xb),
	   \\
	\alpha_{n-1}\gamma|x&\mapsto \alpha_n\beta|(\chi_{n+1}bx-\chi_n cx+xa)+\alpha_n\gamma|[(-1)^{n+1}ax+xb]
	\\
	&\phantom{ \mapsto \; }
	+\alpha_{n-1}\beta_2|(\chi_{n+1}cx-\chi_n bx+xc),
	\\
	\alpha_{n-2}\beta_2|x &\mapsto \alpha_n\beta|(\chi_{n+1}cx-\chi_n bx+xb)+\alpha_n\gamma|(\chi_{n+1}bx-\chi_n cx+xc)
	\\
	&\phantom{ \mapsto \; }
	+\alpha_{n-1}\beta_2|[(-1)^{n+1}ax+xa].  
	\end{split}
\end{equation}
The $\Bbbk$-linear maps $f^n:\operatorname{Hom}_{\Bbbk}((A^!_{-(n+3)})^*,A)\to \operatorname{Hom}_{\Bbbk}((A^!_{-n})^*,A)$
are homogeneous of cohomological degree $-3$ and internal degree $6$. The map $f^0$ is given by 
\begin{align*}
    \alpha_3|x &\mapsto \epsilon^!|[2xbac-2bx(ab+bc)+2cxba+2abxc+2acxb+2bacx],
	\\
	\beta_3|x&\mapsto \epsilon^!|[2xabc-2ax(ba+ac)+2cxab+2bcxa+2baxc+2abcx],
	\\
	\gamma_3|x&\mapsto \epsilon^!|[-2xaba+2axbc+2bxac-2(ab+bc)xb-2(ba+ac)xa-2abax],
	\\
	\alpha_2\beta|x&\mapsto \epsilon^!|[-xabc+ax(ba+ac)-cxab-bcxa-baxc-abcx],
	\\
	\alpha_2\gamma|x &\mapsto \epsilon^!|[xaba-axbc-bxac+(ab+bc)xb+(ba+ac)xa+abax],
	\\
	\alpha\beta_2|x&\mapsto \epsilon^!|[-xbac+bx(ab+bc)-cxba-abxc-acxb-bacx].
	\end{align*}
For $n\in\NN$, $f^n$ is given by
\begin{small}
\allowdisplaybreaks
\begin{align*}
\alpha_{n+3}|x&\mapsto \alpha_n|[2xbac-\chi_n bx(ab+bc)+\chi_n cxba+\chi_n acxb+\chi_n abxc+(-1)^n 2bacx]\\
&\phantom{ \mapsto \; } +\beta_n|[\chi_n xbac-\chi_n bxbc+(-1)^n 2cxba+\chi_n(ba+ac)xb+2abxc+\chi_n bacx]\\
&\phantom{ \mapsto \; } +\gamma_n|	[\chi_n xbac+(-1)^{n+1}2bx(ab+bc)+\chi_n cx(ba+ac)+2acxb-\chi_n bcxc+\chi_n bacx]\\
&\phantom{ \mapsto \; } +\chi_{n+1}\alpha_{n-1}\beta|[axab-(n-2)cxba+cxac-baxa+(n-1)abxc+bcxc]\\
&\phantom{ \mapsto \; } +\chi_{n+1}\alpha_{n-1}\gamma|[axac+(n-1)bxab+(n-2)bxbc+(n-2)acxb-baxb+(ab+bc)xa]\\
&\phantom{ \mapsto \; } +\alpha_{n-2}\beta_2|\{\chi_{n+1}[(n-1)xbac+cx(ab+bc)-bxba+acxc+abxb-(n-1)bacx]\\
&\phantom{ \mapsto \; } +\chi_n (n-2)[xbac-bx(ab+bc)+cxba+acxb+abxc+bacx]\},\\
\beta_{n+3}|x&\mapsto \alpha_n|[\chi_n xabc+(-1)^n2cxab-\chi_n axac+2baxc+\chi_n (ab+bc)xa+\chi_n abcx]\\
&\phantom{ \mapsto \; } +\beta_n|[2xabc+\chi_n cxab-\chi_n ax(ba+ac)+\chi_n baxc+\chi_n bcxa+(-1)^n2abcx]\\
&\phantom{ \mapsto \; } +\gamma_n|[\chi_n xabc+\chi_n cx(ab+bc)+(-1)^{n+1}2ax(ba+ac)-\chi_n acxc+2bcxa+\chi_n abcx]\\
&\phantom{ \mapsto \; } +\chi_{n+1}\alpha_{n-1}\beta|[(n-1)xabc-axab+cx(ba+ac)+baxa+bcxc-(n-1)abcx]\\
&\phantom{ \mapsto \; } +\chi_{n+1}\alpha_{n-1}\gamma|[bxbc+(n-1)axba+(n-2)axac+(ba+ac)xb+(n-2)bcxa-abxa]\\
&\phantom{ \mapsto \; } +\alpha_{n-2}\beta_2|\{\chi_{n+1}[bxba-(n-2)cxab+cxbc+acxc+(n-1)baxc-abxb]\\
&\phantom{ \mapsto \; } +\chi_n(n-2)[xabc+cxab-ax(ba+ac)+baxc+bcxa+abcx]\},
\stepcounter{equation}\tag{\theequation}\label{cofbn}
\\
\gamma_{n+3}|x&\mapsto \alpha_n|[-\chi_n xaba-\chi_n axab+(-1)^n 2bxac-\chi_n baxa-2(ab+bc)xb-\chi_n abax]
\\
&\phantom{ \mapsto \; } +\beta_n|[-\chi_n xaba+(-1)^n 2axbc-\chi_n bxba-2(ba+ac)xa-\chi_n abxb-\chi_n abax]
\\
&\phantom{ \mapsto \; } +\gamma_n |[-2xaba+\chi_n axbc+\chi_n bxac-\chi_n(ba+ac)xa-\chi_n (ab+bc)xb+(-1)^{n+1}2abax]
\\
&\phantom{ \mapsto \; } +\chi_{n+1}\alpha_{n-1}\beta|[-axab-cx(ba+ac)-(n-1)axbc-(n-2)baxa-(n-1)acxa-bcxc]
\\
&\phantom{ \mapsto \; } +\chi_{n+1}\alpha_{n-1}\gamma|[-(n-1)xaba-bxbc-axac-(ba+ac)xb-(ab+bc)xa+(n-1)abax]
\\
&\phantom{ \mapsto \; } +\alpha_{n-2}\beta_2|\{ \chi_{n+1}[-cx(ab+bc)-(n-1)bxac-bxba-acxc-(n-1)bcxb-(n-2)abxb]
\\
&\phantom{ \mapsto \; }
+\chi_n (n-2)[-xaba+axbc+bxac-(ba+ac)xa-(ab+bc)xb-abax]  \},
\\
\alpha_{n+2}\beta|x&\mapsto \alpha_n|(-\chi_n cxab-\chi_n baxc)+\beta_n|(-\chi_n xabc-\chi_n abcx)+\gamma_n|[\chi_n ax(ba+ac)-\chi_n bcxa],
\\
\alpha_{n+2}\gamma|x&\mapsto \alpha_n|[-\chi_n bxac+\chi_n (ab+bc)xb]+\beta_n|[-\chi_n axbc+\chi_n (ba+ac)xa]
+\gamma_n|(\chi_n xaba+\chi_n abax),\\
\alpha_{n+1}\beta_2|x&\mapsto \alpha_n|[-xbac+\chi_{n+1}bxac+\chi_{n+1}cxab+\chi_{n+1}(ab+bc)xb-\chi_{n+1}baxc+(-1)^{n+1}bacx]
\\
&\phantom{ \mapsto \; }
+\beta_n|[-\chi_{n+1}xabc+(-1)^{n+1}cxba+\chi_{n+1}axbc-abxc+\chi_{n+1}(ba+ac)xa+\chi_{n+1}abcx]\\
&\phantom{ \mapsto \; }
+\gamma_n|[\chi_{n+1}xaba-\chi_{n+1}ax(ba+ac)+(-1)^n bx(ab+bc)-\chi_{n+1}bcxa-acxb-\chi_{n+1}abax].
\end{align*}
\end{small}

For the reader’s convenience, we list the images of the differentials $d^n$ and the maps $f^n$ evaluated at elements of the usual $\Bbbk$-basis of the respective domain. 
In the following tables, $d^n_{m}(y|x)$ is the entry appearing in the column indexed by $x$ and the row indexed by $y$, where $m$ is the internal degree of $x$ and $n$ is the internal degree of $y$.

If $n\in\NN$ is odd, the differential $d^n$ is given by 
\begin{table}[H]
	\begin{center}
	\begin{tabular}{|c|c|c|c|c|c|}
		\hline
		\diagbox[width=16mm , height=5mm]{$y$}{$x$}   & $abac$ & $aba$ & $abc$ & $bac$ \\
		\hline
		$\alpha_n$             &   0    & $(\alpha_n\gamma-\alpha_n\beta)|abac$ & $(\alpha_n\gamma-\alpha_n\beta)|abac$ & $0$   \\
		\hline
		$\beta_n$              &   0    & $(\alpha_n\beta-\alpha_n\gamma)|abac$ & $0$ & $(\alpha_n\beta-\alpha_n\gamma)|abac$   \\
		\hline     
		$\gamma_n$             &   0    &  0  & $(\alpha_n\beta-\alpha_n\gamma)|abac$ & $(\alpha_n\gamma-\alpha_n\beta)|abac$   \\
		\hline
        $\alpha_{n-1}\beta$    &   0    & $(\alpha_n\beta-\alpha_n\gamma)|abac$ &  0   & $(\alpha_n\beta-\alpha_n\gamma)|abac$   \\
		\hline
		$\alpha_{n-1}\gamma$   &   0    &  0   & $(\alpha_n\beta-\alpha_n\gamma)|abac$ & $(\alpha_n\gamma-\alpha_n\beta)|abac$   \\
		\hline
		$\alpha_{n-2}\beta_2$  &   0    & $(\alpha_n\gamma-\alpha_n\beta)|abac$ & $(\alpha_n\gamma-\alpha_n\beta)|abac$ &  0   \\
		\hline
	\end{tabular}
	\end{center}
	\caption{Images of $d^n$ for $n\in\NN$ and $n$ odd, where the last three lines are for $n\geqslant 3$ and $n$ odd.}
	\label{odd43}	
	\end{table}

\noindent together with

	\begin{table}[H]
		\begin{center}
		\resizebox{\textwidth}{15mm}{
		\begin{tabular}{|c|c|c|}
			\hline
			\diagbox[width=16mm , height=5mm]{$y$}{$x$}					   & $ab$ & $bc$  \\
			\hline
			$\alpha_n$             & $\alpha_{n+1}|aba+\alpha_n\beta|bac+\alpha_n\gamma|(aba+abc)$ & $\alpha_{n+1}|(abc-aba)-2\alpha_{n}\beta|bac$ \\
			\hline
			$\beta_n$              & $\beta_{n+1}|aba+\alpha_n\beta|abc+\alpha_n\gamma|(aba+bac)$  & $-\beta_{n+1}|bac+\alpha_n\beta|abc-\alpha_n\gamma|(aba+bac)$     \\
			\hline     
			$\gamma_n$             & $\gamma_{n+1}|(abc+bac)+2\alpha_n\beta|aba$ & $-\gamma_{n+1}|bac-\alpha_n\beta|aba+\alpha_n\gamma|(abc-bac)$   \\
			\hline
			$\alpha_{n-1}\beta$    &  $\alpha_n\beta|abc+\alpha_n\gamma|(aba+bac)+\alpha_{n-1}\beta_2|aba$ & $\alpha_n\beta|abc-\alpha_n\gamma|(aba+bac)-\alpha_{n-1}\beta_2|bac$ \\
			\hline
			$\alpha_{n-1}\gamma$   &  $2\alpha_n\beta|aba+\alpha_{n-1}\beta_2|(abc+bac)$ & $-\alpha_n\beta|aba+\alpha_n\gamma|(abc-bac)-\alpha_{n-1}\beta_2|bac$ \\
			\hline
			$\alpha_{n-2}\beta_2$  & $\alpha_n\beta|bac+\alpha_n\gamma|(aba+abc)+\alpha_{n-1}\beta_2|aba$ & $-2\alpha_n\beta|bac+\alpha_{n-1}\beta_2|(abc-aba)$ \\
			\hline
		\end{tabular}
    	}
		\end{center}
		\caption{Images of $d^n$ for $n\in\NN$ and $n$ odd, where the last three lines are for $n\geqslant 3$ and $n$ odd.}
		\label{odd2 ab bc}	
		\end{table}
	
\noindent and	
	
		\begin{table}[H]
			\begin{center}
			\resizebox{\textwidth}{15mm}{
			\begin{tabular}{|c|c|c|}
				\hline
				\diagbox[width=16mm , height=5mm]{$y$}{$x$}					   & $ba$ & $ac$  \\
				\hline
				$\alpha_n$             & $\alpha_{n+1}|aba+\alpha_n\beta|(aba+abc)+\alpha_n\gamma|bac$ & $-\alpha_{n+1}|abc-\alpha_{n}\beta|(aba+abc)+\alpha_n\gamma|bac$ \\
				\hline
				$\beta_n$              & $\beta_{n+1}|aba+\alpha_n\beta|(aba+bac)+\alpha_n\gamma|abc$  & $\beta_{n+1}|(bac-aba)-2\alpha_n\gamma|abc$     \\
				\hline     
				$\gamma_n$             & $\gamma_{n+1}|(abc+bac)+2\alpha_n\gamma|aba$ & $-\gamma_{n+1}|abc+\alpha_n\beta|(bac-abc)-\alpha_n\gamma|aba$   \\
				\hline
				$\alpha_{n-1}\beta$    &  $\alpha_n\beta|(aba+bac)+\alpha_n\gamma|abc+\alpha_{n-1}\beta_2|aba$ & $-2\alpha_n\gamma|abc+\alpha_{n-1}\beta_2|(bac-aba)$ \\
				\hline
				$\alpha_{n-1}\gamma$   &  $2\alpha_n\gamma|aba+\alpha_{n-1}\beta_2|(abc+bac)$ & $\alpha_n\beta|(bac-abc)-\alpha_n\gamma|aba-\alpha_{n-1}\beta_2|abc$ \\
				\hline
				$\alpha_{n-2}\beta_2$  & $\alpha_n\beta|(aba+abc)+\alpha_n\gamma|bac+\alpha_{n-1}\beta_2|aba$ & $-\alpha_n\beta|(aba+abc)+\alpha_n\gamma|bac-\alpha_{n-1}\beta_2|abc$ \\
				\hline
			\end{tabular}
    		}
			\end{center}
			\caption{Images of $d^n$ for $n\in\NN$ and $n$ odd, where the last three lines are for $n\geqslant 3$ and $n$ odd.}\label{odd2 ba ac}	
			\end{table}

\noindent as well as

			\begin{table}[H]
				\begin{center}
				\resizebox{\textwidth}{15mm}{
				\begin{tabular}{|c|c|c|}
					\hline
						\diagbox[width=16mm , height=5mm]{$y$}{$x$}				   & $a$ & $b$  \\
					\hline
					$\alpha_n$             & $-\alpha_n\beta|bc+\alpha_n\gamma|(ba+ac)$ & $\alpha_{n+1}|(ab+ba)-\alpha_{n}\beta|(ba+ac)+\alpha_n\gamma|bc$ \\
					\hline
					$\beta_n$              & $\beta_{n+1}|(ab+ba)+\alpha_n\beta|ac-\alpha_n\gamma|(ab+bc)$  & $\alpha_n\beta|(ab+bc)-\alpha_{n}\gamma|ac$  \\
					\hline     
					$\gamma_n$             & $\gamma_{n+1}|(ac-ab-bc)+\alpha_n\beta|ba+\alpha_{n}\gamma|ab$ & $\gamma_{n+1}|(bc-ba-ac)+\alpha_n\beta|ba+\alpha_{n}\gamma|ab$   \\
					\hline
					$\alpha_{n-1}\beta$    &  $\alpha_n\beta|ac-\alpha_{n}\gamma|(ab+bc)+\alpha_{n-1}\beta_2|(ab+ba)$ & $\alpha_n\beta|(ab+bc)-\alpha_n\gamma|ac$ \\
					\hline
					$\alpha_{n-1}\gamma$   &  $\alpha_n\beta|ba+\alpha_{n}\gamma|ab+\alpha_{n-1}\beta_2|(ac-ab-bc)$ & $\alpha_n\beta|ba+\alpha_n\gamma|ab+\alpha_{n-1}\beta_2|(bc-ba-ac)$ \\
					\hline
					$\alpha_{n-2}\beta_2$  & $-\alpha_n\beta|bc+\alpha_n\gamma|(ba+ac)$ & $-\alpha_n\beta|(ba+ac)+\alpha_n\gamma|bc+\alpha_{n-1}\beta_2|(ab+ba)$ \\
					\hline
				\end{tabular}
				}
				\end{center}
				\caption{Images of $d^n$ for $n\in\NN$ and $n$ odd, where the last three lines are for $n\geqslant 3$ and $n$ odd.}
				\label{odd1 a b}	
				\end{table}
			
\noindent and			
			
				\begin{table}[H]
					\begin{center}
					\resizebox{\textwidth}{15mm}{
					\begin{tabular}{|c|c|c|}
						\hline
							\diagbox[width=16mm , height=5mm]{$y$}{$x$}				   & $c$ & $1$  \\
						\hline
						$\alpha_n$             & $\alpha_{n+1}|(ac-ab-bc)-\alpha_n\beta|(ba+ac)+\alpha_n\gamma|bc$ & $2\alpha_{n+1}|a+(\alpha_{n}\beta+\alpha_n\gamma)|(b+c)$ \\
						\hline
						$\beta_n$              & $\beta_{n+1}|(bc-ba-ac)+\alpha_n\beta|ac-\alpha_n\gamma|(ab+bc)$  & $2\beta_{n+1}|b+(\alpha_n\beta+\alpha_{n}\gamma)|(a+c)$  \\
						\hline     
						$\gamma_n$             & $-\alpha_n\beta|ab-\alpha_{n}\gamma|ba$ & $2\gamma_{n+1}|c+(\alpha_n\beta+\alpha_{n}\gamma)|(a+b)$   \\
						\hline
						$\alpha_{n-1}\beta$    &  $\alpha_n\beta|ac-\alpha_{n}\gamma|(ab+bc)+\alpha_{n-1}\beta_2|(bc-ba-ac)$ & $(\alpha_n\beta+\alpha_n\gamma)|(a+c)+2\alpha_{n-1}\beta_2|b$ \\
						\hline
						$\alpha_{n-1}\gamma$   &  $-\alpha_n\beta|ab-\alpha_{n}\gamma|ba$ & $(\alpha_n\beta+\alpha_n\gamma)|(a+b)+2\alpha_{n-1}\beta_2|c$ \\
						\hline
						$\alpha_{n-2}\beta_2$  & $-\alpha_n\beta|(ba+ac)+\alpha_n\gamma|bc+\alpha_{n-1}\beta_2|(ac-ab-bc)$ & $(\alpha_n\beta+\alpha_n\gamma)|(b+c)+2\alpha_{n-1}\beta_2|a$ \\
						\hline
					\end{tabular}
    				}
					\end{center}
					\caption{Images of $d^n$ for $n\in\NN$ and $n$ odd, where the last three lines are for $n\geqslant 3$ and $n$ odd.}
					\label{odd10 c 1}	
					\end{table}

If $n\geqslant 2$ is even, the differential $d^n$ is given by 
\begin{table}[H]
	\begin{center}
	\resizebox{\textwidth}{15mm}{
	\begin{tabular}{|c|c|c|c|c|c|}
		\hline
		    \diagbox[width=16mm , height=5mm]{$y$}{$x$}                   & $abac$ & $aba$ & $abc$ & $bac$ \\
		\hline
		$\alpha_n$             &   0    & $2\alpha_n\gamma|abac$ & $-2\alpha_n\beta|abac$   & $-2\alpha_{n+1}|abac$   \\
		\hline
		$\beta_n$              &   0    & $2\alpha_n\gamma|abac$ & $-2\beta_{n+1}|abac$     & $-2\alpha_{n-1}\beta_2|abac$   \\
		\hline     
		$\gamma_n$             &   0    & $2\gamma_{n+1}|abac$   & $-2\alpha_{n}\beta|abac$ & $-2\alpha_{n-1}\beta_2|abac$   \\
		\hline
        $\alpha_{n-1}\beta$    &   0    & $\alpha_n\beta|abac+\alpha_{n-1}\beta_2|abac$ & $-\alpha_{n}\gamma|abac-\alpha_{n-1}\beta_2|abac$ & $-\alpha_n\beta|abac-\alpha_n\gamma|abac$   \\
		\hline
		$\alpha_{n-1}\gamma$   &   0    & $\alpha_n\beta|abac+\alpha_{n-1}\beta_2|abac$ & $-\alpha_{n}\gamma|abac-\alpha_{n-1}\beta_2|abac$ & $-\alpha_n\beta|abac-\alpha_n\gamma|abac$   \\
		\hline
		$\alpha_{n-2}\beta_2$  &   0    & $2\alpha_n\gamma|abac$ & $-2\alpha_n\beta|abac$ & $-2\alpha_{n-1}\beta_2|abac$  \\
		\hline
	\end{tabular}
    }
	\end{center}
	\caption{Images of $d^n$ for $n\geqslant 2$ and $n$ even, where the last line is for $n\geqslant 4$ and $n$ even.}
	\label{even43}	
	\end{table}
 
\noindent together with 

\begin{table}[H]
	\begin{center}
	\resizebox{\textwidth}{15mm}{
	\begin{tabular}{|c|c|c|}
		\hline
		     \diagbox[width=16mm , height=5mm]{$y$}{$x$}                  & $ab$ & $bc$  \\
		\hline
		$\alpha_n$             & $\alpha_{n+1}|aba-\alpha_n\beta|aba+\alpha_n\gamma|(abc-bac)$ & $-\alpha_{n+1}|(aba+abc)-\alpha_{n}\beta|bac+\alpha_n\gamma|bac$ \\
		\hline
		$\beta_n$              & $-\beta_{n+1}|aba+\alpha_n\gamma|(abc-bac)+\alpha_{n-1}\beta_2|aba$  & $-\beta_{n+1}|bac+\alpha_n\gamma|bac-\alpha_{n-1}\beta_2|(aba+abc)$     \\
		\hline     
		$\gamma_n$             & $\gamma_{n+1}|(abc-bac)-\alpha_n\beta|aba+\alpha_{n-1}\beta_2|aba$ & $\gamma_{n+1}|bac-\alpha_n\beta|bac-\alpha_{n-1}\beta_2|(aba+abc)$   \\
		\hline
        $\alpha_{n-1}\beta$    &  $\alpha_n\beta|abc-\alpha_{n-1}\beta_2|bac$ & $-\alpha_n\beta|abc-\alpha_n\gamma|aba$ \\
		\hline
		$\alpha_{n-1}\gamma$   &  $\alpha_n\beta|(aba-bac)+\alpha_{n-1}\beta_2|(abc-aba)$ & $\alpha_n\beta|(bac-aba)-\alpha_n\gamma|(abc+bac)$ \\
		\hline
		$\alpha_{n-2}\beta_2$  & $-\alpha_n\beta|aba+\alpha_n\gamma|(abc-bac)+\alpha_{n-1}\beta_2|aba$ & $-\alpha_n\beta|bac+\alpha_n\gamma|bac-\alpha_{n-1}\beta_2|(aba+abc)$ \\
		\hline
	\end{tabular}
    }
	\end{center}
	\caption{Images of $d^n$ for $n\geqslant 2$ and $n$ even, where the last line is for $n\geqslant 4$ and $n$ even.}
	\label{even2 ab bc}	
	\end{table}

\noindent and

	\begin{table}[H]
		\begin{center}
		\resizebox{\textwidth}{15mm}{
		\begin{tabular}{|c|c|c|}
			\hline
					\diagbox[width=16mm , height=5mm]{$y$}{$x$}			   & $ba$ & $ac$  \\
			\hline
			$\alpha_n$             & $-\alpha_{n+1}|aba+\alpha_n\beta|aba+\alpha_n\gamma|(bac-abc)$ & $-\alpha_{n+1}|abc-\alpha_{n}\beta|(aba+bac)+\alpha_n\gamma|abc$ \\
			\hline
			$\beta_n$              & $\beta_{n+1}|aba+\alpha_n\gamma|(bac-abc)-\alpha_{n-1}\beta_2|aba$  & $-\beta_{n+1}|(aba+bac)+\alpha_n\gamma|abc-\alpha_{n-1}\beta_2|abc$     \\
			\hline     
			$\gamma_n$             & $\gamma_{n+1}|(bac-abc)+\alpha_n\beta|aba-\alpha_{n-1}\beta_2|aba$ & $\gamma_{n+1}|abc-\alpha_n\beta|(aba+bac)-\alpha_{n-1}\beta_2|abc$   \\
			\hline
			$\alpha_{n-1}\beta$    &  $\alpha_n\beta|(bac-aba)+\alpha_{n-1}\beta_2|(aba-abc)$ & $-\alpha_n\gamma|(abc+bac)+\alpha_{n-1}\beta_2|(abc-aba)$ \\
			\hline
			$\alpha_{n-1}\gamma$   &  $-\alpha_n\beta|abc+\alpha_{n-1}\beta_2|bac$ & $-\alpha_n\gamma|aba-\alpha_{n-1}\beta_2|bac$ \\
			\hline
			$\alpha_{n-2}\beta_2$  & $\alpha_n\beta|aba+\alpha_n\gamma|(bac-abc)-\alpha_{n-1}\beta_2|aba$ & $-\alpha_n\beta|(aba+bac)+\alpha_n\gamma|abc-\alpha_{n-1}\beta_2|abc$ \\
			\hline
		\end{tabular}
		}
		\end{center}
		\caption{Images of $d^n$ for $n\geqslant 2$ and $n$ even, where the last line is for $n\geqslant 4$ and $n$ even.}
		\label{even2 ba ac}	
		\end{table}

\noindent as well as

\begin{table}[H]
	\begin{center}
	\resizebox{\textwidth}{15mm}{
	\begin{tabular}{|c|c|c|}
		\hline
				\diagbox[width=16mm , height=5mm]{$y$}{$x$}			   & $a$ & $b$  \\
		\hline
		$\alpha_n$             & $\alpha_n\beta|(ab-ba)+\alpha_n\gamma|(ab+bc+ac)$ & $\alpha_{n+1}|(ba-ab)+\alpha_{n}\gamma|(ba+ac+bc)$ \\
		\hline
		$\beta_n$              & $\beta_{n+1}|(ab-ba)+\alpha_n\gamma|(ab+bc+ac)$  & $\alpha_n\gamma|(ba+ac+bc)+\alpha_{n-1}\beta_2|(ba-ab)$  \\
		\hline     
		$\gamma_n$             & $\gamma_{n+1}|(ab+bc+ac)+\alpha_n\beta|(ab-ba)$ & $\gamma_{n+1}|(ba+ac+bc)+\alpha_{n-1}\beta_2|(ba-ab)$   \\
		\hline
		$\alpha_{n-1}\beta$    &  $\alpha_n\beta|ac-\alpha_{n}\gamma|ba+\alpha_{n-1}\beta_2|(2ab+bc)$ & $\alpha_n\beta|(bc-ab)+\alpha_n\gamma|ba+\alpha_{n-1}\beta_2|(ba+ac)$ \\
		\hline
		$\alpha_{n-1}\gamma$   &  $\alpha_n\beta|(ab+bc)+\alpha_{n}\gamma|ab+\alpha_{n-1}\beta_2|(ac-ba)$ & $\alpha_n\beta|(2ba+ac)-\alpha_n\gamma|ab+\alpha_{n-1}\beta_2|bc$ \\
		\hline
		$\alpha_{n-2}\beta_2$  & $\alpha_n\beta|(ab-ba)+\alpha_n\gamma|(ab+bc+ac)$ & $\alpha_n\gamma|(ba+ac+bc)+\alpha_{n-1}\beta_2|(ba-ab)$ \\
		\hline
	\end{tabular}
	}
	\end{center}
	\caption{Images of $d^n$ for $n\geqslant 2$ and $n$ even, where the last line is for $n\geqslant 4$ and $n$ even.}	
	\label{even 1 a b}
	\end{table}

\noindent and 

	\begin{table}[H]
		\begin{center}
		\resizebox{\textwidth}{15mm}{
		\begin{tabular}{|c|c|c|}
			\hline
					\diagbox[width=16mm , height=5mm]{$y$}{$x$}			   & $c$ & $1$  \\
			\hline
			$\alpha_n$             & $-\alpha_{n+1}|(ab+bc+ac)-\alpha_n\beta|(bc+ba+ac)$     & $0$  \\
			\hline
			$\beta_n$              & $-\beta_{n+1}|(bc+ba+ac)-\alpha_{n-1}\beta_2|(ab+bc+ac)$ & $0$  \\
			\hline     
			$\gamma_n$             & $-\alpha_n\beta|(bc+ba+ac)-\alpha_{n-1}\beta_2|(ab+bc+ac)$ & $0$   \\
			\hline
			$\alpha_{n-1}\beta$    &  $-\alpha_n\beta|ac-\alpha_{n}\gamma|(ab+2bc)-\alpha_{n-1}\beta_2|(ba+ac)$ & $\alpha_n\beta|(c-a)+\alpha_n\gamma|(a-b)+\alpha_{n-1}\beta_2|(b-c)$ \\
			\hline
			$\alpha_{n-1}\gamma$   &  $-\alpha_n\beta|(ab+bc)-\alpha_{n}\gamma|(ba+2ac)-\alpha_{n-1}\beta_2|bc$ & $\alpha_n\beta|(a-c)+\alpha_n\gamma|(b-a)+\alpha_{n-1}\beta_2|(c-b)$ \\
			\hline
			$\alpha_{n-2}\beta_2$  & $-\alpha_n\beta|(bc+ba+ac)-\alpha_{n-1}\beta_2|(ab+bc+ac)$ & $0$ \\
			\hline
		\end{tabular}
		}
		\end{center}
		\caption{Images of $d^n$ for $n\geqslant 2$ and $n$ even, where the last line is for $n\geqslant 4$ and $n$ even.}
		\label{even10 c 1}	
		\end{table}

Now we turn to the maps $f^n$. 
Note that $f^n(u)=0$ for $u\in K^{n+3}_{m}$, with $m\in \llbracket 2,4 \rrbracket $ and $n\in\NN_0$ by degree reasons. 
In the following tables, $f^n(y|x)$ is the entry appearing in the column indexed by $x$ and the row indexed by $y$, where $n$ is the internal degree of $y$. The map $f^0$ is given by

\begin{table}[H]
	\begin{center}
	\begin{tabular}{|c|c|c|c|c|}
		\hline
			\diagbox[width=16mm , height=5mm]{$y$}{$x$}				   & $1$ & $a$ & $b$ & $c$ \\
		\hline
		$\alpha_3$             & $4\epsilon^!|(bac-aba+abc)$     & $0$ & $0$ & $0$ \\
		\hline
		$\beta_3$              & $4\epsilon^!|(bac-aba+abc)$     & $0$ & $0$ & $0$ \\
		\hline     
		$\gamma_3$             & $4\epsilon^!|(bac-aba+abc)$     & $0$ & $0$ & $0$  \\
		\hline
		$\alpha_{2}\beta$      & $2\epsilon^!|(aba-abc-bac)$     & $0$ & $0$ & $0$ \\
		\hline
		$\alpha_{2}\gamma$     & $2\epsilon^!|(aba-abc-bac)$     & $0$ & $0$ & $0$ \\
		\hline
		$\alpha\beta_2$        & $2\epsilon^!|(aba-abc-bac)$     & $0$ & $0$ & $0$ \\
		\hline
	\end{tabular}
	\end{center}
	\caption{Images of $f^0$.}
	\label{f0co}	
	\end{table}

If $n\in\NN$ is odd, the map $f^n$ is given by

\begin{table}[H]
	\begin{center}
	\resizebox{\textwidth}{15mm}{
	\begin{tabular}{|c|c|c|c|c|}
		\hline
			\diagbox[width=16mm , height=5mm]{$y$}{$x$}				     & $1$ & $a$ & $b$ & $c$ \\
		\hline
		$\alpha_{n+3}$           & $0$ & $f^n(\alpha_{n+3}|a)$ & $2(\alpha_{n-1}\gamma-\alpha_{n-2}\beta_2)|abac$ & $2(\alpha_{n-1}\beta-\alpha_{n-2}\beta_2)|abac$ \\
		\hline
		$\beta_{n+3}$            & $0$ & $2(\alpha_{n-1}\gamma-\alpha_{n-1}\beta)|abac$ & $f^n(\beta_{n+3}|b)$ & $2(\alpha_{n-2}\beta_2-\alpha_{n-1}\beta)|abac$ \\
		\hline     
		$\gamma_{n+3}$           & $0$ & $2(\alpha_{n-1}\beta-\alpha_{n-1}\gamma)|abac$ & $2(\alpha_{n-2}\beta_2-\alpha_{n-1}\gamma)|abac$ & $f^n(\gamma_{n+3}|c)$ \\
		\hline
		$\alpha_{n+2}\beta$      & $0$ & $0$ & $0$ & $0$ \\
		\hline
		$\alpha_{n+2}\gamma$     & $0$ & $0$ & $0$ & $0$ \\
		\hline
		$\alpha_{n+1}\beta_2$    & $0$ & $-2(\alpha_n+\beta_n+\gamma_n)|abac$ & $-2(\alpha_n+\beta_n+\gamma_n)|abac$ & $-2(\alpha_n+\beta_n+\gamma_n)|abac$ \\
		\hline
	\end{tabular}
	}
	\end{center}
	\caption{Images of $f^n$ for $n\in\NN$ and $n$ odd.}
	\label{fnoddco}	
	\end{table}
\noindent where 
\begin{align*}
	f^n(\alpha_{n+3}|a)&=[4\alpha_n+4\beta_n+4\gamma_n+2(n-2)\alpha_{n-1}\beta+2(n-2)\alpha_{n-1}\gamma+2(n-1)\alpha_{n-2}\beta_2]|abac,\\
	f^n(\beta_{n+3}|b)&= [4\alpha_n+4\beta_n+4\gamma_n+2(n-1)\alpha_{n-1}\beta+2(n-2)\alpha_{n-1}\gamma+2(n-2)\alpha_{n-2}\beta_2]|abac,\\
	f^n(\gamma_{n+3}|c)&= [4\alpha_n+4\beta_n+4\gamma_n+2(n-2)\alpha_{n-1}\beta+2(n-1)\alpha_{n-1}\gamma+2(n-2)\alpha_{n-2}\beta_2]|abac.
\end{align*}

If $n\geqslant 2$ is even, the map $f^n$ is given by 
\begin{align*}
	f^n(\alpha_{n+3}|1)&=2\alpha_n|(2bac-aba+abc)+2\beta_n|(2abc+bac)+2\gamma_n|(bac-2aba)\\
	&\phantom{ =  \;}
	+2(n-2)\alpha_{n-2}\beta_2|(abc-aba+bac),\\
	f^n(\beta_{n+3}|1)&=2\alpha_n|(2bac+abc)+2\beta_n|(2abc-aba+bac)+2\gamma_n|(abc-2aba)\\
	&\phantom{ =  \;}
	+2(n-2)\alpha_{n-2}\beta_2|(abc-aba+bac),\\
    f^n(\gamma_{n+3}|1)&=2\alpha_n|(2bac-aba)+2\beta_n|(2abc-aba)+2\gamma_n|(abc-2aba+bac)\\
	&\phantom{ =  \;}
	+2(n-2)\alpha_{n-2}\beta_2|(abc-aba+bac),\\
	f^n(\alpha_{n+2}\beta|1)&=-2\alpha_n|bac-2\beta_n|abc+2\gamma_n|aba, \\
	f^n(\alpha_{n+2}\gamma|1)&=-2\alpha_n|bac-2\beta_n|abc+2\gamma_n|aba,\\
	f^n(\alpha_{n+1}\beta_2|1)&=-2\alpha_n|bac-2\beta_n|abc+2\gamma_n|aba,
\end{align*}
and $f^n(x)=0$ for $x\in K^{n+3}_1$.

\subsection{Computation of the coboundaries}
\label{subsection: com coboun}

In this subsection, we will explicitly construct bases $\tilde{\mathfrak{B}}^n_m$ and $\mathfrak{B}^n_m$ of the $\Bbbk$-vector spaces $\tilde{B}^n_m=\Img(d^{n-1}_{m-1})$ and $B^n_m=\Img(\partial^{n-1}_{m-1})$ for $m\in\llbracket 0,4 \rrbracket$ and $n\in\NN_0$ respectively, defined before Remark \ref{mrange}.
This will be done by simply applying the corresponding differential $d^{n-1}_{m-1}$ or $\partial^{n-1}_{m-1}$ to the usual basis of its domain and extracting a linearly independent generating subset.

\subsubsection{\texorpdfstring{Computation of $\tilde{\mathfrak{B}}^n_m$}{Computation of Bnm}} 
\label{subsubsection:cob1}

Recall that $\tilde{B}^n_m=\Img(d^{n-1}_{m-1})$ and 
\[     d^n_m:K^n_m=\operatorname{Hom}_{\Bbbk}((A_{-n}^!)^*,A_m)
\to 
K^{n+1}_{m+1}=\operatorname{Hom}_{\Bbbk}((A_{-(n+1)}^!)^*,A_{m+1})     \] 
was defined in Subsection \ref{rdsco}. 
Obviously, 
$ \tilde{B}^0_m=\Img(d^{-1}_{m-1})=0 $
for $m\in \llbracket 0,4 \rrbracket$, and $\tilde{B}^n_0=\Img(d^{n-1}_{-1})=0$ for $n\in\NN$. 
Then we define 
$\tilde{\mathfrak{B}}^0_m =\emptyset$
for $m\in \llbracket 0,4 \rrbracket$, 
and 
$\tilde{\mathfrak{B}}^n_0=\emptyset$
for $n\in \NN$.

Suppose $m=4$. 
If $n=1$, since 
\[ 
\gamma|abac=(1/2)d^{0}_3(\epsilon^!|aba), 
\quad 
\beta|abac=-(1/2)d^0_3(\epsilon^!|abc),
\quad 
\alpha|abac=-(1/2)d^0_3(\epsilon^!|bac),
\] 
we have 
$\tilde{B}^1_4=K^1_4$.
We define a basis of $\tilde{B}^1_4$ by 
the usual basis of $K^1_4$.
If $n\geqslant 3$ is odd, Table \ref{even43} shows that 
\begin{align*}
& \alpha_n|abac=-(1/2)d^{n-1}_3(\alpha_{n-1}|bac),
\quad 
	\beta_n|abac=-(1/2)d^{n-1}_3(\beta_{n-1}|abc), 
	\\
&	\gamma_n|abac=(1/2)d^{n-1}_3(\gamma_{n-1}|aba), \quad 
\alpha_{n-1}\beta|abac=-(1/2)d^{n-1}_3(\alpha_{n-1}|abc),\\
&	\alpha_{n-1}\gamma|abac=(1/2)d^{n-1}_3(\alpha_{n-1}|aba),
\quad  
\alpha_{n-2}\beta_2|abac=-(1/2)d^{n-1}_3(\beta_{n-1}|bac),
\end{align*}
so 
$\tilde{B}^n_4=K^n_4.$ 
We define a basis of $\tilde{B}^n_4$ by the usual basis of $K^n_4$.
If $n\geqslant 2$ is even, it is easy to see that $\tilde{B}^n_4$ is spanned by the element $(\alpha_{n-1}\beta-\alpha_{n-1}\gamma)|abac$ from Table \ref{odd43}, so we define a basis of $\tilde{B}^n_4$ by 
\[ \tilde{\mathfrak{B}}^n_4=\big\{ (\alpha_{n-1}\beta-\suline{\alpha_{n-1}\gamma})|\suline{abac} \big\}.\] 
The dimension of $\tilde{B}^n_4$ is then given by 
\begin{equation}\label{dimcobtilde4}
	\begin{split}
		\operatorname{dim} \tilde{B}^n_4 =
		\begin{cases}
			0,  &  \text{if $n=0$}, 
			\\
			3,  &  \text{if $n=1$}, 
			\\
			1,  &  \text{if $n\geqslant 2$ is even},
			\\
			6,  &  \text{if $n\geqslant 3$ is odd}. 
			\end{cases} 
	\end{split}
\end{equation}

Suppose $m=3$.
If $n=1$, since 
\begin{align*}
	&(\alpha-\beta)|aba+\gamma|(abc-bac)=d^0_2(\epsilon^!|ab)=-d^0_2(\epsilon^!|ba), \\
	& \alpha|(aba+abc)+(\beta-\gamma)|bac=-d^0_2(\epsilon^!|bc),\\ 
	&-\alpha|abc-\beta|(aba+bac)+\gamma|abc=d^0_2(\epsilon^!|ac)=d^0_2(\epsilon^!|ab)+d^0_2(\epsilon^!|bc),
\end{align*}
we define a basis of $\tilde{B}^1_3$ by 
\[ 
\tilde{\mathfrak{B}}^1_3=\big\{
\alpha|(aba+abc)+(\beta-\suline{\gamma})|\suline{bac},
(\alpha-\beta)|aba+\suline{\gamma}|(\suline{abc}-bac)
\big\}. 
\] 
If $n\geqslant 2$ is even, we define the set  
\begin{align*}
	\mathcal{G}^n_3=\big\{ 
	& g^n_{1,3}=(\suline{\alpha_{n-1}\beta}-\alpha_{n-1}\gamma)|\suline{aba}=(1/2)d^{n-1}_2(\gamma_{n-1}|(ab-ba)),\\
	& g^n_{2,3}=(\suline{\alpha_{n-1}\beta}-\alpha_{n-1}\gamma)|\suline{abc}=(1/2)d^{n-1}_2(\beta_{n-1}|(ab+bc+ac)), \\
	& g^n_{3,3}=(\suline{\alpha_{n-1}\beta}-\alpha_{n-1}\gamma)|\suline{bac}=-(1/2)d^{n-1}_2(\alpha_{n-1}|(bc+ba+ac)), \\
	& g^n_{4,3}=\suline{\alpha_n|aba}+\alpha_{n-1}\beta|bac+\alpha_{n-1}\gamma|(aba+abc)=d^{n-1}_2(\alpha_{n-1}|ab),\\
	& g^n_{5,3}=\suline{\alpha_n|abc}-\alpha_{n-1}\beta|bac+\alpha_{n-1}\gamma|(aba+abc)=d^{n-1}_2(\alpha_{n-1}|ab)+d^{n-1}_2(\alpha_{n-1}|bc),\\
	& g^n_{6,3}=\suline{\beta_n|aba}+\alpha_{n-1}\beta|abc+\alpha_{n-1}\gamma|(aba+bac)=d^{n-1}_2(\beta_{n-1}|ab),\\
	& g^n_{7,3}=\suline{\beta_n|bac}-\alpha_{n-1}\beta|abc+\alpha_{n-1}\gamma|(aba+bac)=-d^{n-1}_2(\beta_{n-1}|bc),\\
	& g^n_{8,3}=\suline{\gamma_n|abc}+\alpha_{n-1}\beta|aba+\alpha_{n-1}\gamma|(abc-bac)=d^{n-1}_2(\gamma_{n-1}|ab)+d^{n-1}_2(\gamma_{n-1}|bc),\\
	& g^n_{9,3}=\suline{\gamma_n|bac}+\alpha_{n-1}\beta|aba+\alpha_{n-1}\gamma|(bac-abc)=-d^{n-1}_2(\gamma_{n-1}|bc)
	\big\}.
\end{align*}
Then we define the set 
$\tilde{\mathfrak{B}}^2_3=\mathcal{G}^2_3$, 
and 
\begin{align*}
	\tilde{\mathfrak{B}}^n_3=\mathcal{G}^n_3 \cup \big\{ 
		& g^n_{10,3}=\suline{\alpha_{n-2}\beta_2|aba}+\alpha_{n-1}\gamma|(aba+abc+bac)= d^{n-1}_2(\alpha_{n-2}\beta|ab)-g^n_{2,3},\\
		& g^n_{11,3}=\suline{\alpha_{n-2}\beta_2|abc}+\alpha_{n-1}\gamma|(aba+abc-bac)=-d^{n-1}_2(\alpha_{n-2}\gamma|ac)-g^n_{2,3}+g^n_{3,3},\\
		& g^n_{12,3}=\suline{\alpha_{n-2}\beta_2|bac}+\alpha_{n-1}\gamma|(aba-abc+bac)=-d^{n-1}_2(\alpha_{n-2}\beta|bc)+g^n_{2,3} \big\}
\end{align*}
for $n\geqslant 4$ with $n$ even.
We will show that $\tilde{\mathfrak{B}}^n_3$ is a basis of $\tilde{B}^n_3$ for $n\geqslant 2$ with $n$ even. 
From the definition, we see that $\tilde{\mathfrak{B}}^n_3\subseteq \tilde{B}^n_3$. 
Since 
\begin{align*}
& d^{n-1}_2(\alpha_{n-1}|ab)=g^n_{4,3}, 
\quad 
d^{n-1}_2(\alpha_{n-1}|bc)=g^n_{5,3}-g^n_{4,3},\\
& d^{n-1}_2(\alpha_{n-1}|ba)=g^n_{1,3}+g^n_{2,3}+g^n_{4,3}-g^n_{3,3},\quad 
d^{n-1}_2(\alpha_{n-1}|ac)=-g^n_{1,3}-g^n_{2,3}-g^n_{3,3}-g^n_{5,3},\\
& d^{n-1}_2(\beta_{n-1}|ab)=g^n_{6,3},\quad 
d^{n-1}_2(\beta_{n-1}|bc)=-g^n_{7,3}, \quad 
d^{n-1}_2(\beta_{n-1}|ba)=g^n_{1,3}-g^n_{2,3}+g^n_{3,3}+g^n_{6,3},\\
& d^{n-1}_2(\beta_{n-1}|ac)= 2g^n_{2,3}-g^n_{6,3}+g^n_{7,3}, \quad 
d^{n-1}_2(\gamma_{n-1}|ab)=g^n_{8,3}+g^n_{9,3}, \quad 
d^{n-1}_2(\gamma_{n-1}|bc)=-g^n_{9,3}, 
\\
& d^{n-1}_2(\gamma_{n-1}|ba)= g^n_{8,3}+g^n_{9,3}-2g^n_{1,3},\quad 
d^{n-1}_2(\gamma_{n-1}|ac)=g^n_{1,3}-g^n_{2,3}+g^n_{3,3}-g^n_{8,3}
\end{align*}
for $n\geqslant 2$ with $n$ even, and 
\begin{align*}
& d^{n-1}_2(\alpha_{n-2}\beta|ab)=g^n_{2,3}+g^n_{10,3}, \quad 
d^{n-1}_2(\alpha_{n-2}\beta|bc)=g^n_{2,3}-g^n_{12,3}, 
\\
& d^{n-1}_2(\alpha_{n-2}\beta|ba)=g^n_{1,3}+g^n_{3,3}+g^n_{10,3},\quad 
d^{n-1}_2(\alpha_{n-2}\beta|ac)=g^n_{12,3}-g^n_{10,3}, \\
& d^{n-1}_2(\alpha_{n-2}\gamma|ab)=g^n_{11,3}+g^n_{12,3}+2g^n_{1,3},\quad 
d^{n-1}_2(\alpha_{n-2}\gamma|bc)=-g^n_{1,3}-g^n_{12,3}, \\
& d^{n-1}_2(\alpha_{n-2}\gamma|ba)=g^n_{11,3}+g^n_{12,3},\quad 
d^{n-1}_2(\alpha_{n-2}\gamma|ac)=-g^n_{2,3}+g^n_{3,3}-g^n_{11,3},\\
& d^{n-1}_2(\alpha_{n-3}\beta_2|ab)=g^n_{3,3}+g^n_{10,3},\quad 
d^{n-1}_2(\alpha_{n-3}\beta_2|bc)=g^n_{11,3}-g^n_{10,3}-2g^n_{3,3},\\
& d^{n-1}_2(\alpha_{n-3}\beta_2|ba)=g^n_{1,3}+g^n_{2,3}+g^n_{10,3},\quad 
d^{n-1}_2(\alpha_{n-3}\beta_2|ac)=-g^n_{1,3}-g^n_{2,3}-g^n_{11,3}
\end{align*}
for $n\geqslant 4$ with $n$ even, 
the elements in $\tilde{\mathfrak{B}}^n_3$ span the space $\tilde{B}^n_3$. 
By Fact \ref{indep}, the elements in $\tilde{\mathfrak{B}}^n_3$ are linearly independent, so $\tilde{\mathfrak{B}}^n_3$ is a basis of $\tilde{B}^n_3$ for $n\geqslant 2$ with $n$ even. 
If $n\geqslant 3$ is odd, we define the set 
\begin{align*}
	\mathcal{E}^n_3=\big\{
	e^n_{1,3} & =(\suline{\alpha_n}-\alpha_{n-1}\beta)|\suline{aba}+\alpha_{n-1}\gamma|(abc-bac)=d^{n-1}_2(\alpha_{n-1}|ab), 
	\\
	 e^n_{2,3} & =\suline{\alpha_n}|(aba+\suline{abc})+(\alpha_{n-1}\beta-\alpha_{n-1}\gamma)|bac=-d^{n-1}_2(\alpha_{n-1}|bc),
	 \\
	e^n_{3,3} & =(\alpha_{n-2}\beta_2-\suline{\beta_n})|\suline{aba}+\alpha_{n-1}\gamma|(abc-bac)=d^{n-1}_2(\beta_{n-1}|ab), 
	\\
	e^n_{4,3} & =(\alpha_{n-1}\gamma-\suline{\beta_n})|\suline{bac}-\alpha_{n-2}\beta_2|(aba+abc)=d^{n-1}_2(\beta_{n-1}|bc), 
	\\
	e^n_{5,3} & = (\suline{\gamma_n}-\alpha_{n-1}\beta)|\suline{bac}-\alpha_{n-2}\beta_2|(aba+abc)=d^{n-1}_2(\gamma_{n-1}|bc),
	\\
	e^n_{6,3} & = \suline{\gamma_n}|(\suline{abc}-bac)+(\alpha_{n-2}\beta_2-\alpha_{n-1}\beta)|aba=d^{n-1}_2(\gamma_{n-1}|ab),
	\\
	e^n_{7,3} & = \alpha_{n-1}\beta|abc-\suline{\alpha_{n-2}\beta_2|bac}=d^{n-1}_2(\alpha_{n-2}\beta|ab),
	\\
	e^n_{8,3} & =\alpha_{n-1}\beta|abc+\suline{\alpha_{n-1}\gamma|aba}=-d^{n-1}_2(\alpha_{n-2}\beta|bc), 
	\\
	e^n_{9,3} & = \alpha_{n-1}\beta|(bac-aba)+\alpha_{n-2}\beta_2|(aba-abc)=-d^{n-1}_2(\alpha_{n-2}\gamma|ab),
	\\
	e^n_{10,3} & = \alpha_{n-1}\gamma|(abc+bac)+\alpha_{n-2}\beta_2|(aba-abc)=-d^{n-1}_2(\alpha_{n-2}\gamma|(ab+bc)) 	
	\big\}.
\end{align*}
Then we define the set 
$\tilde{\mathfrak{B}}^3_3=\mathcal{E}^3_3$, 
and 
\begin{align*}
\tilde{\mathfrak{B}}^n_3=\mathcal{E}^n_3 \cup \big\{  
& e^n_{11,3}=(\alpha_{n-2}\beta_2-\alpha_{n-1}\beta)|aba+\alpha_{n-1}\gamma|(abc-bac)=d^{n-1}_2(\alpha_{n-3}\beta_2|ab), \\
&e^n_{12,3}=(\alpha_{n-1}\beta-\alpha_{n-1}\gamma)|bac+\alpha_{n-2}\beta_2|(aba+abc) =-d^{n-1}_2(\alpha_{n-3}\beta_2|bc) 
\big\}
\end{align*}
for $n\geqslant 5$ with $n$ odd. 
We will show that $\tilde{\mathfrak{B}}^n_3$ is a basis of $\tilde{B}^n_3$ for $n\geqslant 3$ with $n$ odd.
By definition, $\tilde{\mathfrak{B}}^n_3\subseteq \tilde{B}^n_3$. Since 
\begin{align*}
  & d^{n-1}_2(\alpha_{n-1}|ba)=-e^n_{1,3},\quad 
	d^{n-1}_2(\alpha_{n-1}|ac)=e^n_{1,3}-e^n_{2,3}, \quad 
	d^{n-1}_2(\beta_{n-1}|ba)=-e^n_{3,3},\\	
  &	d^{n-1}_2(\beta_{n-1}|ac)=e^n_{3,3}+e^n_{4,3},\quad 
  d^{n-1}_2(\gamma_{n-1}|ba)=-e^n_{6,3}, \quad 
  d^{n-1}_2(\gamma_{n-1}|ac)=e^n_{5,3}+e^n_{6,3},\\
  & d^{n-1}_2(\alpha_{n-2}\beta|ba)=e^n_{9,3},\quad 
  d^{n-1}_2(\alpha_{n-2}\beta|ac)=-e^n_{10,3}, \quad 
  d^{n-1}_2(\alpha_{n-2}\gamma|bc)=e^n_{9,3}-e^n_{10,3}, \\
  & d^{n-1}_2(\alpha_{n-2}\gamma|ba)=-e^n_{7,3}, \quad
   d^{n-1}_2(\alpha_{n-2}\gamma|ac)=e^n_{7,3}-e^n_{8,3}
\end{align*}
for $n\geqslant 3$ with $n$ odd, and 
\begin{align*}
	d^{n-1}_2(\alpha_{n-3}\beta_2|ba)=-e^n_{11,3},\quad 
	d^{n-1}_2(\alpha_{n-3}\beta_2|ac)=e^n_{11,3}-e^n_{12,3}
\end{align*}
for $n\geqslant 5$ with $n$ odd, the elements in $\tilde{\mathfrak{B}}^n_3$ span the space $\tilde{B}^n_3$.
By Fact \ref{indep}, the elements $e^n_{\ell,3}$ for $\ell \in \llbracket 1,8\rrbracket$ are linearly independent.  
The reader can easily verify that the elements 
$e^n_{\ell,3}$ for $\ell \in \llbracket 9,12\rrbracket$ are linearly independent. 
Since the underlined terms of $e^n_{\ell,3}$ for $\ell \in \llbracket 1,8\rrbracket$ do not appear in $e^n_{\ell,3}$ for $\ell \in \llbracket 9,12\rrbracket$, 
the elements in $\tilde{\mathfrak{B}}^n_3$ are linearly independent. 
So $\tilde{\mathfrak{B}}^n_3$ is a basis of $\tilde{B}^n_3$.
The dimension of $\tilde{B}^n_3$ is thus given by 
\begin{equation}\label{dimcobtilde3}
	\begin{split}
		\operatorname{dim} \tilde{B}^n_3 =
		\begin{cases}
			0,  & \text{if $n=0$}, 
			\\
			2,  & \text{if $n=1$},
			\\
			9,  & \text{if $n=2$}, 
			\\
			10,  & \text{if $n=3$}, 
			\\
			12,  & \text{if $n\geqslant 4$}. 
			\end{cases} 
	\end{split}
\end{equation}

Suppose $m=2$. 
If $n=1$, since 
\begin{align*}
	& \beta|(ab-ba)+\suline{\gamma}|(\suline{ab}+bc+ac)=d^0_1(\epsilon^!|a),\quad
	(\alpha+\beta+\suline{\gamma})|(ab-\suline{ba})=d^0_1(\epsilon^!|(a-b)),\\
    & (\suline{\alpha}+\beta+\gamma)|(ab+\suline{bc}+ac)=d^0_1(\epsilon^!|(a-c)),
\end{align*}
and these three elements are linearly independent, we define a basis of $\tilde{B}^1_2$ by 
\[ \tilde{\mathfrak{B}}^1_2=\big\{\beta|(ab-ba)+\gamma|(ab+bc+ac), (\alpha+\beta+\gamma)|(ab-ba),(\alpha+\beta+\gamma)|(ab+bc+ac) \big\}.\] 
If $n=2$, we define the set 
\begin{align*}
	\tilde{\mathfrak{B}}^2_2=\big\{
	& g^2_{1,2}=\alpha_2|(ab+ba)-\suline{\alpha\beta}|(\suline{ba}+ac)+\alpha\gamma|bc=d^1_1(\alpha|b),\\ 
	& g^2_{2,2}=\beta_2|(ab+ba)+\alpha\beta|ac-\suline{\alpha\gamma}|(\suline{ab}+bc)=d^1_1(\beta|a),\\
	& g^2_{3,2}=\suline{\gamma_2}|(\suline{bc}-ba-ac)+\alpha\beta|ba+\alpha\gamma|ab=d^1_1(\gamma|b), \\
	& g^2_{4,2}=\suline{\alpha\beta|ab}+\alpha\gamma|ba=-d^1_1(\gamma|c),\\
	& g^2_{5,2}=\suline{\alpha\beta}|(ab+\suline{bc})-\alpha\gamma|ac=d^1_1(\beta|b),\\
	& g^2_{6,2}= \suline{\alpha_2}|(2ab+\suline{bc}+ba-ac)=d^1_1(\alpha|(b-c)),\\
	& g^2_{7,2}=\suline{\beta_2}|(ab-\suline{bc}+2ba+ac)=d^1_1(\beta|(a-c)),\\
	& g^2_{8,2}=\suline{\gamma_2}|(\suline{ab}+2bc-ba-2ac)=d^1_1(\gamma|(b-a)) 
	\big\}.	
\end{align*}
By definition, $\tilde{\mathfrak{B}}^2_{2}\subseteq \tilde{B}^2_2$. Since 
\[ d^1_1(\alpha|a)=g^2_{4,2}-g^2_{5,2}, \]  
the elements in $\tilde{\mathfrak{B}}^2_{2}$ span the space $\tilde{B}^2_2$. 
By Fact \ref{indep}, the elements in $\tilde{\mathfrak{B}}^2_{2}$ are linearly independent, so $\tilde{\mathfrak{B}}^2_{2}$ is a basis of $\tilde{B}^2_2$. 
If $n\geqslant 3$ is odd, 
we define the set 
\begin{align*}
	\tilde{\mathfrak{B}}^n_2=\big\{
	e^n_{1,2} & = \alpha_{n-1}\beta|(ab-ba)+\alpha_{n-1}\gamma|(ab+bc+ac)=d^{n-1}_1(\alpha_{n-1}|a), 
    \\
	 e^n_{2,2} & = (\alpha_{n-1}\beta+\alpha_{n-1}\gamma+\alpha_{n-2}\beta_2)|(ab-ba)=d^{n-1}_1(\alpha_{n-1}|a-\beta_{n-1}|b),
	\\
    e^n_{3,2} & = (\alpha_{n-1}\beta+\alpha_{n-1}\gamma+\alpha_{n-2}\beta_2)|(ab+bc+ac)=d^{n-1}_1(\alpha_{n-1}|a-\gamma_{n-1}|c), 
    \\
	 e^n_{4,2} & =2\alpha_{n-1}\beta|bc+\alpha_{n-1}\gamma|(ab+ba)+2\alpha_{n-2}\beta_2|ac=d^{n-1}_1(\alpha_{n-2}\beta|b+\alpha_{n-2}\gamma|a),
	\\
    e^n_{5,2}& = \alpha_{n-1}\beta|ac-\alpha_{n-1}\gamma|ba+\alpha_{n-2}\beta_2|(2ab+bc)=d^{n-1}_1(\alpha_{n-2}\beta|a), 
    \\
	e^n_{6,2} & =\alpha_{n-1}\gamma|(ab+bc-ac)+\alpha_{n-2}\beta_2|(ac-ab-bc)
	\\
	& =e^n_{3,2}+d^{n-1}_1(\alpha_{n-2}\gamma|c-\alpha_{n-2}\beta|a) ,
	\\
     e^n_{7,2} & = \suline{\beta_n}|(ab-\suline{ba})+\alpha_{n-1}\gamma|(ab+bc+ac)=d^{n-1}_1(\beta_{n-1}|a), 
    \\
	e^n_{8,2} & =(\suline{\beta_n}+\alpha_{n-1}\gamma+\alpha_{n-2}\beta_2)|(ab+\suline{bc}+ac)=d^{n-1}_1(\beta_{n-1}|(a-c)),
	\\
     e^n_{9,2} & =\suline{\gamma_n}|(ab+\suline{bc}+ac)+\alpha_{n-1}\beta|(ab-ba)=d^{n-1}_1(\gamma_{n-1}|a),
    \\
	e^n_{10,2} & =(\suline{\gamma_n}+\alpha_{n-1}\beta+\alpha_{n-2}\beta_2)|(ab-\suline{ba})=d^{n-1}_1(\gamma_{n-1}|(a-b)), 
	\\
	e^n_{11,2} & =(\suline{\alpha_n}+\beta_n+\alpha_{n-1}\gamma)|(ab-\suline{ba})=d^{n-1}_1(\beta_{n-1}|a-\alpha_{n-1}|b), 
	\\
	e^n_{12,2} & =(\suline{\alpha_n}+\gamma_n+\alpha_{n-1}\beta)|(ab+\suline{bc}+ac)=d^{n-1}_1(\gamma_{n-1}|a-\alpha_{n-1}|c)
	\big\}.
\end{align*}
By definition, $\tilde{\mathfrak{B}}^n_2\subseteq \tilde{B}^n_2$.
Since 
\begin{align*}
	& d^{n-1}_1(\alpha_{n-2}\beta|a)=e^n_{5,2},\quad
	d^{n-1}_1(\alpha_{n-2}\beta|b)=e^n_{1,2}-e^n_{2,2}-e^n_{3,2}+e^n_{4,2}+e^n_{5,2},
	\\
	& d^{n-1}_1(\alpha_{n-2}\beta|c)=-e^n_{1,2}+e^n_{2,2}-e^n_{5,2}-e^n_{6,2},\quad 
	d^{n-1}_1(\alpha_{n-2}\gamma|a)=-e^n_{1,2}+e^n_{2,2}+e^n_{3,2}-e^n_{5,2},
	\\
	& d^{n-1}_1(\alpha_{n-2}\gamma|b)=-2e^n_{1,2}+2e^n_{3,2}-e^n_{4,2}-e^n_{5,2}, \quad
	d^{n-1}_1(\alpha_{n-2}\gamma|c)=-e^n_{3,2}+e^n_{5,2}+e^n_{6,2}
\end{align*}
for $n\geqslant 3$ with $n$ odd, and 
\[ 
	d^{n-1}_1(\alpha_{n-3}\beta_2|a)=e^n_{1,2}, \quad 
	d^{n-1}_1(\alpha_{n-3}\beta_2|b)=e^n_{1,2}-e^n_{2,2},\quad 
	d^{n-1}_1(\alpha_{n-3}\beta_2|c)=e^n_{1,2}-e^n_{3,2}
\] 
for $n\geqslant 5$ with $n$ odd, the elements in $\tilde{\mathfrak{B}}^n_{2}$ span the space $\tilde{B}^n_2$. 
The reader can easily verify that the elements 
$e^n_{\ell,3}$ for $\ell \in \llbracket 1,6\rrbracket$ are linearly independent. 
By Fact \ref{indep}, the elements $e^n_{\ell,3}$ for $\ell \in \llbracket 7,12\rrbracket$ are linearly independent.  
Since the underlined terms of $e^n_{\ell,3}$ for $\ell \in \llbracket 7,12\rrbracket$ do not appear in $e^n_{\ell,3}$ for $\ell \in \llbracket 1,6\rrbracket$, 
the elements in $\tilde{\mathfrak{B}}^n_{2}$ are linearly independent.
So $\tilde{\mathfrak{B}}^n_{2}$ is a basis of $\tilde{B}^n_2$. 
If $n\geqslant 4$ is even, we define the set 
\begin{align*}
\tilde{\mathfrak{B}}^n_2=\big\{
	g^n_{1,2}& =\suline{\alpha_n}|(ab+\suline{ba})=d^{n-1}_1(\alpha_{n-1}|b)+g^n_{11,2}+g^n_{12,2},
	\\
	g^n_{2,2} & =\suline{\alpha_n}|(ab+\suline{bc}-ac)=-d^{n-1}_1(\alpha_{n-1}|c)-g^n_{11,2}-g^n_{12,2}
	\\
	g^n_{3,2} & =\suline{\beta_n}|(ab+\suline{ba})=d^{n-1}_1(\beta_{n-1}|a)-g^n_{12,2},
	\\
	g^n_{4,2} & =\suline{\beta_n}|(ab+\suline{bc}-ac)=d^{n-1}_1(\beta_{n-1}|(a+c))-2g^n_{12,2},
	\\
	g^n_{5,2} & =\suline{\gamma_n}|(ab+\suline{ba})=-d^{n-1}_1(\gamma_{n-1}|(a+b))+2g^n_{11,2},
	\\
    g^n_{6,2} & =\suline{\gamma_n}|(ab+\suline{bc}-ac)=-d^{n-1}_1(\gamma_{n-1}|a)+g^n_{11,2},
    \\
	g^n_{7,2} & =\suline{\alpha_{n-2}\beta_2}|(ab+\suline{ba})=(1/3)d^{n-1}_1(\alpha_{n-2}\beta|(a-c)+\alpha_{n-3}\beta_2|(b-c)),
	\\
	g^n_{8,2} & =\suline{\alpha_{n-2}\beta_2}|(ab+\suline{bc}-ac)=(1/3)d^{n-1}_1(2\alpha_{n-3}\beta_2|(b-c)-\alpha_{n-2}\beta|(a-c)), 
	\\
	g^n_{9,2} & =\alpha_{n-1}\beta|ab+\suline{\alpha_{n-1}\gamma|ba}=-d^{n-1}_1(\gamma_{n-1}|c), 
	\\
    g^n_{10,2} & =\alpha_{n-1}\beta|(ab+bc)-\suline{\alpha_{n-1}\gamma|ac}=d^{n-1}_1(\beta_{n-1}|b),
    \\
	g^n_{11,2} & =\suline{\alpha_{n-1}\beta|ba}+\alpha_{n-1}\gamma|ab=d^{n-1}_1(\alpha_{n-2}\gamma|a)+g^n_{8,2},
	\\
	g^n_{12,2} & =\suline{\alpha_{n-1}\beta|ac}-\alpha_{n-1}\gamma|(ab+bc)=d^{n-1}_1(\alpha_{n-2}\beta|a)-g^n_{7,2} 
\big\}.
\end{align*}
By definition, $\tilde{\mathfrak{B}}^n_2\subseteq \tilde{B}^n_2$. 
Since 
\begin{align*}
	& d^{n-1}_1(\alpha_{n-1}|a)=d^{n-1}_1(\alpha_{n-3}\beta_2|a)=g^n_{9,2}-g^n_{10,2},\quad 
	d^{n-1}_1(\alpha_{n-1}|b)=g^n_{1,2}-g^n_{11,2}-g^n_{12,2},\\
	& d^{n-1}_1(\alpha_{n-1}|c)=-g^n_{2,2}-g^n_{11,2}-g^n_{12,2},\quad 
	d^{n-1}_1(\beta_{n-1}|a)=g^n_{3,2}+g^n_{12,2},\\
	& d^{n-1}_1(\beta_{n-1}|b)=d^{n-1}_1(\alpha_{n-2}\beta|b)=g^n_{10,2},\quad 
	d^{n-1}_1(\beta_{n-1}|c)=g^n_{4,2}-g^n_{3,2}+g^n_{12,2},\\
	& d^{n-1}_1(\gamma_{n-1}|a)=-g^n_{6,2}+g^n_{11,2},\quad 
	 d^{n-1}_1(\gamma_{n-1}|b)=-g^n_{5,2}+g^n_{6,2}+g^n_{11,2},\\
	& d^{n-1}_1(\gamma_{n-1}|c)=d^{n-1}_1(\alpha_{n-2}\gamma|c)=-g^n_{9,2}, \quad 
	d^{n-1}_1(\alpha_{n-2}\beta|a)=g^n_{7,2}+g^n_{12,2},\\
	& d^{n-1}_1(\alpha_{n-2}\beta|c)=-g^n_{7,2}+g^n_{8,2}+g^n_{12,2}, \quad 
	d^{n-1}_1(\alpha_{n-2}\gamma|a)=g^n_{11,2}-g^n_{8,2},\\
	& d^{n-1}_1(\alpha_{n-2}\gamma|b)=-g^n_{7,2}+g^n_{8,2}+g^n_{11,2},\quad 
	 d^{n-1}_1(\alpha_{n-3}\beta_2|b)=g^n_{7,2}-g^n_{11,2}-g^n_{12,2},\\
	& d^{n-1}_1(\alpha_{n-3}\beta_2|c)=-g^n_{8,2}-g^n_{11,2}-g^n_{12,2},
\end{align*}
the elements in $\tilde{\mathfrak{B}}^n_2$ span the space $\tilde{B}^n_2$. 
By Fact \ref{indep}, the elements in $\tilde{\mathfrak{B}}^n_2$ are linearly independent, so $\tilde{\mathfrak{B}}^n_2$ is a basis of $\tilde{B}^n_2$. 
Hence, the dimension of $\tilde{B}^n_2$ is given by  
\begin{equation}\label{dimcobtilde2}
	\begin{split}
		\operatorname{dim} \tilde{B}^n_2 =
		\begin{cases}
			0,  & \text{if $n=0$}, 
			\\
			3,  &  \text{if $n=1$}, 
			\\
			8,  &  \text{if $n=2$}, 
			\\
			12,  & \text{if $n\geqslant 3$}. 
			\end{cases}
	\end{split}
\end{equation}

Suppose finally $m=1$. 
If $n=1$, since $d^0_0(\epsilon^!|1)=0$, we have $\tilde{B}^1_1=0$. We define 
$\tilde{\mathfrak{B}}^1_1=\emptyset$. 
If $n\geqslant 3$ is odd, by Table \ref{even10 c 1}, the space $\tilde{B}^n_1$ is spanned by the element  $\alpha_{n-1}\beta|(c-a)+\alpha_{n-1}\gamma|(a-b)+\alpha_{n-2}\beta_2|(b-c)$. So we define a basis of $\tilde{B}^n_1$ by 
\[ \tilde{\mathfrak{B}}^n_1=\big\{\suline{\alpha_{n-1}\beta}|(c-\suline{a})+\alpha_{n-1}\gamma|(a-b)+\alpha_{n-2}\beta_2|(b-c) \big\}. \] 
If $n=2$, by Table \ref{odd10 c 1}, we define a basis of $\tilde{B}^2_1$ by 
\[ \tilde{\mathfrak{B}}^2_1=\big\{2\suline{\alpha_2|a}+(\alpha\beta+\alpha\gamma)|(b+c), 
2\suline{\beta_2|b}+(\alpha\beta+\alpha\gamma)|(a+c),
2\suline{\gamma_2|c}+(\alpha\beta+\alpha\gamma)|(a+b) \big\}. \] 
If $n\geqslant 4$ is even, by Table \ref{odd10 c 1}, we define a basis of $\tilde{B}^n_1$ by 
\begin{align*}
\tilde{\mathfrak{B}}^n_1=\big\{
	& 2\suline{\alpha_n|a}+(\alpha_{n-1}\beta+\alpha_{n-1}\gamma)|(b+c),
	2\suline{\beta_n|b}+(\alpha_{n-1}\beta+\alpha_{n-1}\gamma)|(a+c),
	\\
&  2\suline{\gamma_n|c}+(\alpha_{n-1}\beta+\alpha_{n-1}\gamma)|(a+b),
(\alpha_{n-1}\beta+\alpha_{n-1}\gamma)|(a+c)+2\suline{\alpha_{n-2}\beta_2|b},
\\
& (\alpha_{n-1}\beta+\alpha_{n-1}\gamma)|(a+b)+2\suline{\alpha_{n-2}\beta_2|c},
(\alpha_{n-1}\beta+\alpha_{n-1}\gamma)|(b+c)+2\suline{\alpha_{n-2}\beta_2|a} \big\}.
\end{align*}
In conclusion, the dimension of $\tilde{B}^n_1$ is given by 
\begin{equation}\label{dimcobtilde1}
	\begin{split}
		\operatorname{dim} \tilde{B}^n_1 =
		\begin{cases}
			0,  & \text{if $n=0,1$}, 
			\\
			3,  & \text{if $n=2$}, 
			\\
			1,  & \text{if $n\geqslant 3$ is odd}, 
			\\
			6,  & \text{if $n\geqslant 4$ is even}. 
			\end{cases} 
	\end{split}
\end{equation}

\subsubsection{\texorpdfstring{Computation of $\mathfrak{B}^n_m$}{Computation of Bnm}} 
\label{subsubsection:cob2}

Recall that $B^n_m=\Img(\partial^{n-1}_{m-1})$ and $\partial^n_m:Q^n_m\to Q^{n+1}_{m+1}$. 
Since $d^n_m=\partial^n_m$ for either $m=3$ and $n\geqslant -1$, or $m, n \in \llbracket -1,2 \rrbracket$, we get $B^n_m=\tilde{B}^n_m$ for either $m=4$ and $n\in \NN_0$, or $m, n \in \llbracket 0,3 \rrbracket$. 
So, we define a basis of $B^n_m$ by 
$\mathfrak{B}^n_m=\tilde{\mathfrak{B}}^n_m$ for either $m=4$ and $n\in \NN_0$, or $m, n \in \llbracket 0,3 \rrbracket$.

Suppose $m=3$.
The differential $\partial^{n-1}_2:K^{n-1}_2\oplus \omega^*_1 K^{n-5}_4\to K^n_3$ maps the space $\omega^*_1 K^{n-5}_4$ to zero, 
so $B^n_3=\Img(\partial^{n-1}_{2})=\Img(d^{n-1}_{2})=\tilde{B}^n_3$. 
We define a basis of $B^n_3$ by 
$\mathfrak{B}^n_3=\tilde{\mathfrak{B}}^n_3$.

Suppose $m=2$. 
Consider $\partial^{n-1}_1: K^{n-1}_1\oplus \omega^*_1 K^{n-5}_3\to K^n_2\oplus \omega^*_1 K^{n-4}_4$.
If $n\geqslant 4$ is even, we get $B^n_2=\tilde{B}^n_2\oplus \omega^*_1\tilde{B}^{n-4}_4$, since $f^{n-4}(K^{n-1}_1)=0$ by the last identity of Subsection \ref{subsection:explicit-description-diff-cohomology} for $n > 4$ and $f^{0}(K^{3}_1)=0$ by the last three columns of Table \ref{f0co}, as well as 
$f^{n-8}(u)=0$ for $u\in K^{n-5}_{3}$ by degree reasons. 
If $n\geqslant 5$ is odd, we have $B^n_2=\tilde{B}^n_2\oplus \omega^*_1\tilde{B}^{n-4}_4$, since $\tilde{B}^{n-4}_4=K^{n-4}_4$ as showed in the previous subsubsection.
So, we define a basis of $B^n_2$ by 
$\mathfrak{B}^n_2=\tilde{\mathfrak{B}}^n_2\cup \omega^*_1 \tilde{\mathfrak{B}}^{n-4}_4$
for all integers $n\geqslant 4$.
The dimension of $B^n_2$ is then given by 
\begin{equation}\label{dimcob2}
	\begin{split}
		\operatorname{dim} B^n_2 =
		\begin{cases}
			0,  &  \text{if $n=0$}, 
			\\
			3,  &  \text{if $n=1$}, 
			\\
			8,  &  \text{if $n=2$}, 
			\\
			12,  &  \text{if $n=3,4$}, 
			\\
			15,  &  \text{if $n=5$}, 
			\\
			13,  &  \text{if $n\geqslant 6$ is even},
			\\
			18,  & \text{if $n\geqslant 7$ is odd}. 
			\end{cases} 
	\end{split}
\end{equation}

Suppose $m=1$. 
Consider $\partial^{n-1}_0: K^{n-1}_0\oplus \omega^*_1 K^{n-5}_2\oplus \omega^*_2 K^{n-9}_4\to K^n_1\oplus \omega^*_1 K^{n-4}_3$.
If $n\geqslant 5$ is odd, we have $B^{n}_1=\tilde{B}^n_1\oplus \omega^*_1 B^{n-4}_3=\tilde{B}^n_1\oplus \omega^*_1\tilde{B}^{n-4}_3$, since $f^{n-4}(K^{n-1}_0)=0$ by the second column of Table \ref{fnoddco} and $f^{n'}(u)=0$ for $u\in K^{n'+3}_{m}$, with $m\in \llbracket 2,4 \rrbracket $ and $n'\in\NN_0$, by degree reasons. Then we define a basis of $B^n_1$ by 
$\mathfrak{B}^n_1=\tilde{\mathfrak{B}}^n_1\cup \omega^*_1 \tilde{\mathfrak{B}}^{n-4}_3$. 
If $n=4$, we define the set 
\begin{align*}
\mathfrak{B}^4_1=\big\{
&2\suline{\alpha_4|a}+(\alpha_{3}\beta+\alpha_{3}\gamma)|(b+c)+\omega^*_1 4\epsilon^!|(bac-aba+abc)=\partial^3_0(\alpha_3|1),\\
& 2\suline{\beta_4|b}+(\alpha_{3}\beta+\alpha_{3}\gamma)|(a+c)+\omega^*_14\epsilon^!|(bac-aba+abc)=\partial^3_0(\beta_3|1),\\
&  2\suline{\gamma_4|c}+(\alpha_{3}\beta+\alpha_{3}\gamma)|(a+b)+\omega^*_1 4\epsilon^!|(bac-aba+abc)=\partial^3_0(\gamma_3|1),\\
& (\alpha_{3}\beta+\alpha_{3}\gamma)|(a+c)+2\suline{\alpha_{2}\beta_2|b}+\omega^*_1 2\epsilon^!|(aba-abc-bac)=\partial^3_0(\alpha_2\beta|1),\\
& (\alpha_{3}\beta+\alpha_{3}\gamma)|(a+b)+2\suline{\alpha_{2}\beta_2|c}+\omega^*_1 2\epsilon^!|(aba-abc-bac)=\partial^3_0(\alpha_2\gamma|1),\\
& (\alpha_{3}\beta+\alpha_{3}\gamma)|(b+c)+2\suline{\alpha_{2}\beta_2|a}+\omega^*_1 2\epsilon^!|(aba-abc-bac)=\partial^3_0(\alpha\beta_2|1)  
\big\}.
\end{align*}
Since these six elements are linearly independent by Fact \ref{indep}, $\mathfrak{B}^4_1$ is a basis of $B^4_1$. 
If $n\geqslant 6$ is even, 
we define the set 
\begin{align*}
\mathfrak{B}^n_1=\big\{
	& 2\suline{\alpha_n|a}+(\alpha_{n-1}\beta+\alpha_{n-1}\gamma)|(b+c)+\omega^*_1[2\alpha_{n-4}|(2bac-aba+abc)+2\beta_{n-4}|(2abc+bac)\\
& +2\gamma_{n-4}|(bac-2aba)+2(n-6)\alpha_{n-6}\beta_2|(bac-aba+abc)]=\partial^{n-1}_0(\alpha_{n-1}|1),
\\
& 2\suline{\beta_n|b}+(\alpha_{n-1}\beta+\alpha_{n-1}\gamma)|(a+c)+\omega^*_1[2\alpha_{n-4}|(2bac+abc)+2\beta_{n-4}|(2abc-aba+bac)\\
& +2\gamma_{n-4}|(abc-2aba)+2(n-6)\alpha_{n-6}\beta_2|(bac-aba+abc)]=\partial^{n-1}_0(\beta_{n-1}|1),
\\
&  2\suline{\gamma_n|c}+(\alpha_{n-1}\beta+\alpha_{n-1}\gamma)|(a+b)+\omega^*_1[2\alpha_{n-4}|(2bac-aba)+2\beta_{n-4}|(2abc-aba)\\
& +2\gamma_{n-4}|(bac-2aba+abc)+2(n-6)\alpha_{n-6}\beta_2|(bac-aba+abc)]=\partial^{n-1}_0(\gamma_{n-1}|1),\\
& (\alpha_{n-1}\beta+\alpha_{n-1}\gamma)|(a+c)+2\suline{\alpha_{n-2}\beta_2|b}+\omega^*_1 2(\gamma_{n-4}|aba-\alpha_{n-4}|bac-\beta_{n-4}|abc) \\
& =\partial^{n-1}_0(\alpha_{n-2}\beta|1),\\
& (\alpha_{n-1}\beta+\alpha_{n-1}\gamma)|(a+b)+2\suline{\alpha_{n-2}\beta_2|c} +\omega^*_1 2(\gamma_{n-4}|aba-\alpha_{n-4}|bac-\beta_{n-4}|abc)\\
& =\partial^{n-1}_0(\alpha_{n-2}\gamma|1),\\
& (\alpha_{n-1}\beta+\alpha_{n-1}\gamma)|(b+c)+2\suline{\alpha_{n-2}\beta_2|a} +\omega^*_1 2(\gamma_{n-4}|aba-\alpha_{n-4}|bac-\beta_{n-4}|abc)\\
& =\partial^{n-1}_0(\alpha_{n-3}\beta_2|1) 
\big\}
 \cup \omega^*_1\tilde{\mathfrak{B}}^{n-4}_3.
\end{align*}
Since $f^{n'}(u)=0$ for $u\in K^{n'+3}_{m}$, with $m\in \llbracket 2,4 \rrbracket $ and $n'\in\NN_0$, by degree reasons, 
the previous set is a system of generators of $B^n_1$.
By Fact \ref{indep}, the elements in $\mathfrak{B}^n_1$ are linearly independent, so $\mathfrak{B}^n_1$ is a basis of $B^n_1$. 
The dimension of $B^n_1$ is then given by 
\begin{equation}\label{dimcob1}
	\begin{split}
		\operatorname{dim} B^n_1 =
		\begin{cases}
			0,  &  \text{if $n=0,1$}, 
			\\
			3,  &  \text{if $n=2,5$}, 
			\\
		    1,  &  \text{if $n=3$}, 
		    \\
			6,  &  \text{if $n=4$},
			\\
			15,  &  \text{if $n=6$},
			\\
			11,  &  \text{if $n=7$},
			\\
			18,  &  \text{if $n\geqslant 8$ is even}, 
			\\
			13,  & \text{if $n\geqslant 9$ is odd}. 
			\end{cases} 
	\end{split}
\end{equation}

Suppose finally $m=0$. 
Consider $\partial^{n-1}_{-1}: \omega^*_1K^{n-5}_{1}\oplus \omega^*_2 K^{n-9}_3  \to K^n_0\oplus \omega^*_1 K^{n-4}_2\oplus \omega^*_2 K^{n-8}_4$.
Note that $B^n_0\subseteq  \omega^*_1 K^{n-4}_2\oplus \omega^*_2 K^{n-8}_4$ and $B^n_0=\Img(\partial^{n-1}_{-1})=\Img(\omega^*_1 \partial^{n-5}_1)=\omega^*_1 B^{n-4}_2$. We define a basis of $B^n_0$ by 
$\mathfrak{B}^n_0=\omega^*_1 \mathfrak{B}^{n-4}_2$
for $n\geqslant 4$. 
The dimension of $B^n_0$ is thus given by 
\begin{equation}\label{dimcob0}
	\begin{split}
		\operatorname{dim} B^n_0 =
		\begin{cases}
			0,  & \text{if $n\in \llbracket 0,4\rrbracket$},
			\\
			3,  &  \text{if $n=5$}, 
			\\
			8,  &  \text{if $n=6$}, 
			\\
			12,  &  \text{if $n=7,8$},
			\\
			15,  &  \text{if $n=9$}, 
			\\
			13,  &  \text{if $n\geqslant 10$ is even}, 
			\\
			18,  &  \text{if $n\geqslant 11$ is odd}. 
			\end{cases} 
	\end{split}
\end{equation}

\subsection{Computation of the cocycles}
\label{subsection: cocycles}

As one can remark rather easily, from the computations in the previous subsection we can already deduce the dimensions of the homogeneous components of the spaces of cocycles and thus of the Hochschild cohomology groups. 
However, since we will need specific representatives of  the cohomology classes of bases of the Hochschild cohomology $\operatorname{HH}^{\bullet}(A)$ for computing its algebra structure, we will present them. 
More precisely, in this subsection, we will explicitly construct bases $\tilde{\mathfrak{D}}^n_m$ and $\mathfrak{D}^n_m$ of the $\Bbbk$-vector spaces $\tilde{D}^n_m=\Ker(d^n_m)$ and $D^n_m=\Ker(\partial^n_m)$ for $m\in \llbracket 0, 4\rrbracket $ and $n\in\NN_0$, respectively, defined before Remark \ref{mrange}.

\subsubsection{\texorpdfstring{Computation of $\tilde{\mathfrak{D}}^n_m$}{Computation of Dnm}} 

Recall that $\tilde{D}^n_m=\Ker(d^n_m)$ 
and 
\[     d^n_m:K^n_m=\operatorname{Hom}_{\Bbbk}((A_{-n}^!)^*,A_m)
\to 
K^{n+1}_{m+1}=\operatorname{Hom}_{\Bbbk}((A_{-(n+1)}^!)^*,A_{m+1})     \] 
was defined in Subsection \ref{rdsco}.  
Since
$K^n_m/\tilde{D}^n_m\cong \tilde{B}^{n+1}_{m+1}$, 
we see that 
\begin{equation}\label{dimrelationcotilded}
	\begin{split}
		\operatorname{dim} \tilde{D}^n_m=\operatorname{dim}K^n_m-\operatorname{dim} \tilde{B}^{n+1}_{m+1}.
	\end{split}
\end{equation}
Hence, from the dimension of $\tilde{B}^{n+1}_{m+1}$ computed in Subsubsection \ref{subsubsection:cob1} as well as the dimension of $
K^n_{m}$ (see the last paragraph of Subsection \ref{rdsco}), we deduce the value of the dimension of $\tilde{D}^n_{m}$. 
We will present them explicitly in the computations below. 

For every $(n,m)\in \NN_0 \times \llbracket 0,4 \rrbracket $, we are going to provide a set $\tilde{\mathfrak{D}}^n_{m}\subseteq \tilde{D}^n_{m}$ such that $ \# \tilde{\mathfrak{D}}^n_{m}=\operatorname{dim} \tilde{D}^n_m$ and the elements in $\tilde{\mathfrak{D}}^n_{m}$ are linearly independent. 
As a consequence, $\tilde{\mathfrak{D}}^n_{m}$ is a basis of $\tilde{D}^n_{m}$. If $\tilde{D}^n_{m}=K^n_{m}$, we pick the usual basis of $K^n_m$, defined at the end of Subsection \ref{rdsco}.
We leave to the reader the easy verification in each case that the set $\tilde{\mathfrak{D}}^n_{m}$ satisfies these conditions.

Obviously, 
$\tilde{D}^n_4=K^n_4$  
for $n\in\NN_0$.
Then we define the set $\tilde{\mathfrak{D}}^n_4$ by the usual basis of $K^n_4$. 
The dimension of $\tilde{D}^n_4$ is given by 
\begin{equation}\label{dimcodtilde4}
	\begin{split}
		\operatorname{dim} \tilde{D}^n_4 =
		\begin{cases}
			1,  & \text{if $n=0$},
			\\
			3,  & \text{if $n=1$}, 
			\\
			5,  & \text{if $n=2$}, 
			\\
			6,  & \text{if $n\geqslant 3$}.
			\end{cases}
	\end{split}
\end{equation}

Suppose $m=3$. 
By \eqref{dimcobtilde4},
the dimension of $\tilde{D}^n_3$ is given by 
\begin{equation}\label{dimcodtilde3}
	\begin{split}
		\operatorname{dim} \tilde{D}^n_3 =
		\begin{cases}
			0,  &  \text{if $n=0$},
			\\
			8,  &  \text{if $n=1$}, 
			\\
			9,  &  \text{if $n=2$}, 
			\\
			17,  & \text{if $n\geqslant 3$ is odd}, 
			\\
			12,  & \text{if $n\geqslant 4$ is even}.
			\end{cases} 
	\end{split}
\end{equation}
We define the sets 
$\tilde{\mathfrak{D}}^0_3=\emptyset$, 
\begin{align*}
	\tilde{\mathfrak{D}}^1_3=\big\{
		& \suline{\alpha|bac},
		\suline{\beta|abc},
		\suline{\gamma|aba}, \suline{\alpha}|(aba-\suline{abc}),
		(\alpha+\suline{\beta})|\suline{aba},
		\alpha|aba+\suline{\beta|bac},
		\alpha|aba+\suline{\gamma|abc},
		\\
		& \alpha|aba-\suline{\gamma|bac} \big\}, 
\end{align*}
and 
\begin{align*}
	\tilde{\mathfrak{D}}^2_3=\big\{
		&(\suline{\alpha_2}-\beta_2)|\suline{aba},
		(\suline{\alpha_2}-\gamma_2)|\suline{abc},
		(\suline{\beta_2}-\gamma_2)|\suline{bac}, \alpha_2|abc+\beta_2|bac+2\suline{\alpha\beta|aba},
		\\
& \alpha_2|aba-\beta_2|bac+2\suline{\alpha\beta|abc},
\alpha_2|(aba-abc)+2\suline{\alpha\beta|bac},
\alpha_2|abc+\beta_2|bac+2\suline{\alpha\gamma|aba},\\
& \alpha_2|aba-\beta_2|bac+2\suline{\alpha\gamma|abc},
\alpha_2|(aba-abc)+2\suline{\alpha\gamma|bac} 
\big\}.
\end{align*}
Moreover, if $n\geqslant 3$ is odd, we define 
\begin{align*}
	\tilde{\mathfrak{D}}^n_3=\big\{
	& \suline{\alpha_n|bac},
	\suline{\beta_n|abc},
	\suline{\gamma_n|aba},
	\suline{\alpha_{n-1}\beta|abc},
	\suline{\alpha_{n-1}\gamma|aba},
	\suline{\alpha_{n-2}\beta_2|bac},
	\suline{\alpha_n}|(aba-\suline{abc}),
	(\alpha_n+\suline{\beta_n})|\suline{aba},
	\\
    & \alpha_n|aba+\suline{\beta_n|bac},
    \alpha_n|aba+\suline{\gamma_n|abc},
    \alpha_n|aba-\suline{\gamma_n|bac},
    (\alpha_n+\suline{\alpha_{n-1}\beta})|\suline{aba},
    \\
    & \alpha_n|aba+\suline{\alpha_{n-1}\beta|bac},
    \alpha_n|aba+\suline{\alpha_{n-1}\gamma|abc},
    \alpha_{n}|aba-\suline{\alpha_{n-1}\gamma|bac},
    (\alpha_n-\suline{\alpha_{n-2}\beta_2})|\suline{aba},
    \\
    & \alpha_n|aba-\suline{\alpha_{n-2}\beta_2|abc} 
\big\},
\end{align*}
and if $n\geqslant 4$ is even, we set 
\begin{align*}
\tilde{\mathfrak{D}}^n_3=\big\{
	& (\suline{\alpha_n}-\beta_n)|\suline{aba},
	(\suline{\alpha_n}-\gamma_n)|\suline{abc},
	(\suline{\beta_n}-\gamma_n)|\suline{bac},
	\alpha_n|abc+\beta_n|bac+2\suline{\alpha_{n-1}\beta|aba}, 
	\\
    & \alpha_n|aba-\beta_n|bac+2\suline{\alpha_{n-1}\beta|abc},
    \alpha_n|(aba-abc)+2\suline{\alpha_{n-1}\beta|bac},
    \\
    & \alpha_n|abc+\beta_n|bac+2\suline{\alpha_{n-1}\gamma|aba},
    \alpha_n|aba-\beta_n|bac+2\suline{\alpha_{n-1}\gamma|abc},
    \\
    &\alpha_n|(aba-abc)+2\suline{\alpha_{n-1}\gamma|bac},
    (\suline{\alpha_{n-2}\beta_2}-\alpha_n)|\suline{aba},
    (\suline{\alpha_{n-2}\beta_2}-\alpha_n)|\suline{abc}, 
    \\
    & (\suline{\alpha_{n-2}\beta_2}-\beta_n)|\suline{bac} 
\big\}.
\end{align*}

Suppose $m=2$. 
By \eqref{dimcobtilde3}, 
the dimension of $\tilde{D}^n_2$ is given by 
\begin{equation}\label{dimcodtilde2}
	\begin{split}
		\operatorname{dim} \tilde{D}^n_2 =
		\begin{cases}
			2,  & \text{if $n=0$},
			\\
			3,  &  \text{if $n=1$}, 
			\\
			10,  & \text{if $n=2$}, 
			\\
			12,  & \text{if $n\geqslant 3$}.
			\end{cases} 
	\end{split}
\end{equation}
We define the sets 
\[ 
\tilde{\mathfrak{D}}^0_2= \big\{ \suline{\epsilon^!}|(ab+\suline{ba}), \suline{\epsilon^!}|(ab+bc-\suline{ac}) \big\}, 
\]
\[
\tilde{\mathfrak{D}}^1_2=
\big\{
\beta|(ab-ba)+\suline{\gamma}|(\suline{ab}+bc+ac),
(\alpha+\beta+\suline{\gamma})|(ab-\suline{ba}),
\suline{\alpha}|(ab+\suline{bc}+ac)+\beta|(bc+ba+ac) \big\}, 
\]
\begin{equation}
\begin{split}
\tilde{\mathfrak{D}}^2_2=\big\{ 
& \suline{\alpha_2}|(ab+\suline{ba}),
\suline{\alpha_2}|(ab+\suline{bc}-ac),
\suline{\beta_2}|(ab+\suline{ba}),
\suline{\beta_2}|(ab+\suline{bc}-ac),
\suline{\gamma_2}|(ab+\suline{ba}),
\\
& \suline{\gamma_2}|(ab+\suline{bc}-ac),
\suline{\alpha\beta|ba}+\alpha\gamma|ab,
\suline{\alpha\beta|ab}+\alpha\gamma|ba,
\suline{\alpha\beta}|(ba+\suline{ac})-\alpha\gamma|bc,\\
& \suline{\alpha\beta}|(ab+\suline{bc})-\alpha\gamma|ac
\big\}, 
\end{split}
\end{equation}
and 
$ \tilde{\mathfrak{D}}^n_2=\tilde{\mathfrak{B}}^n_2 $ 
for $n\geqslant 3$.

Suppose $m=1$.
By \eqref{dimcobtilde2}, 
the dimension of $\tilde{D}^n_1$ is given by 
\begin{equation}\label{dimcodtilde1}
	\begin{split}
		\operatorname{dim} \tilde{D}^n_1 =
		\begin{cases}
			0,  &  \text{if $n=0$},
			\\
			1,  &  \text{if $n=1$}, 
			\\
			3,  &  \text{if $n=2$}, 
			\\
			6,  &  \text{if $n\geqslant 3$}.
			\end{cases} 
	\end{split}
\end{equation}
We define the sets 
$\tilde{\mathfrak{D}}^0_1=\emptyset$, 
and  
\[ \tilde{\mathfrak{D}}^1_1=\big\{\suline{\alpha|a}+\beta|b+\gamma|c \big\} .\]
Moreover, if $n\geqslant 2$ is even, we define 
$\tilde{\mathfrak{D}}^n_1=\tilde{\mathfrak{B}}^n_1 , $
and if $n\geqslant 3$ is odd, we define 
\begin{align*}
\tilde{\mathfrak{D}}^n_1=\big\{
&\suline{\alpha_n|a}+\beta_n|b+\gamma_n|c,
(\beta_n-\suline{\alpha_{n-1}\beta})|\suline{b},
(\gamma_n-\suline{\alpha_{n-1}\gamma})|\suline{c},
(\alpha_n-\suline{\alpha_{n-2}\beta_2})|\suline{a},
\\
& \suline{\alpha_{n-1}\beta|c}+\alpha_{n-1}\gamma|a+\alpha_{n-2}\beta_2|b,
\suline{\alpha_{n-1}\beta|a}+\alpha_{n-1}\gamma|b+\alpha_{n-2}\beta_2|c 
\big\}.
\end{align*}

Suppose finally $m=0$. 
By \eqref{dimcobtilde1}, 
the dimension of $\tilde{D}^n_0$ is given by 
\begin{equation}\label{dimcodtilde0}
	\begin{split}
		\operatorname{dim} \tilde{D}^n_0 =
		\begin{cases}
			1,  & \text{if $n=0$},
			\\
			0,  & \text{if $n\geqslant 1$ is odd}, 
			\\
			4,  &  \text{if $n=2$}, 
			\\
			5,  & \text{if $n\geqslant 4$ is even}.
			\end{cases}
	\end{split}
\end{equation}
We define the sets 
\[ \tilde{\mathfrak{D}}^0_0=\big\{ \suline{\epsilon^!|1} \big\},  \] 
and 
\[ \tilde{\mathfrak{D}}^2_0=\big\{ 
	\suline{\alpha_2|1},\suline{\beta_2|1},
	\suline{\gamma_2|1},(\suline{\alpha\beta}+\alpha\gamma)|\suline{1}
\big\}.\]
Moreover, if $n\in\NN$ is odd, we define 
$\tilde{\mathfrak{D}}^n_0=\emptyset$, 
and if $n\geqslant 4$ is even, we define 
\[
\tilde{\mathfrak{D}}^n_0=\big\{ \suline{\alpha_n|1},\suline{\beta_n|1},\suline{\gamma_n|1},(\suline{\alpha_{n-1}\beta}+\alpha_{n-1}\gamma)|\suline{1},\suline{\alpha_{n-2}\beta_2|1} \big\}.
\]

\subsubsection{\texorpdfstring{Computation of $\mathfrak{D}^n_m$}{Computation of Dnm}} 
\label{subsubsection:cod2}

Recall that $D^n_m=\Ker(\partial^n_m)$ and $\partial^n_m:Q^n_m\to Q^{n+1}_{m+1}$.
The isomorphism $Q^n_m/D^n_m\cong B^{n+1}_{m+1}$ tells us that  
\begin{equation}\label{dimrelationcod}
	\begin{split}
		\operatorname{dim}D^n_m=\operatorname{dim} Q^n_m-\operatorname{dim}B^{n+1}_{m+1}.
	\end{split}
\end{equation}
Hence, from the dimension of $B^{n+1}_{m+1}$ computed in Subsubsection \ref{subsubsection:cob2} as well as the dimension of $Q^n_{m}$ (see the last paragraph of Subsection \ref{rdsco}), we deduce the value of the dimension of $D^n_{m}$. 
We will present them explicitly in the computations below.

For every $(n,m)\in \NN_0 \times \llbracket 0,4 \rrbracket $, we are going to provide a set $\mathfrak{D}^n_{m}\subseteq D^n_{m}$ such that $ \# \mathfrak{D}^n_{m}=\operatorname{dim} D^n_m$ and the elements in $\mathfrak{D}^n_{m}$ are linearly independent. 
As a consequence, $\mathfrak{D}^n_{m}$ is a basis of $D^n_{m}$.
We leave to the reader the easy verification in each case that the set $\mathfrak{D}^n_{m}$ satisfies these conditions.

For either $m\in \llbracket 3,4\rrbracket$ and $ n\in\NN_0$, or $m, n \in \llbracket 0,2 \rrbracket$, note that $\partial^n_m=d^n_m$,  
then $D^n_m=\tilde{D}^n_m$. 
So we define the basis of $D^n_m$ by 
$\mathfrak{D}^n_m=\tilde{\mathfrak{D}}^n_m$.

Suppose $m=2$.
By $B^n_3=\tilde{B}^n_3$ and \eqref{dimcobtilde3}, 
the dimension of $D^n_2$ is given by 
\begin{equation}\label{dimcod2}
	\begin{split}
		\operatorname{dim} D^n_2 =
		\begin{cases}
			2,  & \text{if $n=0$},
			\\
			3,  & \text{if $n=1$},
			\\
			10,  & \text{if $n=2$},
			\\
			12,  &  \text{if $n=3$},
			\\
			13,  &  \text{if $n=4$},
			\\
			15,  &  \text{if $n=5$},
			\\
			17,  &  \text{if $n=6$},
			\\
			18,  &  \text{if $n\geqslant 7$}.
			\end{cases}
	\end{split}
\end{equation}
We define the sets 
$\mathfrak{D}^3_2=\tilde{\mathfrak{D}}^3_2$
and 
$\mathfrak{D}^n_2=\tilde{\mathfrak{D}}^n_2\cup \omega^*_1 \tilde{\mathfrak{D}}^{n-4}_4 $
for $n\geqslant 4$.

Suppose $m=1$. 
By \eqref{dimcob2}, 
the dimension of $D^n_1$ is given by 
\begin{equation}\label{dimcod1}
	\begin{split}
		\operatorname{dim} D^n_1 =
		\begin{cases}
			0,  & \text{if $n=0$},
			\\
			1,  & \text{if $n=1$},
			\\
			3,  & \text{if $n=2$},
			\\
			6,  & \text{if $n=3,4$},
			\\
			14,  & \text{if $n=5$},
			\\
			15,  &  \text{if $n=6$},
			\\
			23,  &  \text{if $n\geqslant 7$ is odd},
			\\
			18,  & \text{if $n\geqslant 8$ is even}.
			\end{cases}
	\end{split}
\end{equation}
We define the set 
$\mathfrak{D}^3_1=\tilde{\mathfrak{D}}^3_1$.
Moreover, if $n\geqslant 4$ is even, 
we define 
$\mathfrak{D}^n_1=\mathfrak{B}^n_1$,
and if $n\geqslant 5$ is odd, we define 
$\mathfrak{D}^n_1=\tilde{\mathfrak{D}}^n_1\cup \omega^*_1\tilde{\mathfrak{D}}^{n-4}_3 $.

Suppose finally $m=0$.
By \eqref{dimcob1},
the dimension of $\mathfrak{D}^n_0$ is given by 
\begin{equation}\label{dimcod0}
	\begin{split}
		\operatorname{dim} D^n_0 =
		\begin{cases}
			1,  & \text{if $n=0$},
			\\
			0,  & \text{if $n=1,3$},
			\\
			4,  &  \text{if $n=2$},
			\\
			7,  &  \text{if $n=4$},
			\\
			3,  &  \text{if $n=5$},
			\\
			15,  & \text{if $n=6,9$},
			\\
			12,  & \text{if $n=7$},
			\\
			18,  &  \text{if $n=8$ or $n\geqslant 11$ is odd},
			\\
			22,  &  \text{if $n=10$},
			\\
			23,  & \text{if $n\geqslant 12$ is even}.
			\end{cases}
	\end{split}
\end{equation}
We define the set 
$\mathfrak{D}^3_0=\emptyset$.
Moreover, if $n\geqslant 4$ is even, we define the set
$\mathfrak{D}^n_0=\tilde{\mathfrak{D}}^n_0\cup \omega^*_1 \mathfrak{D}^{n-4}_2  $,
and if $n\geqslant 5$ is odd, we define 
$\mathfrak{D}^n_0=\omega^*_1 \mathfrak{D}^{n-4}_2 $.

\subsection{Hochschild cohomology}
\label{subsection: cohomology}

In this subsection, we will explicitly construct a subspace $H^n_m\subseteq D^n_m$ such that  
$D^n_m= H^{n}_m\oplus B^n_m$ 
for $(n,m)\in \NN_0\times \ZZ_{\leqslant 4}$, 
and we define $H^n_m=0$ for $(n,m)\in \ZZ^2\setminus ( \NN_0\times \ZZ_{\leqslant 4})$.
By Proposition \ref{bdnmco}, we have the following similar recursive description. 

\begin{cor}
\label{H_{-1}}
For integers $m\leqslant 1$ and $n\in\NN_0$, we have 
\begin{equation}\label{euqation:Hnm coho}	
\begin{split}
H_{m}^n\cong
\begin{cases}
\omega^*_{\frac{1-m}{2}}H_{1}^{n+2m-2},  & \text{if $m$ is odd}, 
\\
\omega^*_{-\frac{m}{2}}H_0^{n+2m},  & \text{if $m$ is even}.
\end{cases} 
\end{split}
\end{equation}
\end{cor}

So it is also sufficient to compute the case $m\in\llbracket 0,4\rrbracket $. 
Recall that 
\begin{equation}\label{equation:relation co dim hdb}
	\begin{split}
		\operatorname{dim} H^n_m=\operatorname{dim}D^n_m-\operatorname{dim}B^n_m=\operatorname{dim}Q^n_m-\operatorname{dim}B^{n+1}_{m+1}-\operatorname{dim}B^n_m.
	\end{split}
\end{equation}
Hence, from the dimension of $D^n_{m}$ computed in Subsubsection \ref{subsubsection:cod2} as well as the dimension of 
$B^n_{m}$ computed in Subsubsection \ref{subsubsection:cob2}, 
we deduce the value of the dimension of $H^n_{m}$. 
We will present them explicitly in the computations below.

For every $(n,m)\in \NN_0\times \llbracket 0,4 \rrbracket $, we are going to provide a set $\mathfrak{H}^n_{m}\subseteq D^n_{m}$ such that $ \# \mathfrak{H}^n_{m}=\operatorname{dim} H^n_{m}$ and $ \mathfrak{H}^n_{m}\cup \mathfrak{B}^n_{m}$ is linearly independent. 
As a consequence, the space $H^n_{m}$ spanned by $\mathfrak{H}^n_{m}$ satisfies $D^n_m= H^{n}_m\oplus B^n_m$.
We leave to the reader the easy verification in each case that the set $\mathfrak{H}^n_{m}$ satisfies these conditions. Note that,
unless stated otherwise, the linear independence of the elements in $ \mathfrak{H}^n_{m}\cup \mathfrak{B}^n_{m}$ follows from Fact \ref{indep}, where we put the elements in $\mathfrak{H}^n_{m}$ before the elements in $\mathfrak{B}^n_{m}$.

Suppose $m=4$. 
The dimension of $H^n_4$ is given by 
\begin{equation}\label{dimcoh4}
	\begin{split}
		\operatorname{dim} H^n_4 =
		\begin{cases}
			1,  & \text{if $n=0$},
			\\
			0,  &  \text{if $n\in\NN$ is odd}, 
			\\
			4,  & \text{if $n=2$},
			\\
			5,  &  \text{if $n\geqslant 4$ is even}.
			\end{cases}
	\end{split}
\end{equation}
We define the sets 
\[ \mathfrak{H}^0_4=\big\{ \suline{\epsilon^!|abac} \big\},  \]  
and 
\[  \mathfrak{H}^2_4=\big\{\suline{\alpha_2|abac},\suline{\beta_2|abac},\suline{\gamma_2|abac}, \suline{\alpha\beta|abac} \big\}  .\]
Moreover, if $n\in\NN$ is odd, we define 
$\mathfrak{H}^n_4=\emptyset$, 
and if $n\geqslant 4$ is even, we define 
\[ \mathfrak{H}^n_4=\big\{ \suline{\alpha_n|abac}, \suline{\beta_n|abac}, \suline{\gamma_n|abac}, \suline{\alpha_{n-1}\beta|abac}, \suline{\alpha_{n-2}\beta_2|abac}  \big\}.  \]

Suppose $m=3$.
The dimension of $H^n_3$ is given by 
\begin{equation}\label{dimcoh3}
	\begin{split}
		\operatorname{dim} H^n_3 =
		\begin{cases}
			0,  &  \text{if $n\in \NN_0$ is even},
			\\
			6,  & \text{if $n=1$}, 
			\\
			7,  & \text{if $n=3$},
			\\
			5,  & \text{if $n\geqslant 5$ is odd}.
			\end{cases}
	\end{split}
\end{equation}
We define the sets 
\[ 
\mathfrak{H}^1_3=\big\{
\suline{\alpha|bac},
\suline{\beta|abc},
\suline{\gamma|aba},
\suline{\alpha}|(aba-\suline{abc}),
(\alpha+\suline{\beta})|\suline{aba},
\alpha|aba+\suline{\beta|bac} 
\big\},
\] 
and 
\[ 
\mathfrak{H}^3_3=\big\{ \alpha_3|bac,
\beta_3|abc,
\gamma_3|aba,
\alpha_2\beta|abc,
\alpha_3|(aba-abc),
(\alpha_3+\beta_3)|aba,
\alpha_3|aba+\beta_3|bac 
\big\}. 
\] 
Moreover, if $n\in\NN_0$ is even, we define 
$ \mathfrak{H}^n_3=\emptyset$,
and if $n\geqslant 5$ is odd, we define the set
\[ \mathfrak{H}^n_3=\big\{ \alpha_n|bac,
\beta_n|abc,
\gamma_n|aba,
\alpha_{n-1}\beta|abc,
\alpha_n|(aba-abc)  
\big\}.\] 
The reader can easily verify that the set $ \mathfrak{H}^n_{3}\cup \mathfrak{B}^n_{3}$ for $n\geqslant 3$ and $n$ odd is linearly independent.

Suppose $m=2$.
The dimension of $H^n_2$ is given by 
\begin{equation}\label{dimcoh2}
	\begin{split}
		\operatorname{dim} H^n_2 =
		\begin{cases}
			2,  &  \text{if $n=0,2$}, 
			\\	
			0,  &  \text{if $n\in\NN$ is odd},
			\\
			1,  &  \text{if $n=4$}, 
			\\
			4,  &  \text{if $n=6$},
			\\
			5,  &  \text{if $n\geqslant 8$ is even}.
			\end{cases}. 
	\end{split}
\end{equation}
We define the sets 
\[ \mathfrak{H}^0_2=\big\{ \suline{\epsilon^!}|(ab+\suline{ba}),\suline{\epsilon^!}|(ab+bc-\suline{ac})   \big\} ,
\]
and 
\[ \mathfrak{H}^2_2= \big\{\suline{\alpha_2}|(\suline{ab}+ba),\suline{\beta_2}|(\suline{ab}+ba) \big\} . \]
Moreover, if $n\in\NN$ is odd, we define
$\mathfrak{H}^n_2=\emptyset $, 
and if $n\geqslant 4$ is even, we define 
$\mathfrak{H}^n_2=\omega^*_1 \mathfrak{H}^{n-4}_4 $.

Suppose $m=1$.
The dimension of $H^n_1$ is given by 
\begin{equation}\label{dimcoh1}
	\begin{split}
		\operatorname{dim} H^n_1 =
		\begin{cases}
			0,  & \text{if $n\in\NN_0$ is even},
			\\
			1,  & \text{if $n=1$}, 
			\\
			5,  & \text{if $n=3$}, 
			\\
			11,  & \text{if $n=5$},
			\\
			12,  &  \text{if $n=7$}, 
			\\
			10,  &  \text{if $n\geqslant 9$ is odd}.
			\end{cases} 
	\end{split}
\end{equation}
We define the sets 
\[ \mathfrak{H}^1_1=\big\{\suline{\alpha|a}+\beta|b+\gamma|c \big\} ,\] 
and 
\[ \mathfrak{H}^3_1=\big\{ \suline{\alpha_3|a}+\beta_3|b+\gamma_3|c,
(\beta_3-\suline{\alpha_2\beta})|\suline{b},
(\gamma_3-\suline{\alpha_2\gamma})|\suline{c},
(\alpha_3-\suline{\alpha\beta_2})|\suline{a},
\suline{\alpha_2\beta|c}+\alpha_2\gamma|a+\alpha\beta_2|b   
\big\}  .\] 
Moreover, if $n\in\NN_0$ is even, we define 
$\mathfrak{H}^n_1=\emptyset $,
and if $n\geqslant 5$ is odd, we define 
\begin{align*}
	\mathfrak{H}^n_1=\big\{ 
		& \suline{\alpha_n|a}+\beta_n|b+\gamma_n|c,
		(\beta_n-\suline{\alpha_{n-1}\beta})|\suline{b},
		(\gamma_n-\suline{\alpha_{n-1}\gamma})|\suline{c},
		(\alpha_n-\suline{\alpha_{n-2}\beta_2})|\suline{a},
		\\
		& \suline{\alpha_{n-1}\beta|c}+\alpha_{n-1}\gamma|a+\alpha_{n-2}\beta_2|b
		\big\} \cup \omega^*_1 \mathfrak{H}^{n-4}_3.
\end{align*}

Suppose finally $m=0$. 
The dimension of $H^n_0$ is given by 
\begin{equation}\label{dimcoh0}
	\begin{split}
		\operatorname{dim} H^n_0 =
		\begin{cases}
			1,  &  \text{if $n=0$}, 
			\\	
			0,  &  \text{if $n\in\NN$ is odd},
			\\
			4,  &  \text{if $n=2$}, 
			\\
			7,  &  \text{if $n=4,6$}, 
			\\
			6,  &  \text{if $n=8$},
			\\
			9,  &   \text{if $n=10$}, 
			\\
			10,  &  \text{if $n\geqslant 12$ is even}.
			\end{cases} 
	\end{split}
\end{equation}
We define the sets 
\[ \mathfrak{H}^0_0=\big\{ \suline{\epsilon^!|1} \big\} ,\] 
and 
\[ \mathfrak{H}^2_0=\big\{ \suline{\alpha_2|1}, \suline{\beta_2|1}, \suline{\gamma_2|1}, (\suline{\alpha\beta}+\alpha\gamma)|\suline{1} \big\} .\] 
Moreover, 
if $n\in\NN$ is odd, we define 
$\mathfrak{H}^n_0=\emptyset$,
and if $n\geqslant 4$ is even, we set 
\[ \mathfrak{H}^n_0=\big\{ 
	\suline{\alpha_n|1},\suline{\beta_n|1},\suline{\gamma_n|1}, (\suline{\alpha_{n-1}\beta}+\alpha_{n-1}\gamma)|\suline{1},\suline{\alpha_{n-2}\beta_2|1}
\big\} \cup \omega^*_1\mathfrak{H}^{n-4}_2    . \]

The previous results can be restated as follows.

\begin{cor}\label{tildecohonm}
Let $m\in \llbracket 0,4 \rrbracket$ and $n\in\NN_0$. Then 
$H^n_m=\tilde{H}^n_m\oplus \omega^*_1 H^{n-4}_{m+2}$. Here, $\tilde{H}^n_m$ is the $\Bbbk$-vector space spanned by the set $\tilde{\mathfrak{H}}^n_m$, which is defined as follows. 
If $m\in\llbracket 3,4\rrbracket $, we define the set 
$ \tilde{\mathfrak{H}}^n_m=\mathfrak{H}^n_m $
for $n\in\NN_0$. 
If $m=2$, we define the sets 
\[ \tilde{\mathfrak{H}}^0_2=\big\{ \epsilon^!|(ab+ba),\epsilon^!|(ab+bc-ac) \big\},  
\quad
 \tilde{\mathfrak{H}}^2_2=\big\{\alpha_2|(ab+ba),\beta_2|(ab+ba) \big\} ,\] 
and 
$\tilde{\mathfrak{H}}^n_2=\emptyset$
for $n=1$ and $n\geqslant 3$.
If $m=1$, 
we define the set
\[ \tilde{\mathfrak{H}}^1_1=\big\{ \alpha|a+\beta|b+\gamma|c \big\} ,\] 
and $\tilde{\mathfrak{H}}^n_1=\emptyset$ 
for $n\in\NN_0$ with $n$ even, together with  
\begin{align*}
	\tilde{\mathfrak{H}}^n_1=\big\{
		& \alpha_n|a+\beta_n|b+\gamma_n|c,(\beta_n-\alpha_{n-1}\beta)|b,(\gamma_n-\alpha_{n-1}\gamma)|c,(\alpha_n-\alpha_{n-2}\beta_2)|a,\\
	& \alpha_{n-1}\beta|c+\alpha_{n-1}\gamma|a+\alpha_{n-2}\beta_2|b 
	\big\}
	\end{align*}
for $n\geqslant 3$ with $n$ odd.
If $m=0$, we define the sets 
\[ \tilde{\mathfrak{H}}^0_0=\big\{\epsilon^!|1 \big\} ,
\quad 
  \tilde{\mathfrak{H}}^2_0=\big\{  \alpha_2|1,\beta_2|1, \gamma_2|1, (\alpha\beta+\alpha\gamma)|1 \big\}     ,\]  
and $\tilde{\mathfrak{H}}^n_0=\emptyset$
for $n\in\NN$ with $n$ odd, together with  
\[ \tilde{\mathfrak{H}}^n_0=\big\{ \alpha_n|1,\beta_n|1,\gamma_n|1,(\alpha_{n-1}\beta+\alpha_{n-1}\gamma)|1,\alpha_{n-2}\beta_2|1 \big\}
\]
for $n\geqslant 4$ with $n$ even. 
Moreover, 
if we define $\tilde{H}^n_m=0$ for $(n,m)\in \ZZ^2 \setminus (\NN_0 \times \llbracket 0,4 \rrbracket )$, 
then $H^n_m=\tilde{H}^n_m\oplus \omega^*_1 H^{n-4}_{m+2}$ holds for $m, n  \in \ZZ$ by applying Corollary \ref{H_{-1}}.
\end{cor}

\begin{rk}
The reader can easily check that $\tilde{D}^n_m=\tilde{H}^n_m \oplus \tilde{B}^n_m$ for $m\in \llbracket 0,4 \rrbracket$ and $n\in\NN_0$.	
\end{rk}

Recall that the Hochschild cohomology is decomposed as $\operatorname{HH}^n(A)=\oplus_{m\leqslant 4}H^n_m$ for $n\in\NN_0$.

\begin{prop}\label{linear structure of cohomology}
For $n\in\NN_0$, 
\[ 
\operatorname{HH}^n(A) = \mathop{\bigoplus}\limits_{\substack{i\in \llbracket 0, \lfloor n/4 \rfloor \rrbracket, \\
m\in \llbracket 0,4 \rrbracket}}\omega^*_{i}  \tilde{H}^{n-4i}_{m}.
\] 
\end{prop}

\begin{proof}
%
%
By Corollary \ref{tildecohonm}, we have 
\begin{equation}\label{hco024}
    \begin{split}
    H^n_2=\tilde{H}^n_2\oplus \omega^*_1 \tilde{H}^{n-4}_4,
    \quad 
        H^n_1=\tilde{H}^n_1\oplus \omega^*_1 \tilde{H}^{n-4}_3, \quad 
        H^n_0=\tilde{H}^n_0\oplus \omega^*_1 \tilde{H}^{n-4}_2\oplus \omega^*_2 \tilde{H}^{n-8}_4
    \end{split}
\end{equation}
for $n\in \NN_0$. Using Corollary \ref{H_{-1}} and \eqref{hco024}, we get 
\begin{align*}
\operatorname{HH}^n(A)
& = \bigoplus_{m\in\llbracket -2\lfloor n/4\rfloor, 4\rrbracket }H^n_m 
\\
& = H^n_4\oplus H^n_3\oplus H^n_2\oplus \bigg( \bigoplus_{i\in  \llbracket 0, \lfloor n/4 \rfloor \rrbracket }\omega^*_i H^{n-4i}_1 \bigg) 
\oplus 
\bigg( \bigoplus_{i\in  \llbracket 0, \lfloor n/4 \rfloor \rrbracket }\omega^*_i H^{n-4i}_0 \bigg) 
\\
& = \tilde{H}^n_4
\oplus 
\tilde{H}^n_3
\oplus
(\tilde{H}^n_2\oplus \omega^*_1 \tilde{H}^{n-4}_4) 
\oplus 
\bigg( \bigoplus_{i\in  \llbracket 0, \lfloor n/4 \rfloor  \rrbracket }\omega^*_i (\tilde{H}^{n-4i}_1 \oplus \omega^*_1 \tilde{H}^{n-4i-4}_3)\bigg) \\
& \phantom{= \; }
\oplus \bigg( \bigoplus_{i\in  \llbracket 0, \lfloor n/4 \rfloor \rrbracket }\omega^*_i ( \tilde{H}^{n-4i}_0\oplus \omega^*_1 \tilde{H}^{n-4i-4}_2\oplus \omega^*_2 \tilde{H} ^{n-4i-8}_4)\bigg) 
\\
& = \tilde{H}^n_4
\oplus 
\tilde{H}^n_3
\oplus
\tilde{H}^n_2\oplus \omega^*_1 \tilde{H}^{n-4}_4
\oplus \bigg( \bigoplus_{i\in  \llbracket 0, \lfloor n/4 \rfloor  \rrbracket }\omega^*_i \tilde{H}^{n-4i}_1 \bigg) 
\oplus \bigg( \bigoplus_{i\in  \llbracket 0, \lfloor n/4 \rfloor  \rrbracket }\omega^*_{i+1} \tilde{H}^{n-4i-4}_3 \bigg)\\
& \phantom{= \; }
\oplus \bigg( \bigoplus_{i\in  \llbracket 0, \lfloor n/4 \rfloor \rrbracket }\omega^*_i  \tilde{H}^{n-4i}_0 \bigg)
\oplus \bigg( \bigoplus_{i\in  \llbracket 0, \lfloor n/4 \rfloor \rrbracket }\omega^*_{i+1}  \tilde{H}^{n-4i-4}_2 \bigg)
\oplus \bigg( \bigoplus_{i\in  \llbracket 0, \lfloor n/4 \rfloor \rrbracket }\omega^*_{i+2}  \tilde{H}^{n-4i-8}_4 \bigg)
\\
& =\tilde{H}^n_4
\oplus 
\tilde{H}^n_3
\oplus
\tilde{H}^n_2\oplus \omega^*_1 \tilde{H}^{n-4}_4
\oplus \bigg( \bigoplus_{i\in  \llbracket 0, \lfloor n/4 \rfloor  \rrbracket }\omega^*_i \tilde{H}^{n-4i}_1 \bigg) 
\oplus \bigg( \bigoplus_{i\in  \llbracket 1, \lfloor n/4 \rfloor  \rrbracket }\omega^*_{i} \tilde{H}^{n-4i}_3 \bigg)\\
& \phantom{= \; }
\oplus \bigg( \bigoplus_{i\in  \llbracket 0, \lfloor n/4 \rfloor \rrbracket }\omega^*_i  \tilde{H}^{n-4i}_0 \bigg)
\oplus \bigg( \bigoplus_{i\in  \llbracket 1, \lfloor n/4 \rfloor \rrbracket }\omega^*_{i}  \tilde{H}^{n-4i}_2 \bigg)
\oplus \bigg( \bigoplus_{i\in  \llbracket 2, \lfloor n/4 \rfloor \rrbracket }\omega^*_{i}  \tilde{H}^{n-4i}_4 \bigg)
\\
& =\mathop{\bigoplus}\limits_{\substack{i\in \llbracket 0, \lfloor n/4 \rfloor \rrbracket, \\
m\in \llbracket 0,4 \rrbracket}}\omega^*_{i}  \tilde{H}^{n-4i}_{m}.
\end{align*}
\end{proof}

\begin{rk}
Let $\tilde{H}^n=\oplus_{m\in \llbracket 0,4\rrbracket}\tilde{H}^n_m$. 
Proposition \ref{linear structure of cohomology} shows that $\operatorname{HH}^n(A)=\oplus_{i\in \llbracket 0, \lfloor n/4 \rfloor \rrbracket}\omega^*_{i}  \tilde{H}^{n-4i}$. 
Using Corollary \ref{tildecohonm}, it is easy to compute that $\operatorname{dim}\tilde{H}^0=4$, $\operatorname{dim}\tilde{H}^1=7$, $\operatorname{dim}\tilde{H}^2=10$, $\operatorname{dim}\tilde{H}^3=12$ and $\operatorname{dim}\tilde{H}^n=10$ for $n\geqslant 4$. 
\end{rk}

Using the previous remark, we get the dimension of $\operatorname{HH}^n(A)$.

\begin{prop}
\label{proposition:cohomology}
The dimension of $\operatorname{HH}^n(A)$ is given by 
\[ 
\operatorname{dim}\operatorname{HH}^n(A)=
\begin{cases}
\frac{5}{2}n+4, &  \text{if $n=4r$ for $r\in\NN_0$}, 
\\
\frac{5}{2}n+5, &  \text{if $n=4r+2$ for $r\in\NN_0$},
\\
\frac{5n+9}{2}, &  \text{if $n=2r+1$ for $r\in\NN_0$}.
\end{cases} 
\] 
\end{prop}

The Hilbert series of $\operatorname{HH}^n(A)$ is $ h^n(t)=\sum_{m\leqslant 4}\operatorname{dim}(H^n_m)t^{m-n}$ for $n\in\NN_0$. Note that $m-n$ is the internal degree of $H^n_m$.
\begin{cor}
\label{corollary:cohomology-hilbert-series}
The Hilbert series $h^n(t)$ of $\operatorname{HH}^n(A)$ is given as follows.
Let $n\geqslant 8$.
Then 
\[ 
h^n(t)=t^{-n}
\big[ 
5\chi_n t^4+5\chi_{n+1} t^3+5\chi_n t^2 +10 \sum\limits_{i=0}^{\lfloor \frac{n}{4} \rfloor-3}t^{\chi_{n+1}-2i}+t^{-2\lfloor \frac{n}{4} \rfloor} p^n(t)
\big],
\] 
where
\[ 
p^n(t)=
\begin{cases}
6t^4+7t^2+1,& \text{if $n\equiv 0$ $(\operatorname{mod}$ $4)$},
\\
10t^5+11t^{3}+t,&\text{if $n\equiv 1$ $(\operatorname{mod}$ $4)$},
\\
9t^4+7t^2+4,&\text{if $n\equiv 2$ $(\operatorname{mod}$ $4)$},
\\
10t^5+12t^{3}+5t,&\text{if $n\equiv 3$ $(\operatorname{mod}$ $4)$}.
\end{cases}
\] 
%
%
%
Moreover, 
\begin{align*}
h^0(t)&=t^4+2t^2+1,\quad &
h^1(t)&=6t^2+1,
\\
h^2(t)&=4t^2+2+4t^{-2},\quad & 
h^3(t)&=7+5t^{-2},
\\
h^4(t)&=5+t^{-2}+7t^{-4}+t^{-6},\quad &
h^5(t)&=5t^{-2}+11t^{-4}+t^{-6},
\\
h^6(t)&=5t^{-2}+4t^{-4}+7t^{-6}+4t^{-8},\quad &
h^7(t)&=5t^{-4}+12t^{-6}+5t^{-8}.
\end{align*}
\end{cor}

\section{Algebraic structure of Hochschild cohomology} 
\label{section:algebra-cohomology}

In this section we will explicitly determine the algebra structure of the Hochschild cohomology of $A$ given by the cup product $\cup$. 
To do so, we will first find a generating set of the $\Bbbk$-algebra  $\operatorname{HH}^{\bullet}(A)=\oplus_{n\in \NN_0}\operatorname{HH}^n(A)$ (see Proposition \ref{finitely generated}). 
Then, after extracting a minimal generating set 
from the previous set of generators, we will 
find an explicit presentation of the algebra $\operatorname{HH}^{\bullet}(A)$ as a quotient of a free algebra $F$ by the ideal generated by an explicit set $\mathcalboondox{R}$ of homogeneous relations. 
This is done by using a Gröbner basis of $\mathcalboondox{R}$, which allows us to compute the Hilbert series of the quotient $F/(\mathcalboondox{R})$, and then comparing the Hilbert series of the quotient and that of $\operatorname{HH}^{\bullet}(A)$. 
We refer the reader to Sections 2.1 and 2.2 of the very nice book \cite{Sarah} for the usual method for computing the cup product.

\begin{rk}\label{n1n2}
It is easy to see that $f\cup g\in H^{n_1+n_2}_{m_1+m_2}$ for all $f\in H^{n_1}_{m_1}$, $g\in H^{n_2}_{m_2}$. 
Moreover, it is well-known that the cup product on Hochschild cohomology is graded commutative, \textit{i.e.} $f\cup g=(-1)^{mn}g\cup f$ for $f\in \operatorname{HH}^m(A)$, $g\in \operatorname{HH}^n(A)$ (see \cite{Sarah}, Thm. 1.4.6).
\end{rk}

\begin{lem}\label{g}
	Let $g=\omega^*_1\epsilon^!|1\in \omega^*_1 \tilde{H}^0_0=H^4_{-2}$. Then $f\cup g=\omega^*_1 f$ for all $f\in \operatorname{HH}^{\bullet}(A)$.
\end{lem}

\begin{proof}
 The map $g$ can be extended to a chain map $g_{\bullet} : P^b_{\bullet}\to P^b_{\bullet}$ with $g_n(\omega^*_i x)=\omega^*_{i-1}x$ for $x\in K^b_{n+4-4i}$ and $i\in \llbracket 0,\lfloor n/4 \rfloor+1 \rrbracket$. 
 Hence, given $f\in \operatorname{HH}^m(A)$, we get $f\cup g=fg_m=\omega^*_1 f$.
\end{proof}

By Lemma \ref{g} and Proposition \ref{linear structure of cohomology}, the set 
\[ 
\bigg(\underset{\text{\begin{tiny}$\begin{matrix} m\in \llbracket 0, 4 \rrbracket, \\ n\in \NN_0 \end{matrix}$\end{tiny}}}{\bigcup} \tilde{\mathfrak{H}}^n_m 
	\cup \big\{ \omega^*_1\epsilon^!|1 \big\}\bigg) \setminus \big\{ \epsilon^!|1 \big\}
\] 
%
is a generating set of $\operatorname{HH}^{\bullet}(A)$ as $\Bbbk$-algebra.

\begin{fact}\label{fact chain map}
Assume $x,y\in A$. Let $n\geqslant 2$ be even. 
The map $\alpha_n|1$ can be extended to the chain map $g_{\bullet}:P^b_{\bullet}\to P^b_{\bullet}$ satisfying 
\begin{align*}
	& g_0(x|\alpha_{n}|y)=x|\epsilon^!|y,
	\\
	& g_0(x|\beta_{n}|y)=g_0(x|\gamma_{n}|y)=g_0(x|\alpha_{n-1}\beta|y)=g_0(x|\alpha_{n-1}\gamma|y)=g_0(x|\alpha_{n-2}\beta_2|y)=0,\\
	& g_1(x|\alpha_{n+1}|y)=x|\alpha|y,\quad 
	g_1(x|\beta_{n+1}|y)=g_1(x|\gamma_{n+1}|y)=g_1(x|\alpha_{n-1}\beta_2|y)=0,\\
	& g_1(x|\alpha_{n}\beta|y)=x|\beta|y,\quad 
	g_1(x|\alpha_{n}\gamma|y)=x|\gamma|y,\\
	& g_2(x|\alpha_{n+2}|y)=x|\alpha_2|y,\quad 
	g_2(x|\beta_{n+2}|y)=g_2(x|\gamma_{n+2}|y)=0, \quad 
	g_2(x|\alpha_{n+1}\beta|y)=x|\alpha\beta|y,\\
	& g_2(x|\alpha_{n+1}\gamma|y)=x|\alpha\gamma|y,\quad 
	g_2(x|\alpha_{n}\beta_2|y)=x|(\beta_2+\gamma_2)|y.
\end{align*}
Moreover, if $n=2$, the chain map $g_{\bullet}$ satisfies
\begin{align*}
    g_2(\omega_1x|\epsilon^!|y)
    & =x(2|\alpha\beta|ac+1|\alpha\beta|ba+a|\alpha\beta|b-b|\alpha\beta|c+b|\alpha\gamma|a-c|\alpha\gamma|b-b|\alpha_2|b-3c|\alpha_2|c\\
    &  \phantom{ = \;}  
    +a|\beta_2|a+a|\gamma_2|a+2b|\gamma_2|b-2bc|\alpha\gamma|1-ab|\alpha\gamma|1)y.
\end{align*}
If $n=4$, the chain map $g_{\bullet}$ satisfies
\begin{align*}
    & g_0(\omega_1 x|\epsilon^!|y)=0,
    \\
    & g_1(\omega_1 x|\alpha|y)=2x( 1|\beta|ac+b|\alpha|c+ba|\gamma|1 )y, 
    \\
    & g_1(\omega_1x|\beta|y)=-2x(c|\beta|a+a|\beta|c+b|\alpha|a+a|\gamma|a+a|\alpha|b )y,
    \\
    & g_1(\omega_1x|\gamma|y)=-2x(b|\gamma|a+a|\gamma|b+a|\beta|a+a|\alpha|c+c|\alpha|a )y.
\end{align*}
The map $\beta_n|1$ 
can be extended to the chain map $g_{\bullet}:P^b_{\bullet}\to P^b_{\bullet}$ satisfying 
\begin{align*}
	& g_0(x|\beta_{n}|y)=x|\epsilon^!|y,\\
	& g_0(x|\alpha_{n}|y)=g_0(x|\gamma_{n}|y)=g_0(x|\alpha_{n-1}\beta|y)=g_0(x|\alpha_{n-1}\gamma|y)=g_0(x|\alpha_{n-2}\beta_2|y)=0,\\
	& g_1(x|\beta_{n+1}|y)=x|\beta|y,\quad 
	g_1(x|\alpha_{n+1}|y)=g_1(x|\gamma_{n+1}|y)=g_1(x|\alpha_{n}\beta|y)=0,\\
	& g_1(x|\alpha_{n}\gamma|y)=x|\gamma|y,\quad 
	g_1(x|\alpha_{n-1}\beta_2|y)=x|\alpha|y,\\
	& g_2(x|\beta_{n+2}|y)=x|\beta_2|y,\quad 
	g_2(x|\alpha_{n+2}|y)=g_2(x|\gamma_{n+2}|y)=0, \quad 
	g_2(x|\alpha_{n+1}\beta|y)=x|\alpha\beta|y,\\
	& g_2(x|\alpha_{n+1}\gamma|y)=x|\alpha\gamma|y,\quad 
	g_2(x|\alpha_{n}\beta_2|y)=x|(\alpha_2+\gamma_2)|y.
\end{align*}
Moreover, if $n=2$, the chain map $g_{\bullet}$ satisfies
\begin{align*}
    g_2(\omega_1x|\epsilon^!|y)
    & =x(1|\alpha\beta|ba-1|\alpha\beta|ac+b|\alpha\beta|c-c|\alpha\beta|a+c|\alpha\gamma|b-a|\alpha\gamma|c-c|\beta_2|c 
    -3a|\beta_2|a
    \\
    &  \phantom{ = \;}
    +b|\gamma_2|b+b|\alpha_2|b+2c|\alpha_2|c+2ab|\alpha\gamma|1+bc|\alpha\gamma|1)y.
\end{align*}
If $n=4$, the chain map $g_{\bullet}$ satisfies
\begin{align*}
    & g_0(\omega_1 x|\epsilon^!|y)=0, \\
    & g_1(\omega_1 x|\alpha|y)=-2x( b|\alpha|c+c|\alpha|b+a|\beta|b+b|\gamma|b+b|\beta|a )y, \\
    & g_1(\omega_1 x|\beta|y)= 2x(1|\alpha|bc+a|\beta|c+ab|\gamma|1)y, \\
    & g_1(\omega_1 x|\gamma|y)=-2x( a|\gamma|b+b|\gamma|a+b|\alpha|b+b|\beta|c+c|\beta|b )y. 
\end{align*}
The map $\gamma_n|1$ can be extended to the chain map $g_{\bullet}:P^b_{\bullet}\to P^b_{\bullet}$ satisfying 
\begin{align*}
	& g_0(x|\gamma_{n}|y)=x|\epsilon^!|y,\\
	& g_0(x|\alpha_{n}|y)=g_0(x|\beta_{n}|y)=g_0(x|\alpha_{n-1}\beta|y)=g_0(x|\alpha_{n-1}\gamma|y)=g_0(x|\alpha_{n-2}\beta_2|y)=0,\\
	& g_1(x|\gamma_{n+1}|y)=x|\gamma|y,\quad 
	g_1(x|\alpha_{n+1}|y)=g_1(x|\beta_{n+1}|y)=g_1(x|\alpha_{n}\gamma|y)=0,\\
	& g_1(x|\alpha_{n}\beta|y)=x|\beta|y,\quad 
	g_1(x|\alpha_{n-1}\beta_2|y)=x|\alpha|y,\\
	& g_2(x|\gamma_{n+2}|y)=x|\gamma_2|y,\quad 
	g_2(x|\alpha_{n+2}|y)=g_2(x|\beta_{n+2}|y)=0, \quad 
	g_2(x|\alpha_{n+1}\beta|y)=x|\alpha\beta|y,\\
	& g_2(x|\alpha_{n+1}\gamma|y)=x|\alpha\gamma|y,\quad 
	g_2(x|\alpha_{n}\beta_2|y)=x|(\alpha_2+\beta_2)|y.
\end{align*}
Moreover, if $n=2$, the chain map $g_{\bullet}$ satisfies
\begin{align*}
   g_2(\omega_1 x|\epsilon^!|y) 
   & = -x( 1|\alpha\beta|ac+2|\alpha\beta|ba+2a|\alpha\beta|b+b|\alpha\gamma|a+c|\alpha\gamma|b-a|\beta_2|a+2a|\gamma_2|a+3b|\gamma_2|b 
    \\
    &  \phantom{ = \;}
   +ba|\alpha\beta|1+ab|\alpha\gamma|1  )y.
\end{align*}
If $n=4$, the chain map $g_{\bullet}$ satisfies
\begin{align*}
    & g_0(\omega_1 x|\epsilon^!|y)=0, \\
    & g_1(\omega_1 x|\alpha|y)=2x(a|\beta|b+b|\gamma|b+b|\beta|a)y, \quad 
    g_1(\omega_1 x|\beta|y)= 2x ( b|\alpha|a+a|\gamma|a+a|\alpha|b  )y, \\
    & g_1(\omega_1 x|\gamma|y)= -2x(1|\alpha|ba+a|\beta|a+ab|\alpha|1  )y.
\end{align*}
The map $(\alpha_{n-1}\beta+\alpha_{n-1}\gamma)|1$ can be extended to the chain map $g_{\bullet}:P^b_{\bullet}\to P^b_{\bullet}$ satisfying 
\begin{align*}
	& g_0(x|\alpha_{n}|y)=g_0(x|\beta_{n}|y)=g_0(x|\gamma_{n}|y)=g_0(x|\alpha_{n-2}\beta_2|y)=0,\\
	& g_0(x|\alpha_{n-1}\beta|y)=g_0(x|\alpha_{n-1}\gamma|y)=x|\epsilon^!|y,\\
	& g_1(x|\alpha_{n+1}|y)=g_1(x|\beta_{n+1}|y)=g_1(x|\gamma_{n+1}|y)=0,\quad 
	g_1(x|\alpha_{n}\beta|y)=x|(\alpha+\gamma)|y,\\
	& g_1(x|\alpha_{n}\gamma|y)=x|(\alpha+\beta)|y,\quad 
	g_1(x|\alpha_{n-1}\beta_2|y)=x|(\beta+\gamma)|y,\\
	& g_2(x|\alpha_{n+2}|y)=g_2(x|\beta_{n+2}|y)=g_2(x|\gamma_{n+2}|y)=0, \quad 
	g_2(x|\alpha_{n+1}\beta|y)=x|(\alpha\gamma+\alpha_2+\beta_2+\gamma_2)|y,\\
	& g_2(x|\alpha_{n+1}\gamma|y)=x|(\alpha\beta+\alpha_2+\beta_2+\gamma_2)|y,\quad 
	g_2(x|\alpha_{n}\beta_2|y)=x|(\alpha\beta+\alpha\gamma)|y.
\end{align*}
Moreover, if $n=2$, the chain map $g_{\bullet}$ satisfies
\begin{align*}
    g_2(\omega_1x|\epsilon^!|y)
    & =x[1|\alpha_2|(ba+ac)-1|\alpha_2|bc+1|\beta_2|(ab+bc)-1|\beta_2|ac-1|\gamma_2|ab-1|\gamma_2|ba-b|\alpha_2|c
    \\
    & \phantom{ = \;}
    -c|\alpha_2|b-a|\beta_2|c-c|\beta_2|a-a|\gamma_2|b-b|\gamma_2|a
    +(ba+ac)|\alpha_2|1-bc|\alpha_2|1-ac|\beta_2|1
    \\
    & \phantom{ = \;}
    +(ab+bc)|\beta_2|1-ab|\gamma_2|1-ba|\gamma_2|1]y.
\end{align*}
If $n=4$, the chain map $g_{\bullet}$ satisfies
\begin{align*}
   & g_0(\omega_1x|\epsilon^!|y)=0, \\
   & g_1(\omega_1 x|\alpha|y)=g_1(\omega_1 x|\beta|y)=g_1( \omega_1 x|\gamma|y)=0.
\end{align*}
Now let $n\geqslant 4$ be even, then the map $\alpha_{n-2}\beta_2|1$ can be extended to the chain map $g_{\bullet}:P^b_{\bullet}\to P^b_{\bullet}$ satisfying 
\begin{align*}
	& g_0(x|\alpha_{n-2}\beta_2|y)=x|\epsilon^!|y, \\
	& g_0(x|\alpha_n|y)=g_0(x|\beta_n|y)=g_0(x|\gamma_n|y)=g_0(x|\alpha_{n-1}\beta|y)=g_0(x|\alpha_{n-1}\gamma|y)=0, \\
	& g_1(x|\alpha_{n+1}|y)=g_1(x|\beta_{n+1}|y)=g_1(x|\gamma_{n+1}|y)=0, \quad 
	g_1(x|\alpha_{n}\beta|y)=x|\beta|y,\\
	& g_1(x|\alpha_n\gamma|y)=x|\gamma|y,\quad 
	g_1(x|\alpha_{n-1}\beta_2|y)=x|\alpha|y.
\end{align*}
Moreover, if $n=4$, the chain map $g_{\bullet}$ satisfies 
\begin{align*}
    g_0(\omega_1 x|\epsilon^!|y)=0.
\end{align*}
Let $n\in\NN$ be odd, then the map $\alpha_n|a+\beta_n|b+\gamma_n|c$ can be extended to the chain map $g_{\bullet}:P^b_{\bullet}\to P^b_{\bullet}$ satisfying
\begin{align*}
    & g_0(x|\alpha_n|y)=x|\epsilon^!|ay,
    \quad 
    g_0(x|\beta_n|y)=x|\epsilon^!|by,
    \quad 
    g_0(x|\gamma_n|y)=x|\epsilon^!|cy,
    \\
    & g_0(x|\alpha_{n-1}\beta|y)=g_0(x|\alpha_{n-1}\gamma|y)=g_0(x|\alpha_{n-2}\beta_2|y)=0,
    \\
    & g_1(x|\alpha_{n+1}|y)=-x|\alpha|ay,
    \quad 
    g_1(x|\beta_{n+1}|y)=-x|\beta|by,
    \quad 
    g_1(x|\gamma_{n+1}|y)=-x|\gamma|cy, 
    \\
    & g_1(x|\alpha_n\beta|y)=-x|\alpha|by-x|\beta|cy-x|\gamma|ay,
    \quad 
    g_1(x|\alpha_n\gamma|y)=-x|\alpha|cy-x|\beta|ay-x|\gamma|by,\\
    & g_1(x|\alpha_{n-1}\beta_2|y)=0, \\
   &  g_2(x|\alpha_{n+2}|y)=x|\alpha_2|ay,
    \quad
    g_2(x|\beta_{n+2}|y)=x|\beta_2|by,
    \quad 
    g_2(x|\gamma_{n+2}|y)=x|\gamma_2|cy,
    \\
    & g_2(x|\alpha_{n+1}\beta)|y)=x|\alpha\beta|cy+x|\alpha\gamma|ay+x|(\alpha_2+\gamma_2)|by,
    \\
    & g_2(x|\alpha_{n+1}\gamma|y)=x|\alpha\beta|ay+x|\alpha\gamma|by+x|(\alpha_2+\beta_2)|cy,
    \\
    & g_2(x|\alpha_n\beta_2|y)=x|\alpha\beta|by+x|\alpha\gamma|cy+x|(\beta_2+\gamma_2)|ay.
\end{align*}
Moreover, if $n=3$, the chain map $g_{\bullet}$ satisfies
\begin{align*}
    g_1(\omega_1 x|\epsilon^!|y)& =2x[ 1|\alpha|bac+1|\beta|abc-1|\gamma|aba+c|\alpha|(ba+ac)-a|\beta|ac-b|\gamma|ba-(ba+ac)|\gamma|a 
    \\
    & \phantom{ = \;}
    +ac|\alpha|b+ba|\beta|c  ]y,
   \\
   g_2(\omega_1 x|\alpha|y) & =
   2x[ -2|\alpha_2|bac+a|\alpha\gamma|ab-c|\alpha_2|bc+c|\beta_2|ab-b|\gamma_2|ba-bc|\alpha\beta|c+ab|\alpha\gamma|a
    \\
   & \phantom{ = \;}
   -ac|\alpha_2|c
   -2ba|\beta_2|c+ab|\gamma_2|b-ba|\gamma_2|c-abc|\alpha\gamma|1+bac|\alpha_2|1 ]y, 
   \\
   g_2(\omega_1 x|\beta|y) & = 
   2x[ -2|\beta_2|abc+b|\alpha\gamma|bc+a|\beta_2|(ab+bc)+a|\gamma_2|bc+c|\alpha_2|(ba+ac)
   \\
   & \phantom{ = \;}
   +(ab+bc)|\alpha\beta|a 
   +bc|\alpha\gamma|b-ba|\beta_2|a+2(ba+ac)|\gamma_2|a+bc|\alpha_2|c+(ba+ac)|\alpha_2|a
   \\
   & \phantom{ = \;}
   +aba|\alpha\gamma|1
   +abc|\beta_2|1 ]y, 
   \\
   g_2(\omega_1 x|\gamma|y) & = 
   2x[ 2|\gamma_2|aba-c|\alpha\gamma|(ab+bc)-b|\gamma_2|ab-b|\alpha_2|(ab+bc)-a|\beta_2|ac-ab|\alpha\beta|b
   \\
    & \phantom{ = \;}
   -(ab+bc)|\alpha\gamma|c+(ba+ac)|\gamma_2|b-2ac|\alpha_2|b-(ab+bc)|\beta_2|a-ac|\beta_2|b-bac|\alpha\gamma|1 
   \\
    & \phantom{ = \;}
   -aba|\gamma_2|1 ]y.
\end{align*}
\end{fact}

\begin{prop}\label{finitely generated}
The set 
\[ 
\mathscr{S}  = \bigg(\underset{\text{\begin{tiny}$\begin{matrix} m\in \llbracket 0, 4 \rrbracket, \\ n\in \llbracket 0, 3 \rrbracket \end{matrix}$\end{tiny}}}{\bigcup} \tilde{\mathfrak{H}}^n_m 
	\cup \big\{ \omega^*_1\epsilon^!|1 \big\}\bigg) \setminus \big\{ \epsilon^!|1 \big\} 
\] 
is a generating set of the $\Bbbk$-algebra $\operatorname{HH}^{\bullet}(A)$. 
Hence, $\operatorname{HH}^{\bullet}(A)$ is a finitely generated $\Bbbk$-algebra.
\end{prop}

\begin{proof}
We will prove the proposition by induction on $n$.
Let $n\geqslant 4$. 
Assume that $\tilde{\mathfrak{H}}^{n'}_m$ for $m\in \llbracket 0, 4 \rrbracket$ and $n'\in \llbracket 0, n-1 \rrbracket$	is generated by the elements of $\mathscr{S}$. 
We check that $\tilde{\mathfrak{H}}^n_m$ for $m\in \llbracket 0, 4 \rrbracket$ is generated by the elements of $\mathscr{S}$. 
First, we suppose that $n$ is even. 
Note that $\tilde{\mathfrak{H}}^n_0 = \{ \alpha_n|1, \beta_n|1, \gamma_n|1, (\alpha_{n-1}\beta+\alpha_{n-1}\gamma)|1, \alpha_{n-2}\beta_2|1 \}$.
By Fact \ref{fact chain map}, we have 
\begin{align*}
& \alpha_2|1\cup \alpha_{n-2}|1\in \alpha_n|1+\omega^*_1 H^{n-4}_2,\quad 
\beta_2|1\cup \beta_{n-2}|1\in \beta_n|1+\omega^*_1 H^{n-4}_2, \\
& \gamma_2|1\cup \gamma_{n-2}|1\in \gamma_n|1+\omega^*_1 H^{n-4}_2,\quad 
\gamma_2|1\cup \alpha_{n-2}|1\in \alpha_{n-2}\beta_2|1+\omega^*_1 H^{n-4}_2, \\
& (\alpha\beta+\alpha\gamma)|1\cup \alpha_{n-2}|1\in (\alpha_{n-1}\beta+\alpha_{n-1}\gamma)|1+\omega^*_1 H^{n-4}_2.	
\end{align*}
Hence, the elements in $\tilde{\mathfrak{H}}^n_0$ are generated by the elements in $\mathscr{S}$. 
Note that $\tilde{\mathfrak{H}}^n_1=\tilde{\mathfrak{H}}^n_2=\tilde{\mathfrak{H}}^n_3=\emptyset$. 
Finally, we notice that $\tilde{\mathfrak{H}}^n_4 = \{ \alpha_n|abac, \beta_n|abac, \gamma_n|abac, \alpha_{n-1}\beta|abac, \alpha_{n-2}\beta_2|abac \}$. 
Since
\begin{align*}
&\epsilon^!|abac\cup \alpha_n|1=\alpha_n|abac, \quad 
\epsilon^!|abac\cup \beta_n|1=\beta_n|abac, \quad 
\epsilon^!|abac\cup \gamma_n|1=\gamma_n|abac, \\
&\epsilon^!|abac\cup (\alpha_{n-1}\beta+\alpha_{n-1}\gamma)|1=2\alpha_{n-1}\beta|abac,\quad 
\epsilon^!|abac\cup \alpha_{n-2}\beta_2|1=\alpha_{n-2}\beta_2|abac,
\end{align*}
the elements in $\tilde{\mathfrak{H}}^n_4$ are also generated by the elements of $\mathscr{S}$. 

Next, we suppose that $n$ is odd.
Note first that $\tilde{\mathfrak{H}}^n_0=\tilde{\mathfrak{H}}^n_2=\tilde{\mathfrak{H}}^n_4=\emptyset$, and 
\begin{align*}
	\tilde{\mathfrak{H}}^n_1=\big\{
		& \alpha_n|a+\beta_n|b+\gamma_n|c,(\beta_n-\alpha_{n-1}\beta)|b,(\gamma_n-\alpha_{n-1}\gamma)|c,(\alpha_n-\alpha_{n-2}\beta_2)|a,\\
	& \alpha_{n-1}\beta|c+\alpha_{n-1}\gamma|a+\alpha_{n-2}\beta_2|b 
	\big\}.
\end{align*}
By Fact \ref{fact chain map}, we see that 
\begin{align*}
&(\alpha|a+\beta|b+\gamma|c)	\cup \alpha_{n-1}|1\in \alpha_n|a+\alpha_{n-1}\beta|b+\alpha_{n-1}\gamma|c+\omega^*_1\tilde{H}^{n-4}_3,\\
&(\alpha|a+\beta|b+\gamma|c)\cup \beta_{n-1}|1\in \beta_n|b+\alpha_{n-1}\gamma|c+\alpha_{n-2}\beta_2|a+\omega^*_1\tilde{H}^{n-4}_3,\\
&(\alpha|a+\beta|b+\gamma|c)\cup \gamma_{n-1}|1\in \gamma_n|c+\alpha_{n-1}\beta|b+\alpha_{n-2}\beta_2|a+\omega^*_1 \tilde{H}^{n-4}_3
,\\
&(\alpha|a+\beta|b+\gamma|c)\cup (\alpha_{n-2}\beta+\alpha_{n-2}\gamma)|1\in 2(\alpha_{n-1}\beta|c+\alpha_{n-1}\gamma|a+\alpha_{n-2}\beta_2|b)+\omega^*_1\tilde{H}^{n-4}_3,\\
&(\alpha|a+\beta|b+\gamma|c)\cup \alpha_{n-3}\beta_2|1\in \alpha_{n-1}\beta|b+\alpha_{n-1}\gamma|c+\alpha_{n-2}\beta_2|a+\omega^*_1\tilde{H}^{n-4}_3.
\end{align*}
It is easy to see that the five elements $\alpha_n|a+\alpha_{n-1}\beta|b+\alpha_{n-1}\gamma|c$,
$\beta_n|b+\alpha_{n-1}\gamma|c+\alpha_{n-2}\beta_2|a $,
$\gamma_n|c+\alpha_{n-1}\beta|b+\alpha_{n-2}\beta_2|a $,
$2(\alpha_{n-1}\beta|c+\alpha_{n-1}\gamma|a+\alpha_{n-2}\beta_2|b) $ 
and $\alpha_{n-1}\beta|b+\alpha_{n-1}\gamma|c+\alpha_{n-2}\beta_2|a $ 
are linear combinations of elements of $\tilde{\mathfrak{H}}^n_1$. 
Moreover, they form a basis of $\tilde{H}^n_1$. 
The elements in $\tilde{H}^n_1$, so \textit{a fortiori} $\tilde{\mathfrak{H}}^n_1$, are thus generated by the elements of $\mathscr{S}$. 
Note finally that 
\[ \tilde{\mathfrak{H}}^n_3 = \{ \alpha_n|bac, \beta_n|abc, \gamma_n|aba, \alpha_{n-1}\beta|abc, \alpha_n|(aba-abc) \}.
\] 
Since 
\begin{align*}
&\alpha|bac\cup \alpha_{n-1}|1=\alpha_n|bac,\quad 
\beta|abc\cup \beta_{n-1}|1=\beta_n|abc,\quad 
\gamma|aba\cup \gamma_{n-1}|1=\gamma_n|aba,\\
&\beta|abc\cup \alpha_{n-1}|1=\alpha_{n-1}\beta|abc,\quad 
\alpha|(aba-abc)\cup \alpha_{n-1}|1=\alpha_n|(aba-abc),	
\end{align*}
the elements in $\tilde{\mathfrak{H}}^n_3$ are generated by the elements of $\mathscr{S}$.
\end{proof}

\begin{prop} 
\label{proposition:generators-cohomology}
The set of $14$ elements given by 
\begin{equation}     
\label{eq:minimal-generating-set-cohomology}
\begin{split}
\mathcalboondox{S} &= \{ \epsilon^!|(ab+ba), \epsilon^!|(ab+bc-ac), \epsilon^!|abac, \alpha|a+\beta|b+\gamma|c, \alpha|bac, \beta|abc, \gamma|aba,\\
& \phantom{ = \{ ,}
\alpha|(aba-abc), 
 \alpha_2|1, \beta_2|1, \gamma_2|1, (\alpha\beta+\alpha\gamma)|1, \alpha_3|a+\beta_3|b+\gamma_3|c, \omega^*_1 \epsilon^!|1 \} \subseteq \operatorname{HH}^{\bullet}(A). 
\end{split}
\end{equation}
is a minimal generating set of the $\Bbbk$-algebra $\operatorname{HH}^{\bullet}(A)$.
\end{prop}

\begin{proof}
By Proposition \ref{finitely generated}, the $33$ element set 
\begin{align*}
	\mathscr{S}
	  =\big\{ 
	  & \epsilon^!|(ab+ba), \epsilon^!|(ab+bc+ac), \epsilon^!|abac, 
	  \alpha|bac,\beta|abc,\gamma|aba,\alpha|(aba-abc),(\alpha+\beta)|aba, \\
	  &  \alpha|aba+\beta|bac, 
	  \alpha|a+\beta|b+\gamma|c,
	  \alpha_2|abac,\beta_2|abac,\gamma_2|abac, \alpha\beta|abac, 
	  \alpha_2|(ab+ba), \\
	  & \beta_2|(ab+ba), \alpha_2|1, \beta_2|1, \gamma_2|1, (\alpha\beta+\alpha\gamma)|1, 
	  \alpha_3|bac,\beta_3|abc,\gamma_3|aba,\alpha_2\beta|abc,\alpha_3|(aba-abc), \\
	  & (\alpha_3+\beta_3)|aba,  \alpha_3|aba+\beta_3|bac, 
	  \alpha_3|a+\beta_3|b+\gamma_3|c, (\beta_3-\alpha_2\beta)|b, (\gamma_3-\alpha_2\gamma)|c, \\ & (\alpha_3-\alpha\beta_2)|a, \alpha_2\beta|c+\alpha_2\gamma|a+\alpha\beta_2|b, \omega^*_1 \epsilon^!|1  
	  	\big\}
\end{align*} 
is a generating set of $\operatorname{HH}^{\bullet}(A)$. 
By Fact \ref{fact chain map} and the computation of coboundaries in Subsubsections \ref{subsubsection:cob1} and \ref{subsubsection:cob2}, we get 
\begin{align*}
&\alpha_2|(ab+ba)=\epsilon^!|(ab+ba)\cup \alpha_2|1, \quad 
\beta_2|(ab+ba)=\epsilon^!|(ab+ba)\cup \beta_2|1, \\
& \alpha_2|abac=\epsilon^!|abac\cup \alpha_2|1, \quad 
\beta_2|abac=\epsilon^!|abac\cup \beta_2|1, \quad 
\gamma_2|abac=\epsilon^!|abac\cup \gamma_2|1, \\
& \alpha\beta|abac=(1/2)\epsilon^!|abac\cup (\alpha\beta+\alpha\gamma)|1, \quad 
\alpha_3|bac=\alpha|bac\cup \alpha_2|1, \quad
\beta_3|abc=\beta|abc\cup\beta_2|1, \\
& \gamma_3|aba=\gamma|aba\cup \gamma_2|1,\quad 
\alpha_2\beta|abc= \alpha|bac\cup \beta_2|1,\\ 
& (\alpha_3-\alpha\beta_2)|a=(1/2)[ (\alpha_3|a+\beta_3|b+\gamma_3|c)+ (\alpha|a+\beta|b+\gamma|c) \cup (\alpha_2|1-\beta_2|1-\gamma_2|1)],
\stepcounter{equation}\tag{\theequation}\label{17}
\\
& (\beta_3-\alpha_2\beta)|b=(1/2)[ (\alpha_3|a+\beta_3|b+\gamma_3|c)+ (\alpha|a+\beta|b+\gamma|c)\cup (\beta_2|1-\alpha_2|1-\gamma_2|1)],\\
& (\gamma_3-\alpha_2\gamma)|c=(1/2)[ (\alpha_3|a+\beta_3|b+\gamma_3|c)+ (\alpha|a+\beta|b+\gamma|c) \cup (\gamma_2|1-\alpha_2|1-\beta_2|1)],\\
& \alpha_2\beta|c+\alpha_2\gamma|a+\alpha\beta_2|b=(1/2) (\alpha|a+\beta|b+\gamma|c) \cup (\alpha\beta+\alpha\gamma)|1,\\
& \alpha_3|(aba-abc)= \alpha|(aba-abc)\cup \alpha_2|1, \quad 
(\alpha_3+\beta_3)|aba= \gamma|aba \cup (\alpha\beta+\alpha\gamma)|1, \\
& \alpha_3|aba+\beta_3|bac= \alpha|(aba-abc)\cup (\beta_2|1+\alpha_2|1)-2 \gamma|aba\cup (\alpha\beta+\alpha\gamma)|1, 
\\
& (\alpha+\beta)|aba=(1/2)\epsilon^!|(ab+bc-ac)\cup (\alpha|a+\beta|b+\gamma|c)+\alpha|(aba-abc),\\
& \alpha|aba+\beta|bac=(1/2)[\epsilon^!|(ab+ba)\cup (\alpha|a+\beta|b+\gamma|c)-\epsilon^!|(ab+bc-ac)\cup (\alpha|a+\beta|b+\gamma|c)].
\end{align*}
Hence, the set $\mathcalboondox{S}$ obtained from $\mathscr{S}$ by removing the nineteen elements in \eqref{17} is still a generating set. 
Similarly, it is easy to check that
\begin{equation}\label{0cup0}
	\begin{split}
		\epsilon^!|(ab+ba)\cup \epsilon^!|(ab+ba)
		&=\epsilon^!|(ab+bc-ac)\cup\epsilon^!|(ab+bc-ac)\\
		&=\epsilon^!|(ab+ba)\cup\epsilon^!|(ab+bc-ac)=0.
	\end{split}
\end{equation}
By Remark \ref{n1n2}, Fact \ref{fact chain map} and \eqref{0cup0}, it is easy to check that any one of the fourteen elements of $\mathcalboondox{S}$ can't be generated by the other thirteen elements, so the generating set $\mathcalboondox{S}$ is minimal.
\end{proof}

Let us number the elements of the set $\mathcalboondox{S}$ given in \eqref{eq:minimal-generating-set-cohomology} by $X_1=\epsilon^!|(ab+ba)$, $X_2=\epsilon^!|(ab+bc-ac)$, $X_3=\epsilon^!|abac$, $X_4=\alpha|bac$, $X_5=\beta|abc$, $X_6=\gamma|aba$, $X_7=\alpha|(aba-abc)$, $X_8=\alpha|a+\beta|b+\gamma|c$, $X_9=\alpha_2|1$, $X_{10}=\beta_2|1$, $X_{11}=\gamma_2|1$, $X_{12}=(\alpha\beta+\alpha\gamma)|1$, $X_{13}=\alpha_3|a+\beta_3|b+\gamma_3|c$ and $X_{14}=\omega^*_1\epsilon^!|1$.
We define the well-ordered set $\{x_i,i\in \llbracket 1, 14 \rrbracket \}$ with $x_i\succ x_j$ for all $i>j$.
Let $F$ be the noncommutative associative free $\Bbbk$-algebra generated by $x_i$ for $i\in \llbracket 1, 14 \rrbracket $, with length-lexicographic order. 
We endow the algebra $F$ with the unique grading over $\ZZ^2$ given by setting the bidegree of $x_i$ to be the same as that of $X_i$ for $i\in\llbracket 1,14\rrbracket$.
Let $\mathcalboondox{R}_1 \subseteq F$ be the set consisting of the following $97$ homogeneous elements 
\allowdisplaybreaks
\begin{align*}
& x_1x_2-x_2x_1, x_1x_3-x_3x_1, x_1x_4-x_4x_1, x_1x_5-x_5x_1, x_1x_6-x_6x_1, x_1x_7-x_7x_1, x_1x_8-x_8x_1,\\
& x_1x_9-x_9x_1, x_1x_{10}-x_{10}x_1, x_1x_{11}-x_{11}x_1, x_1x_{12}-x_{12}x_1, x_1x_{13}-x_{13}x_1, x_1x_{14}-x_{14}x_1,\\
& x_2x_3-x_3x_2, x_2x_4-x_4x_2, x_2x_5-x_5x_2, x_2x_6-x_6x_2, x_2x_7-x_7x_2, x_2x_8-x_8x_2, x_2x_9-x_9x_2,\\
& x_2x_{10}-x_{10}x_2, x_2x_{11}-x_{11}x_2, x_2x_{12}-x_{12}x_2, x_2x_{13}-x_{13}x_2, x_2x_{14}-x_{14}x_2, x_3x_4-x_4x_3,\\
& x_3x_5-x_5x_3, x_3x_6-x_6x_3, x_3x_7-x_7x_3, x_3x_8-x_8x_3, x_3x_9-x_9x_3, x_3x_{10}-x_{10}x_3, x_3x_{11}-x_{11}x_3,\\
& x_3x_{12}-x_{12}x_3, x_3x_{13}-x_{13}x_3, x_3x_{14}-x_{14}x_3, x_4x_5+x_5x_4, x_4x_6+x_6x_4, x_4x_7+x_7x_4,\\
& x_4x_8+x_8x_4, x_4x_9-x_9x_4, x_4x_{10}-x_{10}x_4, x_4x_{11}-x_{11}x_4, x_4x_{12}-x_{12}x_4, x_4x_{13}+x_{13}x_4,\\
& x_4x_{14}-x_{14}x_4, x_5x_6+x_6x_5, x_5x_7+x_7x_5, x_5x_8+x_8x_5, x_5x_9-x_9x_5, x_5x_{10}-x_{10}x_5,
\stepcounter{equation}\tag{\theequation}\label{eq:relations-cohomology-1}
\\
& x_5x_{11}-x_{11}x_5, x_5x_{12}-x_{12}x_5,  x_5x_{13}+x_{13}x_5, x_5x_{14}-x_{14}x_5, x_6x_7+x_7x_6, x_6x_8+x_8x_6, \\
& x_6x_9-x_9x_6, x_6x_{10}-x_{10}x_6, x_6x_{11}-x_{11}x_6, x_6x_{12}-x_{12}x_6, x_6x_{13}+x_{13}x_6, x_6x_{14}-x_{14}x_6, \\
& x_7x_8+x_8x_7, x_7x_9-x_9x_7, x_7x_{10}-x_{10}x_7, x_7x_{11}-x_{11}x_7, x_7x_{12}-x_{12}x_7, x_7x_{13}+x_{13}x_7,\\
& x_7x_{14}-x_{14}x_7, x_8x_9-x_9x_8, x_8x_{10}-x_{10}x_8, x_8x_{11}-x_{11}x_8, x_8x_{12}-x_{12}x_8, x_8x_{13}+x_{13}x_8, \\
& x_8x_{14}-x_{14}x_8, x_9x_{10}-x_{10}x_9,  x_9x_{11}-x_{11}x_9, x_9x_{12}-x_{12}x_9, x_9x_{13}-x_{13}x_9, x_9x_{14}-x_{14}x_9,\\
& x_{10}x_{11}-x_{11}x_{10}, x_{10}x_{12}-x_{12}x_{10}, x_{10}x_{13}-x_{13}x_{10}, x_{10}x_{14}-x_{14}x_{10}, x_{11}x_{12}-x_{12}x_{11},
\\
& x_{11}x_{13}-x_{13}x_{11}, x_{11}x_{14}-x_{14}x_{11},  x_{12}x_{13}-x_{13}x_{12}, x_{12}x_{14}-x_{14}x_{12}, x_{13}x_{14}-x_{14}x_{13}, x_4^2, x_5^2, x_6^2, \\
& x_7^2, x_8^2, x_{13}^2.
\end{align*}

\begin{rk}
\label{remark:quotient-free-graded-algebra}
Note that the quotient of the free algebra $F$ 
generated by $x_i$ for $i\in \llbracket 1, 14 \rrbracket $ modulo the (homogeneous) ideal generated by the previous set $\mathcalboondox{R}_{1}$ is precisely the free graded-commutative (for the homological grading) algebra 
$C$ generated by the same generators $x_i$ for $i\in \llbracket 1, 14 \rrbracket $. 
\end{rk}

Let $\mathcalboondox{R}_2 \subseteq F$ be the set consisting of the following $63$ homogeneous elements 
\begin{align*}
& 
x_1^2, x_1x_2, x_1x_3, x_2^2, x_2x_3, x_3^2,  x_1x_4, x_1x_5, x_1x_6, x_1x_7, x_2x_4, x_2x_5, x_2x_6, x_2x_7, x_3x_4, x_3x_5, 
x_3x_6, 
x_3x_7,
\\
&
x_3x_8, x_4x_5, x_4x_6, x_4x_7,  x_5x_6,  x_5x_7, x_6x_7, x_1x_{11}-2x_1x_9-2x_1x_{10}, x_1x_{12}-x_1x_9-x_1x_{10}, 
\\
& 
x_2x_9+x_1x_9, x_2x_{10}-2x_1x_{10}, x_2x_{11}-x_1x_9-x_1x_{10}, x_2x_{12}-x_1x_{10}, x_3x_9+x_8x_4, x_3x_{10}+x_8x_5,
\\
& x_3x_{11}-x_8x_6, x_3x_{12}-x_8x_7, x_9x_5+x_9x_6, x_9x_5-x_{10}x_4, 
 x_9x_5+x_{10}x_6,  x_9x_5-x_{11}x_4,
 \\ 
& x_9x_5-x_{11}x_5,
 x_{12}x_4-(1/3)x_9x_7+(4/3)x_{10}x_7, x_{12}x_5+(1/3)x_9x_7-x_{12}x_6+(5/3)x_{10}x_7,
 \\ 
& x_{10}x_7-x_{11}x_7, 
x_{12}x_7+2x_9x_5-(1/3)x_9x_7-(2/3)x_{10}x_7, x_9x_{10}-x_9x_{11}, x_9x_{10}-x_{10}x_{11},
\\
& x_9x_{12}-x_{12}x_{12}+2x_9x_{10}-3x_{14}x_1+3x_{14}x_2, x_{10}x_{12}-x_{12}x_{12}+2x_9x_{10}-3x_{14}x_2,
\stepcounter{equation}\tag{\theequation}\label{eq:relations-cohomology-2}
\\
& x_{11}x_{12}-x_{12}x_{12}+2x_9x_{10}+3x_{14}x_1, x_1x_{13}-4x_{12}x_6+4x_{10}x_7, \\
& x_2x_{13}+(4/3)x_9x_7-4x_{12}x_6+(8/3)x_{10}x_7, x_3x_{13}, x_8x_{13}-6x_{14}x_3, x_{13}x_4+x_9x_9x_3,\\
& x_{13}x_5+x_{10}x_{10}x_3, x_{13}x_6-x_{11}x_{11}x_3, x_{13}x_7-x_{12}x_{12}x_3+2x_9x_{10}x_3, x_1x_8x_{12}-2x_{12}x_6+2x_{10}x_7, \\
& x_2x_8x_{12}+(2/3)x_9x_7-2x_{12}x_6+(4/3)x_{10}x_7, x_9x_{13}-x_9x_9x_8+6x_{14}x_4, \\
& x_{10}x_{13}-x_{10}x_{10}x_8+6x_{14}x_5, x_{11}x_{13}-x_{11}x_{11}x_8-6x_{14}x_6, \\
& x_{12}x_{13}-x_{11}x_{12}x_8-6x_{14}x_7-3x_{14}x_2x_8.
\end{align*} 
By abuse of notation, we will also
identify $\mathcalboondox{R}_{2}$ 
with its image under the canonical projection $F \rightarrow F/(\mathcalboondox{R}_{1}) = C$. 

The following theorem is the main result of this article.

\begin{thm}
	\label{thm:HHcohomologymain}
Let $I$ be the two-sided ideal of $F$ generated by the set $\mathcalboondox{R} = \mathcalboondox{R}_1\cup \mathcalboondox{R}_2$ of $160$ homogeneous elements and let $D=F/I$. 
Define the morphism $\varphi:F\to \operatorname{HH}^{\bullet}(A)$ of bigraded $\Bbbk$-algebras by setting $\varphi(x_i)= X_i$ for $i\in \llbracket 1,14 \rrbracket $. 
It is easy to check that $\varphi$ is surjective and $I\subseteq \operatorname{Ker}(\varphi)$,
so $\varphi$ induces the surjective morphism $\bar{\varphi}:D\to  \operatorname{HH}^{\bullet}(A)$. 
Moreover, $\bar{\varphi}$ is an isomorphism, \textit{i.e.} $\operatorname{Ker}(\varphi) = I$.
\end{thm}

Before presenting the proof of the previous theorem, let us provide some auxiliary results.
We refer the reader to \cite{Ufnarovskij} (see also \cite{Varadarajan}) for the theory of Gröbner bases, as well as the usual terminology we will follow.
Using GAP (see \cite{cohenknopper}), one shows that a Gröbner basis $G$ of $I$ is given by the following $184$ elements
	\allowdisplaybreaks
\begin{align*}
& x_1^2,
x_1x_2,
x_1x_3,
x_1x_4,
x_1x_5,
x_1x_6,
x_1x_7,
x_1x_{11} - 2x_1x_{10} - 2x_1x_9,
x_1x_{12} - x_1x_{10} - x_1x_9,\\
& x_2x_1 - x_1x_2,
x_2^2,
x_2x_3,
x_2x_4,
x_2x_5,
x_2x_6,
x_2x_7,
x_2x_9 + x_1x_9,
x_2x_{10} - 2x_1x_{10},\\
& x_2x_{11} - x_1x_{10} - x_1x_9,
x_2x_{12} - x_1x_{10},
x_3x_1 - x_1x_3,
x_3x_2 - x_2x_3,
x_3^2,
x_3x_4,
x_3x_5,
x_3x_6,
x_3x_7,
x_3x_8,\\
& x_3x_{13},
x_4x_1 - x_1x_4,
x_4x_2 - x_2x_4,
x_4x_3 - x_3x_4,
x_4^2,
x_4x_5,
x_4x_6,
x_4x_7,
x_4x_8 - x_3x_9,
x_4x_{11} - x_4x_{10},\\
& x_5x_1 - x_1x_5,
x_5x_2 - x_2x_5,
x_5x_3 - x_3x_5,
x_5x_4 + x_4x_5,
x_5^2,
x_5x_6,
x_5x_7,
x_5x_8 - x_3x_{10},
x_5x_9 - x_4x_{10},\\
& x_5x_{11} - x_4x_{10},
x_5x_{12} - x_4x_{12} - (1/2)x_2x_{13} + (1/4)x_1x_{13},
x_6x_1 - x_1x_6,
x_6x_2 - x_2x_6,\\
& x_6x_3 - x_3x_6,
x_6x_4 + x_4x_6,
x_6x_5 + x_5x_6,
x_6^2,
x_6x_7,
x_6x_8 + x_3x_{11},
x_6x_9 + x_5x_9,
x_6x_{10} + x_4x_{10},\\
& x_6x_{12} + (1/2)x_5x_{12} + (1/2)x_4x_{12} - (3/8)x_1x_{13},
x_7x_1 - x_1x_7,
x_7x_2 - x_2x_7,
x_7x_3 - x_3x_7,\\
& x_7x_4 + x_4x_7,
x_7x_5 + x_5x_7,
x_7x_6 + x_6x_7,
x_7^2,
x_7x_8 + x_3x_{12},
x_7x_9 - 4x_6x_{12} - 3x_4x_{12} + x_1x_{13},\\
& x_7x_{10} - (1/4)x_7x_9 + (3/4)x_4x_{12},
x_7x_{11} - x_7x_{10},\\
& x_7x_{12} + x_4x_{12} + 2x_4x_{10} + (1/2)x_2x_{13} - (1/2)x_1x_{13},
x_8x_1 - x_1x_8,
x_8x_2 - x_2x_8,
x_8x_3 - x_3x_8,\\
& x_8x_4 + x_4x_8,
x_8x_5 + x_5x_8,
x_8x_6 + x_6x_8,
x_8x_7 + x_7x_8,
x_8^2,
x_8x_{13} - 6x_3x_{14},
x_9x_1 - x_1x_9,\\
& x_9x_2 - x_2x_9,
x_9x_3 - x_3x_9,
x_9x_4 - x_4x_9,
x_9x_5 - x_5x_9,
x_9x_6 - x_6x_9,
x_9x_7 - x_7x_9,
x_9x_8 - x_8x_9,\\
& x_9x_{11} - x_9x_{10},
x_{10}x_1 - x_1x_{10},
x_{10}x_2 - x_2x_{10},
x_{10}x_3 - x_3x_{10},
x_{10}x_4 - x_9x_5,
x_{10}x_5 - x_5x_{10},\\
& x_{10}x_6 + x_9x_5,
x_{10}x_7 - x_7x_{10},
x_{10}x_8 - x_8x_{10},
x_{10}x_9 - x_9x_{10},
x_{10}x_{11} - x_9x_{10},\\
& x_{10}x_{12} - x_9x_{12} - 6x_2x_{14} + 3x_1x_{14},
x_{11}x_1 - x_1x_{11},
x_{11}x_2 - x_2x_{11},
x_{11}x_3 - x_3x_{11},
x_{11}x_4 - x_9x_5,\\
& x_{11}x_5 - x_9x_5,
x_{11}x_6 - x_6x_{11},
x_{11}x_7 - x_{10}x_7,
x_{11}x_8 - x_8x_{11},
x_{11}x_9 - x_9x_{11},
x_{11}x_{10} - x_{10}x_{11},\\
& x_{11}x_{12} - x_{10}x_{12} + 3x_2x_{14} + 3x_1x_{14},
x_{12}x_1 - x_1x_{12},
x_{12}x_2 - x_2x_{12},
x_{12}x_3 - x_3x_{12},
\stepcounter{equation}\tag{\theequation}\label{eq:184}\\
& x_{12}x_4 + (4/3)x_{10}x_7 - (1/3)x_9x_7,
x_{12}x_5 - x_5x_{12},
x_{12}x_6 - x_6x_{12},
x_{12}x_7 - x_7x_{12},
x_{12}x_8 - x_8x_{12},\\
& x_{12}x_9 - x_9x_{12},
x_{12}x_{10} - x_{10}x_{12},
x_{12}x_{11} - x_{11}x_{12},
x_{12}^2 - x_{11}x_{12} - 2x_9x_{10} - 3x_1x_{14},\\
& x_{13}x_1 - x_1x_{13},
x_{13}x_2 - x_2x_{13},
x_{13}x_3 - x_3x_{13},
x_{13}x_4 + x_4x_{13},
x_{13}x_5 + x_5x_{13},
x_{13}x_6 + x_6x_{13},\\
& x_{13}x_7 + x_7x_{13},
x_{13}x_8 + x_8x_{13},
x_{13}x_9 - x_9x_{13},
x_{13}x_{10} - x_{10}x_{13},
x_{13}x_{11} - x_{11}x_{13},
x_{13}x_{12} - x_{12}x_{13},\\
& x_{13}^2,
x_{14}x_1 - x_1x_{14},
x_{14}x_2 - x_2x_{14},
x_{14}x_3 - x_3x_{14},
x_{14}x_4 - x_4x_{14},
x_{14}x_5 - x_5x_{14},
x_{14}x_6 - x_6x_{14},\\
& x_{14}x_7 - x_7x_{14},
x_{14}x_8 - x_8x_{14},
x_{14}x_9 - x_9x_{14},
x_{14}x_{10} - x_{10}x_{14},
x_{14}x_{11} - x_{11}x_{14},
x_{14}x_{12} - x_{12}x_{14},\\
& x_{14}x_{13} - x_{13}x_{14},
x_1x_8x_{12} - 2x_{12}x_6 + 2x_{10}x_7,
x_2x_8x_{12} - 2x_{12}x_6 + (4/3)x_{10}x_7 + (2/3)x_9x_7,\\
& x_3x_9^2 - x_4x_{13},
x_3x_9x_{12} + x_7x_{13},
x_3x_{10}^2 - x_5x_{13},
x_3x_{11}^2 + x_6x_{13},
x_8x_9^2 - x_9x_{13} - 6x_4x_{14},\\
& x_8x_9x_{12} + 6x_2x_8x_{14} - 6x_1x_8x_{14} - x_{12}x_{13} + 6x_7x_{14},
x_8x_{10}^2 - x_{10}x_{13} - 6x_5x_{14},\\
& x_8x_{11}^2 - x_{11}x_{13} + 6x_6x_{14},
x_1x_8x_9 + (1/2)x_2x_{13} - (1/2)x_1x_{13},
x_1x_8x_{10} + x_1x_8x_9 - (1/2)x_1x_{13},\\
& x_1x_8x_{11} - 2x_1x_8x_{10} - 2x_1x_8x_9,
x_1x_9^2,
x_1x_9x_{10} + 2x_1x_9^2,
x_1x_9x_{12} + x_1x_9^2,
x_1x_9x_{13},
x_1x_{10}^2,\\
& x_1x_{10}x_{13}, 
x_2x_8x_9 + x_1x_8x_9,
x_2x_8x_{10} + 2x_1x_8x_9 - x_1x_{13},
x_2x_8x_{11} - (1/2)x_1x_{13},
x_3x_9x_{13},\\
& x_3x_{10}x_{13},
x_3x_{11}x_{13},
x_3x_{12}x_{13},
x_4x_{10}^2 - x_4x_9x_{10},
x_8x_9x_{13} - 6x_3x_9x_{14},
x_8x_{10}x_{13} - 6x_3x_{10}x_{14},\\
& x_8x_{11}x_{13} - 6x_3x_{11}x_{14},
x_8x_{12}x_{13} - 6x_3x_{12}x_{14},
x_9x_{10}^2 - x_9^2x_{10},
x_3x_9x_{10}x_{13},\\
& x_8x_9x_{10}x_{13} - 6x_3x_9x_{10}x_{14}.
\end{align*}

We will now compute the standard words with respect to $G$, \textit{i.e.} the monomials on the letters $x_{i}$, $i \in \llbracket 1, 14\rrbracket$, that are not divisible by the leading terms of the elements of $G$. 
This is a direct but tedious computation. 
We recall that the set of standard words forms a $\Bbbk$-basis $S$ of $D$. 
Obviously, $1\in S$ and $x_i\in S$ for $i\in \llbracket 1, 14 \rrbracket $.
The elements in $S$ generated by $2$ elements are given 
by the following $46$ elements 
\allowdisplaybreaks
\begin{align*}
		&x_1x_8,x_1x_9,x_1x_{10},x_1x_{13},x_1x_{14},\\
&x_2x_8,x_2x_{13},x_2x_{14},\\
&x_3x_{9},x_3x_{10},x_3x_{11},x_3x_{12},x_3x_{14},\\
&x_4x_{9},x_4x_{10},
x_4x_{12},x_4x_{13},x_4x_{14},\\
&x_5x_{10},x_5x_{13},x_5x_{14},\\
&x_6x_{11},x_6x_{13},x_6x_{14},\\
&x_7x_{13},x_7x_{14},\\
&x_8x_{9},x_8x_{10},x_8x_{11},x_8x_{12},x_8x_{14},
\stepcounter{equation}\tag{\theequation}\label{basis2}\\
&x_9^2,x_9x_{10},x_9x_{12},x_9x_{13},x_9x_{14},\\
&x_{10}^2,x_{10}x_{13},x_{10}x_{14},\\
&x_{11}^2,x_{11}x_{13},x_{11}x_{14},\\
&x_{12}x_{13},x_{12}x_{14},\\
&x_{13}x_{14},\\
&x_{14}^2.
\end{align*}
Analogously, the elements in $S$ generated by $3$ elements are given 
by the following $68$ elements 
\begin{small}
\allowdisplaybreaks
\begin{align*}
&x_1x_8x_{14},x_1x_9x_{14},x_1x_{10}x_{14},x_1x_{13}x_{14},x_1x_{14}^2,\\
&x_2x_8x_{14},x_2x_{13}x_{14},x_2x_{14}^2,\\
&x_3x_9x_{10},x_3x_9x_{14},x_3x_{10}x_{14},x_3x_{11}x_{14},x_3x_{12}x_{14},x_3x_{14}^2,\\
&x_4x_9^2,x_4x_9x_{10},x_4x_9x_{12},x_4x_9x_{13},x_4x_9x_{14},x_4x_{10}x_{13},x_4x_{10}x_{14},x_4x_{12}x_{13},x_4x_{12}x_{14},x_4x_{13}x_{14},x_4x_{14}^2,\\
&x_5x_{10}^2,x_5x_{10}x_{13},x_5x_{10}x_{14},x_5x_{13}x_{14},x_5x_{14}^2,\\
&x_6x_{11}^2,x_6x_{11}x_{13},x_6x_{11}x_{14},x_6x_{13}x_{14},x_6x_{14}^2,\\
&x_7x_{13}x_{14},x_7x_{14}^2,
\stepcounter{equation}\tag{\theequation}\label{basis3}\\
&x_8x_9x_{10},x_8x_9x_{14},x_8x_{10}x_{14},x_8x_{11}x_{14},x_8x_{12}x_{14},x_8x_{14}^2,\\
&x_9^3,x_9^2x_{10},x_9^2x_{12},x_9^2x_{13},x_9^2x_{14},x_9x_{10}x_{13},x_9x_{10}x_{14},x_9x_{12}x_{13},x_{9}x_{12}x_{14},x_9x_{13}x_{14},x_9x_{14}^2,\\
&x_{10}^3,x_{10}^2x_{13},x_{10}^2x_{14},x_{10}x_{13}x_{14},x_{10}x_{14}^2,\\
&x_{11}^3,x_{11}^2x_{13},x_{11}^2x_{14},x_{11}x_{13}x_{14},x_{11}x_{14}^2,\\
&x_{12}x_{13}x_{14},x_{12}x_{14}^2,\\
&x_{13}x_{14}^2,\\
&x_{14}^3.
	\end{align*}
\end{small}
Finally, the elements in $S$ generated by $4$ elements are 
given by the following $89$ elements 
\begin{small}
\allowdisplaybreaks
\begin{align*}
	& x_1x_8x_{14}^2, x_1x_9x_{14}^2,x_1x_{10}x_{14}^2,x_1x_{13}x_{14}^2,x_1x_{14}^3,\\
	& x_2x_8x_{14}^2,x_2x_{13}x_{14}^2,x_2x_{14}^3,\\
	& x_3x_9x_{10}x_{14},x_3x_9x_{14}^2,x_3x_{10}x_{14}^2,x_3x_{11}x_{14}^2,x_3x_{12}x_{14}^2,x_3x_{14}^3,\\
	& x_4x_9^3, x_4x_9^2x_{10}, x_4x_9^2x_{12}, x_4x_9^2x_{13}, x_4x_9^2x_{14}, x_4x_9x_{10}x_{13}, x_4x_9x_{10}x_{14}, x_4x_9x_{12}x_{13}, x_4x_9x_{12}x_{14},x_4x_9x_{13}x_{14},\\
	& x_4x_9x_{14}^2,x_4x_{10}x_{13}x_{14},x_4x_{10}x_{14}^2,x_4x_{12}x_{13}x_{14},x_4x_{12}x_{14}^2,x_4x_{13}x_{14}^2,x_4x_{14}^3,\\
	& x_5x_{10}^3, x_5x_{10}^2x_{13},x_5x_{10}^2x_{14}, x_5x_{10}x_{13}x_{14}, x_5x_{10}x_{14}^2, x_5x_{13}x_{14}^2, x_5x_{14}^3, \\
	& x_6x_{11}^3, x_6x_{11}^2x_{13}, x_6x_{11}^2x_{14}, x_6x_{11}x_{13}x_{14}, x_6x_{11}x_{14}^2, x_6x_{13}x_{14}^2, x_6x_{14}^3, \\
	& x_7x_{13}x_{14}^2, x_7x_{14}^3, 
	\stepcounter{equation}\tag{\theequation}\label{basis4}\\
	& x_8x_9x_{10}x_{14}, x_8x_9x_{14}^2, x_8x_{10}x_{14}^2, x_8x_{11}x_{14}^2, x_8x_{12}x_{14}^2, x_8x_{14}^3, \\
	& x_9^4, x_9^3x_{10}, x_9^3x_{12}, x_9^3x_{12}, x_9^3x_{13}, x_9^3x_{14}, x_9^2x_{10}x_{13}, x_9^2x_{10}x_{14}, x_9^2x_{12}x_{13}, x_9^2x_{12}x_{14}, x_9^2x_{13}x_{14},x_9^2x_{14}^2,\\
	& x_9x_{10}x_{13}x_{14},x_9x_{10}x_{14}^2,x_9x_{12}x_{13}x_{14},x_{9}x_{12}x_{14}^2,x_9x_{13}x_{14}^2,x_9x_{14}^3,\\
	& x_{10}^4,x_{10}^3x_{13}, x_{10}^3x_{14}, x_{10}^2x_{13}x_{14}, x_{10}^2x_{14}^2, x_{10}x_{13}x_{14}^2, x_{10}x_{14}^3, \\
	& x_{11}^4, x_{11}^3x_{13}, x_{11}^3x_{14}, x_{11}^2x_{13}x_{14}, x_{11}^2x_{14}^2, x_{11}x_{13}x_{14}^2, x_{11}x_{14}^3, \\
	& x_{12}x_{13}x_{14}^2, x_{12}x_{14}^3, \\
	& x_{13}x_{14}^3, \\
	& x_{14}^4.
\end{align*}
\end{small}
\begin{lem}\label{equivalent}
Let $x$ be a word generated by $r$ elements, where $r\geqslant 5$. Then the following statements are equivalent: 
\begin{enumerate}[label=(\arabic*)]
    \item\label{item:1} $x\in S$; 
    \item\label{item:2} $y\in S$ for any subword $y\subsetneq x$;
    \item\label{item:3} $y\in S$ for any subword $y\subsetneq x$ generated by $r-1$ elements; 
    \item\label{item:4} $y\in S$ for any subword $y\subsetneq x$ generated by $4$ elements.
\end{enumerate}
\end{lem}
\begin{proof}
The implications \ref{item:1} $\Rightarrow$ \ref{item:2} $\Rightarrow$ \ref{item:3} $\Rightarrow$ \ref{item:4} follow immediately from the  definition of standard word. 
On the other hand, to prove the implication 
\ref{item:4} $\Rightarrow$ \ref{item:1}, it suffices to note that the leading word of any element in $G$ is generated by at most than $4$ elements.  
\end{proof}

\begin{lem}
\label{xx14}
Let $x$ be a word. 
Then $x \in S$ if and only if $xx_{14}\in S$. 
\end{lem}

\begin{proof}
To prove the direct implication, suppose that $x \in S$ is generated by $r$ elements. 
If $r\in \llbracket 0,3\rrbracket$, we get that $xx_{14}\in S$ directly from \eqref{basis2} - \eqref{basis4}. 
If $r\geqslant 4$, write $x=yz$, where $y,z$ are words and $z$ is generated by $3$ elements. Obviously, $z\in S$ and $zx_{14}\in S$. 
By Lemma \ref{equivalent}, $xx_{14}=yzx_{14}\in S$.
Finally, note that the converse follows from the definition of standard word. 
\end{proof}

The following result is a direct consequence of the previous lemma. 

\begin{cor}
Let $\tilde{S}$ be the elements in $S$ generated by $x_i$ for $i\in \llbracket 1,13 \rrbracket$, 
and $\tilde{D}$ the subspace of $D$ generated by the elements in $\tilde{S}$. Then 
\begin{equation}\label{Dtilde}
\begin{split}
	D\cong \tilde{D}\otimes \Bbbk[x_{14}]
\end{split}
\end{equation}
 as graded $\Bbbk$-vector spaces.
\end{cor}

We are now ready to prove Theorem \ref{thm:HHcohomologymain}. 

\begin{proof}[Proof of Theorem \ref{thm:HHcohomologymain}]
It is easy to check that the morphism $\varphi$ vanishes on the set $\mathcalboondox{R}_1$ since the algebra $\operatorname{HH}^{\bullet}(A)$ is graded commutative,
and it also vanishes on the set $\mathcalboondox{R}_2$, as the reader can check using Remark \ref{n1n2}, Fact \ref{fact chain map}, \eqref{0cup0} and the coboundaries in Subsubsections \ref{subsubsection:cob1} and \ref{subsubsection:cob2}.
Hence, $I\subseteq \operatorname{Ker}(\varphi)$. 
By Proposition \ref{linear structure of cohomology}, we have 
\begin{equation}\label{Dtilde01234}
	\begin{split}
		\operatorname{HH}^{\bullet}(A)\cong \big( \underset{ \text{\begin{tiny}$\begin{matrix} m\in \llbracket 0, 4 \rrbracket, \\ n\in \NN_0 \end{matrix}$\end{tiny}}}{\bigoplus}\tilde{H}^n_m \big) \otimes \Bbbk[\omega^*_1\epsilon^!|1]
	\end{split}
	\end{equation}
as graded $\Bbbk$-vector spaces.
Let $S^n_m$ be the elements in $\tilde{S}$ with cohomological degree $n\in\NN_0$ and internal degree $m-n$, where $m\in\ZZ$.
To prove that $\bar{\varphi}$ is an isomorphism, it is sufficient to prove that 
the cardinality of $S^n_m$ is as same as the dimension of $\tilde{H}^n_m$. 

Take $x\in S^n_m$. Since the words $x_ix_j$ for $i>j$ are leading terms of elements in $G$, we may assume that $x$ is of the form $x_1^{r_1}x_2^{r_2}\cdots x_{13}^{r_{13}}$, where $r_i\in\NN_0$ for $i\in \llbracket 1,13\rrbracket $.
Since $x_i^2$ for $i= 1,2,3,4,5,6,7,8,12,13$ are leading terms of elements in $G$, we assume $r_i \in \llbracket 0,1\rrbracket$ for $i= 1,2,3,4,5,6,7,8,12,13$. 
By degree reasons, we have 
\begin{equation}\label{r14n}
	\begin{split}
	r_4+r_5+r_6+r_7+r_8+2r_9+2r_{10}+2r_{11}+2r_{12}+3r_{13}=n,
	\end{split}
\end{equation}
and 
\begin{equation}\label{r14nm}
	\begin{split}
	2r_1+2r_2+4r_3+2r_4+2r_5+2r_6+2r_7-2r_9-2r_{10}-2r_{11}-2r_{12}-2r_{13}=m-n.
	\end{split}
\end{equation}
Adding the two equations together, we get
\begin{equation}\label{r14m}
	\begin{split}
	&2r_1+2r_2+4r_3+3r_4+3r_5+3r_6+3r_7+r_8+r_{13}=m.
	\end{split}
\end{equation}
Note that the previous identity tells us that $m \in\NN_0$. 

First, we will first prove that $m\leqslant 4$. 
If $r_1=1$, then $r_i=0$ for $i=2,3,4,5,6,7,11,12$, since $x_{1} x_{i}$ is the leading term of an element of the Gröbner basis $G$ for $i=2,3,4,5,6,7,11,12$. 
The equation \eqref{r14m} then shows that $m=2+r_8+r_{13}\leqslant 4$. 
Assume for the rest of the paragraph that $r_1=0$. 
If $r_2=1$, then $r_i=0$ for $i=3,4,5,6,7,9,10,11,12$, since $x_{2} x_{i}$ is the leading term of an element of the Gröbner basis $G$ for $i=3,4,5,6,7,9,10,11,12$. 
The equation \eqref{r14m} thus shows that $m=2+r_8+r_{13}\leqslant 4$. 
Suppose for the rest of the paragraph that $r_2=0$. 
If $r_3=1$, then $r_i=0$ for $i=4,5,6,7,8,13$, 
since $x_{3} x_{i}$ is the leading term of an element of the Gröbner basis $G$ for $i=4,5,6,7,8,13$. 
The equation \eqref{r14m} hence shows that $m=4$. 
Assume for the rest of the paragraph that $r_3=0$. 
If $r_4=1$, then $r_i=0$ for $i=5,6,7,8,11$, 
since $x_{4} x_{i}$ is the leading term of an element of the Gröbner basis $G$ for $i=5,6,7,8,11$. 
Then, the equation \eqref{r14m} shows that 
$m=3+r_{13}\leqslant 4$. 
Suppose for the rest of the paragraph that $r_4=0$. 
If $r_5=1$, then $r_i=0$ for $i=6,7,8,9,11,12$, since $x_{5} x_{i}$ is the leading term of an element of the Gröbner basis $G$ for $i=6,7,8,9,11,12$. 
The equation \eqref{r14m} thus shows that $m=3+r_{13}\leqslant 4$.
Assume for the rest of the paragraph that $r_5=0$. 
If $r_6=1$, then $r_i=0$ for $i=7,8,9,10,12$, 
since $x_{6} x_{i}$ is the leading term of an element of the Gröbner basis $G$ for $i=7,8,9,10,12$. 
The equation \eqref{r14m} then shows that $m=3+r_{13}\leqslant 4$.
Suppose further that $r_6=0$. 
If $r_7=1$, then $r_i=0$ for $i=8,9,10,11,12$, since $x_{7} x_{i}$ is the leading term of an element of the Gröbner basis $G$ for $i=8,9,10,11,12$. 
The equation \eqref{r14m} shows that $m=3+r_{13}\leqslant 4$.
Finally, assume also that $r_7=0$. 
Then $m=r_8+r_{13}\leqslant 2\leqslant 4$.

Then we suppose $m=4$. 
The equation \eqref{r14m} then becomes 
\begin{equation}\label{r14m4}
	\begin{split}
	2r_1+2r_2+4r_3+3r_4+3r_5+3r_6+3r_7+r_8+r_{13}=4.
	\end{split}
\end{equation}
If $r_1=1$, then $r_i=0$ for $i=2,3,4,5,6,7,11,12$, and equation \eqref{r14m4} shows $r_8+r_{13}=2$, which gives $r_8=r_{13}=1$. 
This is impossible since $x_1x_8$ can only be followed by $x_{14}$ by \eqref{basis3}.
Assume for the rest of the paragraph that $r_1=0$. 
If $r_2=1$, then $r_i=0$ for $i=3,4,5,6,7,9,10,11,12$, 
and equation \eqref{r14m4} shows $r_8+r_{13}= 2$. 
In the same way as before, this case is also impossible. 
Suppose for the rest of the paragraph that $r_2=0$. If $r_3=1$, then $r_i=0$ for $i=4,5,6,7,8,13$. 
Then $x=x_3$, $x_{3}x_9$, $x_3x_{10}$, $x_3x_{11}$, $x_3x_{12}$ or $x_3x_9x_{10}$.
Suppose for the rest of the paragraph that $r_3=0$. If $r_4=1$, then $r_i=0$ for $i=5,6,7,8,11$, and equation \eqref{r14m4} shows $r_{13}=1$. 
By $\eqref{r14n}$, $n$ is even. 
Moreover, $x=x_4x_9^{r_9}x_{13}$, $x_4x_9^{r_9}x_{10}x_{13}$ or $x_4x_9^{r_9}x_{12}x_{13}$ for $r_9\in\NN_0$.
Suppose for the rest of the paragraph that $r_4=0$. If $r_5=1$, then $r_i=0$ for $i=6,7,8,9,11,12$, and equation \eqref{r14m4} shows $r_{13}=1$. 
Then $n$ is even by \eqref{r14n}, and $x=x_5x_{10}^{r_{10}}x_{13}$ for $r_{10}\in\NN_0$. 
Suppose for the rest of the paragraph that $r_5=0$. If $r_6=1$, then $r_i=0$ for $i=7,8,9,10,12$, and equation \eqref{r14m4} shows $r_{13}=1$. Then $n$ is even by \eqref{r14n}, and $x=x_6x_{11}^{r_{11}}x_{13}$ for $r_{11}\in\NN_0$. 
Suppose for the rest of the paragraph that $r_6=0$. If $r_7=1$, then $r_i=0$ for $i=8,9,10,11,12$, and equation \eqref{r14m4} shows $r_{13}=1$, which implies that $x=x_7x_{13}$. 
Finally, assume also that $r_7=0$. 
Then, equation \eqref{r14m4} shows $4=r_8+r_{13}\leqslant 2$, which is impossible. 
To sum up, we have 
\begin{align*}
S^0_4 & =\{x_3\}, \quad 
S^2_4=\{x_3x_9, x_3x_{10}, x_3x_{11}, x_3x_{12}\}, \quad 
S^4_4=\{x_4x_{13}, x_5x_{13}, x_6x_{13}, x_7x_{13}, x_3x_9x_{10}\},
\\
S^n_4 & =\{x_4x_9^{(n-4)/2}x_{13}, x_4x_9^{(n-6)/2}x_{10}x_{13}, x_4x_9^{(n-6)/2}x_{12}x_{13}, x_5x_{10}^{(n-4)/2}x_{13}, x_6x_{11}^{(n-4)/2}x_{13}\}  
\end{align*}
if $n\geqslant 6$ is even, and $S^n_4=\emptyset$ if $n$ is odd.

Suppose $m=3$. 
Then \eqref{r14m} becomes
\begin{equation}\label{r14m3}
	\begin{split}
	&2r_1+2r_2+4r_3+3r_4+3r_5+3r_6+3r_7+r_8+r_{13}=3.
	\end{split}
\end{equation}
If $r_1=1$, then $r_i=0$ for $i=2,3,4,5,6,7,11,12$, and equation \eqref{r14m3} shows $r_8+r_{13}=1$. 
Then $r_8+r_{13}$ is odd. 
We have thus either $r_8=1$ and $r_{13}=0$, or $r_8=0$ and $r_{13}=1$. 
Both cases imply that $n$ is odd by \eqref{r14n}. 
If $r_8=1$ and $r_{13}=0$, then $r_9=0$ and $r_{10}=0$ by \eqref{basis3}, so $x=x_1x_8$. 
If $r_8=0$ and $r_{13}=1$, then $x$ has the form $x_1x_{9}^{r_9}x_{10}^{r_{10}}x_{13}$. 
By \eqref{basis3}, $x_1x_9$ and $x_1x_{10}$ can only be followed by $x_{14}$, so $x=x_1x_{13}$. 
Now assume for the rest of the paragraph that $r_1=0$. 
If $r_2=1$, then $r_i=0$ for $i=3,4,5,6,7,9,10,11,12$, and equation \eqref{r14m3} shows $r_8+r_{13}=1$. 
We have either $r_8=1$ and $r_{13}=0$, or $r_8=0$ and $r_{13}=1$. 
Moreover, $n$ is odd. 
So, $x=x_2x_8$ or $x_2x_{13}$. 
Suppose for the rest of the paragraph that $r_2=0$. Then $r_3=0$ by \eqref{r14m3}. 
If $r_4=1$, then $r_i=0$ for $i=5,6,7,8,11$, and equation \eqref{r14m3} shows $r_{13}=0$. 
Hence, $n$ is odd by \eqref{r14n}, and $x$ has the form $x_4x_{9}^{r_9}x_{10}^{r_{10}}x_{12}^{r_{12}}$. If $r_9=0$, then $x$ can only be $x_4$, $x_4x_{10}$ or $x_4x_{12}$. 
If $r_9\neq 0$, then $x=x_4x_9^{r_9}$, $x_4x_9^{r_9}x_{10}$ or $x_4x_9^{r_9}x_{12}$. 
Suppose for the rest of the paragraph that $r_4=0$. If $r_5=1$, then $r_i=0$ for $i=6,7,8,9,11,12$. 
Then, equation \eqref{r14m3} shows $r_{13}=0$. 
Then $n$ is odd by \eqref{r14n} and $x=x_5x_{10}^{r_{10}}$.
Suppose for the rest of the paragraph that $r_5=0$. If $r_6=1$, then $r_i=0$ for $i=7,8,9,10,12$, and equation \eqref{r14m3} shows $r_{13}=0$. 
So, $n$ is odd by \eqref{r14n}, and $x=x_6x_{11}^{r_{11}}$.
Suppose for the rest of the paragraph that $r_6=0$. If $r_7=1$, then $r_i=0$ for $i=8,9,10,11,12$, and equation \eqref{r14m3} shows $r_{13}=0$. 
So, $x=x_7$. 
Suppose for the rest of the paragraph that $r_7=0$. If $r_8=1$, then $r_{13}=0$, and equation \eqref{r14m3} shows $1=3$, which is impossible.
Finally, assume also that $r_8=0$. Then equation \eqref{r14m3} shows $r_{13}=3$, which is impossible. To sum up, we have
\begin{align*}
S^1_3 & =\{ x_4, x_5, x_6, x_7, x_1x_8 , x_2x_8\}, \quad 
 S^3_3=\{x_1x_{13}, x_2x_{13}, x_4x_9, x_4x_{10}, x_4x_{12}, x_5x_{10},  x_6x_{11}  \}, 
 \\
 S^n_3 & =\{ x_4x_9^{(n-1)/2}, x_4x_9^{(n-3)/2}x_{10}, x_4x_9^{(n-3)/2}x_{12}, x_5x_{10}^{(n-1)/2}, x_6x_{11}^{(n-1)/2}   \}
\end{align*}
if $n\geqslant 5$ is odd, and $S^n_3=\emptyset$ if $n$ is even. 

Suppose $m=2$. 
Then \eqref{r14m} becomes  
\begin{equation}\label{r14m2}
	\begin{split}
	2r_1+2r_2+4r_3+3r_4+3r_5+3r_6+3r_7+r_8+r_{13}=2.
	\end{split}
\end{equation}
Then $r_i=0$ for $i=3,4,5,6,7$.
If $r_1=1$, then $r_i=0$ for $i=2,3,4,5,6,7,11,12$, and equation \eqref{r14m2} shows $r_8=r_{13}=0$.
Hence, $n$ is even by \eqref{r14n}, and $x=x_1$, $x_1x_9$ or $x_1x_{10}$. 
Assume for the rest of the paragraph that $r_1=0$. 
If $r_2=1$, then $r_i=0$ for $i=3,4,5,6,7,9,10,11,12$, and equation \eqref{r14m2} shows $r_8=r_{13}=0$, so $x=x_2$. 
Suppose for the rest of the paragraph that $r_2=0$. 
If $r_8=1$, \eqref{r14m2} shows $r_{13}=1$, which is impossible. 
Finally, if $r_8=0$, then $r_{13}=2$, which is also impossible. 
We thus have 
\[ S^0_2=\{ x_1,x_2\}  , 
\quad 
S^2_2=\{ x_1x_9,x_1x_{10} \} ,\] 
and $S^n_2=\emptyset$ if $n=1$ and $n\geqslant 3$. 

Suppose $m=1$. 
Then \eqref{r14m} becomes 
\begin{equation}\label{r14m1}
	\begin{split}
	2r_1+2r_2+4r_3+3r_4+3r_5+3r_6+3r_7+r_8+r_{13}=1.
	\end{split}
\end{equation}
If $r_8=1,r_{13}=0$, then $x=x_8$, $x_8x_9$, $x_8x_{10}$, $x_8x_{11}$, $x_8x_{12}$ or $x_8x_9x_{10}$.
If $r_8=0, r_{13}=1$, then $x=x_{13}$, $x_9^{r_9}x_{13}$, $x_9^{r_9}x_{10}x_{13}$, $x_9^{r_9}x_{12}x_{13}$, $x_{10}^{r_{10}}x_{13}$, $x_{11}^{r_{11}}x_{13}$ or $x_{12}x_{13}$. 
We then get  
\begin{align*}
S^1_1 & =\{x_8\}, \quad 
S^3_1=\{x_{13}, x_8x_9,x_8x_{10}, x_8x_{11}, x_8x_{12}\}, 
\\
S^5_1 & =\{ x_9x_{13}, x_{10}x_{13}, x_{11}x_{13}, x_{12}x_{13}, x_8x_9x_{10}\}, 
\\
S^n_1 & =\{ x_9^{(n-3)/2}x_{13}, x_9^{(n-5)/2}x_{10}x_{13}, x_9^{(n-5)/2}x_{12}x_{13}, x_{10}^{(n-3)/2}x_{13}, x_{11}^{(n-3)/2}x_{13} \} 
\end{align*}
if $n\geqslant 7$ is odd, and $S^n_1=\emptyset$ if $n$ is even.

Suppose $m=0$. Then \eqref{r14m} becomes 
\begin{equation}\label{r14m0}
	\begin{split}
	2r_1+2r_2+4r_3+3r_4+3r_5+3r_6+3r_7+r_8+r_{13}=0.
	\end{split}
\end{equation}
Then $r_i=0$ for $i=1,2,3,4,5,6,7,8,13$.
If $r_i=0$ for all $i\in \llbracket 1,14 \rrbracket$, then $x=1$. 
Otherwise, $x=x_{9}^{r_9}$, $x_{9}^{r_9}x_{10}$, $x_{9}^{r_9}x_{12}$, $x_{10}^{r_{10}}$, $x_{11}^{r_{11}}$ or $x_{12}$. 
We thus have
\begin{align*}
S^0_0=\{1\},   \quad 
S^2_0=\{ x_9, x_{10}, x_{11}, x_{12}  \} , \quad 
S^n_0=\{ x_9^{n/2}, x_{9}^{(n-2)/2}x_{10}, x_{9}^{(n-2)/2}x_{12}, x_{10}^{n/2}, x_{11}^{n/2}\}  \end{align*}
if $n\geqslant 4$ is even, and $S^n_0=\emptyset$ if $n$ is odd.

Finally, we leave to the reader the easy task to check that the cardinality of $S^n_m$ is as same as the dimension of $\tilde{H}^n_m$. 
\end{proof}

As a direct consequence of Remark \ref{remark:quotient-free-graded-algebra} and Theorem \ref{thm:HHcohomologymain} 
we get the following result. 
 
\begin{cor}
\label{cor:HHcohomologymain}
Recall that $C=F/(\mathcalboondox{R}_{1})$ is precisely the free graded-commutative (for the homological degree) algebra generated by the elements $x_i$ for $i\in \llbracket 1, 14 \rrbracket $, where $\mathcalboondox{R}_{1}$ is the set given in \eqref{eq:relations-cohomology-1}.
Let $D'=C/J$, where $J$ is the two-sided ideal of $C$ generated by the elements in $\mathcalboondox{R}_2$ given in \eqref{eq:relations-cohomology-2}.
Define the morphism $\varphi':C\to \operatorname{HH}^{\bullet}(A)$ of bigraded $\Bbbk$-algebras by setting $\varphi'(x_i)= X_i$ for $i\in \llbracket 1,14 \rrbracket $. 
It is easy to check that $\varphi'$ is surjective and $J\subseteq \operatorname{Ker}(\varphi')$,
so $\varphi'$ induces the surjective morphism $\bar{\varphi}':D'\to  \operatorname{HH}^{\bullet}(A)$. 
Moreover, $\bar{\varphi}'$ is an isomorphism, \textit{i.e.} $\operatorname{Ker}(\varphi') = J$.
\end{cor}

\bibliographystyle{model1-num-names}

\end{document}